\documentclass[12pt]{amsart}
\usepackage{amsmath}
\usepackage{bm}
\usepackage{amssymb}
\usepackage{amscd}
\usepackage{amsthm}
\usepackage{aliascnt}  
\usepackage{hyperref}
\usepackage{cleveref} 
\usepackage{fullpage,enumitem}
\usepackage{units}
\usepackage[all]{xy}
\usepackage{xcolor} 
\usepackage{mathrsfs} 
\usepackage{url}
\usepackage{tikz-cd}
\usepackage{mathtools}
\usepackage{stmaryrd} 

\usepackage[OT2,T1]{fontenc}
\DeclareSymbolFont{cyrletters}{OT2}{wncyr}{m}{n}
\DeclareMathSymbol{\Sha}{\mathalpha}{cyrletters}{"58}

\usepackage{import}
\usepackage{dsfont} 

\theoremstyle{definition}
\newtheorem{theorem}{Theorem}[section]
\newtheorem*{theorem*}{Theorem}
\numberwithin{theorem}{subsection} 

\newaliascnt{lemma}{theorem}
\newtheorem{lemma}[lemma]{Lemma}
\aliascntresetthe{lemma}

\newaliascnt{corollary}{theorem}
\newtheorem{corollary}[corollary]{Corollary}
\aliascntresetthe{corollary}

\newaliascnt{proposition}{theorem}
\newtheorem{proposition}[proposition]{Proposition}
\aliascntresetthe{proposition}

\newaliascnt{claim}{theorem}

\aliascntresetthe{claim}

\newaliascnt{goal}{theorem}

\newtheorem*{goal*}{Goal}
\aliascntresetthe{goal}

\newaliascnt{remark}{theorem}
\newtheorem{remark}[remark]{Remark}
\aliascntresetthe{remark}

\newaliascnt{definition}{theorem}
\newtheorem{definition}[definition]{Definition}
\aliascntresetthe{definition}

\newaliascnt{notation}{theorem}
\newtheorem{notation}[notation]{Notation}
\aliascntresetthe{notation}

\newaliascnt{convention}{theorem}
\newtheorem{convention}[convention]{Convention}
\aliascntresetthe{convention}

\newaliascnt{context}{theorem}
\newtheorem{context}[context]{Context}
\aliascntresetthe{context}

\newaliascnt{question}{theorem}

\aliascntresetthe{question}

\newaliascnt{problem}{theorem}

\newtheorem*{problem*}{Problem}
\aliascntresetthe{problem}

\newaliascnt{example}{theorem}

\aliascntresetthe{example}

\newaliascnt{conjecture}{theorem}
\newtheorem{conjecture}[conjecture]{Conjecture}
\aliascntresetthe{conjecture}

\newaliascnt{solution}{theorem}

\aliascntresetthe{solution}

\crefname{theorem}{theorem}{theorems}
\Crefname{theorem}{Theorem}{Theorems}

\crefname{lemma}{lemma}{lemmas}
\Crefname{lemma}{Lemma}{Lemmas}

\crefname{corollary}{corollary}{corollaries}
\Crefname{corollary}{Corollary}{Corollaries}

\crefname{proposition}{proposition}{propositions}
\Crefname{proposition}{Proposition}{Propositions}

\crefname{claim}{claim}{claims}
\Crefname{claim}{Claim}{Claims}

\crefname{goal}{goal}{goals}
\Crefname{goal}{Goal}{Goals}

\crefname{remark}{remark}{remarks}
\Crefname{remark}{Remark}{Remarks}

\crefname{definition}{definition}{definitions}
\Crefname{definition}{Definition}{Definitions}

\crefname{notation}{notation}{notations}
\Crefname{notation}{Notation}{Notations}

\crefname{convention}{convention}{conventions}
\Crefname{convention}{Convention}{Conventions}

\crefname{context}{context}{contexts}
\Crefname{context}{Context}{Contexts}

\crefname{question}{question}{questions}
\Crefname{question}{Question}{Questions}

\crefname{problem}{problem}{problems}
\Crefname{problem}{Problem}{Problems}

\crefname{example}{example}{examples}
\Crefname{example}{Example}{Examples}

\crefname{conjecture}{conjecture}{conjectures}
\Crefname{conjecture}{Conjecture}{Conjectures}

\newcommand\sub\subset
\newcommand\longto\longrightarrow

\DeclareMathOperator\Aut{Aut}

\DeclareMathOperator\Ext{Ext}

\DeclareMathOperator\Frob{Frob}

\DeclareMathOperator\op{op}

\newcommand\bbC{\mathbb{C}}

\newcommand\bbF{\mathbb{F}}

\newcommand\bbQ{\mathbb{Q}}

\newcommand\bbZ{\mathbb{Z}}

\newcommand{\Fq}{{\bbF_q}}

\newcommand{\Fqbar}{{\overline{\bbF}_q}}

\usepackage{soul}
\usepackage{xcolor}
\sethlcolor{lightgray}

\makeatletter 
\newcommand\CrefIfExists[1]{%
  \@ifundefined{r@#1}{%
  }{%
    \space(\Cref{#1})%
  }%
}
\makeatother

\makeatletter
\newcommand{\hyperrefIfExists}[2]{%
  \@ifundefined{r@#1}{%
    #2%
  }{%
    \hyperref[#1]{#2}%
  }%
}
\makeatother

\newcommand{\CrefAndHyperrefIfExist}[2]{
\hyperrefIfExists{#1}{#2}\CrefIfExists{#1}}

\makeatletter 
\newcommand\TextIfExists[2]{%
  \@ifundefined{r@#1}{%
  }{%
    #2 %
  }%
}
\makeatother

\makeatletter 
\newcommand\TextIfExistsElse[3]{%
  \@ifundefined{r@#1}{%
    #3 
  }{%
    #2 %
  }%
}
\makeatother

\graphicspath{ {images/} }
\usepackage{caption}
\usepackage{subcaption}
\usepackage{float}
\usepackage[margin=1in]{geometry}
\usepackage{fancyhdr}
\usepackage{setspace}
\usepackage{amsmath}
\usepackage{amssymb}
\usepackage{amsthm}
\usepackage{tikz-cd} 
\usepackage{tikz}
\usetikzlibrary{arrows,calc,matrix}

\usepackage[
    style=alphabetic, 
    backend=biber,    
    backref=true      
]{biblatex}
\newcommand{\HncGn}{{\operatorname{Hur}^c_{G,n}}}

\newcommand{\Fqt}{\mathbb{F}_q(t)}

\numberwithin{equation}{section}

\title{Cohen-Lenstra heuristics and vanishing of zeta functions for superelliptic curves over finite fields}
\author{Hyun Jong Kim}
\date{\today}

\begin{document}
\pagenumbering{roman}

\begin{abstract}
We prove a Cohen-Lenstra type result commenting on the distribution of class group structures amongst function fields of superelliptic curves over finite fields. We also prove a result commenting on the sparsity of such superelliptic curves whose zeta functions vanish at a fixed complex number. These results are proven via point-counting techniques on certain Hurwitz schemes generalizing techniques of Ellenberg-Venkatesh-Westerland and Ellenberg-Li-Shusterman. To obtain these point counts, We prove a unitary big monodromy theorem generalizing big monodromy theorems of Yu and Achter-Pries by appealing to an arithmeticity theorem of Venkataramana. 
\end{abstract}

\maketitle

\tableofcontents

\section{Introduction}
\pagenumbering{arabic} 

This work is based on the author's PhD Thesis.

Given a \CrefAndHyperrefIfExist{definition:global_field}{global field} $K$, write $\operatorname{Cl}(K)$ for its ideal class group. Given an abelian group $A$, write $A_\ell$ for the Sylow $\ell$-subgroup of $A$ for a prime number $\ell$, and write $A_o$ for the odd part of $A$, i.e. the subgroup of $A$ of elements of odd order. The original Cohen-Lenstra heuristics \cite{cohen_lenstra_hcgnf} are conjectures for the distribution of the odd parts of class groups of imaginary and real quadratic fields. One can state these conjectures as follows, cf \cite[Conjecture 5.10.1, Conjecture 5.10.2]{cohen_icnt}:

\begin{conjecture}
    For $X > 0$, let $S_X$ be the set of discriminants $D$ of imaginary quadratic number fields such that $|D| < X$. For any odd prime $\ell$ and any finite abelian $\ell$-group $A$,
    $$\lim_{X \to \infty} \frac{\#\{D \in S_X: \operatorname{Cl}(\mathbb{Q}(\sqrt{D}))_\ell \cong A \}}{\#S_X} = \frac{\prod_{i=1}^\infty (1-\ell^{-i})}{|\operatorname{Aut}(A)|}.$$
\end{conjecture}

\begin{conjecture}
    For $X > 0$, let $S_X$ be the set of discriminants $D$ of real quadratic number fields such that $|D| < X$. For any finite abelian group $A$ of odd order,
    $$\lim_{X \to \infty} \frac{\#\{D \in S_X: \operatorname{Cl}(\mathbb{Q}(\sqrt{D}))_o \cong A \}}{\#S_X} = \frac{1}{2|A| \cdot \left( \prod_{i=1}^\infty (1-2^{-i}) \right) \cdot \left( \prod_{k>2} \zeta(k) \right) \cdot |\operatorname{Aut}(A)|},$$
    where $\zeta(k)$ is the Riemann zeta function.
\end{conjecture}

\noindent There have been various conjectures and results about analogous distributions for both number fields and function fields since the inception of these conjectures. See \cite{wood_clhlc}, \cite{liu_wood_zureick-brown_pd}, and \cite{wang_wood_miclm} for some relevant works in the literature. One of the main results of this works is a generalization of Ellenberg, Venkatesh, and Westerland's \cite{ellenberg_venkatesh_westerland} result for function fields resembling Cohen and Lenstra's original conjecture for imaginary quadratic fields:


\begin{theorem}[{\cite[1.2 Theorem]{ellenberg_venkatesh_westerland}}] \label{theorem:EVW_main_theorem}
    Let $\ell > 2$ be prime and $A$ a finite abelian $\ell$-group. Write $\delta^+$ (resp. $\delta^-$) for the upper density (resp. lower density) of imaginary\footnote{i.e. ramified at $\infty$, cf. \Cref{definition:imaginary}} quadratic extensions of $\mathbb{F}_q(t)$ for which the $\ell$-part of the class group is isomorphic to $A$. Then $\delta^+(q)$ and $\delta^-(q)$ converge, as $q \to \infty$ with $q \not\equiv 1 \pmod{\ell}$ to $\frac{\prod_{i > 1}(1-\ell^{-i})}{|\operatorname{Aut}(A)|}$.
\end{theorem}

\noindent
The following generalizes the above result to Kummer extensions of $\Fqt$ of the form $y^d = f(t)$. 
\begin{theorem}[{\Cref{corollary:generalizing_EVW_main_theorem_to_cubic_covers}}]\label{theorem:generalizing_EVW_main_theorem_to_cubic_covers_intro}
Let $d \geq 3$ be an integer. Let $\ell \nmid d$ be a prime number. Let $\mathbb{Z}_\ell[\zeta_d] = \mathbb{Z}_\ell[\zeta_d]_{\zeta_d \neq 1} = \mathbb{Z}_\ell[X]/\left( \sum_{i=0}^{d-1} X^i \right)$. Let $\mathcal{L}_{d,\ell}$ be the set of isomorphism classes of $\mathbb{Z}_\ell[\zeta_d]$-modules with finite cardinality. For $n \geq 1$, let $\mathfrak{S}_n$ be the set of isomorphism classes of extensions $L$ of $\mathbb{F}_q(t)$ of the form $\mathbb{F}_q(t)[y]/(y^d-f(t))$ where $f(t) \in \mathbb{F}_q[t]$ is squarefree and of degree $n$.

For all but finitely many primes $\ell$ not dividing $2d$, for any $[A] \in \mathcal{L}_{d,\ell}$, write 
    \begin{align*}
        \delta^+(q) &= \limsup_{\substack{n \to \infty \\ \gcd(n,d) = 1}} \frac{|\{L \in \mathfrak{S}_n: (\operatorname{Cl}_L)_\ell \cong_{\mathbb{Z}_\ell[\zeta_d]} A \}|}{|\mathfrak{S}_n|} \\
        \delta^-(q) &= \liminf_{\substack{n \to \infty \\ \gcd(n,d) = 1}} \frac{|\{L \in \mathfrak{S}_n: (\operatorname{Cl}_L)_\ell \cong_{\mathbb{Z}_\ell[\zeta_d]} A \}|}{|\mathfrak{S}_n|}. 
    \end{align*}
Then $\delta^+(q)$ and $\delta^-(q)$ converge, as $q \to \infty$ with $\ell \nmid 2dq(q-1)$ and $q \equiv 1 \pmod{d}$, to $\mu_{d,\ell}([A]) = \frac{C_{d,\ell}}{|\operatorname{Aut}_{\mathbb{Z}_\ell[\zeta_d]}(A)|}$, where 
$$C_{d,\ell} = \left( \sum_{[M] \in \mathcal{L}_{d,\ell}} \frac{1}{|\operatorname{Aut}_{\mathbb{Z}_\ell[\zeta_d]}(M)|} \right)^{-1}.$$
\end{theorem}

In particular, our main results assume that $\Fq$ contains a primitive $d$th-root of unity or equivalently that $q \equiv 1 \pmod{d}$. Note that both \Cref{theorem:EVW_main_theorem} and \Cref{theorem:generalizing_EVW_main_theorem_to_cubic_covers_intro} concern limits which first let $\deg f(t) \to \infty$ and then let $q \to \infty$. In contrast, \cite{liu_wood_zureick-brown_pd} concerns limits of distributions for function fields which first let $q \to \infty$ and then let the branch locus degree approach infinity.

Our method for proving this generalization mimics the method that \cite{ellenberg_venkatesh_westerland} used to prove \Cref{theorem:EVW_main_theorem} above --- one asymptotically counts $\Fq$-points on certain Hurwitz schemes which we denote by $\mathbf{X}_n/\Fq$ in \Cref{section:generalizing_EVW_section_8}. Letting 
\begin{itemize}
    \item $d \geq 2$ be an integer,
    \item $\ell$ be a prime number relatively prime to $d$,
    \item $A$ be a module over $\mathbb{Z}_\ell[\zeta_d] = \mathbb{Z}_\ell[X]/(\sum_{i=0}^{d-1} X^i)$, 
    \item $G$ be the semidirect product $A \rtimes \mathbb{Z}/d\mathbb{Z}$ where $1 \in \mathbb{Z}/d\mathbb{Z}$ acts on $A$ by multiplication by $\zeta_d$, and
    \item $c$ is the conjugacy class of $G$ consisting of elements of the form $(a,\zeta_d)$ for $a \in A$,
\end{itemize}
the Hurwitz scheme $\mathbf{X}_n$ parameterizes tamely ramified $G$-covers $f:C \to \mathbb{P}^1$ with monodromy type (see \Cref{section:monodromy_type}) $c$ away from $\infty$ such that the quotient $C \to C/A$ is unramified above the points in $C/A$ above $\infty \in \mathbb{P}^1$. By the Grothendieck-Lefschetz fixed point formula, 
\begin{equation}
    |\mathbf{X}_n(\Fq)| = \sum_{j=1}^{2n} (-1)^j \operatorname{Tr}(\operatorname{Frob}_q | H_c^j(\overline{\mathbf{X}}_n; \mathbb{Q}_\lambda)).
\end{equation}
Ellenberg, Venkatesh, and Westerland's homological stability results \cite[Theorem 6.1]{ellenberg_venkatesh_westerland} asymptotically bound the $j < 2n$ terms, which are $o(q^n)$. We then show that the $j = 2n$ term is $q^n$ whenever $n$ is sufficiently large. In turn, we show this via \Cref{corollary:generalizing_EVW_Lemma 8.9_to_Z_dZ_covers}, which is a statement about Hermitian spaces, and via \Cref{theorem:big_monodromy_statement_in_characteristic_0} and \Cref{proposition:monodromy_for_char_p_is_bigger_than_for_char_0}, which give desirable ``big monodromy results'', i.e. that certain monodromy representations have sufficiently large image. 

The main difference between the case of $d = 2$ proved by \cite{ellenberg_venkatesh_westerland} and the case of more general $d$ is the action of $\zeta_d$ on $\operatorname{Jac}(C)$, where $C$ is the curve corresponding to the function field --- when $d = 2$, the action is simply negation and when $d > 2$, the action becomes more complicated. Moreover, the Weil pairing on $\operatorname{Jac}(C)$ yields a non-trivial Hermitian pairing when $d > 2$. Hence, \Cref{corollary:generalizing_EVW_Lemma 8.9_to_Z_dZ_covers} is a statement about Hermitian spaces unlike its predecessor \cite[Lemma 8.9]{ellenberg_venkatesh_westerland}, which is about symplectic spaces, and the monodromy representations of interest are unitary representations and not merely symplectic representations. 

After the author submitted his PhD dissertation, Landesman and Levy \cite{landesman_levy_hshsa} proved a stable homology result for Hurwitz spaces, which allowed them to deduce, among other applications, non-abelian Cohen-Lenstra moment asymptotics for global function fields over a fixed base field $\mathbb{F}_q$ (provided $q > C(H)$ for a target group $H$). We note that \Cref{theorem:EVW_main_theorem} and \Cref{theorem:generalizing_EVW_main_theorem_to_cubic_covers_intro} obtain the heuristics uniformly as $q \to \infty$.


Furthermore, Ellenberg, Li, and Shusterman \cite{ellenberg_li_shusterman} used the machinery in Ellenberg, Venkatesh, and Westerland's paper \cite{ellenberg_venkatesh_westerland} to prove a statement about the vanishing of zeta functions for hyperelliptic curves over $\Fq$. 
\begin{theorem}[{\cite[Theorem 1.2/3.2]{ellenberg_li_shusterman}}]
    Fix a prime $p$, and a complex number $s = \frac{1}{2} +it$. For $k \geq 1$, write $\mathcal{H}_g(\bbF_{p^k})$ for the family of genus $g$ hyperelliptic curves over $\bbF_{p^k}$ and write $Z_C(s)$ for the zeta function of a smooth, proper, geometrically integral curve over $\bbF_{p^k}$. As $k \to \infty$, we have
    $$\sup_g \frac{|\{C \in \mathcal{H}_g(\bbF_{p^k}): Z_C(s) = 0\}|}{|\mathcal{H}_g(\bbF_{p^k})} \ll p^{-k/276}$$
\end{theorem}
\noindent The following similarly generalizes this result for superelliptic curves corresponding to the above Kummer extensions of $\mathbb{P}^1$.
\begin{theorem}[\Cref{theorem:research_generalizing_ELS_main_theorem_to_trielliptic_curves}]\label{theorem:research_generalizing_ELS_main_theorem_to_trielliptic_curves_intro} 
Let $d \geq 3$. Let $q$ be a prime power such that $q \equiv 1 \pmod{d}$. Given $n \geq 1$, let $Q_{n,q}$ be the set of squarefree polynomials over $\Fq$ of degree $n$. Given $f \in \mathbb{F}_q[t]$, write $C_f$ for the smooth completion of the curve defined by $y^d = f(t)$, write $P_f \in \mathbb{Z}[T]$ for the reverse characteristic polynomial of geometric Frobenius acting on the Jacobian of $C_f$, and write $Z_{C_f}(T)$ for the zeta function of $C_f$. 

There exists a constant $C_{d} > 0$ depending only on $d$ such that, for any complex number $s \neq 0, 1$,
    $$\limsup_{\substack{n \to \infty \\ \gcd(d,n) = 1 \text{ or } d \mid n}} \frac{|\{f \in Q_{n,q} | Z_{C_f}(q^{-s}) = 0\}|}{|Q_{n,q}|} \ll q^{-C_{d}},$$
or equivalently
    $$\limsup_{\substack{n \to \infty \\ \gcd(d,n) = 1 \text{ or } d \mid n}} \frac{|\{f \in Q_{n,q} | L(s,\chi_f^j) = 0 \text{ for some } j \in \mathbb{Z}/d\mathbb{Z} \setminus \{0\}\}|}{|Q_{n,q}|} \ll q^{-C_{d}},$$
where $\chi_f^j$ is as in \Cref{definition:L_series_of_dirichlet_characters}, and the asymptotic bound is with respect to the limit taken over powers $q$ of a fixed prime power $q_0 \equiv 1 \pmod{d}$. In fact, $C_d = \frac{1}{(138d + 62) \cdot (d-1)}$ suffices.
\end{theorem}

Once again, that the monodromy representations of interest are unitary representations when $d > 2$ presents difficulties not present in the $d = 2$ case. More specifically, \Cref{lemma:ring_of_EVW_central_element_degree}, which generalizes \cite[Lemma 3.5]{ellenberg_venkatesh_westerland}, is more cumbersome to prove when $d > 2$ compared to when $d = 2$.

\Cref{section:conjugacy_structure_of_A_rtimes_Z_dZ} discusses aspects of the structure of groups of the form $A \rtimes \mathbb{Z}/d\mathbb{Z}$ where $A$ is a  finite $\ell$-group for a prime $\ell \nmid d$. \Cref{lemma:c_is_nonsplitting_conjugacy_class} and \Cref{lemma:semidirect_product_is_center_free_and_A_has_trivial_invariant_subgroup} show that the homological stability results of \cite{ellenberg_venkatesh_westerland} are applicable to the Hurwitz scheme $\mathbf{X}_n$. \Cref{section:monodromy_type} discusses monodromy types for covers of $\mathbb{P}^1$. \Cref{proposition:monodromy_type_for_kummer_curves} provides a convenient way to understand the monodromy types for curves given by $y^d = \prod_i (t-t_i)^{e_i}$ --- the monodromy type of $t = t_i$ can be regarded as $e_i \in \mathbb{Z}/d\mathbb{Z}$.  \Cref{section:hermitian_spaces} discusses the Hermitian space theory leading up to \Cref{corollary:generalizing_EVW_Lemma 8.9_to_Z_dZ_covers} needed in later parts of this work.

\Cref{section:burau_representations} discusses the unreduced and reduced Burau representations, which are representations of the Artin braid group $B_n$. The chapter presents Venkataramana's result \cite{venkataramana_ibrd} that the reduced Burau representations evaluated at roots of unity are arithmetic, i.e. their images are of finite index in the largest possible ``trivial'' codomain. \Cref{subsection:hurwitz_schemes} introduces notations for the aforementioned Hurwitz schemes of interest, \Cref{subsection:big_monodromy_hypothesis} proves relevant big monodromy results. There, Venkataramana's arithmeticity results establish base cases and Achter and Pries' clutching methods \cite{achter_pries_imht} demonstrate the inductive step of an induction argument. \Cref{subsection:counting_surjections_from_class_groups} puts together the details from the previous chapters and sections to prove \Cref{theorem:generalizing_EVW_Theorem_8.8_to_cubic_covers}, which asymptotically counts the $\Fq$-points of $\mathbf{X}_n$ as described above. \Cref{section:cohen_lenstra_for_imaginary_Z_3Z_extensions} then uses \Cref{theorem:generalizing_EVW_Theorem_8.8_to_cubic_covers} to finish proving \Cref{theorem:generalizing_EVW_main_theorem_to_cubic_covers_intro}. The details presented in this chapter are basically identical to those in \cite[8.1 to 8.4 Lemmas]{ellenberg_venkatesh_westerland} except for the fact that the abelian group surjections involved all must be equivariant for the $\zeta_d$-action as well. We nevertheless include these details for completeness. 

\Cref{section:generalizing_ELS_section_2} and \Cref{section:generalizing_ELS_section_3} generalize ideas in \cite[Section 2, Section 3]{ellenberg_venkatesh_westerland}. \Cref{section:generalizing_ELS_section_2} counts points on twists of $\mathbf{X}_n$ and \Cref{section:generalizing_ELS_section_3} bounds the number of superelliptic curves whose zeta functions vanish at a fixed complex number. Finally, \Cref{section:orbits_of_burau_representation} proves details, culminating in \Cref{lemma:ring_of_EVW_central_element_degree}, about the orbits of the Burau representations evaluated at roots of unity and reduced modulo $\ell$ ultimately needed to conclude \Cref{theorem:research_generalizing_ELS_main_theorem_to_trielliptic_curves}.

\subsection{Notations} \label{subsection:notations}

The following notations apply throughout the rest of this paper, unless otherwise specified. Let $d$ be a positive integer. Let $\ell$ be a prime number not dividing $d$. Depending on the context, $\zeta_d$ may be a fixed generator of a cyclic group $\langle \zeta_d \rangle$ of order $d$ or a primitive $d$th root of unity of a field. The group operation of $\langle \zeta_d \rangle$ is written multiplicatively and the identity element is written as $1$. 
Given a set $A$ with an action by $\langle \zeta_d \rangle$, let $A^{\zeta_d} = A^{\langle \zeta_d \rangle}$ denote the subset of invariants. Let $\Phi_d(X) \in \bbZ[X]$ be the $d$th-cyclotomic polynomial.

Given a commutative ring $R$, let 
\begin{align*}
R[\zeta_d]_{\text{all}} &\coloneq R[X]/(X^d-1) \\
R[\zeta_d]_{\zeta_d\neq 1} &\coloneq R[X]/\left( \sum_{i=0}^{d-1} X^i \right) \\
R[\zeta_d]_{\text{prim}} & \coloneq R[X]/\left( \Phi_d(X) \right) \\
R[\zeta_d]_{\zeta_d \neq -1} &\coloneq \begin{cases} R[\zeta_d]_{\text{all}} &\text{if } 2 \nmid d \\ R[X]/\left(\frac{X^d-1}{X+1}\right) &\text{if } 2 \mid d\end{cases} \\
R[\zeta_d]_{\zeta_d \neq \pm 1} &\coloneq \begin{cases} R[\zeta_d]_{\zeta_d \neq 1} &\text{if } 2 \nmid d \\ R[X]/\left(\frac{\sum_{i=0}^{d-1} X^i }{X+1} \right) &\text{if } 2 \mid d. \end{cases}
\end{align*}
We often simply refer to each of these algebras as $R[\zeta_d]$ and indicate which we are referring to, if distinctions are necessary, before we use this simplified notation in context. Note that there is an involution, which we denote by $\overline{\cdot}$, on any such algebra sending $\zeta_d$ to $\zeta_d^{-1}$. All of these algebras are quotients of $R[\zeta_d]_{\text{all}}$, which in turn is isomorphic to the group ring $R[\langle \zeta_d \rangle]$ under the identification of $\zeta_d$ with $X$. Moreover, $R[\zeta_d]_{\zeta_d \neq -1}$ and $R[\zeta_d]_{\zeta_d \neq \pm 1}$ are quotients of $R[\zeta_d]_{\text{all}}$ and $R[\zeta_d]_{\zeta_d \neq 1}$ respectively. 

The algebras $R[\zeta_d]$, regarded as quotients of $R[X]/(X^d-1)$, can have a factor of $R[X]/(X+1)$ only if $2 \mid d$. Given an $R$-algebra direct product factor $S$ of $R[\zeta_d]_{\text{all}}$, we let $\varepsilon_{S}$ be the element of $S$ corresponding to $(1,-1)$ in a factorization $S \cong S' \times R[X]/(X+1)$ if such a factorization exists, and be $1$ otherwise.

Note that \Cref{section:cohen_lenstra_for_imaginary_Z_3Z_extensions} and \Cref{section:generalizing_ELS_section_2} both let $\mathbb{Z}_\ell[\zeta_d] = \mathbb{Z}_\ell[\zeta_d]_{\zeta_d\neq 1}$.


\subsection*{Acknowledgements}

We thank Jeff Achter, Tejasi Bhatnagar, Jordan S.~Ellenberg, Chris Hall, Aaron Landesmann, Wanlin Li, Sun Woo Park, Rachel Pries, and Will Sawin for helpful discussions concerning this project. 

\subsection*{AI Use Disclosure}

AI has been used in this work for the following purposes:
\begin{enumerate}
    \item To aid with reorganizing ideas in this work; this includes, but may not be limited to:
    \begin{itemize}
        \item devising \Cref{proposition:conj_classes_in_a_semidirect_prod_of_ab_grp_by_grp_are_classified_by_orbits_by_centralizers_on_coinvariants},
        \item formulating \Cref{theorem:semidirect_product_of_an_abelian_group_by_a_group_is_center_free_whenever_fixed_point_subgroup_by_some_non_identity_elements_is_trivial_and_1_minus_non_identity_elements_are_non_zero_endomorphisms} as a means to generalize \Cref{lemma:semidirect_product_is_center_free_and_A_has_trivial_invariant_subgroup},
        \item correcting a previous version of \Cref{proposition:ell_adic_image_of_arithmetic_group_contains_SU} and its proof,
        \item planning the splitting of the proof of \Cref{proposition:generalizing_EVW_Lemma 8.9_to_Z_dZ_covers} into \Cref{lemma:existence_of_a_hyperbolic_pair}, \Cref{lemma:extending_a_hyperbolic_frame}, and \Cref{lemma:isometry_from_matched_hyperbolic_frames},
        \item planning notational changes in \Cref{section:orbits_of_burau_representation}.
    \end{itemize}
    \item To find and correct logical gaps as follows:
    \begin{enumerate}
        \item Correcting \Cref{proposition:ell_adic_image_of_arithmetic_group_contains_SU} to add the hypotheses that the involution on $L$ be nontrivial and $h$ be isotropic, correcting its proof, and correcting statements that depend on it, 
        \item Filling a gap in \Cref{lemma:generalizing_EVW_lemma_8_4_to_equivariant_surjections} by stating and proving \Cref{lem:reduction_matrix_counting}, \Cref{lem:column_space_row_corank}, \Cref{lem:stacked_matrix_inequality}, and \Cref{lem:single_bump_comparison}. 
    \end{enumerate}

    \item To identify many typos and minor errors. 
\end{enumerate}
AI has not been used to make automated edits to this work. Instead, manual edits were made based on AI-generated reports of errors. We note that the statements and arguments for \Cref{lem:reduction_matrix_counting}, \Cref{lem:column_space_row_corank}, \Cref{lem:stacked_matrix_inequality}, and \Cref{lem:single_bump_comparison} were initially AI-generated, and the author subsequently thoroughly edited the arguments to be more precise. Moroever, the main ideas in this work were already present in the author's PhD thesis, which was mostly drafted in 2023 and defended in 2024.

\section{Conjugacy structure of \texorpdfstring{$A \rtimes \mathbb{Z}/d\mathbb{Z}$}{A rtimes Z/dZ}} \label{section:conjugacy_structure_of_A_rtimes_Z_dZ}
In this section, we prove some lemmas concerning the group structure of $A \rtimes \mathbb{Z}/d\mathbb{Z}$ where $A$ is a $\mathbb{Z}_\ell[\zeta_d]$-module.
Unless otherwise specified, throughout this section, $\mathbb{Z}_\ell[\zeta_d]$-module means module over $\mathbb{Z}_\ell[\zeta_d]_{\zeta_d \neq 1}$ (as in \Cref{subsection:notations}); when we use the full group ring $\mathbb{Z}_\ell[\zeta_d]_{\text{all}} = \mathbb{Z}_\ell[X]/(X^d-1)$ in the decomposition below, we write it explicitly.


\begin{lemma}\label{lemma:research_orbits_of_burau_representations_2_semidirect_product_structure}
    Let $d \geq 2$ be an integer, and let $A$ be a $\mathbb{Z}_\ell$-module with a $\langle \zeta_d \rangle \cong \mathbb{Z}/d\mathbb{Z}$-action where $\ell$ is a prime that does not divide $d$. The following hold for all $k,k_1,k_2 \in \mathbb{Z}$ and $a_1,a_2 \in A$:
    \begin{enumerate}
    \item $(a_1, \zeta_d^{k_1}) \cdot (a_2, \zeta_d^{k_2}) = (a_1 + \zeta_d^{k_1} a_2, \zeta_d^{k_1+k_2})$ \label{item:semidirect_product_group_law}
    \item $(a_1, \zeta_d^k)^{-1} = (-\zeta_d^{-k} a_1, \zeta_d^{-k})$ \label{item:semidirect_product_inverse}
    \item $(a_1, \zeta_d^{k_1}) \cdot (a_2, \zeta_d^{k_2}) \cdot (a_1, \zeta_d^{k_1})^{-1} = (a_1 + \zeta_d^{k_1} a_2 - \zeta_d^{k_2} a_1, \zeta_d^{k_2})$ \label{item:semidirect_product_conjugation_by_zeta_d_on_zeta_d}
    \item $(a_1, \zeta_d^k) \cdot (a_2, \zeta_d^k)^{-1} = (a_1-a_2, 1)$ \label{item:semidirect_product_product_of_zeta_d_by_zeta_d_inverse}
    \item $(a_1, \zeta_d^k) \cdot (a_2, 1) \cdot (a_1, \zeta_d^k)^{-1} = (\zeta_d^k a_2, 1)$ \label{item:semidirect_product_conjugation_by_zeta_d_on_1}
    \item $(a_1, 1) \cdot (a_2, \zeta_d^k) \cdot (a_1, 1)^{-1} = ( (1-\zeta_d^k) a_1 + a_2, \zeta_d^k)$ \label{item:semidirect_product_conjugation_by_1_on_zeta_d}
    \end{enumerate}
\end{lemma}

\begin{proof}
    All of these are immediate calculations.
\end{proof}

We are further concerned with the case in which $A^{\zeta_d}$ is trivial. \Cref{corollary:equivalent_conditions_for_A_to_have_trivial_zeta_d_invariant_subgroup} lists equivalent conditions. To prove  \Cref{corollary:equivalent_conditions_for_A_to_have_trivial_zeta_d_invariant_subgroup}, we discuss the structure theory of finitely generated $\mathbb{Z}_\ell$-modules with a $\mathbb{Z}/d\mathbb{Z}$-action.

Since $\ell \nmid d$, $X^d - 1$ is square free over $\mathbb{Z}_\ell$. Write 
$$X^d-1 = \prod_i f_i(X)$$
for the prime factorization of $X^d-1$ over $\mathbb{Z}_\ell$. We have an isomorphism 
$$\mathbb{Z}_\ell[\zeta_d]_{\text{all}} = \mathbb{Z}_\ell[X]/(X^d-1) \cong \prod_{i} \mathbb{Z}_\ell[X]/f_i(X),$$ 
and the $\mathbb{Z}_\ell[X]/f_i(X)$ are the rings of integers of the unramified extensions $\mathbb{Q}_\ell[X]/(f_i(X))$ of $\mathbb{Q}_\ell$. Let $B_i = \mathbb{Z}_\ell[X]/f_i(X)$. In particular, a $\mathbb{Z}_\ell$-module with a $\langle \zeta_d \rangle$ action, i.e. a $(\mathbb{Z}_\ell[X]/(X^d-1))$-module after identifying $X$ with $\zeta_d$, is the product of its $B_i$-components $A_i$. Given an element $a \in A$, write $a_i \in A_i$ for its $B_i$-component: 
\begin{align*}
    A \cong \prod_{i} A_i. 
\end{align*}
Moreover, the $B_i$ are discrete valuation rings with uniformizer $\ell$, so if $A_i$ is a finitely generated $B_i$ module, then the structure theorem for finitely generated modules over PIDs yields a decomposition
    \begin{align}
        A_i \cong \bigoplus_{j=1}^{k_i} B_i / (\ell^{d_{i,j}}) \label{eq:structure_theorem_for_PIDs_decomposition_of_Ai}
    \end{align}
for finitely many integers $d_{i,j} \geq 0$. 

\begin{proposition} \label{proposition:equivalent_conditions_for_reduction_map_of_linear_endomorphism_over_a_local_noetherian_ring_to_not_have_eigenvalue_1}
Let $R$ be a Noetherian local ring with maximal ideal $\mathfrak{m}$ and residue field $k$. Let $M$ be a finitely generated $R$-module, and let $\phi \colon M \to M$ be an $R$-linear endomorphism. The following are equivalent:
\begin{enumerate}
    \item The induced map $\bar{\phi} \colon M/\mathfrak{m}M \to M/\mathfrak{m}M$ does not have $1 \in k$ as an eigenvalue (i.e., $\bar{\phi} - \operatorname{id}$ is injective). \label{part:no_1_eigenvalue_on_M_mod_m}
    \item The endomorphism $(\phi - \operatorname{id})$ is invertible. \label{part:phi_minus_id_is_invertible}
\end{enumerate}
When they hold, $\ker(\phi - \operatorname{id}) = 0$ as well. 
\end{proposition}
\begin{proof}
    Clearly, \eqref{part:phi_minus_id_is_invertible} implies \eqref{part:no_1_eigenvalue_on_M_mod_m}. Assuming \eqref{part:no_1_eigenvalue_on_M_mod_m}, the map $\bar{\phi} - \operatorname{id}$ is surjective. By \cite[Chapter 2, Exercise 10]{atiyah_macdonald}, $\phi - \operatorname{id}$ is thus surjective. By \cite[Chapter 6, Exercise 1(i)]{atiyah_macdonald}, it is in fact an isomorphism, so \eqref{part:phi_minus_id_is_invertible} holds. 

    Finally, \eqref{part:phi_minus_id_is_invertible} trivially implies that $\ker(\phi - \operatorname{id}) = 0$.
\end{proof}

\begin{corollary} \label{corollary:equivalent_conditions_for_A_to_have_trivial_zeta_d_invariant_subgroup}
    Let $d \geq 2$ be an integer, let $\ell$ be a prime not dividing $d$, and let $A$ be a finitely generated $\mathbb{Z}_\ell$-module with a $\langle \zeta_d \rangle$-action. The following are equivalent:
    \begin{enumerate}
        \item The induced action of $\zeta_d$ on $A/\ell A$ does not have $1 \in \mathbb{F}_\ell$ as an eigenvalue. \label{part:no_1_eigenvalue_on_A_ell_A}
        \item $A$ is a module over $\mathbb{Z}_\ell[\zeta_d] := \mathbb{Z}_\ell[X]/(X^{d-1} + \cdots + 1)$. \label{part:A_is_module_over_Z_ell_zeta_d}
        \item $A^{\zeta_d}$ is trivial. \label{part:trivial_zeta_d_invariant_subgroup_of_A}
    \end{enumerate}
\end{corollary}

\begin{proof}
    The equivalence of \eqref{part:no_1_eigenvalue_on_A_ell_A} and \eqref{part:A_is_module_over_Z_ell_zeta_d} is due to \Cref{proposition:equivalent_conditions_for_reduction_map_of_linear_endomorphism_over_a_local_noetherian_ring_to_not_have_eigenvalue_1}.
    Note that $A^{\zeta_d} = eA$ for $e = \frac{1}{d}\sum_{i=0}^{d-1} \zeta_d^i$. Thus, \eqref{part:A_is_module_over_Z_ell_zeta_d} and \eqref{part:trivial_zeta_d_invariant_subgroup_of_A} are equivalent. 
    
\end{proof}

    
    

Under the equivalent conditions of \Cref{corollary:equivalent_conditions_for_A_to_have_trivial_zeta_d_invariant_subgroup}, we show that the set of elements of the form $(a,\zeta_d)$ forms a conjugacy class and prove some needed properties of this conjugacy class. \Cref{proposition:conj_classes_in_a_semidirect_prod_of_ab_grp_by_grp_are_classified_by_orbits_by_centralizers_on_coinvariants} enumerates conjugacy classes for more general contexts.
\begin{proposition} \label{proposition:conj_classes_in_a_semidirect_prod_of_ab_grp_by_grp_are_classified_by_orbits_by_centralizers_on_coinvariants}
    Let $G$ be a group acting on an abelian group $A$. Write $A \rtimes G$ for the semidirect product. Fix $g \in G$.

    The set of conjugacy classes of $A \rtimes G$ whose images under the canonical projection $A \rtimes G \twoheadrightarrow G$ contain $g$ is in natural bijection with the orbit space
    $$(A/(1-g) A) / C(g)$$
    where $C(g)$ is the centralizer of $g$ in $G$ and the action of $C(g)$ on $A/(1-g) A$ is induced by the action of $C(g)$ on $A$.
\end{proposition}
\begin{proof}
    Given $(a,g) \in A \rtimes G$, its conjugacy class consists of all elements of the form
    $$(b, \tau) \cdot (a, g) \cdot (b,\tau)^{-1} = (b + \tau \cdot a - \tau g \tau^{-1} \cdot b, \, \tau g \tau^{-1}).$$
    \CrefIfExists{lemma:conjugation_in_semidirect_product_of_an_abelian_group_by_a_group}

    Let $S$ be the set of conjugacy classes of $A \rtimes G$ whose images under the canonical projection $A \rtimes G \twoheadrightarrow G$ contain $g$. We construct a bijection
    $$\Phi: S \longrightarrow (A/(1-g) A) / C(g)$$
    as follows: given a $C \in S$, choose an element of $C$ of the form $(a,g)$. Define $\Phi(C)$ to be the orbit of the class $[a] \in A / (1 - g) A$.

    To prove that $\Phi$ is well-defined, suppose that $(a, g)$ and $(a', g)$ are both elements of $C$. In particular, there exists some $(b, \tau) \in A \rtimes G$ such that
    $$(b, \tau) \cdot (a, g) \cdot (b, \tau)^{-1} = (a', g),$$
    i.e. such that
    $$a' = (1 - g)\cdot b + \tau \cdot a \quad \text{and} \quad \tau g \tau^{-1} = g  .$$
    \CrefIfExists{lemma:conjugation_in_semidirect_product_of_an_abelian_group_by_a_group}
    The second equality implies that $\tau \in C(g)$. The first implies that the elements $a'$ and $\tau \cdot a$ are equivalent in the quotient space $A / (1-g)A$. Thus, $\Phi$ is well-defined.

    We construct the inverse map $\Psi: (A/(1-g)A) / C(g) \to S$ by sending the orbit of a class $[a] \in A / (1-g)A$ to the conjugacy class of $(a, g)$. To show that this is well defined, we must show well-definedness with respect to both a choice of another representative of the class $[a] \in A/(1-g)A$ and the $C(g)$-action. Any representative of the class $[a] \in A/(1-g)A$ is of the form $a + (1-g) b$ for $b \in A$; the elements $(a,g)$ and $(a + (1-g) b,g)$ are conjugates, so $\Psi$ is well defined with respect to representatives of $[a]$.

    The action of $C(g)$ on $A/(1-g)A$ is given by
    $$\tau \cdot [a] = [\tau \cdot a].$$
    By definition, $\Psi$ sends the orbits of $[a]$ and $[\tau \cdot a]$ to the conjugacy classes of $(a, g)$ and $( \tau \cdot a, g)$ respectively; these conjugacy classes coincide. Therefore, $\Psi$ is well defined with respect to the $C(g)$-action.

    By construction, $\Phi \circ \Psi$ and $\Psi \circ \Phi$ are the identity maps on their respective sets, establishing the explicit natural bijection.
\end{proof}


\begin{lemma} \label{lemma:research_orbits_of_burau_representations_2_c_is_conjugacy_class}
Let $d \geq 2$ be an integer, and let $A$ be a finitely generated $\mathbb{Z}_\ell$-module with a $\langle \zeta_d \rangle \cong \mathbb{Z}/d\mathbb{Z}$-action where $\ell$ is a prime that does not divide $d$. Suppose that $A^{\zeta_d}$ is trivial.
For any fixed integer $k$ such that $\gcd(k,d) = 1$, the set 
    $$c = \{(a, \zeta_d^k) \in A \rtimes \langle \zeta_d \rangle: a \in A\}$$
is a conjugacy class of $A \rtimes \langle \zeta_d^k \rangle$. In fact, any two elements of $c$ are conjugate via an element of $c$ and are conjugate via an element of $A$.
\begin{proof}
    Note that $1-\zeta_d^k$ is invertible as a map $A \to A$ because $k$ is assumed to be relatively prime to $d$ and because the action of $\zeta_d$ on $A/\ell A$ does not have an eigenvalue of $1 \in \mathbb{F}_\ell$. By \Cref{proposition:conj_classes_in_a_semidirect_prod_of_ab_grp_by_grp_are_classified_by_orbits_by_centralizers_on_coinvariants}, the set of conjugacy classes mapping to $\zeta_d^k$ is identified with the orbit space $(A/(1-\zeta_d^k)A)/C(\zeta_d^k)$, where $C(\zeta_d^k)$ is the centralizer of $\zeta_d^k$. Since $1 - \zeta_d^k$ is invertible on $A$, we have that $A/(1-\zeta_d^k)A = 0$, and hence the orbit space consists of a single point.
    
    
     
    More explicitly, by letting $a_1 = \frac{-\zeta_d^k a_2 + a_3}{1-\zeta_d^k}$, \Cref{lemma:research_orbits_of_burau_representations_2_semidirect_product_structure}\eqref{item:semidirect_product_conjugation_by_zeta_d_on_zeta_d} shows that 
        $$(a_1, \zeta_d^k) (a_2, \zeta_d^k) (a_1, \zeta_d^k)^{-1} = (a_3, \zeta_d^k),$$
    and by letting $a_1 = \frac{a_3 - a_2}{1-\zeta_d^k}$, \Cref{lemma:research_orbits_of_burau_representations_2_semidirect_product_structure}\eqref{item:semidirect_product_conjugation_by_1_on_zeta_d} shows that 
        $$(a_1, 1)(a_2, \zeta_d^k)(a_1, 1)^{-1} = (a_3, \zeta_d^k).$$
\end{proof}
\end{lemma}

Ellenberg, Venkatesh, and Westerland's homological stability result \cite[Theorem 6.1]{ellenberg_venkatesh_westerland} as well as its consequences in loc.~cit.~depend on having a non-splitting conjugacy class --- we say that a conjugacy class $c$ of a group $G$ is \emph{non-splitting} if $c$ generates $G$ and for any subgroup $H \leq G$, the intersection $c \cap H$ is either empty or a conjugacy class in $H$. We show that the conjugacy classes of \Cref{lemma:research_orbits_of_burau_representations_2_c_is_conjugacy_class} are nonsplitting.

\begin{lemma} \label{lemma:c_is_nonsplitting_conjugacy_class}
    Let $d \geq 2$ be an integer, and let $A$ be a finite abelian $\ell$-group with a $\langle \zeta_d \rangle \cong \mathbb{Z}/d\mathbb{Z}$-action where $\ell$ is a prime that does not divide $d$. Suppose that $A^{\zeta_d}$ is trivial.
    For any fixed integer $k$ that is relatively prime to $d$, the conjugacy class $c$ of elements $A \rtimes \langle \zeta_d \rangle$ of the form $(a, \zeta_d^k)$ is non-splitting.
\end{lemma}
\begin{proof}
    The conjugacy class $c$ generates $A \rtimes \langle \zeta_d \rangle$ because $c$ generates $A \rtimes \{1\}$ and contains $(0, \zeta_d^k)$.

    Say that $H$ is a subgroup of $A \rtimes \langle \zeta_d \rangle$ such that $c \cap H \neq \emptyset$. We want to show that $c \cap H$ is a conjugacy class of $H$. Suppose that $(a_2, \zeta_d^k), (a_3, \zeta_d^k) \in H$. We will show that these two elements are conjugates by some $(a,1) \in H$. In particular, $H$ has $(a_2, \zeta_d^k)(a_3, \zeta_d^k)^{-1} = (a_2-a_3, 1)$. Moreover, since $k$ is relatively prime to $d$, $H$ has elements of the form $(\alpha_i, \zeta_d^i)$ for every $i$. Therefore, $H$ has elements
        $$(\alpha_i, \zeta_d^i) \cdot (a_2-a_3, 1) \cdot (\alpha_i, \zeta_d^i)^{-1} = (\zeta_d^i(a_2-a_3), 1)$$
    for all $i$. Let $V = H \cap (A \rtimes \{1\})$, which is a subgroup of $A \rtimes \{1\}$ containing $(a_2-a_3, 1)$ and closed under the $\zeta_d$-action; hence $V$ is closed under $\mathbb{Z}[\zeta_d]$-linear combinations (meaning repeated addition and application of $\zeta_d^i$).

    Note that $A$ can be regarded as a $\mathbb{Z}/\ell^n\mathbb{Z}[\zeta_d]$-module for $n$ sufficiently large and in particular, $V$ can be regarded as a $\mathbb{Z}/\ell^n\mathbb{Z}[\zeta_d]$-submodule of $A$. Since $\ell \nmid d$ and $\gcd(k,d)=1$, the element $(1-\zeta_d^k)$ is invertible in $\mathbb{Z}/\ell^n\mathbb{Z}[\zeta_d]$. Let $a = (1-\zeta_d^k)^{-1}(a_3-a_2)$. Conjugating by $(a,1)$ gives
        $$(a,1)(a_2, \zeta_d^k)(a,1)^{-1} = ((1-\zeta_d^k)a + a_2, \zeta_d^k) = (a_3, \zeta_d^k)$$
    as desired.
\end{proof}

The following is a general sufficient condition for the semidirect product of an abelian group by a group to be center-free.
\begin{theorem}
\label{theorem:semidirect_product_of_an_abelian_group_by_a_group_is_center_free_whenever_fixed_point_subgroup_by_some_non_identity_elements_is_trivial_and_1_minus_non_identity_elements_are_non_zero_endomorphisms}
Let $A$ be a nontrivial abelian group and let $G$ be a group acting on $A$ via a homomorphism $\phi: G \to \text{Aut}(A)$. For each $g \in G$, denote the action of $g$ on $a \in A$ by $g \cdot a$. Suppose that the following hold:
\begin{enumerate}
    \item \label{cond:fixed_points_exist} There exists a non-identity element $g_0 \in G \setminus \{1\}$ such that the fixed-point subgroup of $g_0$ is trivial, i.e., $A^{g_0} = \{a \in A \mid g_0 \cdot a = a\} = \{0\}$.
    \item \label{cond:image} For every non-identity element $g \in G \setminus \{1\}$, the map $(1 - g): A \to A$ defined by $a \mapsto a - g \cdot a$ is not the zero map.
\end{enumerate}
The semidirect product $A \rtimes G$ is \CrefAndHyperrefIfExist{definition:center_of_a_group}{center-free}.
\end{theorem}

\begin{proof}
Let $(a_1, g_1) \in Z(A \rtimes G)$. By definition, $(a_1, g_1)$ must commute with all elements $(a_2, g_2) \in A \rtimes G$. The equality $(a_1, g_1)(a_2, g_2) = (a_2, g_2)(a_1, g_1)$ is equivalent to
$$(a_1 + g_1 \cdot a_2, \, g_1 g_2) = (a_2 + g_2 \cdot a_1, \, g_2 g_1).$$
In particular, we must have 
\begin{equation}
\label{eq:commutation_relation}
(1 - g_2) \cdot a_1 = (1 - g_1) \cdot a_2 \quad \text{for all } a_2 \in A, \, g_2 \in G.
\end{equation}
We must show that $(a_1, g_1) = (0, 1)$. 

Suppose for contradiction that $g_1 \neq 1$. Letting $g_2 = 1$ in \eqref{eq:commutation_relation}, the left-hand side becomes $0$, reducing the relation to:
$$0 = (1 - g_1) \cdot a_2 \quad \text{for all } a_2 \in A.$$
This implies that the map $(1 - g_1): A \to A$ is identically the zero map, which contradicts Condition~\eqref{cond:image}. Hence, $g_1 = 1$.

With $g_1 = 1$, the right-hand side of \eqref{eq:commutation_relation} becomes $0$, so the relation simplifies to $(1 - g_2) \cdot a_1 = 0$ for all $g_2 \in G$. Letting $g_2$ be the element $g_0$ of Condition~\eqref{cond:fixed_points_exist}, the assumption that $A^{g_0} = \{0\}$ implies that $a_1 = 0$. Thus, $(a_1,g_1) = (0,1)$ as desired.
\end{proof}

\begin{lemma} \label{lemma:semidirect_product_is_center_free_and_A_has_trivial_invariant_subgroup}
    Let $d \geq 2$ be an integer, and let $A$ be a nontrivial $\mathbb{Z}_\ell$-module with a $\langle \zeta_d \rangle \cong \mathbb{Z}/d\mathbb{Z}$-action where $\ell \nmid d$ is a prime. Suppose that the induced action of $\zeta_d$ on $A/\ell A$ does not have an eigenvalue of $1$ and has a primitive $d$th root of unity in $\overline{\mathbb{F}}_\ell$ as an eigenvalue. The group $A \rtimes \langle \zeta_d \rangle$ is center free.
\end{lemma}
\begin{proof}
    The assumptions yield the hypotheses of \Cref{theorem:semidirect_product_of_an_abelian_group_by_a_group_is_center_free_whenever_fixed_point_subgroup_by_some_non_identity_elements_is_trivial_and_1_minus_non_identity_elements_are_non_zero_endomorphisms}.
\end{proof}

\begin{remark}
    Let $d$ and $\ell$ be distinct prime numbers. Any nontrivial $\mathbb{Z}_\ell$-module $A$ with a $\langle \zeta_d \rangle$-action whose $\zeta_d$-invariant subgroup is trivial must have a primitive $d$th root of unity in $\overline{\mathbb{F}}_\ell$ as an eigenvalue, thereby satisfying the hypotheses in \Cref{lemma:semidirect_product_is_center_free_and_A_has_trivial_invariant_subgroup}.
\end{remark}

\section{Monodromy types of \texorpdfstring{$G$}{G}-covers of stable genus \texorpdfstring{$0$}{0}-curves} \label{section:monodromy_type}
In this section, we relate various ways to define monodromy of $G$-covers of stable genus $0$-curves for finite groups $G$.

\subsection{Definitions} \label{subsection:definitions_of_tame_G_covers}

\begin{definition}[{cf. \cite[Section 7.1, 7.3]{ellenberg_venkatesh_westerland}, \cite[Proof of Proposition 2.1]{ellenberg_li_shusterman}, \cite[Section 11.1]{liu_wood_zureick-brown_pd}}]
    A {\em curve over a scheme $S$} is a 
    proper morphism $C \to S$ whose geometric fibers are connected and $1$-dimensional. A {\em cover of a curve $C \to S$} is a finite, flat, and surjective morphism $Y \to C$ of $S$-schemes from a curve $Y \to S$. Such a cover is {\em tame} or {\em tamely ramified} if the ramification index at every point is prime to the characteristic of the residue field at the point. For a cover $f: Y \to C$, let $\operatorname{Aut}(f)$ be the automorphism group of $f: Y \to C$. Such a cover $f$ is {\em Galois} if $f$ is separable and if $\operatorname{Aut} f$ acts transitively on fibers of geometric points of $C$. Given a tame Galois cover $f: Y \to C$, its {\em branch locus} $D$ is the reduced divisor of $C$ such that $D \to S$ is \'etale, and $f$ is \'etale over $C - D$, and $C - D$ is the maximal open subscheme of $C$ with these properties. If there is some constant $n$ such that the geometric fibers of $D \to S$ are all of degree $n$, then $f$ is said to have {\em $n$ branch points}. Such a constant exists whenever $S$ is connected.

    In the case that $S = \operatorname{Spec} k$, given a finite group $G$, a \em{tame $G$-cover of $\mathbb{P}^1$} consists of
    \begin{itemize}
        \item a tamely ramified finite Galois cover $f: C \to \mathbb{P}^1$ where $C/k$ is a smooth proper geometrically connected curve, and
        \item an isomorphism $\phi: G \to \operatorname{Aut}(f)$.
    \end{itemize}
    We will often only write $f: C \to \mathbb{P}^1$ when referring to the tame $G$-cover of $\mathbb{P}^1$ whilst leaving $\phi$ implicit. 
\end{definition}

More generally, we can consider labeled admissible stable $G$-covers of curves. We generalize the definitions in \cite[Section 2.1]{achter_pries_imht}, which considers the case where the group $G$ is a cyclic group of prime order.

\begin{definition}[{cf. \cite[Section 2.1]{achter_pries_imht}, \cite[Definition 1.2]{ekedahl_bbhs}}]
    Let $G = \mathbb{Z}/d\mathbb{Z}$ and let $S$ be an irreducible scheme over $\operatorname{Spec} \mathbb{Z}[1/d, \zeta_d]$. Let $k$ be an algebraically closed field equipped with a ring morphism $\mathbb{Z}[1/d, \zeta_d] \to k$.
    
    A {\em semi-stable curve} $\psi: C \to S$ is a flat and proper morphism whose geometric fibers are connected, reduced curves whose only singularities are ordinary double points. Given a point $s$ of $S$, let $C_s$ denote the fiber of $C$ over $s$. Let $\operatorname{Sing}_S(C)$ denote the set of $z \in C$ such that $z$ is a singular point on the fiber $C_{\psi(z)}$.

    A {\em mark} $\Xi$ on $C/S$ is a closed subscheme of $C - \operatorname{Sing}_S(C)$ which is finite and \'etale over $S$. The {\em degree} of $\Xi$ is the number of points in any geometric fiber of $\Xi \to S$. Given a semi-stable curve $\psi: C \to S$ and a mark $\Xi$ on $C/S$, the pair $(C/S, \Xi)$ is called a {\em marked semi-stable curve}. Furthermore, $(C/S, \Xi)$ is a {\em stably marked semi-stable curve} if every irreducible component in every geometric fiber of $C$ has at least three points which are either in $\operatorname{Sing}_S(C)$ or in $\Xi$. We say that a mark $\Xi$ has a labeling if $\Xi$ is an ordered disjoint union of sections $S \to C$. In this case, a {\em labeling} of the mark $\Xi$ is a map $\eta: \{1,\ldots,r\} \to \Xi$ where $r$ is the degree of $\Xi$.

    Given a $G$-action $\iota_0: G \hookrightarrow \operatorname{Aut}_S(C)$ on $C/S$, let $R$ denote the ramification locus of the cover $C \to C/\iota_0(G)$, and let $R_{\text{sm}} = R - (R \cap \operatorname{Sing}_S(C))$ be the smooth ramification locus. The pair $(C/S, \iota_0)$ is a {\em stable $G$-curve} if $C/S$ is a semi-stable curve, $R_{\text{sm}}$ is a mark on $C/S$, and $(C/S, R_{\text{sm}})$ is stably marked. 

    
    Given a stable $G$-curve $(C/S, \iota_0)$, and a geometric point $z$ of $R \cap \operatorname{Sing}_S(C)$, let $C_{z,1}$ and $C_{z,2}$ be the two components of the formal completion of $C_{\psi(z)}$. We say that $(C/S, \iota_0)$ is {\em admissible} if, for every geometric point $z \in R \cap \operatorname{Sing}_S(C)$, $\iota_0(1)$ stabilizes each branch $C_{z,i}$ and the characters of the action of $\iota_0$ on the tangent spaces of $C_{z,1}$ and $C_{z,2}$ at $z$ are inverses. A {\em labeling} of an admissible stable $G$-curve $(C/S, \iota_0)$ is a labeling $\eta$ of $R_{\text{sm}}$. Note that a labeling $\eta$ of an admissible stable $G$-curve $(C/S, \iota_0)$ induces a labeling $\eta_0: \{1,\ldots,r\} \to B_{\text{sm}}$ on the quotient $C/\iota_0(G)$ where $B_{\text{sm}}$ is the smooth branch locus of the covering $C \to C/\iota_0(G)$.

    Let $s$ be a geometric point of $S$ with residue field $k$, and let $a$ be a point of the fiber $R_{\text{sm}, s}$. In particular, $G$ acts on the tangent space of $C_s$ at $a$ via a character $\chi_a: G \to k^\times$. There is a unique choice of $\gamma_a \in G \setminus \{0\}$ so that $\chi_a(1) = \zeta_d^{\gamma_a}$. This value of $\gamma_a$ is called the {\em canonical generator of inertia at $a$}. Given a labeled admissible stable $G$-curve $(C/S, \iota_0, \eta)$, where $G = \mathbb{Z}/d\mathbb{Z}$, its {\em class vector} is the set map $\gamma: \{1,\ldots,r\} \to G \setminus \{0\}$ given by $\gamma(i) = \gamma_{\eta(i)}$. We often write $\gamma = (\gamma(1),\ldots,\gamma(r))$.
\end{definition}

We are mostly concerned with admissible stable $G$-curves $(C/S, \iota_0)$ such that $G = \mathbb{Z}/d\mathbb{Z}$ and the quotient $C/\iota_0(G)$ has (arithmetic) genus $0$. In this case, $C/\iota_0(G)$ is a stably marked curve by \cite[Proposition 1.4]{ekedahl_bbhs}.

Using conventions in \cite[Sections 5.1, 5.2]{wood_alievw}, we define monodromy types of tame $G$-covers of $\mathbb{P}^1$.

\begin{definition}
    Let $k$ be a field. Given an indeterminate $z$, the {\em Puiseux series field} $k((z^{1/\infty}))$ is the field extension of the Laurent series field $k((z))$ generated by $z^{1/m}$ for $m$ relatively prime to $\operatorname{char} k$. 

    When $k = \bar{k}$, $k((z^{1/\infty}))$ is the maximal prime-to-$(\operatorname{char} k)$ extension of $k((z))$. In this case, there is an isomorphism
        \begin{equation} \label{eq:isomorphism_of_galois_group_of_puiseux_series_field_with_roots_of_unity}
            \operatorname{Gal}(k((z^{1/\infty})) / k((z)) ) \simeq \hat{\mathbb{Z}}(1)_k
        \end{equation}
    given by $\sigma \mapsto \{\sigma_m\}$ where $\hat{\mathbb{Z}}(1)_k = \varprojlim_m \mu_m(k)$ and $\sigma_m = \frac{\sigma(z^{1/m})}{z^{1/m}}$. 
\end{definition}

Let $k$ be an algebraically closed field. For $t_0 \in k$, write $z = t-t_0$. Given any Galois prime-to-($\operatorname{char} k$) extension $L$ of $k(t) = k(z)$, there is a $\operatorname{Gal}(L/k(t))$-conjugacy class of homomorphisms $\operatorname{Gal}(k((z^{1/\infty})) / k((z))) \to \operatorname{Gal}(L/k(t))$ which correspond to the $k(t)$-field injections $L \to k((z^{1/\infty}))$. Composing such a homomorphism with the isomorphism \eqref{eq:isomorphism_of_galois_group_of_puiseux_series_field_with_roots_of_unity} yields the homomorphism 
    $$r_{t_0, Q}: \hat{\mathbb{Z}}(1)_k \to \operatorname{Gal}(L/k(t))$$
whose image is an inertia group $I(Q|P) \subseteq \operatorname{Gal}(L/k(t))$ for a place (more precisely, a system of places of finite Galois prime-to-$(\operatorname{char} k)$-extensions of $k(t)$) $Q$ of $L$ above the place $P$ of $k(t)$ corresponding to $t_0$. In particular, the $\operatorname{Gal}(L/k(t))$-conjugacy class of $r_{t_0, Q}$ is independent of the choice of $k(t)$-embedding $L \to k((z^{1/\infty}))$. Let $r_{t_0}$ denote this conjugacy class. 

Moreover, choose some base point $\ast$ for $\mathbb{P}^1_k$ and let $U \subseteq \mathbb{P}^1_k$ be an open subset containing $\ast$\footnote{\cite[Section 5.2]{wood_alievw} specified $\infty$ to be the base point of $\mathbb{P}^1$, but we do not make the same specification here as we study covers of $\mathbb{P}^1$ that are possibly branched at $\infty$.}. Let $\pi_1^t(U,\ast)$ be the tamely ramified fundamental group of $U$ (see \cite[Expos\'e XIII, 2.1.3]{SGA_1} for a more detailed discussion). Let $L = L^{\text{max}}$ be the maximal prime-to-($\operatorname{char} k$)-extension of $k(t)$ unramified at points of $U$ so that $\pi_1^t(U,\ast) \cong \operatorname{Gal}(L^{\text{max}}/k(t))$. Let $t_1,\ldots,t_n$ be the $k$-points of $\mathbb{P}^1 - U$. By Grothendieck's comparison of \'etale and topological fundamental groups \cite[Expos\'e XIII, Corollaire 2.12]{SGA_1}, there exist elements $\gamma_1,\ldots,\gamma_n \in \pi_1^t(U,\ast)$ such that 
\begin{itemize}
    \item $\gamma_1 \cdots \gamma_n = 1$,
    \item $\gamma_i$ topologically generates an inertia group at $t_i$, i.e. $\gamma_i = r_{t_i, Q}(\zeta_i)$ for some topological generators $\zeta_i \in \hat{\mathbb{Z}}(1)_k$ and some $Q$ above $t_i$, and
    \item $\gamma_1,\ldots,\gamma_{n-1}$ freely generate $\pi_1^t(U,\ast)$ as a profinite group.
\end{itemize}
In fact, the condition $\gamma_1 \cdots \gamma_n = 1$ applied to the extension 
$$k\left(t, \sqrt[m]{\left(t-t_{1}\right) /\left(t-t_{2}\right)}, \sqrt[m]{\left(t-t_{2}\right) /\left(t-t_{3}\right)}, \cdots, \sqrt[m]{\left(t-t_{n}\right) /\left(t-t_{1}\right)}\right)$$
of $k(t)$ shows that the $\zeta_i$ are all equal, say to some $\zeta \in \hat{\mathbb{Z}}(1)_k$. We refer to $\gamma_1,\ldots,\gamma_n$ and $\zeta$ as \emph{generator data for $U$}.

\begin{definition}[{cf. \cite[Section 7.3]{ellenberg_venkatesh_westerland}, \cite[Proof of Proposition 2.1]{ellenberg_li_shusterman}}] \label{definition:monodromy_type}
   Let $k$ be a field, let $G$ be a finite group with $\operatorname{char} k \nmid |G|$, and let $f: C \to \mathbb{P}^1$ be a tamely ramified $G$-cover over $k$ equipped with an isomorphism $\phi: G \to \operatorname{Aut}(f) \cong \operatorname{Gal}(L/k(t))$ where $L$ is the $G$-extension of $k(t)$ corresponding to $C$. Let $U \subseteq \mathbb{P}^1_{\bar{k}}$ be the unbranched locus of $f$ and let $t_1,\ldots,t_n$ be the points of $\mathbb{P}^1_{\bar{k}} \backslash U$. Choose generator data $\gamma_1,\ldots,\gamma_n \in \pi_1^t(U,\ast)$ and $\zeta \in \hat{\mathbb{Z}}(1)_{\bar{k}}$ for $U$ as above.
   

   Let $c \subseteq G$ be closed under conjugation. We say that {\em the monodromy type (with respect to the choice of generator data) of $f$ at/above $t_i$ is $c$} if the image of $\gamma_i$ under the surjection $\pi_1^t(U,\ast) \twoheadrightarrow \operatorname{Aut}(f) \xrightarrow{\phi^{-1}} G$ given by the covering $f: C_{\bar{k}} \to \mathbb{P}^1_{\bar{k}}$ is an element of $c$. If $P \in \mathbb{P}^1_{\bar{k}}$ is not a branch point of $f$, then we say that {\em the monodromy type of $f$ at/above $P$ is trivial or is $1 \in G$}. Furthermore, we say that $f$ has {\em monodromy type $c$ (with respect to the choice of generator data)} if all finite branch points $P$ of $f$ have monodromy type $c$. 
   
\end{definition}

\begin{remark}
    The notion of $f$ having monodromy of type $c$ does not specify whether $\infty$ is a branch point of $f$ and what the monodromy type above $\infty \in \mathbb{P}^1$ is. We use this convention to later consider in \Cref{section:generalizing_EVW_section_8} moduli of $G$-covers of $\mathbb{P}^1$ with monodromy of a fixed type (away from $\infty$).
\end{remark}

\begin{remark}
   Monodromy types as defined above generally depend on the choice of the generator data $\gamma_1,\ldots,\gamma_n \in \pi_1^t(U,\ast)$ and $\zeta \in \hat{\mathbb{Z}}(1)_{\bar{k}}$. A different choice of generator data replaces $c$ with $c^\alpha := \{g^\alpha: g \in c\}$ for some $\alpha$. Since these transformations form a group, there must be some $\beta$ such that $c = (c^\alpha)^{\beta}$. When discussing monodromy types, we leave the choice of generator data implicit. In the case of interest, we let $G$ be $A \rtimes \langle \zeta_d \rangle$ where $A$ is a finite $\ell$-group such that $A^{\langle \zeta_d \rangle} = 1$, we let $c = c_k$ be $\{(a, \zeta_d^k): a \in A\}$, which is a non-splitting conjugacy class by \Cref{lemma:c_is_nonsplitting_conjugacy_class}, where $k$ is a fixed integer that is relatively prime to $d$, and we consider $G$-covers of $\mathbb{P}^1$ with monodromy type $c$. Changing the choice of generator data in this case changes the monodromy type from $c_k$ to some $c_{k'}$ for some $k'$ that is also relatively prime to $d$. In fact, there is a choice of monodromy data, obtained by replacing $\zeta$ if necessary, that makes the monodromy types of the covers into $c_1$.
\end{remark}

\Cref{proposition:monodromy_types_respect_compositions} establishes that monodromy types respect subcovers when $G$ is a semi-direct product.

\begin{proposition} \label{proposition:monodromy_types_respect_compositions}
    Let $k$ be a field. Let $\Gamma, H$ be finite groups with $\Gamma$ acting on $H$ and such that $\operatorname{char} k \nmid |\Gamma|, |H|$ and let $H \rtimes \Gamma$ be the semidirect product with respect to this action. Let $f: C \to \mathbb{P}^1_k$ be a tame $(H \rtimes \Gamma)$-cover of smooth curves over a field $k$. In particular, $f$ factors as tame Galois covers $C \to C/H$ and $g: C/H \to \mathbb{P}^1_k$ of smooth curves over $k$; the former cover is an $H$-cover and the latter is a $\Gamma$-cover. Fix some generator data $\gamma_1,\ldots,\gamma_n \in \pi_1^t(U_f,\ast)$ and $\zeta \in \hat{\mathbb{Z}}(1)_{\bar{k}}$ for the unbranched locus $U_f$ of $f$ in $\mathbb{P}^1_{\bar{k}}$. 
    \begin{enumerate}
        \item If $t_{i_1},\ldots,t_{i_m}$ are the branch points of $g$, then the images of $\gamma_{i_1},\ldots,\gamma_{i_m}$ under the natural surjection $\pi: \pi_1^t(U_f, \ast) \to \pi_1^t(U_g, \ast)$ and $\zeta$ form generator data for the unbranched locus $U_g$ of $g$ in $\mathbb{P}^1_{\bar{k}}$. 
        \item Write $i_f(P)$ and $i_g(P)$ for the monodromy types of $f$ and $g$ above $P \in \mathbb{P}^1_{\bar{k}}$ with respect to the above generator data; these are conjugacy classes in $H \rtimes \Gamma$ and $\Gamma$ respectively. The image of $i_f(P)$ in $\Gamma$ is $i_g(P)$.
    \end{enumerate}
    
\end{proposition}
\begin{proof}
    \begin{enumerate}
        \item $\pi(\gamma_i)$ is trivial whenever $i \not\in \{i_1,\ldots,i_m\}$, so $\pi(\gamma_{i_1}) \cdots \pi(\gamma_{i_m}) = 1$. Moreover, $\gamma_{i_j}$ is a topological generator of an inertia subgroup of $\pi_1^t(U_f, \ast)$ at $t_{i_j}$, so its image under the natural surjection is a topological generator of an inertia subgroup of $\pi_1^t(U_g, \ast)$ at $t_{i_j}$. Lastly, under the comparison of \'etale and topological fundamental groups \cite[Expos\'e XIII, Corollaire 2.12]{SGA_1} (see also \cite[Expos\'e XII, Corollaire 5.2]{SGA_1}), one sees that $\pi(\gamma_{i_1}),\ldots,\pi(\gamma_{i_m})$ freely generate $\pi_1^t(U_g,\ast)$ as a profinite group. 
        \item If $P$ is unbranched for $f$, then $i_f(P)$ and $i_g(P)$ are both trivial. Otherwise, $i_f(P)$ is the $(H \rtimes \Gamma)$-conjugacy class of the image of $\gamma_i$ under the surjection $\pi_1^t(U_f, \ast) \to H \rtimes \Gamma$ given by $f$, where $P$ is identified with $t = t_i$. If $P$ is branched for $g$ in this case, then $i_g(P)$ is the $\Gamma$-conjugacy class of the image of $\pi(\gamma_i)$ under the surjection $\pi_1^t(U_g, \ast) \to \Gamma$ given by $g$. This coincides with the image of $\gamma_i$ under the composition $\pi_1^t(U_f, \ast) \to H \rtimes \Gamma \to \Gamma$. If $P$ is unbranched for $g$ instead, then $i_g(P)$ is trivial by definition. Moreover, $\pi(i_f(P))$ is the image of $\gamma_i$ under the composition $\pi_1^t(U_f, \ast) \to H \rtimes \Gamma \to \Gamma$, which equals the composition $\pi_1^t(U_f, \ast) \xrightarrow{\pi} \pi_1^t(U_g, \ast) \to \Gamma$. Thus, $\pi(i_f(P))$ is trivial. In any case, the image of $i_f(P)$ in $\Gamma$ is $i_g(P)$.
    \end{enumerate}
\end{proof}

The ramification behavior above points can be calculated from the monodromy types.
\begin{proposition} \label{proposition:ramification_index_from_monodromy_type}
   Let $f: C \to \mathbb{P}^1_k$ be a tame $G$-cover of smooth curves over a field $k$ where $\operatorname{char} k \nmid |G|$. Given a closed point $P \in \mathbb{P}^1_k$ and a point $Q \in C$ such that $f(Q) = P$, the ramification index of $Q$ over $P$ equals the order of the monodromy type of $f$ at $Q$.
\end{proposition}
\begin{proof}
   It suffices to prove this when $k$ is algebraically closed. The ramification index of $Q$ over $P$ equals $\# I(Q|P) \cdot [\kappa(Q): \kappa(P)]_i$ where $I(Q|P)$ is the inertia group of $Q$ over $P$ and $[\kappa(Q): \kappa(P)]_i$ is the inseparable degree of the residue field at $Q$ over the residue field at $P$ (see, for example, \cite[Corollary 7.10]{sutherland18785fall2017}). Since $\operatorname{char} k \nmid |G|$, the inseparable degree is $1$. Moreover, $I(Q|P)$ is generated by the monodromy type of $f$ at $Q$ (with respect to a chosen generator of $\hat{\mathbb{Z}}(1)_k$ and isomorphism $\operatorname{Aut} f \cong G$). Therefore, the ramification index of $Q$ over $P$ equals the order of this monodromy type. 
\end{proof}

The moduli of $G$-covers of $\mathbb{P}^1$ parameterizing monodromy types in $c$ discussed later in \Cref{section:generalizing_EVW_section_8} will be rational over the base finite field $\mathbb{F}_q$ as long as $c$ is rational as below:

\begin{definition} \label{definition:rational_conjugacy_class}
    Let $G$ be a finite group and $c$ be a union of conjugacy classes in $G$. We say that $c$ is \emph{rational} if $c$ is closed under the assignment $x \mapsto x^N$ for all $N$ relatively prime to $|G|$. Similarly, given a prime power $q$, we say that $c$ is \emph{$\mathbb{F}_q$-rational} if $c$ is closed under the assignment $x \mapsto x^q$.
\end{definition}

\subsection{Monodromy for Kummer extensions of \texorpdfstring{$\mathbb{P}^1$}{P1}}

Let $G = \mathbb{Z}/d\mathbb{Z}$ for an integer $d \geq 2$, and suppose that $k$ is a field with primitive $d$th-roots of unity. By Kummer theory, $\mathbb{Z}/d\mathbb{Z}$-covers of $\mathbb{F}_q(t)$ are of the form $\Fqt[y]/(y^d - f(t))$ where $f(t) \in \Fqt$. Equivalently, the $\mathbb{Z}/d\mathbb{Z}$-covers of $\mathbb{P}^1_{\Fq}$ are smooth completions of the curve given by $y^d = f(t)$ and in fact these covers are tame. We can take $f(t) \in \Fq[t]$ to be a polynomial without any $d$th or higher power factors. In this subsection, we describe monodromy types for such $\mathbb{Z}/d\mathbb{Z}$-covers of $\mathbb{P}^1$.

\begin{lemma} \label{lemma:simple_covers_of_P_1}
    Let $k$ be a field, let $d \geq 2$ be an integer not divisible by $\operatorname{char} k$, and let $\varphi: C \to \mathbb{P}^1_k$ be the covering map from the smooth completion $C$ of the affine plane curve $C_A$ given by $y^d = f(t)$ where $f(t) \in k(t)$. Assume that $C$ is geometrically connected. Without loss of generality, say that $f(t)$ is a polynomial without any $d$th or higher power factors. Write $f(t) = a \prod_{i=1}^n f_i(t)^{e_i}$ where $a \in k^\times$ and $f_i(t)$ are monic irreducible polynomials. For each $i$, $C$ has exactly $\gcd(d,e_i)$ distinct points over the vanishing locus of $f_i(t)$ in $\mathbb{P}^1_k$. Moreover, $C$ has exactly $\gcd(d, \deg f)$ distinct points over $\infty$. Furthermore, the rational function $y$ on $C$ 
    \begin{enumerate}
        \item has vanishing order $\frac{e_i}{\gcd(d, e_i)}$ at each of the points above the vanishing point of $f_i(t)$ in $\mathbb{P}^1_k$, \label{part:vanishing_of_y_on_simple_cover_of_P1_is_given_by_exponent}
        \item has poles of order $\frac{\deg f}{\gcd(d, \deg f)}$ at the points above $\infty$, and \label{part:poles_of_y_on_simple_cover_of_P1_over_infinity}
        \item does not vanish and does not have a pole at other points on $C$. \label{part:y_does_not_vanish_at_unramified_points_of_simple_cover_of_P1}
    \end{enumerate}
    Moreover, $\varphi$ is branched above $\infty$ if and only if $\deg f \not\equiv 0 \pmod{d}$.
\end{lemma}
\begin{proof}
    It suffices to demonstrate these results in the case that $k$ is algebraically closed. In this case, $f_i(t)$ can be written as $t-t_i$ where $t_i \in k$. For each $j$, let $f_j$ be a positive integer such that $e_j f_j \equiv \gcd(d,e_j) \pmod{d}$; write $e_jf_j = \gcd(d,e_j) + dk_j$. For each $j$, the equality $y^d = f(t) = a(t-t_j)^{e_j}\prod_{i \neq j} (t-t_i)^{e_i}$ of rational functions of $C$ is equivalent to 
        $$y_{j,1}^d = a^{f_j} (t-t_j)^{\gcd(d,e_j)} \prod_{i \neq j} (t-t_i)^{e_i f_j}$$
    where $y_{j,1} = \frac{y^{f_j}}{(t-t_j)^{k_j}}$.
    In turn, this equation is equivalent to
        \begin{equation} \label{eq:affine_equation_for_smooth_curve_for_simple_cover_of_P1}
            y_{j,2}^{\gcd(d,e_j)} = a^{f_j} \prod_{i \neq j}(t-t_i)^{e_i f_j}
        \end{equation}
    where $y_{j,2} = \frac{y_{j,1}^{d/\gcd(d,e_j)}}{t-t_j}$. The affine curve given by the equation \eqref{eq:affine_equation_for_smooth_curve_for_simple_cover_of_P1} has $\gcd(d,e_j)$ smooth points above $t = t_j$. At each of these smooth points $Q$, the rational function $t-t_j$ vanishes, so $y_{j,1}$ also vanishes because $y_{j,2} = \frac{y_{j,1}^{d/\gcd(d,e_j)}}{t-t_j}$. In particular, the vanishing order of $t-t_j$ at $Q$ is a multiple of $\frac{d}{\gcd(d,e_j)}$. This vanishing order equals the ramification index of $Q$, and the sum of the ramification indices of the smooth points above $t = t_j$ must be exactly $d$, so in fact these smooth points $Q$ are exactly the points of $C$ above $t = t_j$, the function $y_{j,1}$ is a uniformizer at each $Q$, and $t-t_j$ vanishes with order exactly $\frac{d}{\gcd(d,e_j)}$.

    Since $y_{j,1} = \frac{y^{f_j}}{(t-t_j)^{k_j}}$, the vanishing order of $y^{f_j}$ at each of these $Q$ is exactly $1 + \frac{dk_j}{\gcd(d,e_j)} = \frac{e_j f_j}{\gcd(d,e_j)}$. Thus, the vanishing order of $y$ at these $Q$ is $\frac{e_j}{\gcd(d,e_j)}$.

    Note that the affine curve given by $y^d = f(t)$ is smooth at the points that are not above $t = t_i$ and not above $\infty$. Therefore, such points are points of $C$. Moreover, $y$ does not vanish and does not have poles at such points because $f(t)$ does not vanish and does not have poles away from $t = t_i$ and $t = \infty$.

    To prove \eqref{part:poles_of_y_on_simple_cover_of_P1_over_infinity}, write $s = \frac{1}{t}$ so that $t = \infty$ corresponds to $s = 0$. The equation $y^d = f(t)$ is equivalent to 
    \begin{align}
        y^d &= f\left(\frac{1}{s} \right) \nonumber \\ 
        (s^{\lceil \frac{\deg f}{d} \rceil} y)^d  &= s^{d \lceil \frac{\deg f}{d} \rceil} f\left(\frac{1}{s} \right). \label{eq:affine_curve_equation_above_infinity}
    \end{align}
    Note that $s^{d \lceil \frac{\deg f}{d} \rceil} f\left(\frac{1}{s} \right)$ is a polynomial in $s$ of degree $d \lceil \frac{\deg f}{d} \rceil$, and has exactly $d \lceil \frac{\deg f}{d} \rceil - \deg f$ factors of $s$. Part \eqref{part:vanishing_of_y_on_simple_cover_of_P1_is_given_by_exponent} shows that $C$ has exactly $\gcd(d, d \lceil \frac{\deg f}{d} \rceil - \deg f) = \gcd(d, \deg f)$ distinct points over $t = \infty$. Moreover, the rational function $s^{\lceil \frac{\deg f}{d} \rceil} y$ vanishes with order exactly $\frac{d \lceil \frac{\deg f}{d} \rceil - \deg f}{\gcd(d, \deg f)}$ above each point over $s = 0$. As a rational function of $C$, the proof of part \eqref{part:vanishing_of_y_on_simple_cover_of_P1_is_given_by_exponent} above establishes that $s$ vanishes at each point over $s = 0$ with order $\frac{d}{\gcd(d, \deg f)}$. Therefore, the vanishing order of $y$ at each of these points is $-\frac{\deg f}{\gcd(d, \deg f)}$. In other words, $y$ has a pole of order $\frac{\deg f}{\gcd(d, \deg f)}$ at each of these points.
    
\end{proof}

\begin{proposition} \label{proposition:monodromy_type_for_kummer_curves}
    Let $d \geq 2$ be an integer. Let $k$ be a field such that $\operatorname{char} k \nmid d$ and let $\varphi:C \to \mathbb{P}^1_k$ be a tame $\mathbb{Z}/d\mathbb{Z}$-cover over $k$. Let $\zeta = (\zeta_m)_m \in \hat{\mathbb{Z}}(1)_{\bar{k}}$ and $\gamma_1,\ldots,\gamma_n \in \pi_1^t(U,\ast)$ be a choice of generator data for $U$, where $U \subseteq \mathbb{P}^1_{\bar{k}}$ is the complement of the branch locus of $\varphi \times_k \bar{k}$. After base changing to an extension $k_d$ of $k$ with a primitive $d$th-root of unity, say that $C$ is given by the smooth completion of the affine plane curve given by $y^d = f(t)$ where $f(t) \in k_d[t]$ is a polynomial without any $d$th or higher power factors. After this base change, further say that $\phi: \mathbb{Z}/d\mathbb{Z} \to \operatorname{Aut}(\varphi)$ is chosen to send $1 \in \mathbb{Z}/d\mathbb{Z}$ to the automorphism of $C$ given by $(t,y) \mapsto (t, \zeta_d y)$. Write $f(t) = a \prod_{i=1}^n f_i(t)^{e_i}$ where $a \in k_d^\times$, $f_i(t)$ are monic irreducible polynomials, and $1 \leq e_i \leq d-1$. The monodromy type of $\varphi$ above the vanishing point of $f_i(t)$ in $\mathbb{P}^1_k$ is $e_i \in \mathbb{Z}/d\mathbb{Z}$, and the monodromy type of $\varphi$ above $\infty \in \mathbb{P}^1_k$ is $-\deg f$.
\end{proposition}

\begin{proof}
    It suffices to demonstrate the desired results in the case that $k$ is algebraically closed. In this case, write $z_i = f_i(t) = t-t_{0,i}$ for $t_{0,i} \in k$ and let $P \in \mathbb{P}^1_k$ be the (finite) point corresponding to $t_{0,i}$. Writing $L = k(t)[y]/(y^d-f(t))$ for the function field of $C$, consider any $k(t)$-embedding $\iota: k(t)[y]/(y^d-f(t)) \hookrightarrow k((z_i^{1/\infty}))$ --- such an embedding is given by sending $y$ to $\alpha \cdot (t-t_{0,i})^{\frac{e_i}{d}} \cdot w$ where $\alpha$ is any choice of $d$th root of $a$ in $k = \bar{k}$, and $w$ is any choice of $d$th root of $\prod_{j \neq i}(t - t_{0,j})^{e_j}$ in $k((z))$. Under the isomorphism \eqref{eq:isomorphism_of_galois_group_of_puiseux_series_field_with_roots_of_unity}, $\zeta$ corresponds to the element $\sigma$ of $\operatorname{Gal}(k((z_i^{1/\infty})) / k((z_i)))$ given by $z_i^{1/m} \mapsto \zeta_m z_i^{1/m}$. 
    In turn, $\sigma$ sends $\alpha (t-t_{0,i})^{\frac{e_i}{d}} \cdot w$ to $\alpha \cdot (\zeta_d)^{e_i} (t-t_{0,i})^{\frac{e_i}{d}} \cdot w = \alpha \cdot \zeta_d^{e_i} (t-t_{0,i})^{\frac{e_i}{d}} \cdot w$. In other words, the homomorphism $\operatorname{Gal}(k((z_i^{1/\infty})) / k((z_i))) \to \operatorname{Gal}(L/k(t)) \cong \operatorname{Aut} \varphi$ corresponding to $\iota$ is given by sending $\sigma$ to the $k(t)$ automorphism given by sending $y$ to $\zeta_d^{e_i} y$. Under $\phi$, this automorphism corresponds to $e_i \in \mathbb{Z}/d\mathbb{Z}$. Therefore, the composition
    \begin{center}
        \begin{tikzcd}
            \hat{\mathbb{Z}}(1)_k \ar[r, "\cong"] \ar[rr, bend left, "r_{t_{0,i}} = r_{t_{0,i}, Q}"] & \operatorname{Gal}(k((z^{1/\infty})) / k((z))) \ar[r] & \operatorname{Gal}(L^{\text{max}}/k(t)) \ar[r, "\cong"] & \pi_1^t(U, \ast)   \\ \ar[r, twoheadrightarrow] & \operatorname{Gal}(L/k(t)) \ar[r, "\cong"] & \operatorname{Aut}(\varphi) \ar[r, "\phi^{-1}"] & \mathbb{Z}/d\mathbb{Z}
        \end{tikzcd}
    \end{center}
    sends $\zeta$ to $e_i \in \mathbb{Z}/d\mathbb{Z}$. Within the above series of maps, $\zeta$ is sent to $\gamma_i$ in $\pi_1^t(U,\ast)$, so the monodromy type of $\varphi$ above $t=t_{0,i}$ is $e_i \in \mathbb{Z}/d\mathbb{Z}$.

    Let $s = \frac{1}{t}$, as in equation \eqref{eq:affine_curve_equation_above_infinity}, to see that the monodromy type of $\varphi$ above $s = 0$ is $d \lceil \frac{\deg f}{d} \rceil - \deg f$, which is congruent to $-\deg f$ modulo $d$. Thus, the monodromy type of $\varphi$ above $\infty$ is $-\deg f$.
\end{proof}

\begin{remark}
    The sum of the monodromy types above the branch points (counted with multiplicity) is $0$ modulo $d$. This is a consequence of the relations on generators $\gamma_i$ of the inertia groups at the branch points as discussed in \Cref{subsection:definitions_of_tame_G_covers}.
\end{remark}

\begin{remark}
    In the situation of \Cref{proposition:monodromy_type_for_kummer_curves}, changing the generator data and the choice of $\phi$ scales the monodromy types uniformly across the branch points. More precisely, if $(e_1,\ldots,e_n)$ is the $n$-tuple of monodromy types of $\varphi$ at the finite points $t=t_{0,1},\ldots,t_{0,n}$, then the tuple becomes $(\alpha \cdot e_1,\ldots,\alpha \cdot e_n)$ after changing the generator data and/or $\phi$ for some $\alpha \in (\mathbb{Z}/d\mathbb{Z})^\times$. In other words, the tuple $(e_1,\ldots,e_n)$ is a well defined element of the set of $(\mathbb{Z}/d\mathbb{Z})^\times$-orbits of $\{(g_1,\ldots,g_n) \in (\mathbb{Z}/d\mathbb{Z}\setminus\{0\})^n : \sum g_i = 0\}$.
    This element 
    also coincides with the tuple $(e_1',\ldots,e_n')$ of canonical generators of inertia at $Q_1,\ldots,Q_n$ where $Q_i$ is a choice of ramification point in $C$ above the vanishing loci of $f_1,\ldots,f_n$.
\end{remark}

\subsection{Imaginary cyclic extensions of \texorpdfstring{$\Fqt$}{Fqt}}

\begin{definition}\label{definition:imaginary}
    We say that a finite extension of $k(t)$ is \emph{imaginary} (resp. \emph{totally imaginary}) if it is ramified (resp. totally ramified) above $\infty$.
\end{definition}

We will focus on $\mathbb{Z}/d\mathbb{Z}$-extensions of $\Fqt$ in later chapters, particularly ones which are totally imaginary or not imaginary. The below proposition shows that totally imaginary $\mathbb{Z}/d\mathbb{Z}$ extensions exist only when $q \equiv 1 \pmod{d}$ which is the case exactly when $\Fq$ has primitive $d$th roots of unity.


\begin{proposition} \label{proposition:imaginary_cyclic_ell_power_extensions_exist_only_when_Fq_has_ellth_roots_of_unity}
    Let $q$ be a prime power. Let $d \geq 2$ be an integer that is coprime to $q$. Suppose that $L$ is a cyclic extension of $K = \Fqt$ of degree $d$. Let $\mathfrak{q}$ be a place of $L$ that lies over a place $\mathfrak{p}$ of $K$ with ramification degree $e$. If the residue field of $\mathfrak{p}$ is isomorphic to $\mathbb{F}_{q^a}$, then $q^a \equiv 1 \pmod{e}$. In particular, if $L$ is totally imaginary, then $q \equiv 1 \pmod{d}$, or equivalently, $\Fq$ has primitive $d$th roots of unity. 
    
\end{proposition}
\begin{proof}

    Let $\hat{K}_{\mathfrak{p}}$ and $\hat{L}_{\mathfrak{q}}$ denote the completions of $K$ and $L$ with respect to $\mathfrak{p}$ and $\mathfrak{q}$ respectively. Let $K_\mathfrak{p}^{\text{unr}}$ be the maximal abelian unramified extension of $\hat{K}_\mathfrak{p}$, and let $K_\mathfrak{p}^{\text{tr}}$ be the maximal tamely ramified extension of $\hat{K}_\mathfrak{p}$. The Galois group of $K_\mathfrak{p}^{\text{tr}}$ over $K_\mathfrak{p}^{\text{unr}}$ is isomorphic to $\Pi := \prod_{\substack{\ell \text{ prime} \\ \ell \nmid q}} \mathbb{Z}_\ell$ and  the Galois group of $K_\mathfrak{p}^{\text{unr}}$ over $\hat{K}_\mathfrak{p}$ is the cyclic group generated by $\operatorname{Frob}_\mathfrak{p}$. Therefore, the Galois group of $K_\mathfrak{p}^{\text{tr}}$ over $K_\mathfrak{p}$ is the semidirect product $\Pi \rtimes \langle \operatorname{Frob}_\mathfrak{p} \rangle$, where $\operatorname{Frob}_\mathfrak{p}$ acts on $\Pi$ by $q^a$.
    
    Since $d$ and $q$ are coprime, $\mathfrak{q}$ is tamely ramified over $\mathfrak{p}$, so $\hat{L}_{\mathfrak{q}} \subseteq K_\mathfrak{p}^{\text{tr}}$. Moreover, $\operatorname{Gal}(\hat{L}_\mathfrak{q} / \hat{K}_\mathfrak{p})$ is identified with the decomposition group $D(\mathfrak{q}|\mathfrak{p}) \subseteq \operatorname{Gal}(L/K) \cong \mathbb{Z}/d\mathbb{Z}$, and this decomposition group is of order $ef$, where $f$ is the degree of $\kappa(\mathfrak{q})/\kappa(\mathfrak{p})$. Write $g = \frac{d}{ef}$ so that the natural restriction $\operatorname{Gal}(K_\mathfrak{p}^{\text{tr}} / \hat{K}_\mathfrak{p}) \twoheadrightarrow \operatorname{Gal}(\hat{L}_\mathfrak{q} / \hat{K}_\mathfrak{p})$ is identified with a surjection 
    $$\tau : \Pi \rtimes \langle \operatorname{Frob}_\mathfrak{p} \rangle \twoheadrightarrow g \mathbb{Z}/d\mathbb{Z},$$
    factoring through a composition $\Pi \cong \operatorname{Gal}(K_{\mathfrak{p}}^{\text{tr}} / \hat{K}_\mathfrak{p}^{\text{unr}}) \twoheadrightarrow \operatorname{Gal}(\hat{L}_\mathfrak{q} / \hat{L}_\mathfrak{q} \cap K_\mathfrak{p}^{\text{unr}}) \hookrightarrow \operatorname{Gal}(\hat{L}_\mathfrak{q} / \hat{K}_\mathfrak{p}) \cong g \mathbb{Z} / d\mathbb{Z}$. Note that $[\hat{L}_\mathfrak{q} \cap K_\mathfrak{p}^{\text{unr}}:\hat{K}_\mathfrak{p}] = f$, so the image of $\Pi$ in $\operatorname{Gal}(\hat{L}_\mathfrak{q} / \hat{K}_\mathfrak{p})$ is identified with $fg \in \mathbb{Z}/d\mathbb{Z}$. 

    Let $\gamma \in \Pi$ be an element such that $\tau(\gamma) = fg \in \mathbb{Z}/d\mathbb{Z}$. On the one hand, $\gamma^{\operatorname{Frob}_{\mathfrak{p}}} = q^a \gamma$, so $\tau(\gamma^{\operatorname{Frob}_\mathfrak{p}}) = q^a \cdot \tau(\gamma) = q^a \cdot fg \in \mathbb{Z}/d\mathbb{Z}$. On the other hand, $\tau(\gamma^{\operatorname{Frob}_{\mathfrak{p}}}) = \tau(\operatorname{Frob}_{\mathfrak{p}}) + \tau(\gamma) - \tau(\operatorname{Frob}_{\mathfrak{p}}) = \tau(\gamma) = fg$. Therefore, $q^a \cdot fg = fg$ in $\mathbb{Z}/d\mathbb{Z}$. Since the order of $fg$ as an element of $\mathbb{Z}/d\mathbb{Z}$ is $e$, $q^a \equiv 1 \pmod{e}$.
    
\end{proof}

\section{Hermitian Spaces over \texorpdfstring{$\mathbb{Z}_\ell[\zeta_d]$}{Zell zetad}}
\label{section:hermitian_spaces}

For prime numbers $\ell$ and integers $d \geq 2$ not divisible by $\ell$, \Cref{proposition:symplectic_pairings_to_hermitian_forms} induces Hermitian forms over $\mathbb{Z}_\ell[\zeta_d]_{\zeta_d \neq 1}$ from the $\ell$-adic Weil pairings on Jacobians of $\mathbb{Z}/d\mathbb{Z}$-covers of $\mathbb{P}^1$. We also prove \Cref{corollary:generalizing_EVW_Lemma 8.9_to_Z_dZ_covers}, which concerns the theory of such Hermitian forms. Since $d$ is not divisible by $\ell$, $\mathbb{Z}_\ell[\zeta_d]_{\text{all}}$ does not ramify over $\mathbb{Z}_\ell$ and $\mathbb{Z}_\ell[\zeta_d]_{\text{all}} \cong \prod_{d' \mid d} \mathbb{Z}_\ell[\zeta_{d'}]_{\text{prim}}$.

For a discussion of the theory of Hermitian forms over general involution rings, see \cite{knus_qhfr}.
\begin{definition}
    Let $R$ be a (not necessarily commutative) ring with an involution $\overline{\cdot}: R \to R$. The \emph{opposite module} $\overline{M}$ of the left $R$-module $M$ is the additive group $M$ with the right $R$-module structure
    $$mr = \overline{r} m$$
    for all $r \in R$ and $m \in M$.
    The \emph{dual module} $M^\vee$ of the right $R$-module $M$ is the abelian group $\operatorname{Hom}_R(M,R)$ that is a left $R$-module under the action
    $$(rf)(m) = rf(m)$$
    for all $r \in R$ and $m \in M$.
    The \emph{transpose module} $M^*$ of a (right) $R$-module $M$ is defined to be $\overline{M^\vee}$.
\end{definition}

\begin{definition}[{cf. \cite[Chapter I, 2.2, 3.1, 3.2]{knus_qhfr}}] \label{definition:hermitian_form}
    Let $R$ be a ring with an involution $\overline{\cdot}: R \to R$. Let $\varepsilon$ be an element of the center $Z(R)$ of $R$ such that $\varepsilon \cdot \overline{\varepsilon} = 1$. Let $V$ be an $R$-module.
    A {\em Hermitian form} (or more precisely, an \emph{$\varepsilon$-Hermitian form}) on $V$ is a biadditive pairing $H:V \times V \to R$ such that 
    \begin{enumerate}
        \item $H(av, bw) = \overline{a} H(v,w) b$ for all $v,w \in V$ and $a,b \in R$, and
    
        \item $H(v,w) = \varepsilon\overline{H(w,v)}$ for all $v,w \in V$.
    \end{enumerate}
    We will often say that $(V,H)$ is a \emph{module (equipped) with a Hermitian form}. In contrast, when $V$ is a finitely-generated projective $R$-module, we say that $(V,H)$ is a \emph{Hermitian module}. For a Hermitian module $(V,H)$, the form $H$ is said to be {\em nonsingular} if the adjoint map $V \to V^*, v \mapsto H(v, -)$ is an $R$-module isomorphism, and the pair $(V,H)$ is called a \emph{Hermitian space}.
    
    A morphism $(V_1, H_1) \to (V_2, H_2)$ of modules equipped with Hermitian form is an isometry, i.e. an $R$-linear morphism $\varphi: V_1 \to V_2$ such that $H_2(\varphi(v), \varphi(w)) = H_1(v, w)$ for all $v,w \in V_1$.
\end{definition}

For a commutative ring $R$ equipped with trivial involution, note that a symmetric bilinear form over $R$ is equivalent to a $1$-Hermitian form over $R$ and that an alternating form over $R$ is equivalent to a $(-1)$-Hermitian form over $R$.  

We will only need the theory of Hermitian spaces for our applications. We nevertheless present some basic algebraic facts about modules equipped with Hermitian forms.

\begin{remark}
    It is common to use the opposite convention for sesquilinearity, i.e. it is common to define a Hermitian form $H$ to satisfy $H(av, bw) = aH(v,w) \overline{b}$.
\end{remark}

\begin{remark} \label{remark:1_hermitian_forms}
    We mostly need the theory of $\varepsilon_{\mathbb{Z}_\ell[\zeta_d]}$-Hermitian forms where $\ell \nmid d$. As \Cref{lemma:hermitian_modules_over_product_of_involution_rings_split_along_involution_ring_factors} shows, such Hermitian forms correspond to $1$-Hermitian forms over factors $D_i$ as in \eqref{eq:Z_ell_zeta_d_factorization_as_ring_with_involution} where $\zeta_d \neq -1$ and to $(-1)$-Hermitian forms over $\mathbb{Z}_\ell[X]/(X+1)$, the latter of which can occur only when $d$ is even. Since the involution on $\mathbb{Z}_\ell[X]/(X+1)$ is trivial, $(-1)$-Hermitian forms over it are skew-symmetric bilinear forms.
    
    
    
    
\end{remark}

We will use some facts about orthogonality for Hermitian modules.

\begin{definition}[{cf. \cite[Chapter I, 3.4, 3.6]{knus_qhfr}}]
    Let $R$ be a ring with involution, let $\varepsilon \in Z(R)$ satisfy $\varepsilon \cdot \overline{\varepsilon} = 1$. Let $(V,H)$ be an $R$-module with a $\varepsilon$-Hermitian form. 
    \begin{enumerate}
        \item Let $(V_i, H_i)$ be $R$-modules with $\varepsilon$-Hermitian forms for $i = 1,2$. The \emph{orthogonal sum} $(V_1, H_1) \perp (V_2, H_2)$ is defined as the pair
        $$\left( V_1 \oplus V_2, H_1 \perp H_2 \right)$$
        where $H_1 \perp H_2$ is the Hermitian form uniquely characterized by 
        $$H(v,w) = \begin{cases} H_i(v,w) &\text{if } v,w \in V_i \\ 0 &\text{if } v \in V_1, w \in V_2\end{cases}. $$
        In particular,
        $$H(v_1 + v_2, w_1 + w_2) = H_1(v_1, w_1) + H_2(v_2, w_2)$$
        for all $v_1,w_1 \in V_1$ and $v_2,w_2 \in V_2$.
        \item Two subsets $X \subset V$ and $Y \subset V$ are said to be \emph{orthogonal} if $H(x,y) = 0$ for all $x \in X$ and $y \in Y$.
        \item Given a submodule $U \subset V$, we define the \emph{orthogonal complement} $U^\perp$ by
        $$U^\perp = \{x \in V: H(x,u) = 0 \text{ for all } u \in U \}.$$
        \item $(V,H)$ is said to be \emph{diagonalizable} if $V$ is a free $R$-module and there exists a basis $v_1,\ldots,v_n$ of $V$ such that the matrix of $H$ with respect to $v_1,\ldots,v_n$ is diagonal, i.e. $H(v_i, v_j) = 0$ for all $i \neq j$. Such a basis is called an \emph{orthogonal basis}. If $H(v_i, v_i) = 1$ for all $i$, then the orthogonal basis is called an \emph{orthonormal basis}.
    \end{enumerate}
\end{definition}

\begin{definition}
    Let $R$ be a ring with involution, let $\varepsilon \in Z(R)$ satisfy $\varepsilon \cdot \overline{\varepsilon} = 1$. Let $(V,H)$ be an $R$-module with a $\varepsilon$-Hermitian form.

    We say that a submodule $U$ of $(V,H)$ is \emph{totally isotropic} if $H|_U = 0$ or equivalently if $U \subseteq U^\perp$.
\end{definition}

\begin{lemma}[{\cite[Chapter I, Lemma 3.6.2]{knus_qhfr}}] \label{lemma:knus_qhfr_lemma_3.6.2} 
    Let $R$ be an involution ring, let $\varepsilon \in Z(R)$ satisfy $\varepsilon \cdot \overline{\varepsilon} = 1$, and let $(V,H)$ be a Hermitian module over $R$.
    \begin{enumerate}
        \item Let $U$ be a submodule of $V$ which is finitely generated and projective. If $(U, H|_U)$ is nonsingular, then $(V,H) = (U, H|_U) \perp (U^\perp, H|_{U^\perp})$.
        \item If $(V,H) \simeq (V_1, H_1) \perp (V_2, H_2)$ for some Hermitian modules $(V_i, H_i)$ for $i = 1,2$, then the $H_i$ are nonsingular if and only if $H$ is nonsingular.
    \end{enumerate}
\end{lemma}

In the case of interest, $R = A[\zeta_d] = A[x]/(1+x+ \cdots + x^{d-1})$, for a ring $A$ and $\overline{\cdot}: R \to R$ is the $A$-algebra involution homomorphism given by $\zeta_d \mapsto \zeta_d^{-1}$. More specifically, we will apply the theory of Hermitian forms of $A[\zeta_d]$ in the case that $A = \mathbb{Z}_\ell$ ---  \Cref{corollary:generalizing_EVW_Lemma 8.9_to_Z_dZ_covers} is concerned with a $\mathbb{Z}_\ell[\zeta_d]$-module with a perfect symplectic pairing $\omega: V \times V \to \mathbb{Z}_\ell$ that is preserved by $\zeta_d$, i.e. $\omega(\zeta_d v, \zeta_d w) = \omega(v,w)$ for all $v,w \in V$. We show that $\omega$ yields a nonsingular Hermitian form $H: V \times V \to \mathbb{Z}_\ell[\zeta_d]$ in \Cref{proposition:symplectic_pairings_to_hermitian_forms}. Before doing so, we first prove \Cref{lemma:hermitian_modules_over_product_of_involution_rings_split_along_involution_ring_factors}, which is convenient for understanding Hermitian forms over finite products of involution rings.

\begin{lemma} \label{lemma:hermitian_modules_over_product_of_involution_rings_split_along_involution_ring_factors}
    Let $R$ be a ring with involution, and let $\varepsilon \in Z(R)$ satisfy $\varepsilon \cdot \overline{\varepsilon} = 1$. Decompose $R$ as a finite product $\prod_{i=1}^n R_i$ of involution rings, letting $\varepsilon \in R$ correspond to $\varepsilon_i \in R_i$. In particular, $\varepsilon_i \in Z(R_i)$ and $\varepsilon_i \cdot \overline{\varepsilon}_i = 1$. 
    
    Modules with $\varepsilon$-Hermitian forms over $R$ correspond to tuples of modules with $\varepsilon_i$-Hermitian forms over $R_i$. More precisely, letting $V$ be any $R$-module, and writing $V_i = R_i V$, there is a bijection
        \begin{align*}
        \left\{ \begin{array}{l} \text{Tuples } (H_i)_i \text{ where } \\ H_i: V_i \times V_i \to R_i \\ \text{are } \varepsilon_i\text{-Hermitian forms} \end{array} \right\} &\to \left\{ \begin{array}{l} \varepsilon\text{-Hermitian forms } \\ H: V \times V \to R \end{array}\right\}  \\
        (H_i)_i &\mapsto \perp_i H_i,
        \end{align*}
    where we regard $H_i: V_i \times V_i \to R_i$ as a $\varepsilon$-Hermitian form over $R$ via the injection $R_i \hookrightarrow \prod_i R_i \cong R$ sending an element $d_i \in R_i$ to the element of $\prod_i R_i$ whose $i$th component is $d_i$ and whose other components are $0$. 
    
    Furthermore, given that each $V_i$ is a finitely generated, projective $R_i$-module (equivalently $V$ is a finitely generated, projective $R$-module), the $H_i$ are all nonsingular if and only if $\perp_i H_i$ is nonsingular.
\end{lemma}
\begin{proof}
    It suffices to prove this in the case that $R \cong R_1 \times R_2$. We show that claimed bijection has an inverse map that is the map sending a Hermitian form $H: V \times V \to R$ to the pair $(H|_{V_1}, H|_{V_2})$.

    For $i = 1,2$, write $e_i$ for \CrefAndHyperrefIfExist{definition:complete_orthogonal_idempotents_of_a_ring}{complete orthogonal idempotents} of $R_1 \times R_2$.
    Given a $\varepsilon$-Hermitian form $H: V \times V \to R$, note that $H|_{V_1}: V_1 \times V_1 \to R$ are respectively valued in $R_1$ because 
        \begin{align*}
        H|_{V_1}(v,w) = H(e_1 v, e_1 w) = \overline{e}_1 \cdot e_1 H(v, w) = e_1 H(v,w).
        \end{align*}
    In fact, $H|_{V_1}$ is an $\varepsilon_1$-Hermitian form over $R_1$ because
        \begin{align*}
            H|_{V_1}(v,w) = H(e_1 v, e_1 w) = e_1 H(e_1 v,e_1 w) = e_1 \varepsilon \overline{H(e_1 w, e_1 v)} = \varepsilon_1 \overline{H|_{V_1}(w,v)}.
        \end{align*}
    By symmetry, $H|_{V_2}: V_2 \times V_2 \to R$ is valued in $R_2$ and is an $\varepsilon_2$-Hermitian form over $R_2$. Moreover, $V \cong V_1 \oplus V_2$ as $R$-modules, and for any $v_1 \in V_1$ and $v_2 \in V_2$, 
        \begin{align*}
        H(v_1, v_2) = H(e_1 v_1, e_2 v_2) = \overline{e_1 } \cdot e_2 H(v_1, v_2) = e_1 \cdot e_2 H(v_1,v_2) = 0,
        \end{align*}
    so $H = H|_{V_1} \perp H|_{V_2}$.
    
    On the other hand, given $H_i: V_i \times V_i \to R_i$, we show that 
    $$H_i: V_i \times V_i \to R_i \hookrightarrow R_1 \times R_2 \cong R$$ 
    is an $\varepsilon$-Hermitian form over $R$. For all $v,w \in V_i$, 
        \begin{align*}
            H(v,w) = \varepsilon_i \overline{H(w,v)} = \varepsilon \overline{H(w,v)}.
        \end{align*}
    For any $v,w \in V_i$ and $a,b \in R$, since $e_i \cdot a$ and $e_i \in b$ are identifiable as elements of $R_i$,
        \begin{align*}
        H_i(av, bw) = H_i(e_i a v, e_i b w) = \overline{(e_i \cdot a)} H_i(v, w) (e_i \cdot b) = \overline{a} H_i(v,w) b,
        \end{align*}
    so $H_i$ is indeed a Hermitian form over $R$. Note that $(H_1 \perp H_2)_{V_i} = H_i$, so the claimed inverse map is indeed the inverse map to the claimed bijection as desired. By \Cref{lemma:knus_qhfr_lemma_3.6.2}, the $H_i$ are all nonsingular if and only if $\perp_i H_i$ is nonsingular.
\end{proof}

\subsection{Hermitian forms induced from the Weil pairing of a superelliptic curve}

\Cref{proposition:symplectic_pairings_to_hermitian_forms} discusses that Hermitian forms over $A[\zeta_d]$ correspond to symplectic pairings of $A[\zeta_d]$-modules over $A$ that are preserved by $\zeta_d$. 
\begin{definition}
    Let $A$ be a commutative ring. Let $V$ be a module over $A[\zeta_d]$. We say that a bilinear pairing $\omega: V \times V \to A$ over $A$ is \emph{preserved by $\zeta_d$} if $\omega(\zeta_d v, \zeta_d w) = \omega(v,w)$ for all $v,w \in V$.
\end{definition}



\begin{proposition}\label{proposition:symplectic_pairings_to_hermitian_forms}
    Let $A$ be a commutative ring, let $d \geq 2$ be an integer that is invertible in $A$. Let $B$ be an $A$-algebra direct product factor of $A[\zeta_d] = A[\zeta_d]_{\zeta_d \neq 1}$. 

    There is a categorical equivalence\footnote{the morphisms in the categories are isometries.}
    \begin{align*}
    \left\{\begin{array}{c} \text{Symplectic spaces over } A \\ \text{ preserved by } \zeta_d \text{ on } B\text{-modules}  \end{array} \right\} &\to \left\{\begin{array}{c} \text{$\varepsilon_{B}$-Hermitian spaces over } B \end{array} \right\}
    \end{align*}
    where $\varepsilon_B$ is as in \Cref{subsection:notations}. This equivalence fixes the underlying modules and assigns perfect symplectic pairings over $A$ which are preserved by $\zeta_d$ to nonsingular Hermitian forms over $A[\zeta_d]$.


\end{proposition}

\begin{proof}
    The $A$-algebra $B$ is finite \'etale over $A$ because $d$ is invertible in $A$. Thus, $B$ is a symmetric $A$-algebra with involution (see \cite[Chapter I, 7.4]{knus_qhfr} for a definition and discussion). By \cite[Chapter I, Theorem 7.4.1]{knus_qhfr}, there is a categorical equivalence between the symplectic spaces over $A$ preserved by $\zeta_d$ on $B$-modules and the $(-1)$-Hermitian spaces over $B$. 

    It therefore suffices to show that $(-1)$-Hermitian forms over $A[\zeta_d]$ correspond to $\varepsilon_{A[\zeta_d]}$-Hermitian forms over $A[\zeta_d]$. Note that $1-\zeta_d$ is invertible in $B$. Moreover, $1+\zeta_d$ is invertible in any $A$-algebra quotient of $A[X]/\left(\sum_{i=0}^{d-1} (-1)^i X^i \right)$ because $\left(\sum_{i=0}^{d-1} (-1)^i X^i \right)$ evaluates to $d$, which is invertible in $A$, at $X = -1$. Write $\eta \in B$ for the element corresponding to $(\zeta_d-\zeta_d^{-1}, 1)$ if $B$ has an $A$-algebra direct product factorization $B' \times A[X]/(X+1)$, and for $1$ if $B$ has no such factorization. Since $\zeta_d-\zeta_d^{-1} = (\zeta_d+1)(\zeta_d-1)\zeta_d^{-1}$, the element $\eta$ is invertible in $B$.
    Given a $(-1)$-Hermitian space $(V,H)$ over $B$, let $H': V \times V \to B$ be $H' = \eta H$. One verifies that $H'$ is a $\varepsilon_B$-Hermitian form on $V$ over $B$ and $H'$ is furthermore nonsingular because $\eta$ is invertible in $B$.

\end{proof}

We will factor involution rings $R$ into division rings with involution and double division rings and decompose Hermitian forms as Hermitian forms over these factor rings. We will also show in \Cref{proposition:hermitian_space_over_division_or_double_division_ring_has_orthogonal_basis} that Hermitian spaces over nice enough division or double division rings have orthogonal bases. We introduce the definition of double division rings below.

\begin{definition}
    Let $A$ be a (not necessarily commutative) ring. The {\em opposite ring of $A$} is the ring $A^{\text{op}}$ with the same underlying set, the same additive group, and reverse multiplication, i.e. 
        $$a^{\text{op}} \cdot b^{\text{op}} = (b \cdot a )^{\text{op}}$$
    where $a^{\text{op}}$ stands for $a$ as an element of $A^{\text{op}}$. We often omit the superscript $^\text{op}$.
\end{definition}

\begin{definition}{{\cite[Chapter I, 6.7]{knus_qhfr}}}
    Let $A$ be a ring. The {\em hyperbolic ring of $A$} is the product $H(A) = A \times A^{\text{op}}$ with involution
        $$(a, b) \mapsto (b, a).$$
    A {\em double division ring} is a ring of the form $H(A)$ where $A$ is a division ring.
\end{definition}

Continue, as per \Cref{subsection:notations}, to let $d \geq 2$ be an integer and $\ell \nmid d$ a prime number. Fix an algebraic closure $\overline{\mathbb{Q}}_\ell$ of $\mathbb{Q}_\ell$. The algebra $\mathbb{Q}_\ell[\zeta_d]_{\zeta_d\neq 1}$ is an \'etale $\mathbb{Q}_\ell$-algebra. As such, it decomposes into a finite product of finite unramified extensions of $\mathbb{Q}_\ell$. This factorization is determined by the $\operatorname{Gal}(\overline{\mathbb{Q}}_\ell / \mathbb{Q}_\ell)$-orbits on the set of primitive $d'$th roots of unity for all $d' \mid d$ with $d' \neq 1$. Concretely, there are $\mathbb{Q}_\ell$-algebra isomorphisms
\begin{align} \label{eq:Q_ell_zeta_d_decompose_by_cyclotomic_polynomials}
    \mathbb{Q}_\ell[\zeta_d]_{\zeta_d \neq 1} \cong \prod_{\substack{d' \mid d \\ d' \neq 1}} \mathbb{Q}_\ell[\zeta_{d'}]_{\text{prim}} \\
    \mathbb{Q}_\ell[\zeta_{d'}]_{\text{prim}} \cong \prod_{i} L_{d',i}.
\end{align}
Each $\mathbb{Q}_\ell[\zeta_{d'}]_{\text{prim}}$ is a Galois algebra over $\mathbb{Q}_\ell$ whose automorphism group is isomorphic to $(\mathbb{Z}/d'\mathbb{Z})^\times$ via the map 
\begin{align*}
    (\mathbb{Z}/d'\mathbb{Z})^\times \to \Aut_{\bbQ_\ell}(\mathbb{Q}_\ell[\zeta_{d'}]_{\text{prim}}), \quad (a \mapsto (X \mapsto X^a)).
\end{align*}
This Galois group acts transitively on the constituent fields $L_{d',i}$, yielding isomorphisms among them. Moreover, there are
valuation rings $B_{d',i} \subseteq L_{d',i}$ such that
\begin{equation} \label{eq:Z_ell_Phi_d_prime_factorization_as_ring}
    \mathbb{Z}_\ell[\zeta_{d'}]_{\text{prim}} \cong \prod_i B_{d',i}.
\end{equation}
Suppressing the $d'$-index, we write
\begin{equation} \label{eq:Z_ell_zeta_d_factorization_as_ring}
    \mathbb{Z}_\ell[\zeta_{d}]_{\zeta_d\neq 1} \cong \prod_{i} B_i.
\end{equation}

The involution $\overline{\cdot}: \zeta_d \mapsto \zeta_d^{-1}$ on $\mathbb{Z}_\ell[\zeta_d]$ induces a permutation of the factors $\{B_i\}$. For each $i$, it either restricts to an automorphism of $B_i$ or restricts to an isomorphism $B_i \xrightarrow{\sim} B_{\sigma(i)}$ with $i \neq \sigma(i)$ and $\sigma(\sigma(i)) = i$. In the latter case, $B_i \times B_{\sigma(i)}$ is isomorphic (as a ring with involution) to the double division ring $H(B_i) = B_i \times B_i^{\text{op}}$ via the map $\varphi: B_i \times B_i \to B_i \times B_{\sigma(i)}, \; (a,b) \mapsto (a, \bar{b})$; one checks that $\overline{\varphi(a,b)} = \varphi(\overline{(a,b)})$ because
    \begin{align*}
    \overline{\varphi(a,b)} = \overline{(a,\bar{b})} = (\bar{\bar{b}}, \bar{a}) = (b,\bar{a}) = \varphi(b, a) = \varphi(\overline{(a,b)}).
    \end{align*}
Grouping the $B_i$ into orbits of this permutation yields an isomorphism
    \begin{equation} \label{eq:Z_ell_zeta_d_factorization_as_ring_with_involution}
     \mathbb{Z}_\ell[\zeta_d]_{\zeta_d \neq 1} \cong \prod_{i} R_i
    \end{equation}
of involution rings where each $R_i$ is either a division ring with involution or a double division ring. Moreover, if $d$ is prime, then all of the $R_i$ are isomorphic to one another. Factorizations similar to \eqref{eq:Q_ell_zeta_d_decompose_by_cyclotomic_polynomials}, \eqref{eq:Z_ell_Phi_d_prime_factorization_as_ring}, and \eqref{eq:Z_ell_zeta_d_factorization_as_ring_with_involution} hold for the other quotients $\mathbb{Z}_\ell[\zeta_d]$ appearing in \Cref{subsection:notations}.

\begin{notation} \label{notation:fixed_subring_of_ring_with_involution}
    Given a ring $R$ with involution, let $R^{\sigma}$ denote the subring of elements fixed by the involution. 
\end{notation}

\begin{lemma}\label{lemma:every_invertible_element_of_the_subring_of_Z_ell_zeta_d_fixed_by_the_involution_is_the_norm_of_some_element_of_Z_ell_zeta_d}
    Let $d \geq 3$ be a positive integer, let $\ell$ be a prime number not dividing $d$, and let $\mathbb{Z}_\ell[\zeta_d] = \mathbb{Z}_\ell[\zeta_d]_{\zeta_d \neq \pm 1}$. Factorize $\mathbb{Z}_\ell[\zeta_d]$ as $\prod_i D_i$ as in \eqref{eq:Z_ell_zeta_d_factorization_as_ring_with_involution}. In particular, each of the $D_i$ is either an integral domain or a double division ring. 
    
    \begin{enumerate}
    	 \item Every invertible element of the subring $D_i^{\sigma}$ of $D_i$ is the norm of some element of $D_i$.
    	\item Every invertible element of the subring $\mathbb{Z}_\ell[\zeta_d]^{\sigma}$ of $\mathbb{Z}_\ell[\zeta_d]$ is the norm of some element of $\mathbb{Z}_\ell[\zeta_d]$.
    \end{enumerate}
\end{lemma}
    
\begin{proof}
    The norm map $\mathbb{Z}_\ell[\zeta_d] \to \mathbb{Z}_\ell[\zeta_d]^{\sigma}$ is given by $a \mapsto a \cdot \bar{a}$. When $D_i$ is an integral domain, recall that $D_i$ is in fact the valuation ring of a finite extension $L_i$ of $\mathbb{Q}_\ell$. The involution $\overline{\cdot}$ on $D_i$ is induced by an involution on $L_i$. This involution $\overline{\cdot}$ on $L_i$ respects the valuation on $L_i$, so $D_i^{\sigma}$ is the valuation ring of the subfield $L_i^{\sigma}$ of $L_i$. Moreover, since the $L_i$ are unramified over $\mathbb{Q}_\ell$, each $L_i$ is unramified over $L_i^{\sigma}$. Therefore, the norm map on $D_i^\times$ surjects onto $(D_i^{\sigma})^\times$ (see \cite[Corollary V.1.2]{neukirch} or \cite[Theorem 10.22]{sutherland18785fall2017} for example).
    
    When $D_i$ is a double division ring, say $D_i = B_i \times B_i^{\text{op}} = B_i \times B_i$ with involution $\overline{\cdot}: B_i \times B_i, (a,b) \mapsto (b,a)$, then note that $D_i^{\sigma}$ consists exactly of the elements corresponding to the elements of $B_i \times B_i$ of the form $(a,a)$. The norm map on $B_i \times B_i$ also sends $(a,1)$ to $(a, a)$. Therefore, the norm map on $(B_i \times B_i)^\times$ maps onto $\left( (B_i \times B_i)^{\sigma} \right)^{\times}$. 
    
    Either way, the norm map on $D_i^\times$ maps onto $(D_i^{\sigma})^\times$, so the norm map on $\mathbb{Z}_\ell[\zeta_d]^\times$ maps onto $(\mathbb{Z}_\ell[\zeta_d]^{\sigma})^\times$.
\end{proof}

The following proposition establishes that a Hermitian space over a division ring or a double division ring has an orthogonal basis under nice circumstances:

\begin{proposition}\label{proposition:hermitian_space_over_division_or_double_division_ring_has_orthogonal_basis}
   An $\varepsilon$-Hermitian space $(V,H)$ over a ring $D$ with involution has an orthogonal basis if one of the following hold:
   \begin{enumerate}
       \item $D$ is a division ring and the involution of $D$ is not trivial.
       \item $D$ is a division ring, the involution of $D$ is trivial, $\varepsilon = 1$, and $\operatorname{char} D \neq 2$.
       \item $D$ is a double division ring and $2$ is a unit in $D$.
   \end{enumerate}
\end{proposition}
\begin{proof}
    The first two parts restate \cite[Chapter I, Proposition 6.2.4]{knus_qhfr} in the case $\varepsilon = 1$. When $2$ is a unit in $D$, any Hermitian module $(V,H)$ over $D$ is even \cite[Chapter I, 3.1]{knus_qhfr}\footnote{That $(V,H)$ is even means that $H(v,w) = k(v,w) + \overline{k(w,v)}$ for some sesquilinear form $k$ on $V$. We do not use this definition in this thesis.}. Any even Hermitian space over a double division ring has an orthogonal basis \cite[Chapter I, 6.7]{knus_qhfr}.
\end{proof}

\begin{proposition} \label{proposition:free_hermitian_spaces_over_Z_ell_zeta_d_have_orthonormal_bases}
    Let $d \geq 3$ be a positive integer, let $\ell \neq 2$ be a prime number not dividing $d$, and let $\mathbb{Z}_\ell[\zeta_d] = \mathbb{Z}_\ell[\zeta_d]_{\zeta_d \neq \pm 1}$. As in \eqref{eq:Z_ell_zeta_d_factorization_as_ring_with_involution}, write
    $$\mathbb{Z}_\ell[\zeta_d] \cong \prod_{i \in I} D_i$$
    where the product and isomorphism are of involution rings, and each of the $D_i$ is either a division ring or a double division ring.

\begin{enumerate}
	\item For any $\varepsilon \in D_i$ such that $\varepsilon \cdot \overline{\varepsilon} = 1$, any $\varepsilon$-Hermitian space over any of the $D_i$ has an orthonormal basis.
	\item Let $D = \prod_{i \in J} D_i$ with $J \subseteq I$. Any free $1$-Hermitian space over $D$ has an orthonormal basis.
\end{enumerate}
\end{proposition}

\begin{proof}
\begin{enumerate}
	\item By \Cref{proposition:hermitian_space_over_division_or_double_division_ring_has_orthogonal_basis}, any $\varepsilon$-Hermitian space over any of the $D_i$ has an orthogonal basis because the involution on each $D_i$ is not trivial and because $\ell \neq 2$ and hence $2$ is invertible in $\mathbb{Z}_\ell$.
 
    Let $(V, H)$ be an $\varepsilon$-Hermitian space over $D_i$ and let $v_1,\ldots,v_n$ be an orthogonal basis of $V$. In particular, $H$ is nonsingular and $(V,H)$ is the orthogonal sum of the submodules $V_i$ generated by $v_i$ and hence the $H|_{V_i}$ are all nonsingular by \Cref{lemma:knus_qhfr_lemma_3.6.2}. Therefore, $H(v_i, v_i)$ must be a unit of $D_i$.
    
    By \Cref{lemma:every_invertible_element_of_the_subring_of_Z_ell_zeta_d_fixed_by_the_involution_is_the_norm_of_some_element_of_Z_ell_zeta_d}, there exist $a_i \in D_i^\times$ such that $H(v_i, v_i)$ is the norm of $a_i$ in $(D_i^{\sigma})^\times$. We have $H(a_i^{-1} v_i, a_i^{-1} v_i) = 1$. Replacing $v_i$ with $a_i^{-1} v_i$ yields an orthonormal basis of $V$.

    \item Let $(V,H)$ be such a Hermitian space over $D$. Let $V_i$ be the $D_i$-component of $V$, and let $H_i = H|_{V_i}$.
    By \Cref{lemma:hermitian_modules_over_product_of_involution_rings_split_along_involution_ring_factors}, the $H_i$ are nonsingular $1$-Hermitian forms over $D_i$. By the above, $H_i$ has an orthonormal basis over $D_i$, say $v_{i,1},\ldots,v_{i,n_i}$. In particular, the orthogonal decomposition $V \cong \perp_i V_i$ further decomposes into $V \cong \perp_{i, 1 \leq j \leq n_i} \langle v_{i,j} \rangle$ and this decomposition is in fact a decomposition of Hermitian spaces over $D$. Since $V$ is assumed to be a free $D$-module, all of the $n_i$ are equal. Let $n$ be $n_i$. For each $1 \leq j \leq n$, choose $a_{i,j} \in D_i^\times$ such that $a_{i,j} \cdot \overline{a_{i,j}} = H(v_{i,j},v_{i,j})^{-1}$ in $D_i$. Let $v_j := \sum_{i} a_{i,j} v_{i,j}$ and note that $H(v_j,v_j) = \sum_{i} a_{i,j} \cdot \overline{a_{i,j}} H(v_{i,j}, v_{i,j}) = 1$ in $D$. Moreover, $H(v_j, v_{j'}) = 0$ whenever $j \neq j'$ because $H(v_{i,j}, v_{i',j'}) = 0$ for all $i,i'$. In particular, one can show that the $v_j$ are linearly independent over $D$. Letting $e_i \in D$ correspond to $1$ in $D_i$ and to $0$ in $D_{i'}$ for all $i' \neq i$, note that $e_i v_j = e_i a_{i,j} v_{i,j}$. Furthermore, $e_i a_{i,j}$ is invertible as an element of $D_i$, so $v_{i,j}$ is in the $D$-span of $v_j$. Therefore, the $v_j$ form an orthonormal basis of $(V, H)$.
\end{enumerate}
\end{proof}

We now show that the $\ell$-adic Tate module of a prime-order cyclic cover of $\mathbb{P}^1_{\mathbb{F}_q}$ is a projective module over $\mathbb{Z}_\ell[\zeta_d] = \mathbb{Z}_\ell[\zeta_d]_{\zeta_d\neq-1}$ to establish in \Cref{corollary:superelliptic_curve_tate_module_is_orthonormalizable} that the Tate module decomposes as orthonormalizable Hermitian spaces over the division ring and double division ring factors of $\mathbb{Z}_\ell[\zeta_d]_{\zeta_d \neq \pm 1}$ and that the Tate module itself is almost an orthonormalizable Hermitian space when it is a free $\mathbb{Z}_\ell[\zeta_d]_{\zeta_d \neq \pm 1}$-module.

\begin{lemma} \label{lemma:cyclic_sum_of_zeta_d_acts_as_zero_on_order_d_superelliptic_curve}
    Let $k$ be a field. Let $d \geq 2$ be an integer not divisible by $\operatorname{char} k$. Let $C$ be a $\mathbb{Z}/d\mathbb{Z}$-cover of $\mathbb{P}^1_{k}$, and let $\zeta_d$ be a generator of $\operatorname{Aut}(C/\mathbb{P}^1) \cong \mathbb{Z}/d\mathbb{Z}$. Let $\ell$ be a prime number not dividing $d$ and invertible in $k$. The $\ell$-adic Tate module $T_\ell(\operatorname{Jac}(C))$ (over $\bar{k}$) is a projective $\mathbb{Z}_\ell[\zeta_d]_{\zeta_d \neq 1}$-module by letting $\zeta_d \in \mathbb{Z}_\ell[\zeta_d]_{\zeta_d \neq 1}$ act on $T_\ell(\operatorname{Jac}(C))$ by the automorphism induced by $\zeta_d \in \operatorname{Aut}(C/\mathbb{P}^1)$, by abuse of notation.
\end{lemma}
\begin{proof}
    Let $\varphi: C \to \mathbb{P}^1$ denote the covering map, and fix an $\bar{k}$-point $P_0$ of $C$. For any $\bar{k}$-point $P$ of $C$, the divisor $\left( \sum_{i=0}^{d-1} \big( [\zeta_d^i P] - [\zeta_d^i P_0] \big) \right)$ of $C$ is a principal divisor given by the function $\varphi^*(x(\varphi(P)))$ where $\varphi: C \to \mathbb{P}^1$ is the covering map and $x(\varphi(P))$ is the function of $\mathbb{P}^1$ of divisor $\varphi(P) - \varphi(P_0)$. Therefore, the element $\sum_{i=0}^{d-1} \zeta_d^{i} = \frac{X^d - 1}{X - 1}\big|_{X=\zeta_d}$ acts as $0$ on $\operatorname{Jac}(C)$, so $T_\ell(\operatorname{Jac}(C))$ is indeed a $\mathbb{Z}_\ell[X]/(1 + X + \cdots + X^{d-1})$-module where $X$ acts by $\zeta_d$.
    
%

    Factoring $\mathbb{Z}_\ell[\zeta_d]_{\zeta_d \neq 1} \cong \prod_i B_i$ as in \eqref{eq:Z_ell_zeta_d_factorization_as_ring}, each $B_i$ is a PID. 
    The $\mathbb{Z}_\ell[\zeta_d]_{\zeta_d \neq 1}$-module $T_\ell(\operatorname{Jac}(C))$ corresponds to a tuple $(M_i)_{i}$ of $B_i$-modules by $T_\ell(\operatorname{Jac}(C)) \cong \prod_i M_i$. Moreover, $T_\ell(\operatorname{Jac}(C))$ is free (of rank $2g$ where $g$ is the genus of $C$) as a $\mathbb{Z}_\ell$-module, so each $M_i$ must be a free $B_i$-module because otherwise, $M_i$ would have a finite direct summand by the structure theorem for PID's. Therefore, $T_\ell(\operatorname{Jac}(C))$ is a projective $\mathbb{Z}_\ell[\zeta_d]_{\zeta_d \neq 1}$-module.
\end{proof}

Let $k,d,\ell$, and $C$ be as in \Cref{lemma:cyclic_sum_of_zeta_d_acts_as_zero_on_order_d_superelliptic_curve}. We use the character theory of the representation $\chi_C$ of $\langle \zeta_d \rangle$ on $H^1(C_{\bar{k}}, \mathbb{Q}_\ell) \cong (T_\ell(\operatorname{Jac}(C)) \otimes_{\mathbb{Z}_\ell} \mathbb{Q}_\ell)^\vee$. Suppose that the cover $C \to \mathbb{P}^1$ has branch points $P_1,\ldots,P_{r} \in \mathbb{P}^1(\bar{k})$ of monodromy $g_1,\ldots,g_r \in \mathbb{Z}/d\mathbb{Z} \cong \langle \zeta_d \rangle$ respectively. By \cite[Proposition 1.3]{ellenberg_eaj}, which expresses a Riemann-Hurwitz theorem for characters, the character $\chi_C$ of the representation of $\langle \zeta_d \rangle$ on $H^1(C_{\bar{k}}, \mathbb{Q}_\ell) \cong (T_\ell(\operatorname{Jac}(C)) \otimes_{\mathbb{Z}_\ell} \mathbb{Q}_\ell)^\vee$ satisfies
    \begin{align} \label{eq:riemann_hurwitz_for_characters}
    \chi_C &= 2\chi_{\text{triv}} + 2(g(\mathbb{P}^1) - 1) \chi_{\langle 1 \rangle} + \sum_{i} (\chi_{\langle 1 \rangle} - \chi_{\langle g_i \rangle}) \\
    &= 2\chi_{\text{triv}} - 2 \chi_{\langle 1 \rangle} + \sum_{i=1}^r (\chi_{\langle 1 \rangle} - \chi_{\langle g_i \rangle})  \nonumber
    \end{align}
where $\chi_{\text{triv}}$ is the trivial character of $\langle \zeta_d \rangle$ and $\chi_{\langle g \rangle}$ is the character of $\langle \zeta_d \rangle$ induced from the trivial character $\chi_{\mathrm{triv}, \langle g \rangle}$ on $\langle g \rangle$ where $g \in \langle \zeta_d \rangle$. Moreover, the irreducible representations of $\langle \zeta_d \rangle$ over an algebraically closed field of characteristic $0$ are the dimension $1$ representations $\psi_a$ for $a \in \mathbb{Z}/d\mathbb{Z}$ given by $\psi_a(\zeta_d^j) = \zeta_d^{ja}$. In particular, $\chi_{\text{triv}} = \psi_0$. 

\begin{lemma} \label{lemma:character_theory_of_Z_dZ_cover_of_P1}
    Let $k,d,\ell$, and $C$ be as above. 
    With the above notation for $k,d,\ell$, and $C$, 
    \begin{enumerate}
        \item \label{part:chi_g_as_sum_of_irreducible_characters} For $g \in \langle \zeta_d \rangle$,
        $$\chi_{\langle g \rangle} = \sum_{\substack{a \in \mathbb{Z}/d\mathbb{Z} \\ \operatorname{ord} g \mid a}} \psi_a$$
    as characters of representations on $\overline{\mathbb{Q}}_\ell$-vector spaces, where by $\operatorname{ord} g \mid a$ we mean that $a \equiv \alpha \cdot \operatorname{ord} g \pmod{d}$ for some integer $\alpha$. In particular, 
        $$\chi_{\langle g \rangle}(\zeta_d^i) = \begin{cases} \frac{d}{\operatorname{ord} g} &\text{ if } \zeta_d^i \in \langle g \rangle \\ 0  &\text{otherwise} \end{cases}$$
    so $\chi_{\langle g \rangle}$ is valued in $\mathbb{Q}_\ell$.
    
        \item Fix $r_d$ to be a primitive $d$th root of unity in $\overline{\mathbb{Q}}_\ell$. Then, $\chi_{\langle g \rangle}$ is the character of $\overline{\mathbb{Q}}_\ell[X]/\prod_{\substack{a \in \mathbb{Z}/d\mathbb{Z} \\ \operatorname{ord} g \mid a}} (X - r_d^a)$ as a $\langle \zeta_d \rangle$ representation where $\zeta_d$ acts as $X$. \label{part:chi_g_is_character_of_Q_ell_bar_algebra}
    \end{enumerate}
\end{lemma}
\begin{proof}
    \begin{enumerate}
        \item The formula \cite[The start of Section 7.2 and Chapter 7, Proposition 20]{serre_lrfg} that computes the induced character of a character on a subgroup of a group yields
            \begin{align*}
            \chi_{\langle g \rangle}(\zeta_d^i) &= \frac{1}{\operatorname{ord} g} \sum_{\substack{t \in \langle \zeta_d \rangle \\ t^{-1} \zeta_d^i t \in \langle g \rangle}} \chi_{\mathrm{triv}, \langle g \rangle} (t^{-1} \zeta_d^i t) \\
            &= \begin{cases} \frac{d}{\operatorname{ord} g} &\text{ if } \zeta_d^i \in \langle g \rangle \\ 0  &\text{otherwise} \end{cases}
            \end{align*}
        from which the claimed result follows.
        
        \item By (1),
        \begin{align*}
        \chi_{\langle g \rangle}(\zeta_d^i) &= \sum_{\substack{a \in \mathbb{Z}/d\mathbb{Z} \\ \operatorname{ord} g \mid a}} \psi_a(\zeta_d^i) \\
        &= \sum_{\substack{a \in \mathbb{Z}/d\mathbb{Z} \\ \operatorname{ord} g \mid a}} \zeta_d^{a \cdot i}. 
        \end{align*}
        As an element of $\mathbb{Q}_\ell$, this equals $\sum_{\substack{a \in \mathbb{Z}/d\mathbb{Z} \\ \operatorname{ord} g \mid a}} r_d^{a \cdot i}$, which coincides with the trace of $\zeta_d^i$ acting on $\overline{\mathbb{Q}}_\ell[X]/\prod_{\substack{a \in \mathbb{Z}/d\mathbb{Z} \\ \operatorname{ord} g \mid a}} (X - r_d^a)$, so the claimed result holds.
    \end{enumerate}
\end{proof}

\begin{proposition}
\label{proposition:ell_adic_tate_module_of_prime_order_superelliptic_curve_is_a_free_Z_ell_zeta_d_module}

    Let $k,d,\ell$, and $C$ be as above. If the cover $C \to \mathbb{P}^1$ has $r_1$ branch points $P_1,\ldots,P_{r_1} \in \mathbb{P}^1(\bar{k})$ whose monodromies, which are elements of $\mathbb{Z}/d\mathbb{Z}$, each have order exactly $d$, then the $\ell$-adic Tate module $T_\ell(\operatorname{Jac}(C))$ has a free $\mathbb{Z}_\ell[\zeta_d]_{\zeta_d \neq 1}$-module of rank at least $r_1-2$ as a direct summand.  
    
    Moreover, if all of the monodromies of the branch points have order exactly $d$ (including the point $\infty \in \mathbb{P}^1$ assuming that it is a branch point), which happens when $d$ is prime for example, then $T_\ell(\operatorname{Jac}(C))$ is a free $\mathbb{Z}_\ell[\zeta_d]_{\zeta_d \neq 1}$-module of rank $r-2$ where $r$ is the branch locus degree of the cover $C \to \mathbb{P}^1$.
\end{proposition}
\begin{proof}
    Recall \Cref{eq:riemann_hurwitz_for_characters}. Fixing $r_d$ to be a primitive $d$th root of unity in $\overline{\mathbb{Q}_\ell}$, \Cref{lemma:character_theory_of_Z_dZ_cover_of_P1}\eqref{part:chi_g_is_character_of_Q_ell_bar_algebra} shows that $\prod_{\substack{a \in \mathbb{Z}/d\mathbb{Z} \\ \operatorname{ord} g \mid a}} (X - r_d^a)$ is a polynomial over $\mathbb{Q}_\ell$ and $\chi_{\langle g \rangle}$ is the character of $\overline{\mathbb{Q}}_\ell[X]/\prod_{\substack{a \in \mathbb{Z}/d\mathbb{Z} \\ \operatorname{ord} g \mid a}} (X - r_d^a)$ as a $\langle \zeta_d \rangle$ representation where $\zeta_d$ acts as $X$.  Thus, $\chi_{\langle 1 \rangle} - \chi_{\langle g \rangle}$ is the character of $\overline{\mathbb{Q}}_\ell[X]/\prod_{\substack{a \in \mathbb{Z}/d\mathbb{Z} \\ \operatorname{ord} g \nmid a}} (X-r_d^a)$, which is $\overline{\mathbb{Q}}_\ell[\zeta_d]_{\zeta_d \neq 1}$ if $g$ has order exactly $d$. Likewise, $\chi_{\langle 1 \rangle} - \chi_{\text{triv}}$ is also the character of $\overline{\mathbb{Q}}_\ell[\zeta_d]_{\zeta_d \neq 1}$. Representation of finite groups over algebraically closed fields of characteristic $0$ are determined by their characters \cite[Chapter 2, Corollary 2]{serre_lrfg}, so given that the cover $C \to \mathbb{P}^1$ has $r_1$ branch points whose monodromies each have order exactly $d$, $\chi_C$ is the character of a representation over $\overline{\mathbb{Q}}_\ell$ that has a free $\overline{\mathbb{Q}}_\ell[\zeta_d]_{\zeta_d \neq 1}$-module of rank at least $r_1-2$ as a direct summand. Since field extensions are faithfully flat, $H^1(C_{\overline{k}}, \mathbb{Q}_\ell)$ in fact has a free $\mathbb{Q}_\ell[\zeta_d]_{\zeta_d \neq 1}$-module of rank at least $r_1-2$ as a direct summand, so this is also true of $T_\ell(\operatorname{Jac}(C)) \otimes_{\mathbb{Z}_\ell} \mathbb{Q}_\ell$. 
    
    By \Cref{lemma:cyclic_sum_of_zeta_d_acts_as_zero_on_order_d_superelliptic_curve}, $T_\ell(\operatorname{Jac}(C))$ is a $\mathbb{Z}_\ell[\zeta_d]_{\zeta_d \neq 1}$-module. Say that the prime factorization of $\Phi_d(X)$ over $\mathbb{Z}_\ell$ is $\prod_i h_i(X)$; all of the $h_i(X)$ are distinct because $\mathbb{Q}(\zeta_d)$ is unramified above the prime $\ell$ of $\mathbb{Q}$ since $\ell \nmid d$. In particular, $\mathbb{Z}_\ell[\zeta_d]_{\zeta_d \neq 1} \cong \prod_i B_i$, where $B_i = \mathbb{Z}_\ell[X]/h_i(X)$. Since $T_\ell(\operatorname{Jac}(C))$ is a projective $\mathbb{Z}_\ell[\zeta_d]_{\zeta_d \neq 1}$ and hence has free $B_i$-components, and since $T_\ell(\operatorname{Jac}(C)) \otimes_{\mathbb{Z}_\ell} \mathbb{Q}_\ell$ has a free $\mathbb{Q}_\ell[\zeta_d]_{\zeta_d \neq 1}$-module of rank at least $r_1-2$ as a direct summand, $T_\ell(\operatorname{Jac}(C))$ must have a free $\mathbb{Z}_\ell[\zeta_d]_{\zeta_d \neq 1}$-module of rank at least $r_1-2$ as a direct summand.
    
    If all of the monodromies of the branch points have order exactly $d$, then $T_\ell(\operatorname{Jac}(C)) \otimes_{\mathbb{Z}_\ell} \mathbb{Q}_\ell$ is a free $\mathbb{Q}_\ell[\zeta_d]_{\zeta_d \neq 1}$-module of rank $r-2$ and $T_\ell(\operatorname{Jac}(C))$ must be a free $\mathbb{Z}_\ell[\zeta_d]_{\zeta_d \neq 1}$-module of rank $r-2$.
\end{proof}

We now prove \Cref{corollary:superelliptic_curve_tate_module_is_orthonormalizable}, which summarizes many of the ideas of the current subsection.

\begin{corollary}\label{corollary:superelliptic_curve_tate_module_is_orthonormalizable}
    Let $k$ be a field. Let $d \geq 3$ be an integer not divisible by $\operatorname{char} k$. Let $C$ be a $\mathbb{Z}/d\mathbb{Z}$-cover of $\mathbb{P}^1_{k}$, and identify $\operatorname{Aut}(C/\mathbb{P}^1) \cong \mathbb{Z}/d\mathbb{Z}$ with $\langle \zeta_d \rangle$. Let $\ell$ be a prime number not dividing $2d$ and invertible in $k$. Write $\mathbb{Z}_\ell[\zeta_d] = \mathbb{Z}_\ell[\zeta_d]_{\zeta_d \neq 1}$ as the product $\prod_i D_i$ of division rings and double division rings as in \eqref{eq:Z_ell_zeta_d_factorization_as_ring_with_involution}.
    \begin{enumerate}
    \item The $\ell$-adic Weil pairing $\omega: T_\ell(\operatorname{Jac}(C)) \times T_\ell(\operatorname{Jac}(C)) \to \mathbb{Z}_\ell$ induces a nonsingular $\varepsilon_{\mathbb{Z}_\ell[\zeta_d]}$-Hermitian form $H: T_\ell(\operatorname{Jac}(C)) \times T_\ell(\operatorname{Jac}(C)) \to \mathbb{Z}_\ell[\zeta_d]$.
    \item Writing $V_i$ as the $D_i$-component of $T_\ell(\operatorname{Jac}(C))$, and $H_i$ as the restriction of $H$ to $V_i$, the Hermitian form $(V_i, H_i)$ is nonsingular and has an orthonormal basis whenever $\zeta_d$ is not $-1$ in $D_i$. 
    \item If $T_\ell(\operatorname{Jac}(C))$ is a free $\mathbb{Z}_\ell[\zeta_d]$-module, then the $\mathbb{Z}[\zeta_d]_{\zeta_d \neq \pm 1}$-component of $(T_\ell(\operatorname{Jac}(C)), H)$ has an orthonormal basis. In particular, if $d$ is odd, then $(T_\ell(\operatorname{Jac}(C)), H)$ has an orthonormal basis.
    \end{enumerate}
\end{corollary}
\begin{proof}
    Since $\omega$ is a nondegenerate symplectic form preserved by $\zeta_d$, \Cref{proposition:symplectic_pairings_to_hermitian_forms} shows that $\omega$ induces a nonsingular Hermitian form on $T_\ell(\operatorname{Jac}(C))$ over $\mathbb{Z}_\ell[\zeta_d]$. By \Cref{lemma:hermitian_modules_over_product_of_involution_rings_split_along_involution_ring_factors}, $(V_i, H_i)$ are all nonsingular. Since $\ell \neq 2$ and hence $2$ is a unit in each $D_i$,  \Cref{proposition:free_hermitian_spaces_over_Z_ell_zeta_d_have_orthonormal_bases} shows that $(V_i, H_i)$ has an orthonormal basis.

    Furthermore, if $T_\ell(\operatorname{Jac}(C))$ is a free $\mathbb{Z}_\ell[\zeta_d]$-module, its $\mathbb{Z}[\zeta_d]_{\zeta_d \neq \pm 1}$-component is a free $1$-Hermitian space, in which case it also has an orthonormal basis by \Cref{proposition:free_hermitian_spaces_over_Z_ell_zeta_d_have_orthonormal_bases}.
\end{proof}
\begin{remark}
    Note that the notions of norms and hence orthonormal bases of a $(-1)$-Hermitian space over a ring with trivial involution, such as $\mathbb{Z}_\ell[X]/(X+1)$, are not fruitful --- since the Hermitian form of such a space is actually a symplectic space, norms on such a space are zero. In particular, there is no orthonormal basis of a $\varepsilon_{\mathbb{Z}_\ell[\zeta_d]_{\zeta_d \neq 1}}$-Hermitian space over $\mathbb{Z}_\ell[\zeta_d]_{\zeta_d \neq 1}$ when $d$ is even.
\end{remark}

\subsection{Unitary groups} \label{subsection:unitary_groups}

Let $(V,H)$ be an $\varepsilon$-Hermitian module over a ring $R$ with involution $\overline{\cdot}$. For $m \in R^\times$, let $\mathrm{GU}_m(V)$ be the set of $R$-linear automorphisms $\alpha$ of $V$ satisfying $H(\alpha(v_1), \alpha(v_2)) = mH(v_1, v_2)$ for all $v_1,v_2 \in V$. Write $\mathrm{U}(V) = \mathrm{GU}_1(V)$, $\mathrm{GU}(V) = \bigcup_m \mathrm{GU}_m(V)$. When $(V,H)$ is a free Hermitian space, note that $\alpha \in \mathrm{GU}_m(V)$ satisfies $\det \alpha \cdot \overline{\det \alpha} = m$. If $R$ is a product $\prod_i D_i$ of finitely many involution rings and $(V,H)$ corresponds to the tuple $(V_i, H_i)_i$ of Hermitian modules over $D_i$ via \Cref{lemma:hermitian_modules_over_product_of_involution_rings_split_along_involution_ring_factors}, then $\mathrm{GU}_m(V) \cong \prod_i \mathrm{GU}_m(V_i)$.

For a free Hermitian space $(V,H)$, write $\mathrm{SU}(V) = \{\alpha \in \mathrm{U}(V): \det \alpha = 1\}$. More generally, if $R$ is a product $\prod_i D_i$ of finitely many involution rings and $(V,H)$ is a Hermitian space such that the component $(V_i, H_i)$ of $(V,H)$ over $D_i$ is free for every $i$, then write $\mathrm{SU}(V) = \{\alpha \in \mathrm{U}(V): \alpha|_{V_i} \in \mathrm{SU}(V_i) \}$. By construction, $\mathrm{SU}_m(V) \cong \prod_i \mathrm{SU}_m(V_i)$.

\subsection{Counting the number of \texorpdfstring{$\Fq$}{Fq}-rational connected components of Hurwitz schemes} \label{subsection:counting_the_number_of_Fq_rational_connected_components_of_hurwitz_schemes}

In this subsection, we prove \Cref{proposition:generalizing_EVW_Lemma 8.9_to_Z_dZ_covers}, whose statement and purpose are both analogous to those of \cite[Lemma 8.9]{ellenberg_venkatesh_westerland}.
Namely, both \Cref{corollary:generalizing_EVW_Lemma 8.9_to_Z_dZ_covers} and \cite[Lemma 8.9]{ellenberg_venkatesh_westerland} are used to show that certain Hurwitz schemes over $\Fq$ have exactly one $\Fq$-rational connected component assuming such big monodromy results discussed further in \Cref{subsection:big_monodromy_hypothesis}. However, the $\zeta_d$-action introduces subtle differences that prevent the ideas of the proof of \cite[Lemma 8.9]{ellenberg_venkatesh_westerland} from applying directly. In particular, elements $f: V \to A$ of the set $O$ defined in \cite{ellenberg_venkatesh_westerland} require a stabilizer $h: V \to V$ to only satisfy $\omega(h(v), h(w)) = q\omega(v,w)$ for all $v,w \in V$, where $\omega$ is the given symplectic pairing on $V$. In contrast, an element $f$ of the set $O$ defined in \Cref{proposition:generalizing_EVW_Lemma 8.9_to_Z_dZ_covers} requires a stabilizer $h: V \to V$ that is equivariant for the $\zeta_d$-action as well. The proof of \Cref{proposition:generalizing_EVW_Lemma 8.9_to_Z_dZ_covers} generalizes the symplectic theoretic ideas from \cite[Lemma 8.9]{ellenberg_venkatesh_westerland} via the theory of Hermitian forms.

\begin{proposition}\label{proposition:generalizing_EVW_Lemma 8.9_to_Z_dZ_covers}

    Let $d \geq 3$ be an integer, let $\ell$ be a prime number not dividing $2dr$, let $\mathbb{Z}_\ell[\zeta_d] = \mathbb{Z}_\ell[\zeta_d]_{\zeta_d \neq 1}$. Let $A$ be a fixed $\mathbb{Z}_\ell[\zeta_d]$-module of finite cardinality and say that $A$ is a quotient module of $\mathbb{Z}_\ell[\zeta_d]^{\oplus r}$.

    Let $V$ be a projective $\mathbb{Z}_\ell[\zeta_d]$-module equipped with a perfect symplectic pairing $\omega: V \times V \to \mathbb{Z}_\ell$ preserved by $\zeta_d$. Recall from \Cref{proposition:symplectic_pairings_to_hermitian_forms} that $\omega$ induces a nonsingular $\varepsilon_{\mathbb{Z}_\ell[\zeta_d]}$-Hermitian form $H: V \times V \to \mathbb{Z}_\ell[\zeta_d]$. Let $q \in \mathbb{Z}_\ell^\times$ be such that $q-1$ is invertible in $\mathbb{Z}_\ell$.
    Define $O$ as the set of all $\mathbb{Z}_\ell[\zeta_d]$-equivariant surjections $V \to A$ whose stabilizer, inside $\mathrm{GU}(V)$, intersects $\mathrm{GU}_q(V)$ nontrivially, i.e.
    	$$O = \{f: V \to A \text{ surjective, equivariant for } \zeta_d, \text{ and there exists } h \in \mathrm{GU}_q(V) \text{ with } f \circ h = f \}.$$
    If $V$ has a free $\mathbb{Z}_\ell[\zeta_d]$-module of rank at least $2r$ as a direct summand, then $O$ is nonempty and $\mathrm{U}(V)$ acts transitively on $O$. In fact, $\mathrm{SU}(V)$ acts transitively on $O$ if $V$ has a free $\mathbb{Z}_\ell[\zeta_d]$-module of rank more than $2r$ as a direct summand. 
\end{proposition}

We prove \Cref{proposition:generalizing_EVW_Lemma 8.9_to_Z_dZ_covers} via \Cref{proposition:nonempty_O} and \Cref{proposition:nonempty_O_implies_transitive_action_of_U_V}.

\begin{lemma}\label{lemma:D_i_module_decomposes_into_quotients_of_D_i}
    Let $D$ be a principal ideal ring that is a product of integral domains. A finitely generated $D$-module $M$ decomposes in the form $\bigoplus^r_j M_j$ where the $M_j$ are quotients of $D$ and $r < \infty$.
\end{lemma}
\begin{proof}
Say that $D = \prod_k B_k$ where $B_k$ are integral domains. Since $D$ is a principal ideal ring and the $B_k$ are quotients of $D$, the $B_k$ are PID's. Moreover, $M$ corresponds to a tuple $(M_k)_k$ of $B_k$-modules; $M_k$ is the $B_k$-component of $M$ and $M = \prod_k M_k$. For each $k$, the structure theorem for finitely generated modules over PID's yields a decomposition $M_k \cong \bigoplus^{r_k}_j M_{k,j}$ where each $M_{k,j}$ is a $B_k$-module that is a quotient of $B_k$ and $r_k < \infty$. Letting $r = \max_k r_k$, letting $M_{k,j} = 0$ whenever $j > r_k$, and letting $M_j = \prod_k M_{k,j}$, $M_j$ is a quotient of $D$ and $M$ decomposes in the form $\bigoplus^r_j M_j$.
\end{proof}

\begin{proposition} \label{proposition:nonempty_O}
    Let $d \geq 3$ be an integer, let $\ell$ be a prime number not dividing $2d$. Let $D$ be a division ring or a double division ring quotient of $\mathbb{Z}_\ell[\zeta_d]_{\zeta_d \neq \pm 1}$. Fix a $D$-module $A$ of finite cardinality. Let $(V,H)$ be a free Hermitian space over $D$. Let $q \in \mathbb{Z}_\ell^\times$. Decompose $A$ as $\oplus_{i=1}^r A_i$ as in \Cref{lemma:D_i_module_decomposes_into_quotients_of_D_i}. If the rank $s$ of $V$ as a $D$-module is at least $2r$, then the set 
        $$O = \{f: V \to A \text{ surjective, } D\text{-linear, and there exists } h \in \operatorname{\mathrm{GU}_q}(V) \text{ with } f \circ h = f \}$$
    is nonempty.

\end{proposition}

\begin{proof}
    
    Let $v_1,\ldots,v_{s}$ be an orthogonal basis of $V$ with respect to $D$ such that $H(v_{2i-1}, v_{2i-1}) = 1$ and $H(v_{2i}, v_{2i}) = -1$ for $1 \leq i \leq r$ and such that $H(v_{i}, v_i) = 1$ for $2r+1 \leq i \leq s$. Recall that $V$ has an orthonormal basis by \Cref{proposition:free_hermitian_spaces_over_Z_ell_zeta_d_have_orthonormal_bases} and such an orthonormal basis can be normalized in this proposed manner because all elements of $(D^{\sigma})^\times$ are norms of elements of $D^\times$ by \Cref{lemma:every_invertible_element_of_the_subring_of_Z_ell_zeta_d_fixed_by_the_involution_is_the_norm_of_some_element_of_Z_ell_zeta_d}.
    
    Let $f: V \to A$ be given by 
        \begin{align*}
        v_{2i-1}, v_{2i} &\mapsto \frac{1}{2} \in A_{i} \subset A \quad 1 \leq i \leq r \\
        v_{i} &\mapsto 0 \in A  \quad 2r+1 \leq i \leq s
        \end{align*}
    and extended $D$-linearly; here $1 \in A_{i}$ is a $D$-module generator of $A_{i}$, which exists because $A_{i}$ is a quotient of $D$, and $\frac{1}{2} \in A_{i}$ is the element such that $\frac{1}{2} + \frac{1}{2} = 1$.
    
    Let $h: V \to V$ be the $D$-linear automorphism given by 
        \begin{align*}
        v_{2i-1} &\mapsto \frac{1+q}{2} v_{2i-1} + \frac{1-q}{2} v_{2i} \\
        v_{2i} &\mapsto \frac{1-q}{2} v_{2i-1} + \frac{1+q}{2} v_{2i} \quad \text{for } 1 \leq i \leq r \\
        v_i &\mapsto \alpha v_i \quad \text{for } 2r+1 \leq i \leq s
        \end{align*}
    where $\alpha \in D$ is some element whose norm $\alpha \cdot \overline{\alpha}$ is $q$. Note that $h \in \mathrm{GU}_q(V)$ because 
        \begin{align*}
        H(h v_{2i-1}, hv_{2i-1}) &= H\left( \frac{1+q}{2} v_{2i-1} + \frac{1-q}{2} v_{2i}, \frac{1+q}{2} v_{2i-1} + \frac{1-q}{2} v_{2i} \right)  \\
        &= \left( \frac{1+q}{2} \right)^2 - \left( \frac{1-q}{2} \right)^2 = q \\
        H (h v_{2i}, hv_{2i}) &= H\left( \frac{1-q}{2} v_{2i-1} + \frac{1+q}{2} v_{2i}, \frac{1-q}{2} v_{2i-1} + \frac{1+q}{2} v_{2i} \right)  \\
        &= \left( \frac{1-q}{2} \right)^2 - \left( \frac{1+q}{2} \right)^2 = -q \quad \text{ for } 1 \leq i \leq r \\
        H (hv_i, hv_i) &= H\left( \alpha v_i, \alpha v_i \right) \\
        &= N(\alpha) \cdot H(v_i, v_i) = q \quad \text{for } 2r+1 \leq i \leq s.
        \end{align*}
    Moreover, $f \circ h: V \to A$ is the $D$-linear map such that
    
    \begin{align*}
    v_{2i-1} &\mapsto \left( \frac{1+q}{2} \right) \frac{1}{2} + \left( \frac{1-q}{2} \right) \frac{1}{2} = \frac{1}{2}, \\
    v_{2i} &\mapsto \left( \frac{1-q}{2} \right) \frac{1}{2} + \left( \frac{1+q}{2} \right) \frac{1}{2} = \frac{1}{2}, \quad 1 \leq i \leq r \\
    v_i &\mapsto \alpha \cdot 0 = 0 \quad 2r + 1 \leq i \leq s
    \end{align*}
    respectively, so $f \circ h = f$. Therefore, $O$ is nonempty as desired.
\end{proof}

We continue with the notation in the above proposition to give a decomposition of $V$. The decomposition that we will arrive at is similar to the decomposition of $V$ as presented in the proof of \cite[Lemma 8.9]{ellenberg_venkatesh_westerland}. Given $f \in O$, let $h \in \mathrm{GU}_q(V)$ be so that $f \circ h = f$ and thus $\operatorname{im}(h-1) \subseteq \ker f$. Given $\lambda \in \bar{\mathbb{F}}_\ell^\times$, let $V_{\lambda}$ be the sum of the generalized eigenspaces of $h$ on $V$ as a $\mathbb{Z}_\ell$-module for all eigenvalues that reduce to $\lambda$ in $\bar{\mathbb{F}}_\ell$. More precisely, let $\overline{V} = V \otimes \overline{\mathbb{Q}}_\ell$, and then set 
    $$V_{\lambda} = V \cap \bigoplus_{|\nu - \lambda| < 1} \overline{V}_{\nu}$$
where $\overline{V}_{\nu}$ is the generalized $\nu$-eigenspace of $h$ in $\overline{V}$. Equivalently, $V_{\lambda}$ consists of all $v \in V$ for which $g_\lambda(h)^n v \to 0$ as $n \to \infty$ where $g_\lambda(x) \in \mathbb{Z}_\ell[x]$ is a polynomial which reduces to the minimal polynomial of $\lambda$ over $\mathbb{F}_\ell$. 

Let $W_{\lambda}$ be the sum of all other generalized eigenspaces $\overline{V}_\nu$ of $h$ in $\overline{V}$ for all eigenvalues $\nu \neq \lambda$ and $\nu \neq \frac{q}{\lambda}$, i.e. 
$$W_{\lambda} \coloneq \overline{V} \cap \bigoplus_{\substack{\nu \neq \lambda \\ \nu \neq \frac{q}{\lambda}}} \overline{V}_{\nu}.$$
Equivalently, $W_{\lambda}$ consists of all $v \in V$ for which $g_\lambda(h)^n g_{\frac{q}{\lambda}}(h)^n v \to 0$ as $n \to \infty$.

\begin{lemma}\label{lemma:V_i_decomposes}
    Given $h \in \mathrm{GU}_q(V)$, 
    \begin{enumerate}
    	\item  there is a decomposition
    	$$V = \begin{cases} V_{\lambda} \oplus V_{\frac{q}{\lambda}} \oplus W_{\lambda} &\text{if } \lambda \text{ is not a square root of } q \text{ in } \overline{\mathbb{F}}_\ell \\ V_{ \lambda} \oplus W_{\lambda} &\text{if } \lambda \text{ is a square root of } q \text{ in } \overline{\mathbb{F}}_\ell. \end{cases}$$
    	as $D$-modules. 

    	\item $V_{\lambda}$ and $V_{\frac{q}{\lambda}}$ are both totally isotropic and both orthogonal to $W_{\lambda}$ with respect to $H$, i.e. $H|_{V_{\lambda} \times V_{\lambda}}$, $H|_{V_{\frac{q}{\lambda}} \times V_{\frac{q}{\lambda}}}$, $H|_{V_{\lambda} \times W_{\lambda}}$, $H|_{V_{\frac{q}{\lambda}} \times W_{\lambda}}$ are all $0$. In particular, $W_{\lambda}^\perp = V_{\lambda} + V_{\frac{q}{\lambda}}$ whether or not $\lambda$ is a square root of $q$. Moreover, the pairing $H|_{V_\lambda \times V_{q/\lambda}}$ is nonsingular and  $V_\lambda$ and $V_{q/\lambda}$ are paired perfectly by $H$.

    	\item $W_{\lambda}^\perp$ and $W_{\lambda}$ have orthonormal bases.
    \end{enumerate}
\end{lemma}

\begin{proof}
    We show that $V$ decomposes as $V_{\lambda} \oplus V_{\frac{q}{\lambda}} \oplus W_{\lambda}$ in the case that $\lambda$ is not a square root of $q$ in $\overline{\mathbb{F}}_\ell$. For $v \in V$, write $v$ uniquely as $v = v_{\lambda} + v_{\frac{q}{\lambda}} + w$ where $v_\lambda, v_{\frac{q}{\lambda}}, w$ respectively lie in the $\overline{\mathbb{Q}}_\ell$-spans of  $V_{\lambda}, V_{\frac{q}{\lambda}}, W_{\lambda}$. In particular,
        \begin{align*}
        g_\lambda(h)^n g_{\frac{q}{\lambda}}(h)^n v &= g_\lambda(h)^n g_{\frac{q}{\lambda}}(h)^n v_\lambda + g_\lambda(h)^n g_{\frac{q}{\lambda}}(h)^n v_{\frac{q}{\lambda}} + g_\lambda(h)^n g_{\frac{q}{\lambda}}(h)^n w \\
        &= g_{\frac{q}{\lambda}}(h)^n g_\lambda(h)^n  v_\lambda + g_\lambda(h)^n g_{\frac{q}{\lambda}}(h)^n v_{\frac{q}{\lambda}} + g_\lambda(h)^n g_{\frac{q}{\lambda}}(h)^n w.
        \end{align*}
    Since $g_\lambda(h)^n v_\lambda \to 0$ and $g_{\frac{q}{\lambda}}(h)^n v_{\frac{q}{\lambda}}$ converge to $0$ as $n \to \infty$, they are in $V$ for sufficiently large $n$, so $g_\lambda(h)^n g_{\frac{q}{\lambda}}(h)^n w \in W_{\lambda}$ for sufficiently large $n$. Moreover, $g_\lambda(h)$ and $g_{\frac{q}{\lambda}}(h)$ are invertible on $W_{\lambda}$, so in fact $w \in W_{\lambda}$. Thus, $v_\lambda+v_{\frac{q}{\lambda}} \in V$, and proceeding similarly, we have that $v_\lambda \in V_\lambda$ and $v_{\frac{q}{\lambda}} \in V_{\frac{q}{\lambda}}$. Therefore, $V = V_\lambda \oplus V_{\frac{q}{\lambda}} \oplus W_\lambda$.
    Similarly, $V = V_\lambda \oplus W_\lambda$ if $\lambda$ is a square root of $q$ in $\bar{\mathbb{F}_\ell}$.
    
    We show that $V_{\lambda}$ and $V_{\frac{q}{\lambda}}$ are both totally isotropic and both orthogonal to $W_{\lambda}$ with respect to $H$. Showing that $V_{\lambda}$ is totally isotropic and orthogonal to $W_{\lambda}$ is tantamount to showing that $H(x,y) = 0$ whenever $x \in V_{\lambda}$ and $y \in V_{\lambda} \oplus W_{\lambda}$. Note that $g_\lambda(h)^n x \to 0$ as $n \to \infty$ and for all $n$ there exists some $z_n \in V$ such that $y = g_{\frac{q}{\lambda}}(h)^n z_n$ because $g_{\frac{q}{\lambda}}$ is invertible on both $V_{\lambda}$ and $W_{\lambda}$. Write $g_\lambda(x) = x^k + a_{k-1} x^{k-1} + \cdots + a_0$; note that $a_0 \in \mathbb{Z}_\ell^\times$ because $\lambda \in \overline{\mathbb{F}}_\ell^\times$. In particular, $g_{\frac{q}{\lambda}}(x) = x^k + \frac{a_1}{a_0} qx^{k-1} + \cdots + \frac{a_{k-1}}{a_0} q^{k-1} + \frac{q^k}{a_0}$, and
        \begin{align*}
        H(x,y) &= H(x, g_{\frac{q}{\lambda}}(h)^n z_n) \\
        &= H(x, h^k g_{\frac{q}{\lambda}}(h)^{n-1} z_n) 
        + H\left(x, \frac{a_{1}}{a_0} q h^{k-1} g_{\frac{q}{\lambda}}(h)^{n-1} z_n\right) + \cdots + H\left(x, \frac{q^k}{a_0} g_{\frac{q}{\lambda}}(h)^{n-1} z_n \right).
        \end{align*}
    Since $H$ is bilinear over $\mathbb{Z}_\ell$, since the coefficients of $g_\lambda$ are in $\mathbb{Z}_\ell$,  and since $h$ scales $H$ by $q$, the above equals
        \begin{align*}
        &H\left(x, h^k g_{\frac{q}{\lambda}}(h)^{n-1} z_n \right) + H\left( hx, \frac{a_1}{a_0} h^{k} g_{\frac{q}{\lambda}}(h)^{n-1} z_n \right) + \cdots + H\left(h^k x, \frac{1}{a_0} h^k g_{\frac{q}{\lambda}}(h)^{n-1} z_n \right) \\
        =& H\left(x, h^k g_{\frac{q}{\lambda}}(h)^{n-1} z_n \right) + H \left( \frac{a_1}{a_0} hx, h^k g_{\frac{q}{\lambda}}(h)^{n-1} z_n \right) + \cdots + H \left( \frac{h^k}{a_0} x, h^k g_{\frac{q}{\lambda}}(h)^{n-1} z_n \right) \\
        =& H\left( \frac{g_\lambda(h)}{a_0} x, g_{\frac{q}{\lambda}}(h)^{n-1} z_n \right).
        \end{align*}
    Iterating, we have that $H(x,y) = H\left( \left(\frac{g_\lambda(h)}{a_0} \right)^n x, z_n \right)$ and hence that $H(x,y)$ lies in $H(g_\lambda(h)^n x, V)$ for all $n$. Therefore, $H(x,y) = 0$ as desired, so $V_{\lambda}$ is totally isotropic and orthogonal to $W_{\lambda}$. By symmetry, $V_{\frac{q}{\lambda}}$ is totally isotropic and orthogonal to $W_{\lambda}$.    
    
    Since $V$ is a free $D$-module, $W_{\lambda}^\perp$ and $W_{\lambda}$ are projective $D$-modules. Moreover, $H$ is nonsingular and hence restricts to nonsingular Hermitian forms on $W_{\lambda}^\perp$ and $W_{\lambda}$ by \Cref{lemma:knus_qhfr_lemma_3.6.2}. Therefore, $W_{\lambda}^\perp$ and $W_{\lambda}$ are Hermitian spaces with respect to $H$ and have orthonormal bases by \Cref{proposition:free_hermitian_spaces_over_Z_ell_zeta_d_have_orthonormal_bases} and in particular are free $D$-modules. 
\end{proof}

For a double division ring $D = B \times B^{\mathrm{op}}$, let $\pi_1 = (1,0), \pi_2 = (0,1)$ be the orthogonal idempotents. The involution swaps them: $\overline{\pi_i} = \pi_{3-i}$. Any $D$-module $M$ decomposes as $M = \pi_1 M \oplus \pi_2 M$, and a Hermitian form $H: V \times V \to D$ satisfies $H(\pi_i m, \pi_j n) = 0$ whenever $j \neq 3-i$. 
For a division ring $D$, set $\pi_1$ and $\pi_2$ to simply be $1 \in D$.  

\begin{context} \label{context:lemmas_splitting_proposition}
    Let $D$ be a division ring or a double division quotient of $\bbZ_\ell[\zeta_d]_{\zeta_d \neq \pm 1}$. 
    If $D$ is a double division ring, then let $\pi_1 = (1,0), \pi_2 = (0,1)$ be the orthogonal idempotents. The involution swaps them: $\overline{\pi_i} = \pi_{3-i}$. Any $D$-module $M$ decomposes as $M = \pi_1 M \oplus \pi_2 M$, and a Hermitian form $H: V \times V \to D$ satisfies $H(\pi_i m, \pi_j n) = 0$ whenever $j \neq 3-i$. If $D$ is instead a division ring, then let $\pi_1$ and $\pi_2$ simply be $1 \in D$.  
    
    Furthermore, let $(V,H)$ be a nonsingular Hermitian space over $D$, and let $V_1,V_q \subset V$ be totally isotropic submodules paired perfectly by $H$. Let $A \cong \prod_{i=1}^r A_i$ where each $A_i$ is a nonzero quotient of $D$ (if $D$ is an integral domain) or of one component of $B \times B^{\mathrm{op}}$ (if $D$ is a double division ring). Let $e_i \in A$ be the orthogonal idempotent reducing to $1 \in A_i$. For each $i$, let $\kappa_i \in \{1,2\}$ be the index with $\pi_{\kappa_i} e_i = e_i$.
\end{context}

\begin{lemma} \label{lemma:existence_of_a_hyperbolic_pair}
    Let notation be as in \Cref{context:lemmas_splitting_proposition}. Let $f: V \to A$ be $D$-linear and surjective with $f|_{V_q} = 0$. Given $i \in \{1,\ldots,r\}$ and any $v \in V_1$ such that $f(v) = e_i$, there exists $w \in V_q$ such that 
    \begin{enumerate}
        \item $H(v,w) = \overline{\pi}_{\kappa_i}$;
        \item the $D$-submodule $X_i$ of $V$ generated by $v$ and $w$ is free and of rank $2$ if $D$ is a division ring and of rank $1$ if $D$ is a double division ring;
        \item $H|_{X_i}$ is nonsingular;
        \item $w \in \overline{\pi}_{\kappa_i} V_q$.
    \end{enumerate}
\end{lemma}
\begin{proof}
    Replace $v$ with $\pi_{\kappa_i} v$ so that $v \in \pi_{\kappa_i} V_1$. Whether or not $D$ is an integral domain, there is some $w \in V_1 \oplus V_q$ such that $H(v,w)$ is nonzero because $H$ is nonsingular. Since $v = \pi_{\kappa_i} v$, we have $H(v,w) = H(\pi_{\kappa_i} v, w) = \overline{\pi}_{\kappa_i} H(v,w)$, so $H(v,w) \in \overline{\pi}_{\kappa_i} D$. Replace $w$ with its projection in $V_q$ and replace $w$ with an appropriate scalar multiple so that $H(v, w) = \overline{\pi_{\kappa_i}}$; this is possible even when $D = B \times B^{\op}$ is a double division ring because $H(v,w)$ is a nonzero element of $\overline{\pi_{\kappa_i}} \cdot D$ and hence the scalar multiplication can be accomplished with an element of $B^\times$ or $(B^{\op})^\times$. Further replace $w$ with $\overline{\pi_{\kappa_i}} \cdot w$ so that $w \in \overline{\pi_{\kappa_i}} V_q$. Let $X_i$ be the free $D$ submodule of $V_1 \oplus V_q$ generated by $v$ and $w$. Recall that $v$ and $w$ form a basis of $X_i$ if $D$ is an integral domain and that $v+w$ forms a basis of $X_i$ otherwise.
    
    We show that $H|_{X_i}$ is nonsingular, i.e. that the adjoint map $X_i \to X_i^*, \quad v \mapsto H(v, -)$ is a $D$-isomorphism. If $D$ is a double division ring, then the adjoint map sends $a(v+w)$, where $a \in D$ to $H(a(v+w), -)$, which in turn is the map $X_i \to D$ that sends $v+w$ to $\overline{a}$. Thus, the adjoint map is a bijection and hence a $D$-module isomorphism. If $D$ is an integral domain, then the adjoint map sends $av+bw$, where $a,b \in D$ to $H(av+bw, -)$, which in turn is the map $X_i \to D$ sending $v$ to $\overline{b}$ and $w$ to $\overline{a}$. The adjoint map is therefore also a bijection and hence a $D$-module isomorphism in this case as desired. Thus, $V_1 \oplus V_{q} = X_i \perp X_i^\perp$ by \Cref{lemma:knus_qhfr_lemma_3.6.2}. 
\end{proof}

\begin{lemma} \label{lemma:extending_a_hyperbolic_frame}
    Let notation be as in \Cref{context:lemmas_splitting_proposition}. Let $f: V \to A$ be $D$-linear and surjective with $f|_{V_q} = 0$. 
    Let $U_s \subset V_1 \oplus V_q$ be a nonsingular submodule of rank $2s$, and let ${(v_i, w_i)}_{i=1}^s$ be pairs in $V_1 \times V_q$ satisfying:
    \begin{enumerate}
        \item $f(v_i) = e_i$, $v_i \in \pi_{\kappa_i} V_1$, $w_i \in \overline{\pi}_{\kappa_i} V_q$;
        \item $H(v_i, w_j) = \delta_{ij} \overline{\pi}_{\kappa_j}$;
        \item $f(U_s^\perp) \subseteq A_{s+1} \times \cdots \times A_r$.
    \end{enumerate}
    Given $s < r$, there exist $v_{s+1} \in V_1 \cap U_s^\perp$ and $w_{s+1} \in V_q \cap U_s^\perp$ such that:
    \begin{enumerate}
        \item $f(v_{s+1}) = e_{s+1}$, $v_{s+1} \in \pi_{\kappa_{s+1}} V_1$, $w_{s+1} \in \overline{\pi}_{\kappa_{s+1}} V_q$;
        \item $H(v_{s+1}, w_j) = 0$ for $j \leq s$, and $H(v_{s+1}, w_{s+1}) = \overline{\pi}_{\kappa_{s+1}}$;
        \item the free $D$-submodule $X_{s+1}$ generated by $v_1,\ldots,v_{s+1},w_1,\ldots,w_{s+1}$ is nonsingular of rank 2(s+1), so we set $U_{s+1} = U_s \perp X_{s+1}$;
        \item $f(U_{s+1}^\perp) \subseteq A_{s+2} \times \cdots \times A_r$.
    \end{enumerate}
\end{lemma}
\begin{proof}
    Write $M = V_1 \oplus V_q$ for shorthand notation.
    Suppose inductively that $v_1,\ldots,v_s \in V_1$ and $w_1,\ldots,w_s \in V_q$ satisfy the properties where $1 \leq s < r$. 
    By the inductive hypothesis, $M = U_s \perp U_s^\perp$. Let $v \in V_1$ such that $f(v) = e_{s+1}$. Replace $v$ with $\pi_{\kappa_{s+1}} v$ so that $v \in \pi_{\kappa_{s+1}} V_1$ as well. The $U_s$-component $v'$ of $v$ has no components in $v_1,\ldots,v_s$ because $f(v_i) = e_i$. Therefore, $v'$ only has components in $w_1,\ldots,w_s$. Let $v_{s+1}'$ be the $U_s^\perp$ component of $v$, i.e. $v_{s+1}' = v - v'$. In particular, $f(v'_{s+1}) = e_{s+1}$, the element $v'_{s+1}$ is an element of $\pi_{\kappa_{s+1}} V_1$ just as $v$ is, and $H(v'_{s+1}, w_j) = 0$ for $1 \leq j \leq s$ because $w_j \in U_s$. By \Cref{lemma:existence_of_a_hyperbolic_pair} applied to $v'_{s+1}$, there exists $w \in V_q$ such that $H(v'_{s+1}, w) = \overline{\pi}_{\kappa_{s+1}}$ and $w \in \overline{\pi}_{\kappa_{s+1}} V_q$. 
    \begin{equation}
    \label{eq:replace_w_s_plus_1_to_get_symplectic_basis}
    	w'_{s+1} = w - \sum_{i=1}^s H(v_i, w) w_i.
    \end{equation}
    Since $v_i \in \pi_{\kappa_i} M$ and $w \in \overline{\pi}_{\kappa_{s+1}}M$, either $\kappa_i = \kappa_{s+1}$ (so $H(v_i,w) \in \overline{\pi}_{\kappa_i} D$, and $H(v_i, w) w_i$ is unchanged since $w_i \in \overline{\pi}_{\kappa_i} M$), or $\kappa_i \neq \kappa_{s+1}$ (so $H(v_i,w) = 0$).
    For $1 \leq i' \leq s$,
    \begin{align*}
    	H(v_{i'}, w'_{s+1}) &= H\left(v_{i'}, w - \sum_{i=1}^s H(v_{i}, w) w_i \right) \\
    	&= H(v_{i'}, w) - \sum_{i=1}^s H(v_{i'}, H(v_{i}, w) w_i) \\
    	&= H(v_{i'}, w) - H(v_{i'}, H(v_{i'}, w) w_{i'}) \\
    	&= H(v_{i'}, w) - H(v_{i'}, w) H(v_{i'}, w_{i'}) \\
    	&= H(v_{i'}, w) - H(v_{i'}, w) \cdot \overline{\pi}_{\kappa_{i'}}.
    \end{align*}
    The last step holds because $H(v_{i'}, w_{i'}) = \overline{\pi}_{\kappa_{i'}}$ holds whether $D$ is an integral domain or a double division ring. The RHS expression equals $0$; when $\kappa_{i'} = \kappa_{s+1}$, we have $H(v_{i'},w) \in \overline{\pi}_{\kappa_{i'}}D$, and when $\kappa_{i'} \neq \kappa_{s+1}$, the RHS expression is $0 - 0 = 0$.
    
    
    Moreover,
    \begin{align*}
    	H(v'_{s+1}, w'_{s+1}) &= H\left(v'_{s+1}, w - \sum_{i=1}^s H(v_{i}, w) w_i \right) \\
    	&= H(v'_{s+1}, w) - \sum_{i=1}^s H(v'_{s+1}, H(v_{i}, w) w_i) \\
    	&= \overline{\pi}_{\kappa_{s+1}}.
    \end{align*}
    
    Let $\tilde{X}_{s+1}$ be the free $D$-module generated by $v_1,\ldots,v_{s}, v'_{s+1}, w_1,\ldots,w_s,w'_{s+1}$. In particular, $\tilde{X}_{s+1} = U_s \perp \langle v'_{s+1}, w'_{s+1} \rangle$, so $\tilde{X}_{s+1}$ is nonsingular and hence $M = \tilde{X}_{s+1} \perp \tilde{X}_{s+1}^\perp$. Since $U_s \subset \tilde{X}_{s+1}$, we have that $\tilde{X}_{s+1}^\perp \subset U_s^\perp$ and hence $f(\tilde{X}_{s+1}^\perp) \subseteq A_{s+1} \times \cdots \times A_r$ by the inductive hypothesis. Choose a basis $x_1,\ldots,x_k$ of $\tilde{X}_{s+1}^\perp$, let $a_1,\ldots,a_k \in D$ be so that the element $f(x_i - a_i v'_{s+1})$ of $A$ has image $0$ when projected to $A_{s+1}$.
    Note that $x_1 - a_1 v'_{s+1}, \ldots, x_k - a_k v'_{s+1}$ are linearly independent over $D$. Let $X_{s+1}$ be the free module that they generate. In particular, $f(X_{s+1}) \subseteq A_{s+2} \times \cdots \times A_r$, and $X_{s+1}$ is orthogonal to $U_s$, i.e. $U_s \subseteq (X_{s+1})^\perp$. Since $f$ maps $V_1$ surjectively onto $A$, but $f(U_s) \subseteq A_1 \times \cdots \times A_s$ and $f(X_{s+1}) \subseteq A_{s+2} \times \cdots \times A_r$, there must be some element of $U_s^\perp \cap (X_{s+1})^\perp$ that maps to $e_{s+1}$ under $f$. Let $v_{s+1}$ be $\pi_{\kappa_{s+1}}$ times the $V_1$-component of this element so $v_{s+1} \in V_1 \cap U_s^\perp \cap (X_{s+1})^\perp$, $v_{s+1} \in \pi_{\kappa_{s+1}} V_1$, and $f(v_{s+1}) = e_{s+1}$. Furthermore, $H(v_{s+1}, w_j) = 0$ for $1 \leq j \leq s$ because $w_j \in U_s$. 
    Because $v_{s+1} \in (X_{s+1})^\perp$, there exists some element $w \in V_q \cap (X_{s+1})^\perp$ such that $H(v_{s+1}, w)$ is invertible in $\overline{\pi}_{\kappa_{s+1}} D$. Replace $w$ with $\overline{\pi}_{\kappa_{s+1}} \cdot w$, and then with an appropriate scalar multiple so that $w \in \overline{\pi}_{\kappa_{s+1}} V_q$ and $H(v_{s+1}, w) = \overline{\pi}_{\kappa_{s+1}}$.
    Define $w_{s+1}$ as we defined $w'_{s+1}$ in \eqref{eq:replace_w_s_plus_1_to_get_symplectic_basis} so that $w_{s+1} \in \overline{\pi}_{\kappa_{s+1}} V_q$ and $H(v_i, w_j) = \delta_{ij} \overline{\pi}_{\kappa_j}$ for all $1 \leq i,j \leq s+1$.
    In particular, if $D$ is a double division ring, then $H(v_i+w_i, v_j+w_j) = \delta_{ij}$. Let $U_{s+1}$ be the free module generated by $v_1,\ldots,v_{s+1}, w_1,\ldots,w_{s+1}$. In particular, $H|_{U_{s+1}}$ is nonsingular. Moreover, the ranks of $\tilde{X}_{s+1}$ and $\tilde{X}_{s+1}^\perp$ are $2(s+1)$ and $k$ respectively and the rank of $X_{s+1}$ is $k$, so the rank of $(X_{s+1})^\perp$ is $2(s+1)$. Moreover, the rank $2(s+1)$ submodule $U_{s+1}$ of $(X_{s+1})^\perp$ is a nonsingular Hermitian module for $H$, so in fact $U_{s+1}$ is precisely $(X_{s+1})^\perp$. Thus, $f(U_{s+1}^\perp) = f(X_{s+1}) \subseteq A_{s+2} \times \cdots \times A_r$, so we can indeed recursively obtain the elements of $V_1$ and $V_q$ with the desired properties.     
\end{proof}

\begin{lemma} \label{lemma:isometry_from_matched_hyperbolic_frames}
    Let notation be as in \Cref{lemma:extending_a_hyperbolic_frame}, and assume that $q \not\equiv 1 \pmod{\ell}$. Given $f, f' \in O$, let $(V_1 \oplus V_q) \perp W$ and $(V'_1 \oplus V'_q) \perp W'$ be the decompositions of \Cref{lemma:V_i_decomposes} applied with $\lambda = 1$ to $h, h' \in \mathrm{GU}_q(V)$ such that $f \circ h = f$ and $f' \circ h' = f'$. 
    Let ${(v_i, w_i)}_{i=1}^r$ and ${(v'_i, w'_i)}_{i=1}^r$ be tuples obtained as in \Cref{lemma:extending_a_hyperbolic_frame} for $f$ and $f'$ respectively. Then there exists $\varphi \in \mathrm{U}(V)$ such that $f = f' \circ \varphi$, $\varphi(v_i) = v'_i$ and $\varphi(w_i) = w'_i$ for all $i$. If additionally $\operatorname{rank} W > 0$, then $\varphi$ can be chosen with $\det\varphi = 1$.

    
    
\end{lemma}
\begin{proof}
    Applying \Cref{lemma:knus_qhfr_lemma_3.6.2} shows that $V_1 \oplus V_q = U_r \perp U_r^\perp$ and $V'_1 \oplus V'_q = U'_r \perp {U'}_r^\perp$ where $U_r$ and $U'_r$ are respectively the $D$-modules generated by $v_1,\ldots,v_r,w_1,\ldots,w_r$ and $v'_1,\ldots,v'_r,w'_1,\ldots,w'_r$. In particular, we have the orthogonal decompositions 
    $$U_r \perp U_r^\perp \perp W = V = U'_r \perp {U'}_r^\perp \perp W',$$ and by the inductive construction, $f$ vanishes on $U_r^\perp$ and ${U'}_r^\perp$. Let $\varphi: V \to V$ be any $D$-linear automorphism which takes $v_i$ to $v'_i$, takes $w_i$ to $w'_i$, and transforms $U_r^\perp \perp W$ to ${U'}_r^\perp \perp W'$ such that $H(\varphi(v), \varphi(w)) = H(v,w)$ for all $v,w \in U_r^\perp \perp W$. The latter is possible because $U_r^\perp \perp W$ and ${U'}_r^\perp \perp W'$ are nonsingular Hermitian modules over $D$ and hence have orthonormal bases by \Cref{proposition:free_hermitian_spaces_over_Z_ell_zeta_d_have_orthonormal_bases}. In the case that $D$ is a double division ring, the former can be accomplished by letting $\varphi$ send $v_i + w_i$ to $v'_i + w'_i$ and extending $D$-linearly. Note that $H(\varphi(v), \varphi(w)) = H(v,w)$ for all $v,w \in U_r$. Moreover, $f(v_i) = e_i = f'(v'_i)$ and $f$ vanishes on $w_i$, $w'_i$, $U_r^\perp$, ${U'}_r^\perp$, $W$, and $W'$. Therefore, $\varphi$ is indeed an element of $\mathrm{U}(V)$ such that $f = f' \circ \varphi$. The action of $\mathrm{U}(V)$ on $O$ is thus transitive as desired.
    
    We now show that the action of $\mathrm{SU}(V)$ on $O$ is transitive when $s > 2r$. As above, given $f, f' \in O$, let $\varphi \in \mathrm{U}(V)$ such that $f = f' \circ \varphi$. In the decomposition $V = (V_1' \oplus V_q') \perp W'$ with respect to $f'$, the summand $W'$ is a free $D$-module of positive rank. Moreover, $\det \varphi \cdot \overline{\det \varphi} = 1$, so in particular $\det \varphi$ is an invertible element of $D$. Choose an orthonormal basis of $W'$ via \Cref{proposition:free_hermitian_spaces_over_Z_ell_zeta_d_have_orthonormal_bases}. Let $\alpha: V \to V$ be the $D$-linear automorphism acting as the identity on $V_1' \oplus V_q'$ and as the matrix 
        $$\begin{pmatrix} \frac{1}{\det \varphi} & 0 & 0 & \cdots & 0 \\ 0 & 1 & 0 & \cdots & 0 \\ 0 & 0 & 1 & \cdots & 0 \\ \vdots & \vdots & \vdots & \ddots & \vdots \\ 0 & 0 & 0 & \cdots & 1 \end{pmatrix}$$
    on $W'$ with respect to the orthonormal basis. 

    Note that $f = f' \circ (\alpha \circ \varphi)$ because 
    \begin{align*}
        f(v_i) &= e_i = f'(v'_i) = f'(\alpha(\varphi(v_i)))\\
        f(w_i) &= 0 = f'(w'_i) = f'(\alpha(w'_i)) = f'(\alpha(\varphi(w_i))) \\
        f|_{U_r^\perp \perp W} &= 0 = f'|_{{U'_r}^\perp \perp W'} \text{ and } (\alpha \circ \varphi)(U_r^\perp \perp W) \subseteq {U'_r}^\perp \perp W'
    \end{align*}
    Moreover, $\det (\alpha \circ \varphi) = 1$, so the action of $\mathrm{SU}(V)$ on $O$ is transitive as claimed.
\end{proof}

\begin{proposition}\label{proposition:nonempty_O_implies_transitive_action_of_U_V}
    With notation as in \Cref{proposition:nonempty_O} and assuming that $q \not\equiv 1 \pmod{\ell}$, if $O$ is nonempty, then $\mathrm{U}(V)$ acts transitively on $O$. If $O$ is nonempty and the rank $s$ of $V$ as a $D$-module is greater than $2r$, then $\mathrm{SU}(V)$ acts transitively on $O$.
\end{proposition}

\begin{proof}
    Setup notation as in \Cref{context:lemmas_splitting_proposition}.

    Given $f \in O$, let $h \in \mathrm{GU}_q(V)$ be such that $f \circ h = f$. Apply \Cref{lemma:V_i_decomposes} with $\lambda = 1$ to obtain the orthogonal decomposition $(V_1 \oplus V_q) \perp W$; in particular, $H$ pairs $V_1$ and $V_q$ perfectly.

    We inductively construct a list $\{(v_i, w_i)\}_{i=1}^r$ of tuples satisfying the following properties, which are slightly stronger than the hypotheses of \Cref{lemma:extending_a_hyperbolic_frame}: for each $s = 1,\ldots, r$, writing $U_s$ for the $D$-submodule of $V_1 \oplus V_q$ generated by $v_1,\ldots,v_s, w_1,\ldots,w_s$, we have
    \begin{enumerate}
        \item $f(v_i) = e_i$, $v_i \in \pi_{\kappa_i} V_1$, $w_i \in \overline{\pi}_{\kappa_i} V_q$;
        \item $H(v_i, w_j) = \delta_{ij} \overline{\pi}_{\kappa_j}$;
        \item $f(U_s^\perp) \subseteq A_{s+1} \times \cdots \times A_r$.
    \end{enumerate}

    For the base case $s = 1$, choose $v \in V_1$ such that $f(v) = e_1$, and replace $v$ with $\pi_{\kappa_1} v$. Apply \Cref{lemma:existence_of_a_hyperbolic_pair} to obtain $w \in \overline{\pi}_{\kappa_1} V_q$ such that $H(v, w) = \overline{\pi}_{\kappa_1}$; let $\tilde{X}_1$ be the free $D$-submodule generated by $v$ and $w$. By \Cref{lemma:existence_of_a_hyperbolic_pair}, $H|_{\tilde{X}_1}$ is nonsingular, so $V_1 \oplus V_q = \tilde{X}_1 \perp \tilde{X}_1^\perp$.

    To ensure property (3), adjust the orthogonal complement as follows. Choose a basis $x_1, \ldots, x_k$ of $\tilde{X}_1^\perp$, and for each $i$, choose $a_i \in D$ so that $f(x_i - a_i v)$ has image $0$ in $A_1$. Set $y_i = x_i - a_i v$; the elements $y_i$ are linearly independent: if $\sum c_i y_i = 0$, then $\sum c_i x_i = (\sum c_i a_i) v$, where the left side lies in $\tilde{X}_1^\perp$ and the right side in $\tilde{X}_1$, so both must be zero, hence all $c_i = 0$. Let $Y_1$ be the free $D$-submodule generated by $y_1, \ldots, y_k$. Then $f(Y_1) \subseteq A_2 \times \cdots \times A_r$, and $\operatorname{rank} Y_1 = k = \operatorname{rank} \tilde{X}_1^\perp$.

    Since $f$ maps $V_1$ surjectively onto $A$, but $f(Y_1)$ has no $A_1$-component, there exists some element of $(Y_1)^\perp$ whose image under $f$ has a nonzero $A_1$-component. Let $v_1$ be $\pi_{\kappa_1}$ times the $V_1$-component of such an element, scaled so that $f(v_1) = e_1$. Then $v_1 \in (Y_1)^\perp \cap \pi_{\kappa_1} V_1$. Because $H$ pairs $V_1$ and $V_q$ perfectly and $(Y_1)^\perp$ is a nonsingular $D$-submodule, there exists some element of $V_q \cap (Y_1)^\perp$ paired invertibly with $v_1$; scale and project to obtain $w_1 \in \overline{\pi}_{\kappa_1} V_q \cap (Y_1)^\perp$ such that $H(v_1, w_1) = \overline{\pi}_{\kappa_1}$. Let $U_1$ be the free $D$-submodule generated by $v_1, w_1$. Then $U_1 \subset (Y_1)^\perp$, is nonsingular of rank 2, and $(Y_1)^\perp$ has rank 2, so $U_1 = (Y_1)^\perp$. Thus $f(U_1^\perp) = f(Y_1) \subseteq A_2 \times \cdots \times A_r$, establishing property (3).

    Inductively, for $s = 1, \ldots, r-1$, supposing that we have a list $\{(v_i, w_i)\}_{i=1}^s$ of tuples, apply \Cref{lemma:extending_a_hyperbolic_frame} to extend it to $\{(v_i, w_i)\}_{i=1}^{s+1}$. This extended list itself satisfies the hypotheses of the lemma and hence by induction concludes to yield a tuple $\{(v_i, w_i)\}_{i=1}^r$ satisfying these hypotheses.

    Given any $f' \in O$, let $h' \in \mathrm{GU}_q(V)$ be such that $f' \circ h' = f'$, and similarly obtain a decomposition $(V'_1 \oplus V'_q) \perp W'$ and a tuple $\{(v'_i, w'_i)\}_{i=1}^r$. By \Cref{lemma:isometry_from_matched_hyperbolic_frames}, there exists $\varphi \in \mathrm{U}(V)$ such that $f = f' \circ \varphi$. This proves the transitivity of $\mathrm{U}(V)$ on $O$. If additionally $s > 2r$, then $\operatorname{rank} W = s - 2r > 0$, so $\varphi$ can be chosen with $\det\varphi = 1$. Thus $\varphi \in \mathrm{SU}(V)$, proving the transitivity of $\mathrm{SU}(V)$ on $O$.
\end{proof}

\begin{proof}[Proof of \Cref{proposition:generalizing_EVW_Lemma 8.9_to_Z_dZ_covers}]
    Decompose $\mathbb{Z}_\ell[\zeta_d]$ as $\prod_i D_i$ as in \eqref{eq:Z_ell_zeta_d_factorization_as_ring_with_involution} and say that $(V,H)$ corresponds to the tuple $(V_i,H_i)_i$ of Hermitian forms over $D_i$ via  \Cref{lemma:hermitian_modules_over_product_of_involution_rings_split_along_involution_ring_factors}. Since $H$ is nonsingular, all the $H_i$ are nonsingular, so the $V_i$ are free $D_i$-modules. Moreover, let $A_i$ be the $D_i$-component of $A$, so $A \cong \prod_i A_i$.
    
    Let $O_i$ be the set
        $$O_i = \{f_i: V_i \to A_i \text{ surjective, } D_i\text{-linear, and there exists } h \in \mathrm{GU}_q(V_i) \text{ with } f_i \circ h = f_i \}.$$
    The correspondences of $(V,H)$ with the tuple $(V_i, H_i)_i$ and of $A$ with $\prod_i A_i$ induces a bijection $O \xrightarrow{\sim} \prod_i O_i$ equivariant for the actions of $\mathrm{U}(V) \cong \prod_i \mathrm{U}_i(V)$. If $V$ has a free $\mathbb{Z}_\ell[\zeta_d]$-module of rank at least $2r$ as a direct summand, then each $V_i$ is a free $D_i$-module of at least that rank. By \Cref{proposition:nonempty_O} and \Cref{proposition:nonempty_O_implies_transitive_action_of_U_V}, $O_i$ is nonempty and $\mathrm{U}(V_i)$ acts transitively on $O_i$ whenever $\zeta_d \neq -1$ in $D_i$. In fact, $\mathrm{SU}(V_i)$ acts transitively on $O_i$ as well if $V$ has a free $\mathbb{Z}_\ell[\zeta_d]$-module of rank at least $2r+1$ as a direct summand. If $\zeta_d = -1$ in $D_i$, then $H_i$ is in fact a symplectic form, so the unitary group $\mathrm{U}(V_i)$ is the symplectic group $\mathrm{Sp}(V_i)$; note that $V_i$ is of even rank over $D_i$ because $H_i$ is nonsingular. By \cite[Lemma 8.9]{ellenberg_venkatesh_westerland}, $O_i$ is nonempty and $\mathrm{U}(V_i) \cong \mathrm{Sp}(V_i)$ acts transitively on $O_i$ because $2r > \dim_{\mathbb{F}_\ell} A_i / \ell A_i$. In fact, one can argue similarly as in the last paragraph of the proof of \Cref{proposition:nonempty_O_implies_transitive_action_of_U_V} applied to the context of the last paragraph of the proof of \cite[Lemma 8.9]{ellenberg_venkatesh_westerland} to show that $\mathrm{SU}(V_i) = \{\varphi \in \mathrm{U}(V_i) : \det(\varphi) = 1\}$  acts transitively on $O_i$. Therefore, $O$ is nonempty and $\mathrm{U}(V)$ (or $\mathrm{SU}(V)$) acts transitively on $O$.
\end{proof}

\begin{corollary} \label{corollary:generalizing_EVW_Lemma 8.9_to_Z_dZ_covers}
    Let $k$ be a field. Let $d \geq 3$ be a prime power not divisible by the characteristic of $k$. Let $\ell$ be a prime number not dividing $2d$ and invertible in $k$. Let $\mathbb{Z}_\ell[\zeta_d] = \mathbb{Z}_\ell[\zeta_d]_{\zeta_d \neq 1}$. Let $q \geq 2$ be an integer such that $q \not\equiv 1 \pmod{\ell}$. Fix a $\mathbb{Z}_\ell[\zeta_d]$-module $A$ of finite cardinality and say that $A$ is a quotient of $\mathbb{Z}_\ell[\zeta_d]^{\oplus r}$.

    Let $C$ be a $\mathbb{Z}/d\mathbb{Z}$-cover of $\mathbb{P}_{k}^1$, and identify $\operatorname{Aut}(C/\mathbb{P}^1) \cong \mathbb{Z}/d\mathbb{Z}$ with $\langle \zeta_d\rangle$. Write $V = T_\ell \operatorname{Jac}(C)$. Whenever there are at least $2r+3$ branch points of the cover $C \to \mathbb{P}^1$ whose monodromies are elements of $(\mathbb{Z}/d\mathbb{Z})^\times$, the set
        $$O = \{f: V \to A \text{ surjective and } \mathbb{Z}_\ell[\zeta_d]\text{-linear, and there exists } h \in \mathrm{GU}_q(V) \text{ with } f \circ h = f\}$$
    is nonempty and $\mathrm{SU}(V)$ acts transitively on $O$.
\end{corollary}

\begin{proof}
    This follows from \Cref{proposition:ell_adic_tate_module_of_prime_order_superelliptic_curve_is_a_free_Z_ell_zeta_d_module} and \Cref{proposition:generalizing_EVW_Lemma 8.9_to_Z_dZ_covers}.
\end{proof}

\begin{remark}
    To apply \cite[Lemma 8.9]{ellenberg_venkatesh_westerland} in the proof of \Cref{proposition:generalizing_EVW_Lemma 8.9_to_Z_dZ_covers}, we needed to know that $V_i$ is of even rank in case $\zeta_d = -1$ in $D_i$. We alternatively show this via \eqref{eq:riemann_hurwitz_for_characters}. For $\zeta_d = -1$ to hold in $D_i$, $d$ must be even. Moreover, the monodromies of the branch points of the cover $C \to \mathbb{P}^1$ are elements of $\mathbb{Z}/d\mathbb{Z}$ whose sum is $0$ by \Cref{proposition:monodromy_type_for_kummer_curves}. In particular, the number of branch points with monodromy congruent to $1$ mod $2$ must be even. Consequently, in expressing $\chi_C$ as a sum of the irreducible characters $\psi_a$ by applying \eqref{eq:riemann_hurwitz_for_characters} and \Cref{lemma:character_theory_of_Z_dZ_cover_of_P1}\eqref{part:chi_g_as_sum_of_irreducible_characters}, the character $\psi_{\frac{d}{2}}$ appears an even number of times. Thus, the $\mathbb{Z}_\ell[X]/(X+1)$-rank in $T_\ell(\operatorname{Jac}(C))$ is even.
\end{remark}

\section{The Burau representations}
\label{section:burau_representations}

This section presents the unreduced and reduced Burau representations, which will be used in \Cref{subsection:big_monodromy_hypothesis} to establish a big monodromy result. Moreover, \Cref{section:orbits_of_burau_representation} studies orbits of the actions of the braid group on vectors given by these Burau representations and the results there are applied to \Cref{section:generalizing_ELS_section_2} and \Cref{section:generalizing_ELS_section_3}.

\begin{definition}
    Let $n \geq 2$ be an integer. The \emph{Artin braid group $B_n$} is the group with $n-1$ generators $\sigma_1,\sigma_2,\ldots,\sigma_{n-1}$ under the braid relations
    \begin{itemize}
        \item $\sigma_i \sigma_j = \sigma_j\sigma_i$ for all $i,j = 1,2,\ldots,n-1$ with $|i-j| \geq 2$, and
        \item $\sigma_i \sigma_{i+1} \sigma_i = \sigma_{i+1} \sigma_i \sigma_{i+1}$ for $i = 1,2,\ldots,n-2$.
    \end{itemize}
\end{definition}

\begin{definition} \label{definition:unreduced_burau_representation}
    For $n \geq 2$, the \emph{unreduced Burau representation of $B_n$} is the representation $\psi_n: B_n \to \operatorname{GL}_n(\mathbb{Z}[t,t^{-1}])$ sending the standard braid $\sigma_i$ to the block matrix
        $$\left({\begin{array}{c|cc|c} I_{i-1} & 0 & 0 & 0 \\ \hline 0 & 1-t & t & 0 \\ 0 & 1 & 0 & 0 \\ \hline  0 & 0 & 0 & I_{n-i-1} \end{array}} \right)$$
    Given some invertible element $\zeta$ in a ring $R$, we call the composition
        $$B_n \stackrel{\psi_n}{\to} \operatorname{GL}_n(\mathbb{Z}[t,t^{-1}]) \xrightarrow{t = \zeta} \operatorname{GL}_n(R)$$
    the \emph{unreduced Burau representation of $B_n$ evaluated at $t=\zeta$}.

    Given a prime number $\ell$, we call the composition
        $$B_n \stackrel{\psi_n}{\to} \operatorname{GL}_n(\mathbb{Z}[t,t^{-1}]) \xrightarrow{\text{modulo } \ell} \operatorname{GL}_n(\mathbb{Z}/\ell\mathbb{Z}[t,t^{-1}])$$
    the \emph{unreduced Burau representation of $B_n$ modulo $\ell$}.
    
    In the case that $R = \mathbb{Z}/\ell\mathbb{Z}[\zeta]$, where $\zeta$ is invertible, we also refer to the unreduced Burau representation of $B_n$ evaluated at $t = \zeta$ as the \emph{unreduced Burau representation of $B_n$ modulo $\ell$ evaluated at $t = \zeta$}.
\end{definition}

\begin{definition} \label{definition:reduced_burau_representation}
    For $n \geq 2$, the \emph{reduced Burau representation of $B_n$} is the representation $\psi_n^r: B_n \to \operatorname{GL}_{n-1}(\mathbb{Z}[t,t^{-1}])$ defined as follows: for $n \geq 3$, $\psi_n^r$ is the homomorphism given by
    \begin{align*}
    \sigma_{1} &\mapsto \left({\begin{array}{cc|c} -t & 1 & 0 \\ 0 & 1 & 0 \\ \hline 0 & 0 & I_{n-3} \end{array}}\right) \\
    \sigma_{i} & \mapsto \left({\begin{array}{c|ccc|c} I_{i-2} & 0 & 0 & 0 & 0 \\ \hline 0 & 1 & 0 & 0 & 0\\0 & t & -t & 1 & 0 \\ 0 & 0 & 0 & 1 & 0 \\ \hline 0 & 0 & 0 & 0 & I_{n-i-2}\end{array}}\right),\quad 2\leq i\leq n-2, \\
    \sigma _{n-1} &\mapsto \left({\begin{array}{c|cc} I_{n-3} & 0 & 0 \\ \hline 0 & 1 & 0\\0 & t & -t\end{array}}\right).
    \end{align*} 
    For $n = 2$, $\psi_2^r(\sigma_1) = (-t)$.

    We define the \emph{reduced Burau representation of $B_n$ evaluated at $t = \zeta$} (for an invertible element $\zeta$ in a ring $R$), the \emph{reduced Burau representation of $B_n$ modulo $\ell$} (for a prime number $\ell$), and the \emph{reduced Burau representation of $B_n$ modulo $\ell$ evaluated at $t = \zeta$} similarly as in \Cref{definition:unreduced_burau_representation}.
\end{definition}

\begin{remark}
    There are competing conventions for the unreduced and reduced Burau representations that differ by transposes and/or by involutions on $\mathbb{Z}[t,t^{-1}]$, i.e. by an exchange of $t$ and $t^{-1}$. In particular, the convention for the unreduced Burau representation in \Cref{definition:unreduced_burau_representation} coincides with that of \cite[Definition 3.1.1]{kassel_turaev_bg} and is transpose to that of \cite[Definition 2.1]{salter_lcf}. The convention for the reduced Burau representation in \Cref{definition:reduced_burau_representation} coincides with those of \cite[Section 2, Page 200]{squier_bru} and \cite[Section 1.1.2]{venkataramana_ibrd} and is transpose to that of \cite[Between Theorem 3.9 and Lemma 3.10]{kassel_turaev_bg}.
\end{remark}

See \cite[Chapter 3]{kassel_turaev_bg} for a discussion on the Burau representations. When discussing the action of $B_n$ on an $A^{\oplus n}$ (resp. $A^{\oplus (n-1)}$) for a $\mathbb{Z}[t,t^{-1}]$-module $A$ via $\psi_n$ (resp. $\psi_n^r$), we mean the action of the square matrices that are images of $\psi_n$ (resp. $\psi_n^r$) acting on the left of column vectors.

\subsection{The Burau representations are unitary} \label{subsection:unreduced_burau_representation_is_unitary}

The ring $\mathbb{Z}[t,t^{-1}]$ of Laurent polynomials has a natural involution $a \mapsto \overline{a}$ given by $t \mapsto \overline{t} = t^{-1}$. In particular, one can discuss Hermitian forms over $\mathbb{Z}[t,t^{-1}]$ with respect to this involution. Squier \cite{squier_bru} showed that the reduced Burau representation is unitary, i.e. preserves a Hermitian form on $\mathbb{Z}[t,t^{-1}]^{\oplus (n-1)}$. In fact, the unreduced Burau representation is unitary as well --- Salter \cite[Definition 2.3, Lemma 2.4]{salter_lcf} presented an explicit Hermitian matrix whose Hermitian form is preserved by the unreduced Burau representation. 


One can verify that the $n \times n$ matrices $H_n$ over a $\mathbb{Z}[t,t^{-1}]$-involution algebra $R$ such that $H_n^* = H_n$ and $\psi_n(\sigma_i)^* H_n \psi_n(\sigma_i) = H_n$ for every $i=1,\ldots,n-1$ are exactly those given by 
    $$(H_n)_{ij} = \begin{cases} a &\text{if } i = j \\ bt^{j-i-1} & \text{if } j > i \\ \overline{b} t^{1-i+j} & \text{if } i > j \end{cases},$$
where $b \in R$ and $a = \frac{b-t\overline{b}}{t-1} \in R^{\sigma}$ (recall \Cref{notation:fixed_subring_of_ring_with_involution}). More visually,
    \begin{equation}
        H_n = \begin{pmatrix} a & b & bt & bt^2 & \cdots \\ \overline{b} & a & b & bt & \cdots \\ \overline{b}t^{-1} & \overline{b} & a & b \\ \overline{b}t^{-2} & \overline{b} t^{-1} & \overline{b} & a & \cdots \\ \vdots & \vdots & \vdots & \vdots & \ddots \end{pmatrix}. \label{eq:hermitian_matrix_H_a_b}
    \end{equation}
We will also write $H_{n,a,b}$ when we need to construct such a matrix with specific values $a$ and $b$.
    

\begin{remark}
    Note that such matrices $H_n$ are different, even after conjugation and/or transpose, from the matrix $J$ given in \cite[Definition 2.3]{salter_lcf}. There is no immediate reason for $H_n$ and $J$ to be conjugate and/or transpose to each other even though the unreduced Burau representation presented in \Cref{definition:unreduced_burau_representation} is transpose to that in \cite[Definition 2.1]{salter_lcf}.
\end{remark}

Let $h: R^{\oplus n} \times R^{\oplus n} \to E$ be the Hermitian form given by 
    \begin{equation} \label{eq:hermitian_form_from_matrix}
        h(v,w) = c_v^* H_n c_w,
    \end{equation}
where $c_v, c_w$ respectively denote the column vectors corresponding to $v,w \in R^{\oplus n}$ (with respect to the elementary basis vectors $e_1,\ldots,e_n$), the operator $\cdot^*$ denotes conjugate transpose of vectors/matrices, and $E \coloneq R^{\sigma}$ is the subring of $R$ fixed under the involution. Note that $\psi_n$, with matrices acting on the left of (column vectors in) $R^{\oplus n}$, preserves $h$. In particular, $\psi_n$ preserves the norm
    $$v \in R^{\oplus n} \mapsto h(v,v).$$

Now let $H_n^r$ be the $(n-1) \times (n-1)$ matrix 
\begin{equation} \label{eq:hermitian_form_preserved_by_reduced_burau_representation}
    H_n^r = \begin{pmatrix} \frac{(t+1)^2}{t} & -(\frac{1}{t}+1) & 0 & 0 & \cdots & 0 \\ -(t + 1) & \frac{(t+1)^2}{t} & -(\frac{1}{t} + 1) & 0 & \cdots & 0 \\ 0 & -(t+1) & \frac{(t+1)^2}{t} & -(\frac{1}{t} + 1) & \cdots & 0 \\ 0 & 0 &  -(t + 1) & \frac{(t+1)^2}{t} & \cdots & 0 \\ \vdots & \vdots & \vdots & \vdots & \ddots & \vdots \\ 0 & 0 & 0 & 0 & \cdots & \frac{(t+1)^2}{t} \end{pmatrix}.
\end{equation}
One can verify that $\psi_n^r(\sigma_i)^* H_n^r \psi_n^r(\sigma_i) = H_n^r$ for every $i = 1,\ldots,n-1$, i.e. $H_n^r$ induces a Hermitian form preserved by the reduced Burau representation $\psi_n^r$. This matrix is essentially derived from the description of the Hermitian form in \cite[Section 1.1.3]{venkataramana_ibrd}. Furthermore, $\det H_n^r = \left( \frac{t+1}{t} \right)^{n-1} \left( \frac{t^{n} - 1}{t-1} \right)$ by \cite[Lemma 13]{venkataramana_ibrd}. 

\begin{remark}
    $H_n^r$ above is conjugate/transpose to the matrix in \cite[Section 4.1]{venkataramana_ibrd} as the sesquilinearity conditions in the definitions of Hermitian forms of \Cref{definition:hermitian_form} and of \cite[Section 1.1.3]{venkataramana_ibrd} are opposite to each other.
\end{remark}

Since all of the entries of $H_n^r$ have a factor of $t+1$, the matrix becomes zero upon evaluating at $t = -1$. However, dividing all of the entries of $H_n^r$ by $t+1$ and then evaluating at $t = -1$ yields the $(n-1) \times (n-1)$ skew-symmetric matrix
    \begin{equation} \label{eq:skew_symmetric_matrix_for_hermitian_form_evaluated_at_minus_1}
    \tilde{H}_n^r = \begin{pmatrix} 0 & 1 & 0 & 0 & \cdots & 0 \\ -1 & 0 & 1 & 0 & \cdots & 0 \\ 0 & -1 & 0 & 1 & \cdots & 0 \\ 0 & 0 &  -1 & 0 & \cdots & 0 \\ \vdots & \vdots & \vdots & \vdots & \ddots & \vdots \\ 0 & 0 & 0 & 0 & \cdots & 0 \end{pmatrix}.
    \end{equation}
that the $\psi_n^r(\sigma_i)$ preserve; note that the involution on $\mathbb{Z}[t,t^{-1}]|_{t = -1}$ is trivial. In particular, the Hermitian form arising from $\tilde{H}_n^r$ is a symplectic form. Thus, the image of $\psi_n^r$ evaluated at $t = -1$ lies in a symplectic group and norms of any vector with respect to this Hermitian form are all zero. Moreover, 
    \begin{equation} \label{eq:determinant_of_tilde_H_n_r}
        \det \tilde{H}_n^r = \frac{1}{t^{n-1}} \sum_{i=0}^{n-1} t^i = \begin{cases} 1 &\text{if } 2 \nmid n \\ 0 &\text{otherwise}. \end{cases}
    \end{equation}
Assuming that $R$ factorizes as $R = \prod_i R_i$, where the $R_i$ are division or double division rings with involution, let $\tilde{h}_n^r$ be the Hermitian form on $R^{\oplus(n-1)}$ given (via \Cref{lemma:hermitian_modules_over_product_of_involution_rings_split_along_involution_ring_factors}) on $R_i^{\oplus (n-1)}$ by $H_n^r$ if $R_i$ is a division ring where $t \neq -1$ and by $\tilde{H}_n^r$ otherwise. In particular, when $R = A[\zeta_d]_{\text{all}}$ factorizes as $A[\zeta_d]_{\zeta_d \neq -1} \times A[X]/(X+1)$, $\tilde{h}_n^r$ is a $\varepsilon_{A[\zeta_d]_{\text{all}}}$-Hermitian form.

\subsection{The reduced Burau representation at \texorpdfstring{$d$}{d}-th roots of unity is an arithmetic group}

Venkataramana \cite{venkataramana_ibrd} showed that the image of $\psi_n^r$ evaluated at $d$-th roots $\zeta_d$ of unity is an arithmetic group for all large enough $n$ (with respect to $d$). The following is a restatement of the main theorem of \cite{venkataramana_ibrd} fit to the above notations:

\begin{theorem}[{\cite[Theorem 2]{venkataramana_ibrd}}] \label{theorem:main_theorem_of_venkataramana_ibrd}
    Let $d \geq 3$ be an integer, and let $\mathbb{Z}[\zeta_d] = \mathbb{Z}[\zeta_d]_{\text{prim}}$. Equip $\mathbb{Z}[\zeta_d]_{\text{prim}}^{\oplus (n-1)}$ with the $1$-Hermitian form $\tilde{h}_n^r$. In particular, $\psi_n^r|_{t=\zeta_d}: B_n \to \mathrm{GL}_{n-1}(\mathbb{Z}[\zeta_d])$ maps into $U(\mathbb{Z}[\zeta_d]^{\oplus (n-1)})$. If $n \geq 2d+1$, then the image of $\psi_n^r|_{t=\zeta_d}$ is a finite index subgroup of $U(\mathbb{Z}[\zeta_d]^{\oplus (n-1)})$.

    Moreover, writing $K = \mathbb{Q}(\zeta_d)$ for the $d$-cyclotomic field, the (Zariski closure) of this image inside $\mathrm{GL}_{n-1}$ is of $K$-rank at least $2$. 
\end{theorem}
\begin{remark}
    {\cite[Theorem 2]{venkataramana_ibrd}} in fact formulates the rank conclusion in terms of the $\mathbb{Q}$-rank via the restriction of scalars, which preserves ranks. 
\end{remark}

In fact, when $d =3,4,6$, \cite[Theorem 3]{venkataramana_ibrd} shows that the image of $\psi_n^r|_{t = \zeta_d}$ is an arithmetic group for all $n$. Furthermore, \cite{venkataramana_ibrd} combines a theorem of \cite{acampo_tmgs} and \Cref{theorem:main_theorem_of_venkataramana_ibrd} to obtain the following:

\begin{theorem}[{\cite[Theorem 4]{venkataramana_ibrd}}] \label{theorem:main_theorem_of_venkataramana_ibrd_2}
    Let $d \geq 2$ be an integer, and write $\mathbb{Z}[\zeta_d] = \mathbb{Z}[\zeta_d]_{\zeta_d \neq 1}$. Equip $\mathbb{Z}[\zeta_d]^{\oplus (n-1)}$ with the $\varepsilon_{\mathbb{Z}[\zeta_d]}$-Hermitian form $\tilde{h}_n^r$. In particular, $\psi_n^r|_{t=\zeta_d}: B_n \to \mathrm{GL}_{n-1}(\mathbb{Z}[\zeta_d])$ maps into $U(\mathbb{Z}[\zeta_d]^{\oplus (n-1)})$. If $n \geq 2d+1$, then the image of $\psi_n^r|_{t=\zeta_d}$ is a finite index subgroup of $U(\mathbb{Z}[\zeta_d]^{\oplus (n-1)})$.
\end{theorem}
\begin{remark}
    Recall that $\tilde{h}_n^r$ restricted to the factor $\mathbb{Z}[X]/(X+1)$ of $\mathbb{Z}[\zeta_d]_{\text{all}}$ is a symplectic form --- $U\left((\mathbb{Z}[X]/(X+1))^{\oplus(n-1)}\right)$ here thus is the symplectic group $\mathrm{Sp}\left((\mathbb{Z}[X]/(X+1))^{\oplus (n-1)}\right)$. 
\end{remark}

\subsection{Arithmetic groups have large adic images}

Let $L$ be a field with an involution $\overline{\cdot}$. Let $K = L^{\sigma}$ be the fixed field of $L$ under this involution. Given a $(\pm 1)$- Hermitian form $h$ on $L^{\oplus n}$, there is an algebraic group $\mathrm{U}_n(L,h)$ over $K$ such that $\mathrm{U}_n(L,h)(K) = \mathrm{U}(L^{\oplus n})$ (where $L^{\oplus n}$ is equipped with $h$); if $\overline{\cdot}$ is nontrivial, then restriction of scalars realizes $\mathrm{U}_n(L,h)$ as an algebraic subgroup of $\operatorname{Res}_{L/K}\mathrm{GL}_n \cong \mathrm{GL}_{2n,K}$. Similarly, there is an algebraic group $\mathrm{SU}_n(L,h)$ over $K$ such that $\mathrm{SU}_n(L,h)(K) = \mathrm{SU}(L^{\oplus n})$. In the case that $L$ is a number field, $\mathrm{U}_n(L,h)(\mathcal{O}_K) = \mathrm{U}(\mathcal{O}_L^{\oplus n})$ and $\mathrm{SU}_n(L,h)(\mathcal{O}_K) = \mathrm{SU}(\mathcal{O}_L^{\oplus n})$. See \cite[Section 2.3.3]{platonov_rapinchuk_agnt} and \cite[Section 3]{venkataramana_ibrd} for relevant discussions about these algebraic groups.

\begin{proposition} \label{proposition:ell_adic_image_of_arithmetic_group_contains_SU}
    Let $L$ be a number field with an involution $\overline{\cdot}$, let $K = L^{\sigma}$, let $h$ be a nondegenerate, isotropic $(\pm 1)$-Hermitian form on $\mathcal{O}_L^{\oplus n}$, and let $\Gamma \subseteq \mathrm{U}(\mathcal{O}_L^{\oplus n})$ be a finite index subgroup. For a (finite) place $v$ of $K$, let $\mathcal{O}_{L,v} = \mathcal{O}_{K,v} \otimes_{\mathcal{O}_K} \mathcal{O}_L$, which has a $v$-adic topology.
    Assume that the set of $\alpha \in \mathcal{O}_L^\times$ such that $\alpha \cdot \overline{\alpha} = 1$ is finite. Further assume that one of the following holds:
    \begin{itemize}
        \item The involution $\overline{\cdot}$ is nontrivial, or
        \item The involution $\overline{\cdot}$ is trivial and $h$ is $(-1)$-Hermitian. 
    \end{itemize}
    For all but finitely many finite places $v$ of $K$, the closure of the image of $\Gamma$ in $\mathrm{U}(\mathcal{O}_{L,v}^{\oplus n})$ contains $\mathrm{SU}(\mathcal{O}_{L,v}^{\oplus n})$.
\end{proposition}
\begin{proof}
    If an $L$-automorphism $\varphi$ of $L^{\oplus n}$ preserves $h$, then $\det \varphi \cdot \overline{\det \varphi} = 1$. Note that $\Gamma \cap \mathrm{SU}(\mathcal{O}_L^{\oplus n})$ is the kernel of the map $\Gamma \to \{\alpha \in \mathcal{O}_L^\times: \alpha \cdot \overline{\alpha} = 1\}$ sending $\varphi$ to $\det \varphi$, so $\Gamma \cap \mathrm{SU}(\mathcal{O}_L^{\oplus n})$ has finite index in $\Gamma$, which has finite index in $\mathrm{U}(\mathcal{O}_L^{\oplus n})$. Therefore, $\Gamma' := \Gamma \cap \mathrm{SU}(\mathcal{O}_L^{\oplus n})$ has finite index in $\mathrm{SU}(\mathcal{O}_L^{\oplus n})$.
    
    
    The algebraic group $\mathrm{SU}_n(L,h)$ over $K$ is geometrically connected \cite[Proposition 2.15]{platonov_rapinchuk_agnt} and hence connected, so it is the identity component of $\mathrm{U}_n(L,h)$. By \cite[Proposition 7.4]{platonov_rapinchuk_agnt}, $\mathrm{SU}_n(L,h)$ has the weak approximation property (see \cite[Section 7.1]{platonov_rapinchuk_agnt}). 

    Now, let $v_1, \ldots, v_m$ be any finite collection of finite places of $K$ such that the image of $\Gamma'$ under the maps $\mathrm{SU}(\mathcal{O}_L^{\oplus n}) \to \mathrm{SU}(\mathcal{O}_{L,v_i}^{\oplus n})$ is not dense.
    Identifying the $v_i$ with their corresponding primes of $K$, there exist integers $k_1,\ldots,k_m$ such that the image of $\Gamma'$ under $f_i: \mathrm{SU}(\mathcal{O}_L^{\oplus n}) \to \mathrm{SU}((\mathcal{O}_{L}/v_i^{k_i})^{\oplus n})$ is not surjective. 

    
    Moreover, $\mathrm{SU}_n(L,h)$ is reductive over $K$. By the assumption that $h$ is isotropic, there is some totally isotropic vector $v \in \mathcal{O}_L^{\oplus n}$ for $h$, which remains totally isotropic when $(h, \mathcal{O}_L^{\oplus n})$ is base changed over to $K_v$ for any $v \in M_K$. In particular, $h \otimes_L L_v$ has positive Witt index, and hence the rank of $\mathrm{SU}_n(L,h)$ over $K_v$ is positive, so $\mathrm{SU}_n(L,h)$ is isotropic. By \cite[Theorem 3.1]{platonov_rapinchuk_agnt_ed2}, $\mathrm{SU}_n(L,h)(K_v)$ is noncompact. 

    If the involution $\overline{\cdot}$ on $L$ is nontrivial, then it is an involution of the second kind. If the involution $\overline{\cdot}$ is trivial, then since $L = K$ in this case, the involution is the first kind of the orthogonal type (see \cite[2.3]{platonov_rapinchuk_agnt_ed2}). Thus, the group $\mathrm{SU}_n(L,h)$ is simply connected by\cite[Proposition 2.35]{platonov_rapinchuk_agnt_ed2}. By loc.~cit.~, it is also geometrically simple and hence $K$-simple. By above, $\prod_{v \in M_K^\infty} \mathrm{SU}_n(L,h)(K_v)$ is noncompact, where $M_K^\infty$ is the set of infinite places of $K$. By \cite[Theorem 7.12]{platonov_rapinchuk_agnt}, $\mathrm{SU}_n(L,h)$ over $K$ has strong approximation with respect to $M_K^\infty$. 
    

    Therefore, the diagonal map $f: \mathrm{SU}(\mathcal{O}_L^{\oplus n}) \to \prod_{i=1}^m \mathrm{SU}((\mathcal{O}_L/v_i^{k_i})^{\oplus n})$ is surjective. Moreover, $f(\Gamma')$ is contained in $\prod_{i} f_i(\Gamma')$. Since $[\mathrm{SU}((\mathcal{O}_L/v_i^{k_i})^{\oplus n}) : f_i(\Gamma')] \geq 2$, we have $[\mathrm{SU}(\mathcal{O}_L^{\oplus n}): \Gamma'] \geq 2^m$, which gives the bound $m \leq \log_2 [\mathrm{SU}(\mathcal{O}_L^{\oplus n}) : \Gamma']$. Since this bound depends only on the fixed index of $\Gamma'$ and holds for any finite collection of bad places, there can be at most finitely many such places in total. Therefore, for all but finitely many finite places $v$ of $K$, the closure of the image of $\Gamma$ in $\mathrm{U}(\mathcal{O}_{L,v}^{\oplus n})$ contains $\mathrm{SU}(\mathcal{O}_{L,v}^{\oplus n})$.
    
\end{proof}

In the case of interest, $L$ is a cyclotomic field. We show that the set of elements $\alpha \in \mathcal{O}_L^\times$ such that $\alpha \cdot \overline{\alpha} = 1$ is finite.

\begin{lemma} \label{lemma:there_are_finitely_many_norm_1_elements_of_a_finite_abelian_number_field}
\begin{enumerate}
    \item An algebraic integer $\alpha$ whose Galois conjugates all have absolute value $1$ is a root of unity. \label{part:there_are_finitely_many_norm_1_elements_of_a_finite_abelian_number_field}
    
    \item Let $L$ be a finite abelian field over $\mathbb{Q}$. There are only finitely many $\alpha \in \mathcal{O}_L$ such that $\alpha \cdot \overline{\alpha} = 1$. 
\end{enumerate}
\end{lemma}

\begin{proof}
    
\begin{enumerate}
    \item This is proved, for instance, in expository notes by Wang-Erickson \cite[Claim in Proposition 13]{erickson_notes}. 
    \item Given an embedding $L \to \mathbb{C}$, let $\overline{\cdot}$ be complex conjugation. For any $\sigma \in \operatorname{Gal}(L/\mathbb{Q})$, we have $$\sigma(\alpha) \overline{\sigma(\alpha)} = \sigma(\alpha) \sigma(\overline{\alpha}) = \sigma(\alpha \overline{\alpha}) = \sigma(1) = 1.$$
    Therefore, there are only finitely many such $\alpha$ by  \eqref{part:there_are_finitely_many_norm_1_elements_of_a_finite_abelian_number_field}.
\end{enumerate}
\end{proof}

\begin{lemma} \label{lemma:reduced_burau_representation_is_isotropic}
    Let $d \geq 2$ be an integer, and write $\mathbb{Z}[\zeta_d] = \mathbb{Z}[\zeta_d]_{\zeta_d \neq 1}$. 
    
    The $\varepsilon_{\mathbb{Z}[\zeta_d]}$-Hermitian form $\tilde{h}_n^r$ on $\mathbb{Z}[\zeta_d]^{\oplus (n-1)}$ is isotropic.
\end{lemma}
\begin{proof}
    For $d \geq 3$, \Cref{theorem:main_theorem_of_venkataramana_ibrd} shows that the rank of the Zariski closure of the image of $\tilde{h}_n^r$ is at least $2$ and in particular positive. This corresponds to the isotropicness of $\tilde{h}_n^r$.  

    For $d = 2$, \cite[Th\'eor\`eme 1(2) and section 2]{acampo_tmgs} shows that the image of the $\tilde{h}_n^r$ evaluated at $q = -1$ forms an arithmetic group preserving $\tilde{h}_n^r$ that contains a congruence subgroup modulo $2$ of the symplectic group. Since $n \ge 4$, the associated algebraic group has positive $\mathbb{Q}$-rank (at least $2$), which guarantees that the Hermitian form $\tilde{h}_n^r$ is isotropic.
\end{proof}

\begin{corollary} \label{corollary:ell_adic_image_of_reduced_burau_representation_contains_SU}
    Let $d \geq 2$ be an integer, and write $\mathbb{Z}[\zeta_d] = \mathbb{Z}[\zeta_d]_{\zeta_d \neq 1}$. Equip $\mathbb{Z}[\zeta_d]^{\oplus (n-1)}$ with the $\varepsilon_{\mathbb{Z}[\zeta_d]}$-Hermitian form $\tilde{h}_n^r$. In particular, $\psi_n^r|_{t=\zeta_d}: B_n \to \mathrm{GL}_{n-1}(\mathbb{Z}[\zeta_d])$ maps into $U(\mathbb{Z}[\zeta_d]^{\oplus (n-1)})$. If $n \geq 2d+1$ and $e \nmid n$ for every $e \geq 3$ dividing $d$, then the closure of the image of the composition
        $$B_n \xrightarrow{\psi_n^r|_{t=\zeta_d}} \mathrm{U}(\mathbb{Z}[\zeta_d]^{\oplus(n-1)}) \to U(\mathbb{Z}_\ell[\zeta_d]^{\oplus (n-1)})$$
    contains $\mathrm{SU}(\mathbb{Z}_\ell[\zeta_d]^{\oplus (n-1)})$ for all but finitely many prime numbers $\ell$.
\end{corollary}
\begin{proof}
    Apply \Cref{lemma:there_are_finitely_many_norm_1_elements_of_a_finite_abelian_number_field}, \Cref{lemma:reduced_burau_representation_is_isotropic} and \Cref{proposition:ell_adic_image_of_arithmetic_group_contains_SU} to the involution ring factors $\mathbb{Z}[\zeta_e]_{\text{prim}}$ for $e \mid d$. Note that for $e = 2$, the involution on $\mathbb{Z}[\zeta_e]_{\text{prim}} \cong \mathbb{Z}$ is trivial and that $\tilde{h}_n^r$ is $(-1)$-Hermitian. Further note that $\operatorname{im}\psi_n^r$ has finite index in each factor group by \Cref{theorem:main_theorem_of_venkataramana_ibrd_2}  --- finite index in the product implies finite index in each factor projection, and each factor's form is isotropic by the higher-rank statement of \cite[Theorem 2 proof]{venkataramana_ibrd} applied factor-wise to $SU(h)(O_e)$.
\end{proof}

\section{Counting Rational Points on Hurwitz schemes over \texorpdfstring{$\Fq$}{Fq}} \label{section:generalizing_EVW_section_8}
\subsection{Hurwitz schemes} \label{subsection:hurwitz_schemes}

Let $\mathsf{Conf}'_n / \operatorname{Spec} \mathbb{Z}$ denote the configuration space of squarefree degree $n$-divisors on $\mathbb{P}^1$. Letting $\mathbb{P}^n_\mathbb{Z}$ be the projective $n$-space $\mathbb{P}^n_\mathbb{Z}$ with coordinates $a_0: a_1: \cdots :a_n$, which alternatively parameterize binary forms $a_0 X^n + a_1 X^{n-1} W + \cdots + a_n W^n$ up to scaling, one can also concretely construct $\mathsf{Conf}'_n$ as the open subscheme of $\mathbb{P}^n_\mathbb{Z}$ where the discriminant $(\Delta \sum_{i=0}^n a_i X^{n-i} W^i)$ does not vanish. 

For finite groups $G$, Romagny and Wewers \cite[Theorem 4.11]{romagny_wewers} constructed a scheme $\mathcal{H}_{G,n}$ of finite type over $\mathbb{Z}$ with an \'etale cover $\pi: \mathcal{H}_{G,n} \to \mathsf{Conf}'_n$. For any algebraically closed field $k$, there is an $\operatorname{Aut}(k)$-equivariant bijection between $\mathcal{H}_{G,n}(k)$ and the set of isomorphism classes of tame $G$-covers of $\mathbb{P}^1_k$ with $n$ branch points. Moreover, if $x \in \mathcal{H}_{G,n}(k)$ corresponds to a $G$-cover $f: C \to \mathbb{P}^1_k$, then $\pi(x) \in \mathsf{Conf}'_n(k)$ is the point parameterizing the branch locus of $f$ in $\mathbb{P}^1$. \cite[Corollary 4.12]{romagny_wewers} further establishes that, when $G$ is center free, $\mathcal{H}_{G,n}(k)$ is in natural bijection with the set of isomorphism classes of $G$-covers with degree-$n$ branch locus defined over $k$.

We discuss Hurwitz schemes introduced in \cite[Section 7.3]{ellenberg_venkatesh_westerland} and \cite[Proof of Proposition 2.1]{ellenberg_li_shusterman} to generalize point counts over finite fields on these Hurwitz schemes carried out in these papers in order to extend them to the case that $G = A \rtimes \mathbb{Z}/d\mathbb{Z}$ where $A$ is a finite $\ell$-group; note that these papers mainly discuss the case when $d = 2$.

\cite[Section 7.3]{ellenberg_venkatesh_westerland} constructs a Hurwitz scheme $\mathsf{Hn}_{G,n}$ parameterizing tame $G$-covers of $\mathbb{P}^1$ whose branch divisor in $\mathbb{A}^1$ is of degree $n$ --- the branched covers may be either ramified at $\infty$ or not. To construct such a Hurwitz scheme, let $\mathsf{Conf}_n$ be the configuration space of squarefree degree $n$-divisors on $\mathbb{A}^1$. Concretely, one can construct $\mathsf{Conf}_n$ as the closed subscheme of $\mathsf{Conf}'_{n+1}$ cut out by $a_0 = 0$, i.e. as the configuration space of squarefree degree $(n+1)$-divisors on $\mathbb{P}^1$ that contain $\infty$. Alternatively, one can construct $\mathsf{Conf}_n$ as the open subscheme of $\mathsf{Conf}'_{n}$ where $a_0$ does not vanish, i.e. as the configuration of squarefree degree $n$-divisors on $\mathbb{P}^1$ that do not contain $\infty$. Given a finite group $G$, take $\mathsf{Hn}_{G,n}$ to be the disjoint union of $\mathcal{H}_{G,n+1} \times_{\mathsf{Conf}'_{n+1}} \mathsf{Conf}_n$ and $\mathcal{H}_{G,n} \times_{\mathsf{Conf}'_{n}} \mathsf{Conf}_n$.


If $c$ is a rational conjugacy closed subset of $G$, then there is a closed and open subscheme $\HncGn \subseteq \mathsf{Hn}_{G,n}$ parameterizing tame $G$-covers with monodromy of type $c$. In fact, if $c$ is $\mathbb{F}_q$-rational and $\gcd(q, |G|) = 1$, then $\HncGn$ can be constructed over $\mathbb{F}_q$, cf. \cite[the discussion between definition 8.3 and Proposition 8.4]{ellenberg_tran_westerland_fcqsamcff}. In either case, note that $\HncGn$ parameterizes tame $G$-covers of $\mathbb{P}^1$ branched at $n$ points away from $\infty$ with monodromy of type $c$ --- whether these covers are branched at $\infty$ and the monodromy types these covers have at $\infty$ are not specified.

\begin{remark} \label{remark:proposition_7.8_of_EVW_generalizes_to_rational_conjugacy_classes}
    \cite[Proposition 7.8]{ellenberg_venkatesh_westerland} is stated when $c$ is a rational conjugacy class. This proposition also holds when the phrase "rational conjugacy class" is replaced with the phrase "$\mathbb{F}_q$-rational conjugacy class" --- either rationality condition is only needed to ensure that $\mathsf{Hn}_{G,n}^c$ can be defined over $\mathbb{F}_q$.
\end{remark}

We further define Hurwitz schemes of interest. In the case of interest, $d \geq 2$ is an integer, $q \equiv 1 \pmod{d}$, $\ell$ is a prime that does not divide $q$, $G = A \rtimes (\mathbb{Z}/d\mathbb{Z})$ where $A$ is a nontrivial $\mathbb{Z}_\ell[\zeta_d] = \mathbb{Z}_\ell[\zeta_d]_{\zeta_d \neq 1}$ module of finite cardinality, where $1 \in \mathbb{Z}/d\mathbb{Z}$ acts on $A$ by $\zeta_d$, and where $c$ is the conjugacy class of elements of the form $(a,1) \in G$. Note that $c$ is indeed a conjugacy class by \Cref{lemma:research_orbits_of_burau_representations_2_c_is_conjugacy_class}. A tame $G$-cover $f: C \to \mathbb{P}^1_k$ over a field $k$ factors as $C \xrightarrow{g} C/A \xrightarrow{h} \mathbb{P}^1_k$ where $g$ is a tame $A$-cover and $h$ is a tame $\mathbb{Z}/d\mathbb{Z}$-cover. If $f$ is also branched at $n$ points of $\mathbb{A}^1_k$ with monodromy type $c$, then \Cref{proposition:ramification_index_from_monodromy_type} shows that the ramification indices of $f$ and $h$ above these branch points are the orders of $(a,1) \in G$ and of $1 \in \mathbb{Z}/d\mathbb{Z}$ respectively. These orders are $d$ since $(1-\zeta_d)$ is invertible on $A$. Thus, $g$ must be unramified above $\mathbb{A}^1_k$. 

Moreover, the monodromy type of $\infty$ is of the form $(a, -n \pmod{d})$ where $a \in A$ by \Cref{proposition:monodromy_type_for_kummer_curves} and \Cref{proposition:monodromy_types_respect_compositions}.  \Cref{proposition:ramification_index_from_monodromy_type} shows that the ramification indices of $h \circ g$ and $h$ over $\infty$ are the orders of $(a, -n) \in A \rtimes \mathbb{Z}/d\mathbb{Z}$ and of $-n \in \mathbb{Z}/d\mathbb{Z}$ respectively. If $\gcd(n,d) = 1$, then these orders are both $d$. In this case, $h$ is totally ramified above $\infty$ and $g$ is unramified at the point in $C/A$ above $\infty$. Note that such a tame $G$-cover $f$ over $\mathbb{F}_q$ exists only if $q \equiv 1 \pmod{d}$ by \Cref{proposition:imaginary_cyclic_ell_power_extensions_exist_only_when_Fq_has_ellth_roots_of_unity}. If $d \mid n$ instead, $h$ must be unramified above $\infty$, but $g$ might be ramified above $\infty$. In this case, we restrict our attention to a Hurwitz scheme parameterizing exactly the covers $f$ such that $g$ is unramified above $\infty$.

\begin{definition} \label{definition:hurwitz_scheme_X_n}
    Let $d \geq 2$ be an integer, let $\ell$ be a prime number not dividing $d$, let $G = A \rtimes (\mathbb{Z}/d\mathbb{Z})$ where $A$ is a nontrivial $\mathbb{Z}_\ell[\zeta_d] = \mathbb{Z}_\ell[\zeta_d]_{\zeta_d \neq 1}$ module of finite cardinality and where $1 \in \mathbb{Z}/d\mathbb{Z}$ acts on $A$ by $\zeta_d$. Note that $G$ is center-free by \Cref{lemma:semidirect_product_is_center_free_and_A_has_trivial_invariant_subgroup}. Let $c = \{(a,1) \in G: a \in A\}$, which is a conjugacy class by \Cref{lemma:research_orbits_of_burau_representations_2_c_is_conjugacy_class}. 
    
    For integers $n \geq 1$ that are either divisible by $d$ or relatively prime to $d$ and a prime power $q$ coprime to $d\ell$ such that $q \equiv 1 \pmod{d}$, define the scheme $\mathbf{X}_n/\mathbb{F}_q$ as follows:
    \begin{itemize}
        \item If $(n,d) = 1$, then let $\mathbf{X}_n = \mathsf{Hn}_{G,n}^c$. 
        \item If $d \mid n$, then let $\mathbf{X}_n$ be the intersection of $\mathsf{Hn}_{G,n}^c$ with the open and closed subscheme $\mathcal{H}_{G,n} \times_{\mathsf{Conf}'_n} \mathsf{Conf}_n$ of $\mathsf{Hn}_{G,n}$. In particular, $\mathbf{X}_n$ is the disjoint union of connected components of $\mathsf{Hn}^c_{G,n}$.
    \end{itemize}
\end{definition}

\noindent In either case, $\mathbf{X}_n$ parameterizes the covers $f$ for which $g$ is unramified above $\infty$. We will later apply class field theory and Kummer theory to such $g$ as appropriate. The fact that $G$ is center free and \cite[Theorem 2.1]{romagny_wewers} imply that $\mathsf{Hn}_{G,n}^c$ is a fine moduli space.

\Cref{proposition:generalizing_EVW_proposition_8.7_to_cubic_covers} below generalizes \cite[Proposition 8.7]{ellenberg_venkatesh_westerland}; the latter is a special case of the former when $d = 2$ and $\gcd(d,n) = 1$, and the proof of the former that we present below is very similar to the proof of the latter.

Let $k$ be a field. Given a finite field extension $L$ of $k(t)$, write $\operatorname{Cl}_L$ for the class group of $\mathcal{O}_L$, the integral closure of $k[t]$ inside $L$. Assuming that $k$ is algebraically closed in $L$ and assuming that $L$ has a $k$-rational point, this class group is isomorphic to $\operatorname{Jac}(C_L)(k)$ where $C_L$ is the smooth curve over $k$ corresponding to $L$ --- see \cite[Theorem C(ii)]{rosen_suscg} for more details. This happens, for instance, over the base field $k=\Fq$ if $L = \Fqt[y]/(y^d-f(t))$ for a nonconstant $d$th-power free polynomial $f(t)$ with $(\deg f(t), d) = 1$ since $L$ is a totally imaginary extension of $\Fqt$ by \Cref{lemma:simple_covers_of_P_1}. Suppose further that $C_L$ is a tame $(\mathbb{Z}/d\mathbb{Z})$-cover of $\mathbb{P}^1$ over $k$. Identifying $\zeta_d$ with $1 \in \mathbb{Z}/d\mathbb{Z} \cong \operatorname{Aut}(C_L/\mathbb{P}^1)$, there are actions of $\zeta_d$ on $\operatorname{Jac}(C_L)$ and $\operatorname{Cl}_{L}$ in this case.

\begin{proposition} \label{proposition:generalizing_EVW_proposition_8.7_to_cubic_covers}


Let $d \geq 2$ be an integer, let $\ell$ be a prime number not dividing $d$, and let $\mathbb{Z}_\ell[\zeta_d] = \mathbb{Z}_\ell[\zeta_d]_{\zeta_d \neq 1}$. Fix $A$ to be a $\mathbb{Z}_\ell[\zeta_d]$-module of finite cardinality. Let $G = A \rtimes (\mathbb{Z}/d\mathbb{Z})$ where $1 \in \mathbb{Z}/d\mathbb{Z}$ acts on $A$ by $\zeta_d$ and let $c$ be the conjugacy class in $G$ of the elements of the form $(a,1)$. Let $q$ be a power of a prime not dividing $d\ell$ such that $q \equiv 1 \pmod{d}$. In particular, $\mathbb{F}_q$ contains a primitive $d$th root of unity and $\mathbf{X}_n$ is defined over $\mathbb{F}_q$. 

For each integer $n \geq 1$ coprime to $d$, there is a bijection between $\mathbf{X}_n(\Fq)$ and the set of isomorphism classes of pairs $(L,\alpha)$, where $L$ is a $\mathbb{Z}/d\mathbb{Z}$-extension of $\Fqt$ of discriminant degree $n$ of the form $L = \mathbb{F}_q(t)(\sqrt[d]{f(t)})$ for some squarefree $f(t) \in \mathbb{F}_q[t]$, and where $\alpha$ is a surjective homomorphism
    $$\alpha : \operatorname{Cl}_L \to A$$
\noindent equivariant for the $\zeta_d$-actions. Here, two pairs $(L,\alpha)$ and $(L', \alpha')$ are isomorphic if there exists an $\Fqt$-isomorphism $f: L \to L'$ with $f^* \alpha' = \alpha$.

\end{proposition}

\begin{remark}
    The squarefree condition on $f(t)$ ensures that the $\mathbb{Z}/d\mathbb{Z}$-cover $C_L$ of $\mathbb{P}^1$ has consistent monodromy type over all branch points in $\mathbb{A}^1$. Moreover, the base field is geometrically closed inside $L = k(t)(\sqrt[d]{f(t)})$ for such an $f(t)$ because the points of $C_L$ above the branch points in $\mathbb{A}^1$ are all totally (tamely) ramified by \Cref{proposition:monodromy_type_for_kummer_curves} and \Cref{proposition:ramification_index_from_monodromy_type}.
\end{remark}

\begin{proof}
   For $L$ a $\mathbb{Z}/d\mathbb{Z}$ extension of $K := \Fqt$, let $\zeta_d$ be the automorphism of $L$ over $K$ that corresponds to $1$ under the identification $\operatorname{Aut}(L/K) \cong \mathbb{Z}/d\mathbb{Z}$. Let $J_L$ be the Galois group of the maximal abelian everywhere unramified extension $E/L$ with pro-$\ell$ Galois group. In particular, $E/K$ is Galois and $\operatorname{Gal}(E/K)$ has $J_L$ as a normal subgroup. We show that the subgroup $N := \left \langle \sum_{i=0}^{d-1} \zeta_d^i(x) : x \in J_L \right \rangle$ is a normal subgroup of $\operatorname{Gal}(E/K)$. 

   There is a short exact sequence
       $$1 \to \left(J_L = \operatorname{Gal}(E/L) \right) \to \operatorname{Gal}(E/K) \to \left( \operatorname{Gal}(L/K) \cong \langle \zeta_d \rangle \right) \to 1,$$
   and $\operatorname{Gal}(E/K)$ is isomorphic to $J_L \rtimes \operatorname{Gal}(L/K)$. For elements $(y, \zeta_d) \in J_L \rtimes \operatorname{Gal}(L/K)$, calculate
       \begin{align*}
        (y,\zeta_d^k) \cdot \left( \sum_{i=0}^{d-1} \zeta_d^i(x), 0 \right) \cdot (y,\zeta_d^k)^{-1} &= \left(y + \zeta_d\left( \sum_{i=0}^{d-1} \zeta_d^{i+k}(x) \right) - y, 0 \right) \\
        &= \left( \sum_{i=0}^{d-1} \zeta_d^{i}(x), 0 \right).
        \end{align*}
    Thus, $N$ is normal in $\operatorname{Gal}(E/K)$ as desired.

    Now let $F_L$ be the fixed field of $E$ by $N$. In particular, $F_L / K$ is also a Galois extension
    and there is a short exact sequence
        \begin{equation} \label{eq:ses_for_F_L_over_K}
            1 \to J'_L \to \operatorname{Gal}(F_L/K) \to \operatorname{Gal}(L/K) \to 1. 
        \end{equation}
    where $J'_L := \operatorname{Gal}(F_L/L)$. There is also a short exact sequence 
        \begin{align*}
        1 \to \operatorname{Gal}(E/F_L) \to \operatorname{Gal}(E/L) \to \operatorname{Gal}(F_L/L) \to 1
        \end{align*}
    of abelian groups. This short exact sequence is identifiable with 
        $$1 \to N \to J_L \to J'_L \to 1$$
    and hence $J'_L$ is the quotient of $\operatorname{Gal}(E/L)$ by $N$. Since $J_L$ is a pro $\ell$-group by definition, $J'_L$ is also a pro $\ell$-group. Furthermore, the sequence \eqref{eq:ses_for_F_L_over_K} yields an isomorphism 
        \begin{equation} \label{eq:Gal_F_L_K_as_a_semidirect_product}
            \operatorname{Gal}(F_L/K) \cong J'_L \rtimes \langle \tau_d \rangle.
        \end{equation}
    where $\tau_d$ is any element of $\operatorname{Gal}(F_L/K)$ that reduces to $\zeta_d$ in $\operatorname{Gal}(L/K)$.

    Class field theory (see e.g. \cite[Theorem 1.1.4]{toth_gcft} for a statement) yields a short exact sequence:
        \begin{equation} \label{eq:geometric_class_field_theory}
            (\operatorname{Cl}_L)_\ell \hookrightarrow J_L \twoheadrightarrow \hat{\mathbb{Z}}_\ell
        \end{equation}
    where $(\operatorname{Cl}_L)_\ell$ denotes the Sylow-$\ell$ subgroup of $\operatorname{Cl}_L$ and $\hat{\mathbb{Z}}_\ell$ is the pro-$\ell$ part of $\hat{\mathbb{Z}} \cong \operatorname{Gal}(\bar{\mathbb{F}}_q / \mathbb{F}_q)$. Moreover, the surjection $J_L \twoheadrightarrow \hat{\mathbb{Z}}_\ell$ is identified with the homomorphism $(\pi_1^{\text{\'et}}(C_L))^{\text{ab}} \to \pi_1^{\text{\'et}}(\operatorname{Spec} \mathbb{F}_q)$ induced by the structure morphism $C_L \to \mathbb{F}_q$. Note that $\zeta_d$ acts compatibly on all groups in the above short exact sequence. 

    Note that $\zeta_d$ acts trivially on $\hat{\mathbb{Z}}_\ell$, so $\sum_{i=0}^{d-1} \zeta_d^i$ acts as $d$ on $\hat{\mathbb{Z}}_\ell$ and in particular acts as an automorphism. Apply the snake lemma to 
        \begin{center}
            \begin{tikzcd}
                0 \ar[r] & (\operatorname{Cl}_L)_\ell \ar[r] \ar[d, "\sum_{i=0}^{d-1} \zeta_d^i"] & J_L \ar[r] \ar[d, "\sum_{i=0}^{d-1} \zeta_d^i"] & \hat{\mathbb{Z}}_\ell \ar[r] \ar[d, "\sum_{i=0}^{d-1} \zeta_d^i"] & 0 \\
                0 \ar[r] & (\operatorname{Cl}_L)_\ell \ar[r] & J_L \ar[r] & \hat{\mathbb{Z}}_\ell \ar[r] & 0 \\
            \end{tikzcd}
        \end{center}
    to conclude that the cokernels of $\sum_{i=0}^{d-1} \zeta_d^i$ acting on $(\operatorname{Cl}_L)_\ell$ and $J_L$ are isomorphic, i.e. $J_L / \left(\sum_{i=0}^{d-1} \zeta_d^i\right)J_L \cong (\operatorname{Cl}_L)_\ell / \left(\sum_{i=0}^{d-1} \zeta_d^i \right)(\operatorname{Cl}_L)_\ell$. The former of these groups is $J'_L$. Moreover, for any fractional ideal $I$ on $\mathcal{O}_L$, the product $\prod_{i=0}^{d-1} \zeta_d(I)$ is an extension of a fractional ideal from $K$ and hence is principal, i.e. is trivial as an element of $\operatorname{Cl}_L$. Therefore, $(\operatorname{Cl}_L)_\ell / \left( \sum_{i=0}^{d-1} \zeta_d^i \right)(\operatorname{Cl}_L)_\ell = (\operatorname{Cl}_L)_\ell$, so we have a $\zeta_d$-equivariant isomorphism
        \begin{equation}\label{eq:J_prime_L_is_isomorphic_to_ell_part_of_class_group_of_L}
            J'_L \xrightarrow{\sim} (\operatorname{Cl}_L)_\ell.
        \end{equation}
    Thus, given a pair $(L, \alpha)$ where $L$ is an imaginary $\mathbb{Z}/d\mathbb{Z}$-extension of $\mathbb{F}_q(t)$ of discriminant degree $n$, and $\alpha: \operatorname{Cl}_L \to A$ is a $\zeta_d$-equivariant surjection, composing $\alpha$ with \eqref{eq:J_prime_L_is_isomorphic_to_ell_part_of_class_group_of_L} above yields a $\zeta_d$-equivariant surjection $f_\alpha: J'_L \to A$. This map $f_\alpha$ extends to a surjection $J'_L \rtimes \langle \tau_d \rangle \to A \rtimes \langle \zeta_d \rangle$ by sending $(\operatorname{id}, \tau_d)$ to any element of the form $(a,\zeta_d)$. Such a surjection is identifiable with a map $g_\alpha: \operatorname{Gal}(F_L/K) \to G$ by \eqref{eq:Gal_F_L_K_as_a_semidirect_product}. Moreover, since all elements of $A \rtimes \langle \zeta_d \rangle$ of the form $(a, \zeta_d)$ are conjugate under $A$ by \Cref{lemma:research_orbits_of_burau_representations_2_c_is_conjugacy_class}, the choice of $g_\alpha$ is unique up to $A$-conjugacy.

    Now let $F_\alpha$ be the fixed field of $\ker(g_\alpha)$. It is a Galois extension of $K$ with an isomorphism $\operatorname{Gal}(F_\alpha / K) \cong G$ defined up to $A$-conjugacy. Moreover, $\Fq$ is algebraically closed inside $F_\alpha$. In other words, given the pair $(L,\alpha)$, there is an associated geometrically connected curve $Y_\alpha$ whose function field is $F_\alpha$ together with a cover $Y_\alpha \to \mathbb{P}^1$ and an isomorphism $g_\alpha: \operatorname{Aut}(Y_\alpha / \mathbb{P}^1) \to G$. 

    Note that $F_\alpha / L$ is everywhere unramified because $F_\alpha$ is a subfield of $E$, which is by definition everywhere unramified above $L$. Thus, the ramified places of $F_\alpha / K$ are the same as those of $L/K$ and the ramification degrees of $F_\alpha / K$ are the same as those of $L/K$. In fact, the condition that $L$ is of the form $\mathbb{F}_q(t)(\sqrt[d]{f(t)})$ for some squarefree $f(t) \in \mathbb{F}_q[t]$ and \Cref{proposition:monodromy_type_for_kummer_curves} imply (with an appropriate choice of topological generator of $\hat{\mathbb{Z}}(1)$) that the monodromy types above the branch points of $C_L \to \mathbb{P}^1$ in $\mathbb{A}^1$ are all $1 \in \mathbb{Z}/d\mathbb{Z}$ and that the monodromy type above $\infty$ is $-n \in \mathbb{Z}/d\mathbb{Z}$. By \Cref{proposition:monodromy_types_respect_compositions}, the monodromy types above the branch points of $Y_\alpha \to \mathbb{P}^1$ in $\mathbb{A}^1$ are all in $c$ and the monodromy type above $\infty$ is of the form $(a, -n) \in G$.
    
    The isomorphism class of the $G$-cover of $\mathbb{P}^1$ defined by $(Y_\alpha, g_\alpha)$ does not depend on conjugating $g_\alpha$ by elements of $A$. To see this, let $(Y'_\alpha, g'_\alpha)$ be another pair obtained via such a conjugation, say by $(x, 0) \in A \rtimes \mathbb{Z}/d\mathbb{Z} = G$. In other words, starting from the surjection $f_\alpha: J'_L \to A$, if $g_\alpha: \operatorname{Gal}(F_L/K) \to G$ sends $\tau_d$ to $(y,1)$, then $g'_\alpha$ sends $\tau_d$ to 
    $$(x,0)(y,1)(x,0)^{-1}$$
    where $i \mapsto \varphi_i: \mathbb{Z}/d\mathbb{Z} \to \operatorname{Aut}(A)$ denotes the action of $\mathbb{Z}/d\mathbb{Z}$ on $A$ determining $G$. Letting $F_\alpha$ and $F'_\alpha$ respectively be the fixed fields of $\ker(g_\alpha)$ and $\ker(g'_\alpha)$, we obtain isomorphisms $\operatorname{Gal}(F_\alpha/K) \simeq G$ and $\operatorname{Gal}(F'_\alpha/K) \simeq G$. These isomorphisms are naturally identified with isomorphisms $g_\alpha: \operatorname{Aut}(Y_\alpha/\mathbb{P}^1) \to G$ and $g'_\alpha: \operatorname{Aut}(Y'_\alpha/\mathbb{P}^1) \to G$ where $Y_\alpha$ and $Y'_\alpha$ respectively are the curves associated to $F_\alpha$ and $F'_\alpha$.

    To summarize, given an isomorphism class of a pair $(L, \alpha)$, we have produced a tame $G$-cover $F_\alpha \to \mathbb{P}^1$ defined over $\Fq$ branched at $n$ points in $\mathbb{A}^1$ above which the monodromy types are $c$. In turn, such a cover corresponds to a point of $\mathbf{X}_n(\Fq)$, so we have a map
        \begin{equation} \label{eq:map_from_pairs_to_points_o_HncGnim_Fq}
            \{(L,\alpha)\} / (\text{isomorphism}) \to \mathbf{X}_n(\Fq).
        \end{equation}
    Now we produce an inverse of this map. Given a point of $\mathbf{X}_n(\Fq)$, write $X \to \mathbb{P}^1$ for the corresponding tame $G$-cover branched at $n$ points of $\mathbb{A}^1$ with monodromy of type $c$. By the discussion leading up to \Cref{definition:hurwitz_scheme_X_n} and by the construction of $\mathbf{X}_n$, the quotient map $X \to X/A$ is \'etale. 
    
    Let $L$ be the function field of $X/A$. By \Cref{proposition:monodromy_type_for_kummer_curves}, $L$ must be of the form $\mathbb{F}_q(t)[y]/(y^d-f(t))$ for some squarefree $f(t) \in \mathbb{F}_q[t]$. Applying class field theory to the \'etale cover $X \to X/A$ yields a $\zeta_d$-equivariant homomorphism $\alpha: \operatorname{Cl}_L \to A$ and this construction yields the inverse map to \eqref{eq:map_from_pairs_to_points_o_HncGnim_Fq} as desired. Thus, there is indeed a bijection as claimed.
\end{proof}

\begin{remark}
    When $d$ is coprime to $n$, the points of $\mathbf{X}_n$ parameterize totally imaginary $\mathbb{Z}/d\mathbb{Z}$-extensions $L$ of $\mathbb{F}_q(t)$.
\end{remark}

\subsection{Big monodromy hypotheses} \label{subsection:big_monodromy_hypothesis}

We begin this subsection by motivating big monodromy results which imply the conclusion of \Cref{theorem:generalizing_EVW_Theorem_8.8_to_cubic_covers} when they hold. Fix a positive integer $d$ and a fixed prime power $q \equiv 1 \pmod{d}$. Given a positive integer $n$, let $\mathfrak{S}_n$ be the set of isomorphism classes of extensions $L$ of $\mathbb{F}_q(t)$ of the form $\mathbb{F}_q(t)[y]/(y^d-f(t))$ where $f(t) \in \mathbb{F}_q[t]$ is squarefree and of degree $n$. Moreover, for an abelian $\ell$-group $A$ with a $\langle \zeta_d \rangle$-action and a $\mathbb{Z}/d\mathbb{Z} \cong \langle \zeta_d \rangle$-extension $L$ of $\mathbb{F}_q(t)$, write $m_A(L)$ for $|\operatorname{Sur}_{\mathbb{Z}_\ell[\zeta_d]_{\text{all}}}(\operatorname{Cl}(\mathcal{O}_L), A)|$. If $A$ is in fact a $\mathbb{Z}_{\ell}[\zeta_d]_{\zeta_d \neq 1}$-module, then note that $m_A(L) = |\operatorname{Sur}_{\mathbb{Z}_\ell[\zeta_d]_{\zeta_d \neq 1}}(\operatorname{Cl}(\mathcal{O}_L), A)|$.

\Cref{theorem:generalizing_EVW_Theorem_8.8_to_cubic_covers} below, like its predecessor \cite[Theorem 8.8]{ellenberg_venkatesh_westerland}, asymptotically counts $\mathbb{F}_q$-rational points of $\mathbf{X}_n$ to relate $\sum_{L \in \mathfrak{S}_n} m_A(L)$ and $|\mathfrak{S}_n|$. The proof of the theorem achieves this point count by using the Lefschetz trace formula and by considering the traces of $\operatorname{Frob}_q$ on $H^j_c(\mathbf{X}_n \times_{\mathbb{F}_q} \overline{\mathbb{F}}_q, \mathbb{Q}_r)$ for the cases $j < 2n$ and $j = 2n$ separately. Ellenberg, Venkatesh, and Westerland's cohomological stability results for Hurwitz spaces and schemes (\cite[Theorem 6.1, Corollary 6.2, Proposition 7.8]{ellenberg_venkatesh_westerland}) bound the terms given by $j < 2n$. On the other hand, the trace of Frobenius for the $j = 2n$ term is precisely identified by showing that $\mathbf{X}_n \times_{\mathbb{F}_q} \overline{\mathbb{F}}_q$ has precisely one $\mathbb{F}_q$-rational connected component, assuming an appropriate big monodromy result.

In general, one can discuss the mod-$\ell$ or $\ell$-adic monodromy representations of abelian schemes generalizing notions of Galois representations. See \cite[Section 3.1]{achter_pries_imht} for a more expansive discussion of such representations. Let $X/S$ be an abelian scheme of relative dimension $g$ over an irreducible base $S$. For any rational prime $\ell$ invertible on $S$ and $k \geq 1$, the $\ell^k$-torsion subgroup $X[\ell^k]$ is a finite \'etale cover of $S$ and hence $\pi_1^{\text{\'et}}(S,s)$ acts on $X[\ell^k]$ for any geometric point $s$ of $S$. In fact, this action respects the additive structure on $X[\ell^k]$, so there is an induced {\em monodromy representation}
    \begin{equation} \label{eq:mod_ell_k_monodromy_representation}
        \rho_{X \to S, s, \mathbb{Z}/\ell^k\mathbb{Z}}: \pi_1^{\text{\'et}}(S,s) \to \operatorname{Aut}(X[\ell^k]_s) \cong \operatorname{GL}_{2g}(\mathbb{Z}/\ell^k\mathbb{Z}).
    \end{equation}
The isomorphism class of the image of $\rho_{X \to S, s, \ell^k}$ is independent of the choice of base point $s$. Taking the inverse limit over $k$ yields a continuous representation of the $\ell$-adic Tate module of $X$:
    $$\rho_{X \to S, s, \mathbb{Z}_\ell}: \pi_1^{\text{\'et}}(S,s) \to \varprojlim_k \operatorname{Aut}(X[\ell^k]_s) \cong \operatorname{GL}_{2g}(\mathbb{Z}_\ell).$$
Let $M_{\mathbb{Z}/\ell^k\mathbb{Z}}(X \to S, s)$ and $M_{\mathbb{Z}_\ell}(X \to S, s)$ respectively be the images of $\rho_{X \to S, s, \mathbb{Z}/\ell^k\mathbb{Z}}$ and $\rho_{X \to S, s, \mathbb{Z}_\ell}$ and call them the mod-$\ell^k$ and $\ell$-adic monodromy groups of $X \to S$. If the abelian scheme $X$ over $S$ is clear in context, then denote the monodromy representations by $\rho_{S, s, \mathbb{Z}/\ell^k\mathbb{Z}}$ and $\rho_{S, s, \mathbb{Z}_\ell}$ and the monodromy groups by $M_{\mathbb{Z}/\ell^k\mathbb{Z}}(S, s)$ and $M_{\mathbb{Z}_\ell}(S, s)$ respectively. Note that the isomorphism classes of the monodromy groups do not depend on the base point $s$, and we sometimes omit $s$ when notating these groups.

If more conditions are specified, then the image of the monodromy representations may be contained in smaller linear groups. For example, if $X$ is principally polarized, then in fact the image of $\rho_{X \to S, s, \mathbb{Z}_\ell}$ is contained in $\operatorname{GSp}_{2g}(\mathbb{Z}_\ell)$, the group of similitudes of $((T_{\ell} X)_s, \omega)$, where $\omega: T_\ell X \times T_\ell X^\vee \to \bbZ_\ell(1)$ is the $\ell$-adic Weil pairing. Moreover, if a compatible system of primitive $\ell$th roots of unity exists globally on $S$, then $\pi_1(S,s)$ acts trivially on $\mu_{\ell, S}$ and on $\bbZ_\ell(1)$, so the image of $\rho_{X \to S, s, \mathbb{Z}_\ell}$ is contained in $\operatorname{Sp}_{2g}(\mathbb{Z}_\ell)$. 

Given a relative proper semi-stable curve $\psi: C \to S$, the identity component $\operatorname{Pic}^0(C) := \operatorname{Pic}^0_{C/S}$ of the relative Picard functor of $\psi$ is a semiabelian scheme \cite[Proposition 4.3]{deligne_llg} (cf. \cite[Theorem 9.4.1]{bosch_lutkebohmert_nm}). For a geometric point $s$ of $S$ such that $\operatorname{Pic}^0(C_s)$ is an abelian variety, there is a nonempty open neighborhood $S^*$ of $s$ such that $\operatorname{Pic}^0(C|_{S^*})$ is an abelian scheme over $S^*$. By \cite[Example 9.2.8]{bosch_lutkebohmert_nm}, $s$ is such a point whenever $C_s$ is a tree of smooth curves. Define the mod-$\ell^k$ and $\mathbb{Z}_\ell$ monodromy representations of $C$ (with respect to $S^*$) to be those of $\operatorname{Pic}^0(C|_{S^*}) \to S^*$, and define the monodromy groups $M_{\mathbb{Z}/\ell^k \mathbb{Z}}(C \to S, s)$ and $M_{\mathbb{Z}_\ell}(C \to S,s)$ to be the images of these representations. We omit $C$ in these notations when the curve $C$ is clear in context. The monodromy groups do not depend on the choice of $S^*$ by \Cref{lemma:comparing_monodromy_groups_under_morphism_of_bases}\eqref{part:surjection_on_fundamental_groups_implies_that_monodromy_groups_are_equal} below --- in general, if $X \hookrightarrow Y$ is an open dense embedding, then the induced map of \'etale fundamental groups is a surjection by \cite[0BN6]{stacks-project}.

One can generalize such notions of monodromy representations and monodromy groups to families of abelian varieties over algebraic stacks $\mathcal{S}$ by generalizing the formalism of \'etale covers over to stacks; note that \cite{noohi_fgas} studies fundamental groups of algebraic stacks by establishing that the category of Galois \'etale covers of an algebraic stack is a Galois category.

Informally speaking, a ``big monodromy theorem'' would state that the image of a monodromy representation is ``not much smaller'' than the linear group that is the codomain of the representation. Some big monodromy results in the literature include those in \cite{venkataramana_ibrd}, \cite{yu_tpclcffc}, \cite{achter_pries_imht}, \cite{hall_bsomm}, \cite{jain_bmlmfGc}.

In cases of interest, $X$ will be the identity component $\operatorname{Pic}_{C/S}^0$ of the relative Picard group of a proper family of smooth curves over a base scheme $S$. Given $n \geq 1$, recall from \Cref{subsection:hurwitz_schemes} the construction of $\mathsf{Conf}_n$ as the closed subscheme of $\mathsf{Conf}_{n+1}'$ with $a_0 = 0$ and as the open subscheme of $\mathsf{Conf}_n'$ with $a_0 \neq 0$. Here, let $\mathbb{Z}[\zeta_d] = \mathbb{Z}[\zeta_d]_{\text{prim}}$. Let $C_n$ be the smooth proper curve over $\mathsf{Conf}_n \times_{\mathbb{Z}} \mathbb{Z}[1/d, \zeta_d]$ birational to the plane curve
    \begin{equation} \label{eq:plane_curve_over_Conf_n}
        \begin{cases}
            Y^d = a_1 X^{n} + a_2 X^{n-1} W + \cdots + a_{n+1} W^{n} &\text{if } d \nmid n; \text{ embed } \mathsf{Conf}_n \text{ into } \mathsf{Conf}_{n+1}'\\
            Y^d = a_0 X^{n} + a_1 X^{n-1} W + \cdots + a_{n} W^n &\text{if } d \mid n; \text{ embed } \mathsf{Conf}_n \text{ into } \mathsf{Conf}_n'.
        \end{cases}
    \end{equation}
In either case, the branch locus of $C_n$ as a cover of $\mathbb{P}^1_{[X:W]}$ is of degree $n$ away from $\infty = [1:0]$. Moreover, $C_n$ has a $\mathbb{Z}/d\mathbb{Z}$-action where $1 \in \mathbb{Z}/d\mathbb{Z}$ sends $([X:W], Y)$ to $([X:W], \zeta_d Y)$. For any prime $\ell$, any $\mathbb{Z}[1/(d\ell), \zeta_d]$-algebra $R$, and any geometric point $s$ of $\mathsf{Conf}_n \times_{\mathbb{Z}} R$, we have a monodromy representation
    \begin{equation} \label{eq:monodromy_representation_of_Z_dZ_cover_of_P1_over_Conf_n}
        \pi_1^{\text{\'et}}(\mathsf{Conf}_n \times_{\mathbb{Z}} R, s) \to \varprojlim_{k} \operatorname{Aut}(\operatorname{Pic}^0_{C_n/\mathsf{Conf}_n \times_{\mathbb{Z}} R}[\ell^k]_s) \subseteq \operatorname{GL}_{2g}(\mathbb{Z}_\ell)
    \end{equation}
where $g$ is the genus of $C_n$. In fact, the image of the above monodromy representation lies in $\operatorname{Sp}_{2g}(\mathbb{Z}_\ell)$ and further lies in $\mathrm{U}(V)$ (see \Cref{subsection:unitary_groups} for notation) where $V$ is the $\mathbb{Z}_\ell[\zeta_d]_{\zeta_d \neq 1}$-module $T_\ell \operatorname{Pic}^0(C_n|_s)$ equipped with a Hermitian form induced by \Cref{proposition:symplectic_pairings_to_hermitian_forms} from the Weil pairing $\omega$ on $V$ and the $\mathbb{Z}/d\mathbb{Z}$-action on $C_n$. 

We will compare monodromy groups with one another to deduce that the largeness of one implies that of another. We do so by recognizing that monodromy representations factor one another or, in other words, that monodromy representations respect natural maps of fundamental groups.

\begin{lemma} \label{lemma:comparing_monodromy_groups_under_morphism_of_bases}
    Let $T \to S$ be a morphism of irreducible schemes. Let $X/S$ be an abelian scheme. For any geometric point $t$ of $T$, let $s$ be the induced geometric point of $S$. For any rational prime $\ell$ invertible on $S$, the diagram
    \begin{center}
        \begin{tikzcd}
            \pi_1^{\text{\'et}}(T,t) \ar[r] \ar[d] & \operatorname{Aut}((X \times_S T)[\ell^k]_t)  \ar[d, "\cong"] \\
            \pi_1^{\text{\'et}}(S,s) \ar[r] & \operatorname{Aut}(X[\ell^k]_s)
        \end{tikzcd}
    \end{center}
    commutes, where the horizontal maps are the mod $\ell^k$ monodromy representations. Moreover, under the identification $\operatorname{Aut}((X \times_S T)[\ell^k]_t) \cong \operatorname{Aut}(X[\ell^k]_s)$,
    \begin{enumerate} 
        \item The monodromy group of $X/S$ contains the monodromy group of $(X \times_S T)/T$
        \item If $\pi_1^{\text{\'et}}(T,t) \to \pi_1^{\text{\'et}}(S,s)$ is surjective, then the monodromy group of $X/S$ coincides with the monodromy group of $(X \times_S T)/T$. \label{part:surjection_on_fundamental_groups_implies_that_monodromy_groups_are_equal}
    \end{enumerate}
\end{lemma}
\begin{proof}
   The action of $\pi_1^{\text{\'et}}(S, s)$ on the finite \'etale cover $X[\ell^k]$ of $S$ naturally induces an action of $\pi_1^{\text{\'et}}(T,t)$ on the finite \'etale cover $X[\ell^k] \times_S T \cong (X \times_S T)[\ell^k]$ of $T$, so the diagram commutes. In particular, the image of the composed homomorphism
       $$\pi_1^{\text{\'et}}(T,t) \to \pi_1^{\text{\'et}}(S,s) \to \operatorname{Aut}(X[\ell^k]_s)$$
   is contained in the image of its factor
       $$\pi_1^{\text{\'et}}(S,s) \to \operatorname{Aut}(X[\ell^k]_s).$$
   If $\pi_1^{\text{\'et}}(T,t) \to \pi_1^{\text{\'et}}(S,s)$ is surjective, then the images in fact coincide.
\end{proof}

\begin{corollary} \label{corollary:monodromy_group_over_generic_point_coincides_with_monodromy_group}
    Let $S$ be a normal integral scheme and let $X \to S$ be an abelian scheme. Let $\eta$ be the generic point of $S$ and let $\bar{\eta}$ be the spectrum of an algebraically closed extension of $K := \kappa(\eta)$. For any rational prime $\ell$ invertible on $S$, the monodromy group of the representation $\rho_{X \times_S \eta \to \eta, \bar{\eta}, \mathbb{Z}/\ell^k\mathbb{Z}}: \operatorname{Gal}(\bar{K}/K) \cong \pi_1^{\text{\'et}}(\eta, \bar{\eta}) \to \operatorname{Aut}(X[\ell^k]_{\bar{\eta}})$ coincides with the monodromy group of $\rho_{X \to S, \bar{\eta}, \mathbb{Z}/\ell^k\mathbb{Z}}: \pi_1^{\text{\'et}}(S, \bar{\eta}) \to \operatorname{Aut}(X[\ell^k]_{\bar{\eta}})$.
\end{corollary}
\begin{proof}
    By \Cref{lemma:comparing_monodromy_groups_under_morphism_of_bases} it suffices to show that the homomorphism $\operatorname{Gal}(\bar{K} / K) \cong \pi_1^{\text{\'et}}(\eta, \bar{\eta}) \to \pi_1^{\text{\'et}}(S,\eta)$ is surjective, which is true by \cite[\href{https://stacks.math.columbia.edu/tag/0BQM}{Tag 0BQM}]{stacks-project}.
\end{proof}

Now let $R = \mathbb{F}_q$ for a prime power $q$ coprime to $\ell d$. Let $\eta$ be the generic point of $\mathsf{Conf}_{n} \times_{\operatorname{Spec} \mathbb{Z}} \operatorname{Spec} \mathbb{F}_q$, and write $K_n$ for the function field of $\eta$. The embedding $\eta \hookrightarrow \mathsf{Conf}_{n} \times_{\operatorname{Spec} \mathbb{Z}} \operatorname{Spec} \mathbb{F}_q$ induces from \eqref{eq:monodromy_representation_of_Z_dZ_cover_of_P1_over_Conf_n} a monodromy representation
    $$\pi_1^{\text{\'et}}(\eta, \bar{\eta}) \to \pi_1^{\text{\'et}}(\mathsf{Conf}_n \times_{\mathbb{Z}} \mathbb{F}_q, \bar{\eta}) \to \varprojlim_{k} \operatorname{Aut}(\operatorname{Pic}^0_{C_n / \mathsf{Conf_n \times_{\mathbb{Z}} \mathbb{F}_q}}[\ell^k]_{\bar{\eta}})$$
which we identify with the Galois representation
    \begin{equation} \label{eq:galois_representation_of_Z_dZ_covers_of_P1_over_generic_point_of_Conf_n}
        \mu: \operatorname{Gal}(\overline{K}_n / K_n) \cong \pi_1^{\text{\'et}}(\eta, \bar{\eta}) \to \varprojlim_{k} \operatorname{Aut}(\operatorname{Pic}^0_{C_n / K_n}[\ell^k]_{\bar{\eta}})
    \end{equation}
via \Cref{lemma:comparing_monodromy_groups_under_morphism_of_bases}. We will need big monodromy results for such Galois representations in the proof of \Cref{theorem:generalizing_EVW_Theorem_8.8_to_cubic_covers}. \Cref{corollary:monodromy_group_over_generic_point_coincides_with_monodromy_group} shows that the monodromy groups of the representations \eqref{eq:monodromy_representation_of_Z_dZ_cover_of_P1_over_Conf_n} (after choosing $s = \bar{\eta}$) and \eqref{eq:galois_representation_of_Z_dZ_covers_of_P1_over_generic_point_of_Conf_n} coincide, so big monodromy results for \eqref{eq:monodromy_representation_of_Z_dZ_cover_of_P1_over_Conf_n} are equivalent to those of \eqref{eq:galois_representation_of_Z_dZ_covers_of_P1_over_generic_point_of_Conf_n}.

Write $V_n$ for the codomain $\varprojlim_{k} \operatorname{Aut}(\operatorname{Pic}^0_{C_n / K_n}[\ell^k]_{\bar{\eta}})$ of the above Galois representation. Note that there is a short exact sequence
    $$\operatorname{Gal}(\overline{K}_n / \overline{\mathbb{F}}_q K_n) \to \operatorname{Gal}(\overline{K}_n / K_n) \to \operatorname{Gal}(\overline{\mathbb{F}}_q / \mathbb{F}_q).$$
The actions of $\operatorname{Gal}(\overline{K}_n/\overline{\mathbb{F}}_q K_n)$ and $\operatorname{Aut}(C_n / \mathbb{P}^1_{K_n}) \cong \langle \zeta_d \rangle$ on $V_n$ both preserve the Weil pairing $\omega: V_n \times V_n \to \mathbb{Z}_\ell(1)$. In particular, the Weil pairing induces a Hermitian form via \Cref{proposition:symplectic_pairings_to_hermitian_forms} on $V_n$ that $\operatorname{Gal}(\overline{K}_n/\overline{\mathbb{F}}_q K_n)$ preserves. Recalling notation from \Cref{subsection:unitary_groups}, we thus have
    $$\mu \left( \operatorname{Gal} \left( \overline{K}_n / \overline{\mathbb{F}}_q K_n \right) \right) \subseteq \mathrm{U}(V_n).$$
Moreover, if $F$ is an element of $\operatorname{Gal}(\overline{K}_n / K_n )$ lying over Frobenius in $\operatorname{Gal}\left (\overline{\mathbb{F}}_q / \mathbb{F}_q \right)$, then $\mu(F)$ lies in $\mathrm{GU}_{q}(V_n)$. In particular, $\mu(\operatorname{Gal}(\overline{K}_n/K_n)) \subset \mathrm{GU}(V_n)$. Due to \Cref{proposition:nonempty_O_implies_transitive_action_of_U_V} and its consequences \Cref{corollary:generalizing_EVW_Lemma 8.9_to_Z_dZ_covers} and \Cref{proposition:generalizing_EVW_Lemma 8.9_to_Z_dZ_covers}, big monodromy results that state that the image $\mu$ contains $\mathrm{SU}(V_n)$ will suffice for our purposes.

To obtain this desired big monodromy result in the above characteristic $p$ setting, we first use \cite{venkataramana_ibrd} to show that the complex analytic $\ell$-adic monodromy representation has large image. The base change of the family $C_n$ defined in \eqref{eq:plane_curve_over_Conf_n} to $\mathbb{C}$ is an algebraic family $f: C_n \otimes \mathbb{C} \to \mathsf{Conf}_n \otimes \mathbb{C}$ of algebraic manifolds, which is a fibration. In general, given a fibration $f: X \to S$ let $f^{-1}(s)$ be a typical fiber of this family. The topological fundamental group $\pi_1(S(\mathbb{C}), s)$ acts on the cohomology group $H^m(f^{-1}(s), \mathbb{Z})$, and this action is called the \emph{monodromy representation on $H^m(f^{-1}(s), \mathbb{Z})$}. In the case that $f: X \to S$ is a family of smooth projective complex algebraic curves and $\ell \neq p$ is a prime, the monodromy representation on $H^1(f^{-1}(s), \mathbb{Z})$ induces an action of $H^1(f^{-1}(s), \mathbb{Z}/\ell^k \mathbb{Z}) \cong H^1(f^{-1}(s), \mathbb{Z}) \otimes_{\mathbb{Z}} \mathbb{Z}/\ell^k\mathbb{Z}$ compatible with the monodromy representation $\rho_{\operatorname{Pic}^0 X \to S, s, \mathbb{Z}/\ell^k \mathbb{Z}}: \pi_1^{\text{\'et}}(S,s) \to \operatorname{Aut}(\operatorname{Pic}^0 X_s [\ell^k])$ defined in \eqref{eq:mod_ell_k_monodromy_representation} under the identifications $\pi_1^{\text{\'et}}(S,s) \cong \widehat{\pi_1(S(\mathbb{C}),s)}$ \cite[Expos\'e XII, Corollaire 5.2]{SGA_1} and $\operatorname{Pic}^0 X_s [\ell^k] \cong H^1_{\text{\'et}}(X_s, \mathbb{Z}/\ell^k \mathbb{Z}) \cong H^1(X_s(\mathbb{C}), \mathbb{Z}/\ell^k\mathbb{Z})$\footnote{To see the identification $\operatorname{Pic}^0 X_s [\ell^k] \cong H^1_{\text{\'et}}(X_s, \mathbb{Z}/\ell^k \mathbb{Z})$, note that $H^1_{\text{\'et}}(X_s, \mathbb{G}_m) \cong \operatorname{Pic}(X_s)$ \cite[Theorem 13.7]{milne_lec} and use the long exact sequence in \'etale cohomology for the Kummer exact sequence $0 \to \mu_\ell \to \mathbb{G}_m \to \mathbb{G}_m \to 0$. The identification $H^1_{\text{\'et}}(X_s, \mathbb{Z}/\ell^k \mathbb{Z}) \cong H^1(X_s(\mathbb{C}), \mathbb{Z}/\ell^k\mathbb{Z})$ is by \cite[Expos\'e XVI, Lemme 4.4]{SGA4_III}, cf. \cite[Theorem 21.1]{milne_lec}}. Note that $\pi_1(\mathsf{Conf}_n(\mathbb{C}), s) \cong B_n$.

Let $\mathsf{PConf}_n^{\mathbb{A}^1}/\mathbb{Z}$ and $\mathsf{PConf}_n^{\mathbb{P}^1}/\mathbb{Z}$ respectively be the \emph{ordered configuration schemes} of labeled $n$-tuples of distinct sections of $\mathbb{A}^1$ and $\mathbb{P}^1$. We write $\mathsf{PConf}_n$ for $\mathsf{PConf}_n^{\mathbb{A}^1}$. More concretely, these schemes can be constructed as open subschemes of $(\mathbb{A}^1)^n$ and $(\mathbb{P}^1)^n$ described by $n$-tuples $(p_1,\ldots,p_n)$ where $p_i \neq p_j$ for any $i \neq j$. Note that there is a Galois, \'etale $S_n$-cover $\mathsf{PConf}_n^{\mathbb{A}^1} \to \mathsf{Conf}_n$ sending $(p_1,\ldots,p_n)$ to $(X-p_1) \cdots (X-p_n)$. This cover induces a normal injection $\pi_1^{\text{\'et}}(\mathsf{PConf}_n, s) \to \pi_1^{\text{\'et}}(\mathsf{Conf}_n, s)$ with cokernel $S_n$. The topological fundamental group $\pi_1(\mathsf{PConf}_n(\mathbb{C}), s)$ is the \emph{pure braid group} $P_n$. Hence, pulling back the family $f: C_n \otimes \mathbb{C} \to \mathsf{Conf}_n \otimes \mathbb{C}$ of curves over to $\mathsf{PConf}_n \otimes \mathbb{C}$, the monodromy representation of $\pi_1(\mathsf{PConf}_n(\mathbb{C}), s) \cong P_n$ is the restriction of that of $\pi_1(\mathsf{Conf}_n(\mathbb{C}), s) \cong B_n$ under the injection $P_n \hookrightarrow B_n$.

Note that $\mathsf{PConf}_n^{\mathbb{A}^1}$ is identifiable as the closed subscheme of $\mathsf{PConf}_{n+1}^{\mathbb{P}^1}$ in which the last section is $\infty$. The induced embedding of complex points is the fiber of $\infty \in \mathbb{P}^1(\mathbb{C})$ of the fibration $\mathsf{PConf}_{n+1}^{\mathbb{P}^1}(\mathbb{C}) \to \mathbb{P}^1(\mathbb{C})$ sending $(p_1,\ldots,p_{n+1})$ to $p_{n+1}$. This fibration induces an exact sequence $\pi_1(\mathsf{PConf}_n^{\mathbb{A}^1}(\mathbb{C}), s) \to \pi_1(\mathsf{PConf}_{n+1}^{\mathbb{P}^1}(\mathbb{C}), s) \to \pi_1(\mathbb{P}^1(\mathbb{C})) = 1$. Therefore,

\begin{lemma} \label{lemma:map_from_pconf_of_A1_to_pconf_of_P1_gives_surjection_of_fundamental_groups}
    The embedding of $\mathsf{PConf}_n^{\mathbb{A}^1}(\mathbb{C})$ into $\mathsf{PConf}_{n+1}^{\mathbb{P}^1}(\mathbb{C})$ induces a surjection of fundamental groups. By the comparison of topological and \'etale fundamental groups \cite[Expos\'e XII, Corollaire 5.2]{SGA_1}, the embedding of $\mathsf{PConf}_n^{\mathbb{A}^1} \times \mathbb{C} \to \mathsf{PConf}_{n+1}^{\mathbb{P}^1}$ induces a surjection of \'etale fundamental groups. 
\end{lemma}

The following is due to \cite[Theorem 1]{venkataramana_ibrd}, which follows from \cite[Proposition 24]{venkataramana_ibrd} after applying Poincar\'e duality.
\begin{proposition} \label{proposition:venkataramana_proposition_24}
    If $n \geq 2d+1$, then the image of the monodromy representation of $\pi_1(\mathsf{Conf}_n(\mathbb{C}), s) \cong B_n$ on $H^1(C_n)$ is a subgroup of finite index in $\mathrm{U}(H^1(C_n))$. 
\end{proposition}

\begin{corollary} \label{corollary:weak_big_monodromy_statement}
    For all $d \geq 2$, for all $n \geq 2d+1$, for all but finitely many prime numbers $\ell \nmid d$, the closure of the image of $\pi_1(\mathsf{Conf}_n(\mathbb{C}), s)$ under the $\ell$-adic monodromy representation on $H^1(C_n, \mathbb{Z}_\ell)$ contains $\mathrm{SU}(H^1(C_n, \mathbb{Z}_\ell))$. The same holds with $\mathsf{PConf}^{\mathbb{A}^1}_n$ replacing $\mathsf{Conf}_n$.
\end{corollary}
\begin{proof}
    Recognizing that the monodromy representation described here is the reduced Burau representation specialized to $t = \zeta_d$, combine \Cref{lemma:reduced_burau_representation_is_isotropic}, \Cref{proposition:venkataramana_proposition_24}, \Cref{proposition:ell_adic_image_of_arithmetic_group_contains_SU}, and \Cref{lemma:there_are_finitely_many_norm_1_elements_of_a_finite_abelian_number_field}.
    When $d$ is even, note that the involution on the factor $\mathbb{Z}[\zeta_2] \cong \mathbb{Z}$ is trivial and $\tilde h_n^r$ is $(-1)$-Hermitian on that factor, so the second alternative of \Cref{proposition:ell_adic_image_of_arithmetic_group_contains_SU} applies.
\end{proof}

In the above statement, the prime numbers $\ell$ a priori depend on $d$ and $n$. To prove \Cref{theorem:generalizing_EVW_main_theorem_to_cubic_covers} and \Cref{theorem:research_generalizing_ELS_main_theorem_to_trielliptic_curves} we will need a stronger statement, in which $\ell$ does not depend on $n$:

\begin{theorem} \label{theorem:big_monodromy_statement_in_characteristic_0}
Let $d \geq 3$ be an integer. For all but finitely many prime numbers $\ell$, and for all $n \geq 2d+1$, the closure of the image of $\pi_1(\mathsf{Conf}_n(\mathbb{C}), s)$ under the $\ell$-adic monodromy representation on $H^1(C_n|_s, \mathbb{Z}_\ell)$ contains $\mathrm{SU}(H^1(C_n, \mathbb{Z}_\ell))$. Equivalently, $M_{\mathbb{Z}_\ell}(\mathsf{Conf}_n \times \mathbb{C}, s) \supseteq \mathrm{SU}(H^1(C_n|_s, \mathbb{Z}_\ell))$. The same holds with $\mathsf{PConf}_n^{\mathbb{A}^1}$ replacing $\mathsf{Conf}_n$.
\end{theorem}

\cite{achter_pries_imht} studies (among other moduli) the moduli $\widetilde{\mathcal{M}}_{G}$, where $G = \mathbb{Z}/d\mathbb{Z}$ for prime numbers $d$, parameterizing labeled admissible stable $G$-curves $(C/S, \iota_0, \eta)$. We are additionally interested in the possibility that $d \geq 2$ is not necessarily a prime number. \cite[Lemma 2.2]{achter_pries_imht} establishes that $\widetilde{\mathcal{M}}_G$ is a smooth, proper Deligne-Mumford stack over $\operatorname{Spec} \mathbb{Z}[1/d,\zeta_d]$ and the subspace $\widetilde{\mathcal{M}}_G^\circ$ parameterizing smooth curves is open and dense; the lemma in fact holds for any $d \geq 2$. Given a class vector $\gamma: \{1,\ldots,r\} \to \mathbb{Z}/d\mathbb{Z} \setminus \{0\}$ of length $r$, the substack $\widetilde{\mathcal{M}}_G^\gamma$ of $\widetilde{\mathcal{M}}_G$ for which $(C/S,\iota_0,\eta)$ has class vector $\gamma$ is an irreducible connected component \cite[Lemma 2.3]{achter_pries_imht}. In the case that $d \geq 2$ is not necessarily prime, the proof of the lemma generalizes immediately by replacing $\widetilde{\mathcal{M}}_G^{\gamma, \circ}$ with its open and dense subspace of smooth curves $C$ whose $r$ branch points have distinct images in $C/\iota_0(G)$.

Further write $\widetilde{\mathcal{M}}_{g,r}$ for the moduli of triples $(C/S, \Xi, \eta)$ where $C/S$ is a semi-stable curve of genus $g$, $\Xi$ is a mark of degree $r$ on $C$ such that $(C/S,\Xi)$ is stably marked, and $\eta$ is a labeling of $\Xi$. There is a clutching map \cite[Definition 3.8]{knudsen_pmsscii}
    $$\widetilde{\mathcal{M}}_{g_1, r_1} \times \widetilde{\mathcal{M}}_{g_2,r_2} \to \widetilde{\mathcal{M}}_{g_1+g_2, r_1+r_2-2}$$
which glues curves $C_1$ and $C_2$ over $S$ together at the last section of $C_1$ and the first section of $C_2$. The composition of $\widetilde{\mathcal{M}}_G^{\gamma_1} \times \widetilde{\mathcal{M}}_G^{\gamma_2} \to \widetilde{\mathcal{M}}_{g_1, r_1} \times \widetilde{\mathcal{M}}_{g_2,r_2}$ with the above clutching map glues two labeled admissible stable $G$-curves $(C_i/S, \iota_{0,i}, \eta_i)$ to obtain a labeled $G$-curve $C/S$ with class vector $\gamma = (\gamma_1(1),\ldots,\gamma_1(r_1-1), \gamma_2(2),\ldots,\gamma_2(r_2))$. By \cite[Proposition 2.2]{ekedahl_bbhs}, $C/S$ is equivariantly smoothable if and only if $\gamma_1(r_1) \equiv -\gamma_2(1) \pmod{d}$, in which case one says that \emph{$(\gamma_1,\gamma_2)$ deforms to $\gamma$} or that \emph{$\gamma$ degenerates to $(\gamma_1,\gamma_2)$}. Assuming that $\gamma_1(r_1)$ (and equivalently $\gamma_2(1)$) is relatively prime to $d$, the covers $C_1 \to C/\iota_{0,1}(G)$ and $C_2 \to C/\iota_{0,2}(G)$ are totally ramified above the branch points corresponding to $\gamma_1(r_1)$ and $\gamma_2(1)$ respectively. In particular, there is exactly one point above each of these branch points, so the clutching map induces a well-defined map
    \begin{equation} \label{eq:clutching_map_for_moduli_of_G_covers_with_specified_class_vector}
        \kappa: \widetilde{\mathcal{M}}_G^{\gamma_1} \times \widetilde{\mathcal{M}}_G^{\gamma_2} \to \widetilde{\mathcal{M}}_G^{\gamma}.
    \end{equation}
\cite[Theorem 3.4, Theorem 3.8]{achter_pries_imht} use the above clutching map \eqref{eq:clutching_map_for_moduli_of_G_covers_with_specified_class_vector} to prove big monodromy statements for $d = 2,3$ by inducting on the length of $\gamma$. We use these same inductive ideas to prove \Cref{theorem:big_monodromy_statement_in_characteristic_0}, using the results \cite{venkataramana_ibrd} to establish base cases. 

Before doing so, let $\gamma_n$ be the class vector
     $$\gamma_n = \begin{cases} (\underbrace{1,\ldots,1}_{n \text{ times}}, -n \pmod{d}) &\text{if } n \not\equiv 0 \pmod{d} \\ (\underbrace{1,\ldots,1}_{n \text{ times}}) &\text{if } n \equiv 0 \pmod{d}. \end{cases}$$
Write $(\widetilde{\mathcal{M}}_G^{\gamma})^\circ$ for $\widetilde{\mathcal{M}}_G^\gamma \cap \widetilde{\mathcal{M}}_G^\circ$, which is the open and dense subspace of $\widetilde{\mathcal{M}}_G^\gamma$ of smooth curves, and write $(\widetilde{\mathcal{M}}_G^\gamma)^{\text{max}}$ for the open and dense subspace of $\widetilde{\mathcal{M}}_G^\gamma$ of curves whose relative Picard scheme is a family of abelian varieties. By definition, $M_{\mathbb{Z}/\ell^k\mathbb{Z}}(\widetilde{\mathcal{M}}_G^\gamma, s)$ is $M_{\mathbb{Z}/\ell^k\mathbb{Z}}((\widetilde{\mathcal{M}}_G^\gamma)^{\text{max}}, s)$. Note that  $(\widetilde{\mathcal{M}}_G^\gamma)^{\text{max}}$ includes trees of smooth curves by \cite[Example 9.2.8]{bosch_lutkebohmert_nm}. 

There is an embedding $\iota_n: \mathsf{PConf}_n \hookrightarrow (\widetilde{\mathcal{M}}_G^{\gamma_n})^\circ$ sending $(p_1,\ldots,p_n) \in \mathsf{PConf}_n$ to the point in $(\widetilde{\mathcal{M}}_G^{\gamma_n})^\circ$ corresponding to projective curve given by the affine equation $y^d = (x-p_1) \cdots (x-p_n)$; such a curve indeed has class vector $\gamma_n$ by \Cref{proposition:monodromy_type_for_kummer_curves}. 
Moreover, the tautological curve over $(\widetilde{\mathcal{M}}_G^{\gamma_n})^\circ$ pulls back under $\iota_n$ to the family $C_n$.

\Cref{lemma:monodromy_groups_over_PConf_and_tilde_M_G_gamma_are_equal} shows that the embedding $\iota_n \otimes \mathbb{C}$ induces a surjection on \'etale fundamental groups. Since the embedding $(\widetilde{\mathcal{M}}_G)^\circ \times \mathbb{C} \hookrightarrow ( \widetilde{\mathcal{M}}_G^\gamma)^{\text{max}} \times \mathbb{C}$ is open and dense and hence induces a surjection on \'etale fundamental groups as well, the composed embedding $\mathsf{PConf}_n \otimes \mathbb{C} \hookrightarrow (\widetilde{\mathcal{M}}_G^\gamma)^{\text{max}} \times \mathbb{C}$ also induces a surjection on \'etale fundamental groups. By \Cref{lemma:comparing_monodromy_groups_under_morphism_of_bases}\eqref{part:surjection_on_fundamental_groups_implies_that_monodromy_groups_are_equal}, this in turn shows that $M_{\mathbb{Z}/\ell^k \mathbb{Z}}(\mathsf{PConf}_n \times \mathbb{C}, s) = M_{\mathbb{Z}/\ell^k \mathbb{Z}}(\widetilde{\mathcal{M}}_G^\gamma \times \mathbb{C}, s)$.

\begin{lemma} \label{lemma:monodromy_groups_over_PConf_and_tilde_M_G_gamma_are_equal}
    Let $d \geq 2$, and $n \geq 2$ be integers. For any geometric point $s \in \mathsf{PConf}_n \times \mathbb{C}$, the map $\iota_n$ induces a surjection on \'etale fundamental groups.
\end{lemma}

\begin{proof}
If $n \equiv 0 \pmod{d}$, then $\iota_n$ is an open dense embedding, as the points of $\mathsf{PConf}_n$ correspond exactly with the points of $(\widetilde{\mathcal{M}}_G^{\gamma_n})^\circ$ parameterizing $\mathbb{Z}/d\mathbb{Z}$-covers of $\mathbb{P}^1$ unramified above $\infty$. Therefore, $\iota_n \otimes \mathbb{C}$ induces a surjection on \'etale fundamental groups \cite[0BN6]{stacks-project}.

If $n \not \equiv 0 \pmod{d}$, then $\iota_n$ is identifiable with the closed embedding $\mathsf{PConf}_n \hookrightarrow \mathsf{PConf}_{n+1}^{\mathbb{P}^1}$ mentioned before \Cref{lemma:map_from_pconf_of_A1_to_pconf_of_P1_gives_surjection_of_fundamental_groups}. Note that $\mathsf{PConf}_{n+1}^{\mathbb{P}^1}$ can be identified with $(\widetilde{\mathcal{M}}_G^{\gamma_n})^\circ$ by sending $(p_1,\ldots,p_n, p_{n+1}) \in (\mathbb{P}^1)^{n+1}$ to the $\mathbb{Z}/d\mathbb{Z}$-cover of $\mathbb{P}^1$ ramified at each $p_i$ with monodromy type $\gamma_n(i)$; the affine equations for such a cover can again be obtained via \Cref{proposition:monodromy_type_for_kummer_curves}.
\end{proof}

\begin{proof}[Proof of \Cref{theorem:big_monodromy_statement_in_characteristic_0}]

By \Cref{lemma:comparing_monodromy_groups_under_morphism_of_bases}, it suffices to prove the theorem for $\mathsf{PConf}_n$. For all but finitely many primes $\ell$ and for all $2d+1 \leq n \leq 3d+1$, the closure of the image of $\pi_1(\mathsf{PConf}_n(\mathbb{C}), s)$ under the $\ell$-adic monodromy representation on $H^1(C_n|_s, \mathbb{Z}_\ell) \cong T_\ell \operatorname{Pic}^0(C_n|_s)$ contains $\mathrm{SU}(H^1(C_n|_s, \mathbb{Z}_\ell))$ by \Cref{corollary:weak_big_monodromy_statement}. Let $\mathcal{L}$ denote this set of all but finitely many primes $\ell$. Equivalently, by \Cref{lemma:comparing_monodromy_groups_under_morphism_of_bases}\eqref{part:surjection_on_fundamental_groups_implies_that_monodromy_groups_are_equal} and \Cref{lemma:monodromy_groups_over_PConf_and_tilde_M_G_gamma_are_equal}, 
    \begin{equation} \label{eq:big_monodromy_for_tilde_M_G_gamma_n}
        M_{\mathbb{Z}_\ell}(\widetilde{\mathcal{M}}_G^{\gamma_n} \times \mathbb{C}, s) \supseteq \mathrm{SU}(T_\ell \operatorname{Pic}^0(C_n|_s))  
    \end{equation}
for all $2d+1 \leq n \leq 3d+1$ and $\ell \in \mathcal{L}$.

Inductively suppose that the containment \eqref{eq:big_monodromy_for_tilde_M_G_gamma_n} holds for all $2d+1 \leq n \leq m-1$ for some $m \geq 3d+2$. Given class vectors $\gamma_a, \gamma_b$ deforming to $\gamma$ and points $s_i \in \widetilde{\mathcal{M}}_G^{\gamma_i}(\mathbb{C})$ for $i = a,b$, write $s = \kappa(s_a,s_b)$ where $\kappa$ is the clutching map \eqref{eq:clutching_map_for_moduli_of_G_covers_with_specified_class_vector}. \cite[Lemma 3.1]{achter_pries_imht} (generalized to the case where $d$ is not necessarily prime) shows that $(\kappa \otimes \mathbb{C})^* \operatorname{Pic}^0(\mathcal{C}^{\gamma})[\ell] \cong \operatorname{Pic}^0(\mathcal{C}^{\gamma_1})[\ell] \times \operatorname{Pic}^0(\mathcal{C}^{\gamma_2})[\ell]$ and that $M_{\mathbb{Z}/\ell\mathbb{Z}}(\kappa(\widetilde{\mathcal{M}}_G^{\gamma_a} \times \widetilde{\mathcal{M}}_G^{\gamma_b}), s)$ is a subgroup of $M_{\mathbb{Z}/\ell\mathbb{Z}}(\widetilde{\mathcal{M}}_G^\gamma, s)$.

We have a commuting diagram of clutching maps:
\begin{center}
    \begin{tikzcd}
        \widetilde{\mathcal{M}}_G^{\gamma_{d+1}} \times \widetilde{\mathcal{M}}_G^{\gamma_{d+1}} \times \widetilde{\mathcal{M}}_G^{\gamma_{m-2d-2}}  \ar[r, "\kappa_R"] \ar[d,"\kappa_L"] &  \widetilde{\mathcal{M}}_G^{\gamma_{d+1}} \times \widetilde{\mathcal{M}}_G^{\gamma_{m-d-1}} 
 \ar[d]  \\
        \widetilde{\mathcal{M}}_G^{\gamma_{2d+2}} \times \widetilde{\mathcal{M}}_G^{\gamma_{m-2d-2}} \ar[r] &  \widetilde{\mathcal{M}}_G^{\gamma_{m}}.
    \end{tikzcd}
\end{center}
Write $(s_1,s_2,s_3) \in (\widetilde{\mathcal{M}}_G^{\gamma_{d+1}} \times \widetilde{\mathcal{M}}_G^{\gamma_{d+1}} \times \widetilde{\mathcal{M}}_G^{\gamma_{m-2d-2}})(\mathbb{C})$, write $s \in \widetilde{\mathcal{M}}_G^{\gamma_m}$ for the image of $(s_1,s_2,s_3)$, and let $V_1, V_2, V_3, V$ be $\operatorname{Pic}^0(C)[\ell]$ for the curves $C$ parameterized by $s_1,s_2,s_3, s$ respectively. We then have the decomposition $V \cong V_1 \perp V_2 \perp V_3$ as Hermitian spaces. By the inductive hypothesis, $M_{\mathbb{Z}/\ell\mathbb{Z}}(\widetilde{\mathcal{M}}_G^{\gamma_{2d+2}})$ and $M_{\mathbb{Z}/\ell\mathbb{Z}}(\widetilde{\mathcal{M}}_G^{\gamma_{m-d-1}})$ respectively contain the $\mathrm{SU}(V_1 \perp V_2)$ and $\mathrm{SU}(V_2 \perp V_3)$. By \cite[Lemma 3.2(b)(ii)]{achter_pries_imht}, $M_{\mathbb{Z}/\ell\mathbb{Z}}(\widetilde{\mathcal{M}}_G^{\gamma_m}) \supseteq \mathrm{SU}(V)$. A subgroup of $\mathrm{SU}(\mathbb{Z}_\ell)$ surjecting onto $\mathrm{SU}(\mathbb{Z}/\ell)$ is $\mathrm{SU}(\mathbb{Z}_\ell)$ itself (see the proofs of \cite[Corollary 3.5, Corollary 3.10]{achter_pries_imht}), so in fact $M_{\mathbb{Z}_\ell}(\widetilde{\mathcal{M}}_G^{\gamma_m}) \supseteq \mathrm{SU}(T_\ell \operatorname{Pic}^0(C_m|_s))$.
\end{proof}

Now we convert the big monodromy result of \Cref{theorem:big_monodromy_statement_in_characteristic_0} into characteristic $p$ by appealing to the theory of tamely ramified fundamental groups.

\begin{proposition} \label{proposition:monodromy_for_char_p_is_bigger_than_for_char_0}
    Let $d \geq 2$, and $n \geq 2$ be integers. Given a prime number $\ell \nmid d$ and a prime number $p \nmid d \ell$, we have $M_{\mathbb{Z}/\ell^k\mathbb{Z}}( \mathsf{PConf}_n \times \mathbb{C}) \subseteq M_{\mathbb{Z}/\ell^k\mathbb{Z}}(\mathsf{PConf}_n \times \overline{\mathbb{F}}_p)$.
\end{proposition}

\begin{proof}
    \cite[Lemma 16.1.2]{scholze_weinstein_bplg} shows that $\pi_1^{\text{\'et}}(\mathsf{PConf}_n \times \mathbb{C}, s) \cong \pi_1^{\text{\'et}}(\mathsf{PConf}_n \times \overline{\mathbb{Q}}, s) \cong \pi_1^{\text{\'et}}(\mathsf{PConf}_n \times \overline{\mathbb{Q}}_p, s)$, where $s$ is a $\overline{\mathbb{Q}}$-point of $\mathsf{PConf}_n$. 

    Moreover, \cite[Lemma 7.6]{ellenberg_venkatesh_westerland} constructs a smooth and proper compactification $X_n$\footnote{Not to be confused with the Hurwitz scheme which we have been denoting as $\mathbf{X}_n$.} of $\mathsf{PConf}_n$ over $\mathbb{Z}$ such that $X_n \setminus \mathsf{PConf}_n$ is a relative normal crossings divisor. As per \cite[Expos\'e XIII, 2.1.0, 2.1.3]{SGA_1}, one defines from the open embedding $\mathsf{PConf}_n \hookrightarrow X_n$ the tamely ramified fundamental groups $\pi_1^{\text{tr}}$ of $\mathsf{PConf}_n$ over base schemes $S$. In particular, $\pi_1^{\text{\'et}}(\mathsf{PConf}_n \times \overline{\mathbb{Q}}_p, s) \cong \pi_1^{\text{tr}}(\mathsf{PConf}_n \times \overline{\mathbb{Q}}_p, s)$ as all ramifications are tame in characteristic $0$. 

    Write $\mathbb{Z}^{\text{unr}}_p$ for the valuation ring of $\mathbb{Q}^{\text{unr}}_p$. Given a specialization $\overline{s}_1 \to \overline{s}_2$ of geometric points of $\mathbb{Z}^{\text{unr}}_p$, we have a specialization morphism \cite[Expos\'e XIII, 2.10]{SGA_1}
        \begin{align*}
        \pi_1^{\text{tr}}(\mathsf{PConf}_{n,\overline{s}_1}, a_1) \to \pi_1^{\text{tr}}(\mathsf{PConf}_n \times \mathbb{Z}^{\text{unr}}_p, a_1) \xrightarrow{\cong} \pi_1^{\text{tr}}(\mathsf{PConf}_n \times \mathbb{Z}^{\text{unr}}_p, a_2) \xleftarrow{\cong} \pi_1^{\text{tr}}(\mathsf{PConf}_{n,\overline{s}_2}, a_2).
        \end{align*}
    where $a_i$ is a geometric point of $\mathsf{PConf}_{n,\bar{s}_i} = \mathsf{PConf}_n \times \bar{s}_i$ for $i = 1,2$.
    Letting $\overline{s}_1 = \operatorname{Spec} \overline{\mathbb{Q}^{\text{unr}}_p} = \operatorname{Spec} \overline{\mathbb{Q}}_p$ and $\overline{s}_2 = \operatorname{Spec} \overline{\mathbb{F}}_p$, this specialization morphism is a morphism $\pi_1^{\text{tr}}(\mathsf{PConf}_n \times \overline{\mathbb{Q}}_p, a_1) \to \pi_1^{\text{tr}}(\mathsf{PConf}_n \times \overline{\mathbb{F}}_p, a_2)$. 

    We show that the mod $\ell^k$-monodromy representations of $\mathsf{PConf}_n$ (over a base scheme $S$ of characteristic not $\ell$ such that $\mathsf{PConf}_n \times S$ is a normal integral scheme) 
    factor through the natural surjections $\pi_1^{\text{\'et}}(\mathsf{PConf}_n \times S, s) \to \pi_1^{\text{tr}}(\mathsf{PConf}_n \times S, s)$, i.e. that these representations are tamely ramified in $X_n$. Since $\operatorname{Pic}^0(C_n/\mathsf{PConf}_n \times S)$ is an abelian scheme, $J[\ell] := \operatorname{Pic}^0(C_n/(\mathsf{PConf}_n \times S))[\ell]$ is finite \'etale over $\mathsf{PConf}_n \times S$. For $J[\ell]$ to be tamely ramified along $X_n \setminus \mathsf{PConf}_n$ would mean the following \cite[Definition 5.7.1.5]{szamuley_gg}: for each divisor $D$ of $X_n$ not in $\mathsf{PConf}_n$ and for each connected component $Y$ of $J[\ell]$, the closed points of the normalization of $\operatorname{Spec}(\mathcal{O}_{X_n, D})$ in the function field $K(Y)$ have ramification indices prime to the characteristic of $\kappa(D)$. Take $T \to \mathsf{PConf}_n \times S$ to be the normalization of the compositum extension of the function fields of the connected components of $J[\ell]$. In particular, $J[\ell] \times_{\mathsf{PConf}_n \times S} T \cong \mathbb{Z}/\ell\mathbb{Z}^{2g}$ as abelian schemes over $T$ where $g$ is the genus of $C_n$. 
    After the unramified base change adjoining $\mu_\ell$, local Kummer theory exhibits $K(T)/K(X_n)$ as generated by $\ell$th roots, and extracting an $\ell$th root in a DVR of characteristic $\neq \ell$ contributes ramification index dividing $\ell$; hence the total ramification index divides $\ell^{2g}$. Thus, the ramification index of $\operatorname{Spec}(\mathcal{O}_{X_n,D})$ in $K(T)$ must divide $\ell^{2g}$ and hence the same must be true of any ramification index of $\operatorname{Spec}(\mathcal{O}_{X_n,D})$ in any $K(Y)$. Since $S$ is assumed to be of characteristic not $\ell$, $J[\ell]$ is tamely ramified along $X_n \setminus \mathsf{PConf}_n$. We apply this to $S = \mathbb{Q}_p^{\text{unr}}, \mathbb{Z}_p^{\text{unr}}, \overline{\mathbb{F}}_p$.

    To summarize, we have the following sequence of group homomorphisms all of which have compatible mod $\ell^k$-monodromy representations:
        $$\pi_1^{\text{\'et}}(\mathsf{PConf}_n \times \mathbb{C}, s) \xrightarrow{\cong} \pi_1^{\text{\'et}}(\mathsf{PConf}_n \times \overline{\mathbb{Q}}_p, s) \to \pi_1^{\text{tr}}(\mathsf{PConf}_n \times \overline{\mathbb{Q}}_p, s)\to \pi_1^{\text{tr}}(\mathsf{PConf}_n \times \overline{\mathbb{F}}_p, s').$$
    The image of $\pi_1^{\text{\'et}}(\mathsf{PConf}_n \times \mathbb{C}, s)$ under its monodromy representation is thus contained in the image of $\pi_1^{\text{tr}}(\mathsf{PConf}_n \times \overline{\mathbb{F}}_p, s')$ under its monodromy representation. The former image is $M_{\mathbb{Z}/\ell^k\mathbb{Z}}(\mathsf{PConf}_n \times \mathbb{C})$. Since $\pi_1^{\text{\'et}}(\mathsf{PConf}_n \times \overline{\mathbb{F}}_p, s')$ naturally surjects onto $\pi_1^{\text{tr}}(\mathsf{PConf}_n \times \overline{\mathbb{F}}_p, s')$, the latter image equals $M_{\mathbb{Z}/\ell^k\mathbb{Z}}(\mathsf{PConf}_n \times \overline{\mathbb{F}}_p)$.
\end{proof}

\begin{corollary} \label{corollary:big_monodromy_result_needed_for_counting}
    Let $d \geq 3$. For all but finitely many primes $\ell \nmid d$, for all $n \geq 2d+1$, and for any prime power $q$ coprime to $\ell d$, we have 
    $$\mu(\operatorname{Gal}(\overline{K}_n / \Fqbar K_n)) \supseteq \mathrm{SU}(V_n),$$
    where $K_n$ and $\mu$ are as in \eqref{eq:galois_representation_of_Z_dZ_covers_of_P1_over_generic_point_of_Conf_n}.
\end{corollary}
\begin{proof}
    By \Cref{corollary:monodromy_group_over_generic_point_coincides_with_monodromy_group}, the image of $\mu$ coincides with the monodromy group $M_{\mathbb{Z}_\ell}(\mathsf{Conf}_n \times \Fqbar)$. By \Cref{theorem:big_monodromy_statement_in_characteristic_0} and \Cref{proposition:monodromy_for_char_p_is_bigger_than_for_char_0}, $M_{\mathbb{Z}_\ell}(\mathsf{PConf}_n \times \Fqbar)$ contains $\mathrm{SU}(V_n)$. The morphism $\mathsf{PConf}_n \to \mathsf{Conf}_n$, which forgets the ordering of each configuration, induces a homomorphism $\pi_1^{\text{\'et}}(\mathsf{PConf}_n \times \Fqbar, \bar{s}) \to \pi_1^{\text{\'et}}(\mathsf{Conf}_n \times \Fqbar, \bar{s})$ and hence an inclusion $M_{\mathbb{Z}_\ell}(\mathsf{PConf}_n \times \Fqbar) \subseteq M_{\mathbb{Z}_\ell}(\mathsf{Conf}_n \times \Fqbar)$. Therefore, $M_{\mathbb{Z}_\ell}(\mathsf{Conf}_n \times \Fqbar)$ contains $\mathrm{SU}(V_n)$.
    
\end{proof}

\subsection{Counting surjections from class groups} \label{subsection:counting_surjections_from_class_groups}

Recall some notation at the start of \Cref{subsection:big_monodromy_hypothesis}. We first state \Cref{theorem:generalizing_EVW_Theorem_8.8_to_cubic_covers} by including a big monodromy condition as an assumption in the hypothesis. We then state \Cref{corollary:generalizing_EVW_Theorem_8.8_to_cubic_covers}, which applies \Cref{corollary:big_monodromy_result_needed_for_counting} to \Cref{theorem:generalizing_EVW_Theorem_8.8_to_cubic_covers} to remove the big monodromy condition in the hypothesis at the cost of restrictions on $d$, $\ell$, and $n$. If future works should strengthen the big monodromy results of \Cref{subsection:big_monodromy_hypothesis} by allowing for more general combinations of $d$, $\ell$, and $n$, then \Cref{corollary:generalizing_EVW_Theorem_8.8_to_cubic_covers} can be strengthened as well.

\begin{theorem}\label{theorem:generalizing_EVW_Theorem_8.8_to_cubic_covers}
    Let $d \geq 3$ be an integer, let $\ell$ be a prime number not dividing $d$, and let $\mathbb{Z}_\ell[\zeta_d] = \mathbb{Z}_\ell[\zeta_d]_{\zeta_d \neq 1}$. Fix $A$ to be a $\mathbb{Z}_\ell[\zeta_d]$-module of finite cardinality. There are constants $B_{d,\ell,A}$, $C_{d,\ell,A}$, and $N_{d,\ell,A}$ such that 
        $$\left| \frac{\sum_{L \in \mathfrak{S}_n} m_A(L)}{|\mathfrak{S}_n|} - 1 \right| \leq \frac{C_{d,\ell,A}}{\sqrt{q}}$$
    for all $n$,$q$ such that 
    \begin{itemize}
        \item $n$ is relatively prime to $d$,
        \item $q \equiv 1 \pmod{d}$  is a prime power,
        \item $\sqrt{q} > 2 B_{d,\ell,A}$, $n \geq N_{d,\ell,A}$, 
        \item $\ell$ does not divide $2dq(q-1)$, and
        \item $\mathrm{SU}(V_n) \subseteq \mu(\operatorname{Gal}(\overline{K}_n / \overline{\mathbb{F}}_q K_n))$.
    \end{itemize}
    In fact, it suffices for $N_{d,\ell,A} = 2r+3$ if $A$ is a quotient of $\mathbb{Z}_\ell[\zeta_d]^{\oplus r}$. 
\end{theorem}

\begin{corollary}\label{corollary:generalizing_EVW_Theorem_8.8_to_cubic_covers}
   Let $d \geq 3$ be an integer. For all but finitely many prime numbers $\ell \nmid d$, for any module $A$ of finite cardinality for the ring $\mathbb{Z}_\ell[\zeta_d] = \mathbb{Z}_\ell[\zeta_d]_{\zeta_d \neq 1}$, there are constants $B_{d,\ell,A}$, $C_{d,\ell,A}$, and $N_{d,\ell,A}$ such that 
        $$\left| \frac{\sum_{L \in \mathfrak{S}_n} m_A(L)}{|\mathfrak{S}_n|} - 1 \right| \leq \frac{C_{d,\ell,A}}{\sqrt{q}}$$
    for all $n$,$q$ such that 
    \begin{itemize}
        \item $n$ is relatively prime to or divisible by $d$,
        \item $q \equiv 1 \pmod{d}$  is a prime power,
        \item $\sqrt{q} > 2 B_{d,\ell,A}$, $n \geq N_{d,\ell,A}$, and
        \item $\ell$ does not divide $2dq(q-1)$.
    \end{itemize}
    In fact, it suffices for $N_{d,\ell,A} = \max(2d+1,2r+3)$ if $A$ is a quotient of $\mathbb{Z}_\ell[\zeta_d]^{\oplus r}$. 
\end{corollary}

\begin{remark}
    In the proof of \Cref{theorem:generalizing_EVW_Theorem_8.8_to_cubic_covers} below, note that the condition $n \geq 2r+3$ is used to apply \Cref{corollary:generalizing_EVW_Lemma 8.9_to_Z_dZ_covers}.
\end{remark}

\begin{remark}
    The below proof of \Cref{theorem:generalizing_EVW_Theorem_8.8_to_cubic_covers} is based on that of \cite[Theorem 8.8]{ellenberg_venkatesh_westerland}. The main difference between the two proofs is that the former proof needs to account for $\zeta_d$-actions in general whereas the latter proof only needs to do so for $d = 2$, i.e. when $\zeta_d=-1$. In particular, the latter proof counts surjections between $\mathbb{Z}_\ell$-modules and the former proof counts surjections between modules over $\mathbb{Z}_\ell[\zeta_d] = \mathbb{Z}_\ell[\zeta_d]_{\zeta_d \neq 1}$. Moreover, the $\ell$-adic Weil pairing on $\operatorname{Jac}(C_L)$, where $C_L$ is the $\mathbb{Z}/d\mathbb{Z}$-cover of $\mathbb{P}^1$ corresponding to an imaginary $\mathbb{Z}/d\mathbb{Z}$-extension $L$ of $\Fq(t)$, is a symplectic pairing $\omega: T_\ell \operatorname{Jac}(C_L) \times T_\ell \operatorname{Jac}(C_L) \to \mathbb{Z}_\ell(1)$ that is also preserved by the $\zeta_d$ action, i.e. $\omega(\zeta_d v, \zeta_d w) = \omega(v,w)$. Such a symplectic pairing gives rise to a Hermitian pairing by \Cref{proposition:symplectic_pairings_to_hermitian_forms} since $T_\ell \operatorname{Jac}(C_L)$ is a $\mathbb{Z}_\ell[\zeta_d]$-module by \Cref{lemma:cyclic_sum_of_zeta_d_acts_as_zero_on_order_d_superelliptic_curve}. In contrast, the $\ell$-adic Weil pairing for the Jacobian of a hyperelliptic curve is merely a symplectic pairing.
\end{remark}

\begin{proof}[Proof of \Cref{theorem:generalizing_EVW_Theorem_8.8_to_cubic_covers}]

    Let $H^i$ denote the $i$th \'etale cohomology group functor and let $H^i_c$ denote the $i$th compactly supported \'etale cohomology group functor.
    
    Note that if $A$ is the trivial group, then the LHS of the desired inequality is $0$. Now assume that $A$ is nontrivial.
    
    Let $G$ and $c$ be as in \Cref{proposition:generalizing_EVW_proposition_8.7_to_cubic_covers}. In particular, there is a bijection between $\mathbf{X}_n(\mathbb{F}_q)$ and the set of isomorphism classes of pairs $(L,\alpha)$ as stated in the proposition. Each isomorphism class of $(L,\alpha)$ has $d$ elements; the automorphisms of $L/\mathbb{F}_q(t)$ act simply transitively on such an isomorphism class, so 
        $$\sum_{L \in \mathfrak{S_n}} m_A(L) = d|\mathbf{X}_n(\mathbb{F}_q)|.$$
    We count $\mathfrak{S}_n$. To exhaust $\mathfrak{S}_n$, it is sufficient to let $f$ range through a set of representatives for squarefree polynomials of degree $n$ up to the multiplication action of $(\mathbb{F}_q^*)^d$. The number of monic squarefree polynomials of degree $n$ with coefficients in $\mathbb{F}_q$ is equal to $q^{n} - q^{n-1}$ \cite[4, equation (vi)]{carlitz_apgf}. Since $q \equiv 1 \pmod{d}$, the $d$th power elements form an index $d$ subgroup of $\mathbb{F}_q^*$. Therefore,
        $$|\mathfrak{S}_n| = d(q^{n} - q^{n-1})$$
    and hence the LHS of the desired inequality equals
        $$\left| \frac{|\mathbf{X}_n(\mathbb{F}_q)|}{q^{n}-q^{n-1}} - 1 \right|.$$
    It thus suffices to show that 
        \begin{equation} \label{eq:Hurwitz_scheme_has_about_q_n-1_rational_points}
            \left| \frac{|\mathbf{X}_n(\mathbb{F}_q)|}{q^{n}} - 1 \right| \leq \frac{C_{d,\ell,A}}{\sqrt{q}} 
        \end{equation}
    when $n$ and $q$ are sufficiently large relative to $A$.
    
    Let $\overline{\mathbf{X}}_n := \mathbf{X}_n \times \Fqbar$. Recall that $G$ is center-free by \Cref{lemma:semidirect_product_is_center_free_and_A_has_trivial_invariant_subgroup}. Moreover, $c$ is a $q$-rational conjugacy class of $G$ because $q \equiv 1 \pmod{d}$, and it is non-splitting by \Cref{lemma:c_is_nonsplitting_conjugacy_class}. As per \Cref{remark:proposition_7.8_of_EVW_generalizes_to_rational_conjugacy_classes}, an appropriate generalization of \cite[Proposition 7.8]{ellenberg_venkatesh_westerland} applies to $\mathbf{X}_n$ and additionally applying Poincar\'e duality for the smooth $n$-dimensional variety $\overline{\mathbf{X}}_n$, there are constants $K(A)$ and $B_{d,\ell,A}$ such that \footnote{Even though $\mathbf{X}_n$ is not the same as $\mathsf{Hn}_{G,n}^c$ when $d \mid n$, the inequality still applies since $\mathbf{X}_n$ is the union of some connected components of $\mathsf{Hn}_{G,n}^c$ in this case.} 
        \begin{equation} \label{eq:cohomological_stability_bounds}
            \dim H_c^{2n-i}\left(\overline{\mathbf{X}}_n, \mathbb{Q}_\lambda \right) = \dim H^i \left( \overline{\mathbf{X}}_n; \mathbb{Q}_\lambda \right) \leq K(A) \cdot B_{d,\ell,A}^i
        \end{equation}
    for all $i > 0$ as long as $\lambda$ is a prime greater than $\max(|G|, q, n)$ --- for example, it suffices to take $K(A), B_{d,\ell,A} = C(G,c)$ in the notation of \cite[Section 7.8.1]{ellenberg_venkatesh_westerland}. 
    
    Deligne \cite{deligne_lcwi} proved that the absolute value of every eigenvalue of the geometric Frobenius $\operatorname{Frob}_q$ on compactly supported $H^j_c$ of a smooth variety is bounded above by $q^{j/2}$. Consequently,
        \begin{align*} 
        \left| q^{-n} \sum_{j < 2n} (-1)^j \operatorname{Tr}(\operatorname{Frob}_q|H^j_c(\overline{\mathbf{X}}_n; \mathbb{Q}_\lambda)) \right| &\leq q^{-n} \sum_{j=0}^{2n-1} q^{j/2} \dim H^j_c(\overline{\mathbf{X}}_n; \mathbb{Q}_\lambda) \\ 
        & \leq q^{-n} K(A) \sum_{j=0}^{2n-1} B_{d,\ell,A}^{2n-j} q^{j/2} \\
        & \leq K(A) \sum_{k=1}^\infty \left( \frac{B_{d,\ell,A}}{\sqrt{q}} \right)^k \\
        & = K(A) \cdot \frac{ \frac{B_{d,\ell,A}}{\sqrt{q}} }{1 - \frac{B_{d,\ell,A}}{\sqrt{q}} }.
        \end{align*}
    The last quantity is at most $2 \frac{K(A) B_{d,\ell,A}}{\sqrt{q}}$ when $\frac{B_{d,\ell,A}}{\sqrt{q}} \leq \frac{1}{2}$. Let $C_{d,\ell,A} = 2K(A)B_{d,\ell,A}$ so that
        \begin{equation} \label{eq:alternating_sum_for_trace_of_frobenius_for_lower_cohomology_is_small}
        \left| q^{-n} \sum_{j < 2n} (-1)^j \operatorname{Tr}(\operatorname{Frob}_q|H^j_c(\overline{\mathbf{X}}_n; \mathbb{Q}_\lambda)) \right| \leq \frac{C_{d,\ell,A}}{\sqrt{q}} 
        \end{equation}
    whenever $\sqrt{q} \geq 2 B_{d,\ell,A}$.
    
    
    We now show, assuming that $n \geq 2r+3$, that
        $$\operatorname{Tr}(\operatorname{Frob}_q|H^{2n}_c(\overline{\mathbf{X}}_n; \mathbb{Q}_\lambda)) = q^n.$$
    By Poincar\'e duality, this is equivalent to the statement that there is exactly one $\mathbb{F}_q$-rational connected component of $\overline{\mathbf{X}}_n$. This, together with the Lefschetz trace formula and the bound \eqref{eq:alternating_sum_for_trace_of_frobenius_for_lower_cohomology_is_small} will imply \eqref{eq:Hurwitz_scheme_has_about_q_n-1_rational_points}.
    
    

    Let $\eta$ be the generic point of $\mathsf{Conf}_{n} \times_{\operatorname{Spec} \mathbb{Z}} \operatorname{Spec} \mathbb{F}_q$ and let $K_n$ be its function field. The \'etale cover $\mathbf{X}_n \to \mathsf{Conf}_{n} \times_{\operatorname{Spec} \mathbb{Z}} \operatorname{Spec} \mathbb{F}_q$ is determined by a geometric generic fiber $\Sigma$ of $\eta$ together with the action of $\operatorname{Gal}(\bar{K}_n/K_n)$ on that fiber. Associated to the Galois group is a short exact sequence
        $$1 \to \operatorname{Gal}(\bar{K}_n/ \bar{\mathbb{F}_q} K_n) \to \operatorname{Gal}(\bar{K}_n/K_n) \to \operatorname{Gal}(\overline{\mathbb{F}}_q/\mathbb{F}_q) \to 1.$$
    Thus, there is exactly one $\mathbb{F}_q$-rational connected component of $\mathbf{X}_n$ if and only if only one $\operatorname{Gal}(\bar{K}_n/\overline{\mathbb{F}}_q K_n)$-orbit of $\Sigma$ is preserved by the action of $\operatorname{Gal}(\overline{\mathbb{F}}_q / \mathbb{F}_q)$, which we prove 
    by expressing $\Sigma$ in a different way. Recall the construction of $\mathsf{Conf}_{n}$ as the subscheme of $\mathsf{Conf}'_{n+1}$ with $a_0 = 0$. 

    By abuse of notation, let $C_n$ denote the base change of the curve notated as $C_n$ in \Cref{subsection:big_monodromy_hypothesis} over $K_n$. 
    Choose $k$ sufficiently large so that $\ell^k A = 0$. Let $V_{n,k}$ be the $\ell^k$-torsion points of the Jacobian $\operatorname{Jac}(C_n)$ over $\bar{K}_n$. Then $V_{n,k} \simeq (\mathbb{Z}/\ell^k \mathbb{Z})^{2g}$ as a $\mathbb{Z}/\ell^k\mathbb{Z}$-module, where $g$ is the genus of $C_n$. Note that $g = \frac{(m-2)(d-1)}{2}$ by the Riemann-Hurwitz formula where $m$ is the branch locus degree of $C_n \to \mathbb{P}^1$, i.e. $m = \begin{cases} n+1 &\text{if } d \nmid n \\ n &\text{if } d \mid n \end{cases}$. Moreover, we have a mod-$\ell^k$ monodromy representation
        
        $$\mu_{\ell^k}: \operatorname{Gal}(\bar{K}_n/K_n) \to \operatorname{Aut}(V_{n,k})$$
    which factors the Galois/monodromy representation $\mu$ from \eqref{eq:galois_representation_of_Z_dZ_covers_of_P1_over_generic_point_of_Conf_n}. Letting $L$ be the function field of $C_n$, there is also an action of $\operatorname{Aut}(C_n/\mathbb{P}^1) \cong \operatorname{Gal}(L/K_n(t))$ on $V_{n,k}$ induced by the action on $C_n$. Identify the automorphism/Galois group with $\langle \zeta_d \rangle \cong \mathbb{Z}/d\mathbb{Z}$ so that $V_{n,k}$ is a $(\mathbb{Z}/\ell^k\mathbb{Z})[\zeta_d]$-module (and in fact a $\mathbb{Z}_\ell[\zeta_d]$-module) by \Cref{lemma:cyclic_sum_of_zeta_d_acts_as_zero_on_order_d_superelliptic_curve}.
    
    Now consider the set $\operatorname{Sur}_{\mathbb{Z}_\ell[\zeta_d]}(V_{n,k},A)$ of surjective homomorphisms from $V_{n,k}$ to $A$ as $\mathbb{Z}_\ell[\zeta_d]$-modules. This set carries a natural action of $\operatorname{Gal}(\bar{K}_n/K_n)$ derived from $\mu_{\ell^k}$.
    Since $\mathsf{Conf}_{n}$ is a moduli scheme for degree-$n$ squarefree divisors on $\mathbb{A}^1$, there is a universal such divisor on $\mathbb{A}^1 / \mathsf{Conf}_{n}$, which restricts to a canonical degree-$n$ squarefree (i.e. reduced) divisor $D$ on $\mathbb{A}^1/\bar{K}_n$. The set $\mathbf{X}_n(\bar{K}_n)$ of tame $G$-covers of $\mathbb{A}^1/\bar{K}_n$ branched at $D$ (which is to say $\Sigma)$ is naturally identified with $\operatorname{Sur}_{\mathbb{Z}_\ell[\zeta_d]}(V_{n,k},A)$ by the argument of \Cref{proposition:generalizing_EVW_proposition_8.7_to_cubic_covers}, equivariantly for the action of $\operatorname{Gal}(\bar{K}_n/K_n)$ on both sides (The statement of class field theory of \Cref{eq:geometric_class_field_theory} needs to be replaced with the fact that the abelian \'etale covers of $C_n / \overline{K}_n$ with Galois group $A$ are classified by surjections $\operatorname{Jac}(C_n)[\ell^k](\overline{K}_n) \twoheadrightarrow A$, see e.g. \cite[(2.4)]{katz_lang_ftgct}). 
    
    To summarize, the geometric components of $\overline{X}_n / \overline{\mathbb{F}}_q$ are in bijection with the $\operatorname{Gal} \left(\overline{K}_n / \overline{\mathbb{F}}_q K_n \right)$-orbits on $\operatorname{Sur}(V_{n,k}, A)$. It thus suffices to show that exactly one $\operatorname{Gal} \left(\overline{K}_n / \overline{\mathbb{F}}_q K_n \right)$-orbit on $\operatorname{Sur}(V_{n,k},A)$ is preserved by the action of $\operatorname{Gal} \left(\overline{\mathbb{F}}_q / \mathbb{F}_q \right)$ (again, for $n$ large enough). Moreover, an orbit $O$ is defined over $\mathbb{F}_q$ if and only if the stabilizer in $\mathrm{GU}(V_{n,k})$ of some, equivalently every, $x \in O$  has nontrivial intersection with $\mathrm{GU}_q(V_{n,k})$.

    By assumption, $\mathrm{SU}(V_{n,k}) \subseteq \mu_{\ell^k}(\operatorname{Gal}(\overline{K}_n / \overline{\mathbb{F}}_q K_n))$. Moreover, since $\ell^k A = 0$, the pullback via $V_n \to V_{n,k}$ identifies $\operatorname{Sur}(V_{n,k}, A)$ with $\operatorname{Sur}(V_n, A)$. It thus suffices to show that there is a unique $\mathrm{SU}(V_n)$-orbit on $\operatorname{Sur}_{\mathbb{Z}_\ell[\zeta_d]}(V_{n},A)$ defined over $\mathbb{F}_q$ when $n$ is sufficiently large. This follows from \Cref{corollary:generalizing_EVW_Lemma 8.9_to_Z_dZ_covers} for $n \geq 2r+3$.
\end{proof}

We also identify sufficient values for the constants $B_{d,\ell,A}$ and $C_{d,\ell,A}$ afforded by \Cref{corollary:big_monodromy_result_needed_for_counting}. While we do not use these values for \Cref{section:cohen_lenstra_for_imaginary_Z_3Z_extensions}, we will need them for \Cref{section:generalizing_ELS_section_3}.

\begin{proposition}[{cf. \cite[Proposition 2.7]{ellenberg_li_shusterman}}] \label{proposition:constants_of_EVW}
    Let $d \geq 3$ be an integer. For all but finitely many primes $\ell \nmid d$, and for any involution ring quotient $A$ of $\mathbb{Z}/\ell\mathbb{Z}[\zeta_d]$, the conclusion of \Cref{corollary:generalizing_EVW_Theorem_8.8_to_cubic_covers} holds with the constants $B_{d,\ell,A} = (2d|A|)^{39d+18}$ and $C_{d,\ell,A} = 2 \cdot (2d|A|)^{69d+31}$.
\end{proposition}

\begin{proof}

    Recall that the proof of \Cref{theorem:generalizing_EVW_Theorem_8.8_to_cubic_covers} uses constants $K(A)$ and $B_{d,\ell,A}$ such that the bound \eqref{eq:cohomological_stability_bounds} holds. 

    In turn, this bound comes from a stability theorem \cite[Theorem 6.2]{ellenberg_venkatesh_westerland} and an absolute cohomological bound \cite[Proposition 2.5]{ellenberg_venkatesh_westerland}. Letting $\mathbf{X}_n^a$ be as in \Cref{section:generalizing_ELS_section_2} above \Cref{lemma:research_generalizing_ELS_lemma_2.3_to_trielliptic_curves}, the stability theorem implies that there exist constants $\alpha, \beta, \delta$ dependent on $d$ and $\ell$ such that
        $$\dim H^i_{\text{\'et}} (\mathbf{X}_n^a \times \overline{\mathbb{F}}_q ; \mathbb{Q}_\lambda) = \dim H^i_{\text{\'et}} (\mathbf{X}_{n+\delta}^a \times \overline{\mathbb{F}}_q ; \mathbb{Q}_\lambda)$$
    whenever $n > \alpha i+\beta$. Furthermore, the absolute bound (along with some comparison results between the cohomologies of Hurwitz schemes over $\mathbb{C}$ and $\overline{\mathbb{F}}_q$ explained more in detail in the proof of \cite[Proposition 7.8]{ellenberg_venkatesh_westerland}) tells us that 
        $$\dim H^i_{\text{\'et}} (\mathbf{X}_n^a \times \overline{\mathbb{F}}_q ; \mathbb{Q}_\lambda) \leq (2|A \rtimes \mathbb{Z}/d\mathbb{Z}|)^n = (2 d |A|)^n.$$
    These imply that
        $$\dim H^i_{\text{\'et}} (\mathbf{X}_n^a \times \overline{\mathbb{F}}_q ; \mathbb{Q}_\lambda) \leq (2d|A|)^{\alpha i + \beta + \delta}$$
    for every $n$. 

    We now obtain more concrete values of $\alpha$, $\beta$, and $\delta$ using the ideas presented in the proof of \cite[Proposition 2.7]{ellenberg_li_shusterman} and details which we prove in \Cref{subsection:ring_of_connected_components_of_hurwitz_schemes}. The proof shows that $\alpha = 3A_0$, $\beta = 2A_0 + A_2$, and $\delta = \deg U_D$ suffice with the following assignments and lower bounds: 
    \begin{itemize}
        \item $U_D$ is defined in \eqref{eq:definition_of_U_D};  $D = 1$ suffices by \Cref{lemma:ring_of_EVW_central_element_degree}, in which case $\deg U_D = d$.
        \item $A(R) = \max(\deg \ker U_D, \deg \operatorname{coker} U_D)$, which is at most $2d+1$ by \Cref{lemma:ring_of_EVW_central_element_degree}.
        \item $A_0 = 6A(R) + \deg U_D \leq 13d + 6$.
        \item $A_2 = A(R) + \deg U_D \leq 3d + 1$.
    \end{itemize}
    In particular, it suffices for $\alpha$, $\beta$, and $\delta$ to be $39d + 18$, $29d + 13$, and $d$ respectively.

    Thus, $B_{d,\ell,A} = (2d|A|)^\alpha = (2d|A|)^{39d + 18}$ and $K(A) = (2d|A|)^{\beta + \delta} = (2d|A|)^{30d+13}$ suffice. Recall that we let $C_{d,\ell,A} = 2K(A)B_{d,\ell,A}$ in the proof of \Cref{theorem:generalizing_EVW_Theorem_8.8_to_cubic_covers}, so $C_{d,\ell,A} = 2 \cdot (2d|A|)^{69d+31}$ suffices.
\end{proof}

\section{Cohen-Lenstra distribution for imaginary \texorpdfstring{$\mathbb{Z}/d\mathbb{Z}$}{Z/dZ}-extensions of \texorpdfstring{$\Fqt$}{Fqt}} \label{section:cohen_lenstra_for_imaginary_Z_3Z_extensions}
Using \Cref{theorem:generalizing_EVW_Theorem_8.8_to_cubic_covers}, we can prove \Cref{theorem:generalizing_EVW_main_theorem_to_cubic_covers}. Just as \cite[Main Theorem, Theorem 1.2]{ellenberg_venkatesh_westerland} uses \cite[Theorem 8.8]{ellenberg_venkatesh_westerland} and \cite[Proposition 8.3]{ellenberg_venkatesh_westerland}, we will require \Cref{proposition:generalizing_EVW_Proposition_8_3_to_equivariant_surjections}, which is a statement about probability measures similar to \cite[Proposition 8.3]{ellenberg_venkatesh_westerland}. The module-theoretic ideas used to prove \cite[Proposition 8.3]{ellenberg_venkatesh_westerland} readily generalize to prove \Cref{proposition:generalizing_EVW_Proposition_8_3_to_equivariant_surjections}. We nevertheless record the details for completeness. Compare this relative ease of generalization against the difficulties involved in proving \Cref{proposition:generalizing_EVW_Lemma 8.9_to_Z_dZ_covers} and \Cref{corollary:generalizing_EVW_Lemma 8.9_to_Z_dZ_covers} --- where we used the theory of Hermitian forms to generalize the arguments of \cite[Lemma 8.9]{ellenberg_venkatesh_westerland}, which were based on the theory of symplectic forms --- cf. the discussion at the start of \Cref{subsection:counting_the_number_of_Fq_rational_connected_components_of_hurwitz_schemes}.

Throughout this section, write $\mathbb{Z}_\ell[\zeta_d] = \mathbb{Z}_\ell[\zeta_d]_{\zeta_d \neq 1}$ when $d \geq 2$ is an integer and $\ell \nmid d$ is a prime number. Further write $\mathcal{L}_{d,\ell}$ for the set of isomorphism classes of $\mathbb{Z}_\ell[\zeta_d]$-modules with finite cardinality. In particular, such a module has an action of $\langle \zeta_d \rangle$ and its $\zeta_d$-invariant submodule is trivial by \Cref{corollary:equivalent_conditions_for_A_to_have_trivial_zeta_d_invariant_subgroup}. Write $C_{d,\ell} = \left( \sum_{[M] \in \mathcal{L}_{d,\ell}} \frac{1}{|\operatorname{Aut}_{\mathbb{Z}_\ell[\zeta_d]}(M)|} \right)^{-1}$, which converges (see \cite[Th\'eor\`eme 3.6]{cohen_martinet_eh} and cf. \cite[Remark 3.4 and the discussion between Definition 6.1 and Theorem 6.2]{wang_wood_miclm}). The \emph{Cohen-Lenstra distribution} $\mu_{d,\ell}$ in the case of interest is the probability distribution on $\mathcal{L}_{d,\ell}$ such that the $\mu_{d,\ell}$-mass on the isomorphism class of $A$ equals $\frac{C_{d,\ell}}{|\operatorname{Aut}_{\mathbb{Z}_\ell[\zeta_d]}(A)|}$.

As with \Cref{theorem:generalizing_EVW_Theorem_8.8_to_cubic_covers}, we first state \Cref{theorem:generalizing_EVW_main_theorem_to_cubic_covers} to include a big monodromy condition in the hypothesis and then apply \Cref{corollary:big_monodromy_result_needed_for_counting} to state \Cref{corollary:generalizing_EVW_main_theorem_to_cubic_covers}. Again, if \Cref{corollary:big_monodromy_result_needed_for_counting} can be strengthened to include more combinations of $d$, $\ell$, and $n$, then \Cref{corollary:generalizing_EVW_main_theorem_to_cubic_covers} can be strengthened as well.

\begin{theorem}\label{theorem:generalizing_EVW_main_theorem_to_cubic_covers}

Let $d \geq 3$ be an integer, and let $\ell$ be a prime number not dividing $2d$. Fix $[A] \in \mathcal{L}_{d,\ell}$. Assuming that $\mathrm{SU}(V_n) \subseteq \mu(\operatorname{Gal}(\overline{K}_n / \overline{\mathbb{F}}_q K_n))$ for all sufficiently large $n$, write $\delta^+$ and $\delta^-$ respectively for the following upper and lower densities of totally imaginary $\mathbb{Z}/d\mathbb{Z}$ extensions $L$ of $\mathbb{F}_q(t)$ for which the $\ell$-part of the class group is isomorphic to $A$ under an isomorphism equivariant for the $\langle \zeta_d \rangle \cong \operatorname{Gal}(L/\Fqt)$ actions on $A$ and $(\operatorname{Cl}_L)_\ell$:
    \begin{align*}
        \delta^+(q) &= \limsup_{\substack{n \to \infty \\ \gcd(n,d) = 1 \\ \mathrm{SU}(V_n) \subseteq \mu(\operatorname{Gal}(\overline{K}_n / \overline{\mathbb{F}}_q K_n)}} \frac{|\{L \in \mathfrak{S}_n: (\operatorname{Cl}_L)_\ell \cong_{\mathbb{Z}_\ell[\zeta_d]} A \}|}{|\mathfrak{S}_n|} \\
        \delta^-(q) &= \liminf_{\substack{n \to \infty \\ \gcd(n,d) = 1  \\ \mathrm{SU}(V_n) \subseteq \mu(\operatorname{Gal}(\overline{K}_n / \overline{\mathbb{F}}_q K_n)}} \frac{|\{L \in \mathfrak{S}_n: (\operatorname{Cl}_L)_\ell \cong_{\mathbb{Z}_\ell[\zeta_d]} A \}|}{|\mathfrak{S}_n|}. 
    \end{align*}

Then, $\delta^+(q)$ and $\delta^-(q)$ converge, as $q \to \infty$ with $\ell \nmid 2dq(q-1)$ and $q \equiv 1 \pmod{d}$, to $\mu_{d,\ell}([A]) = \frac{C_{d,\ell}}{|\operatorname{Aut}_{\mathbb{Z}_\ell[\zeta_d]}(A)|}$.
\end{theorem}

\begin{corollary}[Restatement of \Cref{theorem:generalizing_EVW_main_theorem_to_cubic_covers_intro}]\label{corollary:generalizing_EVW_main_theorem_to_cubic_covers}
Let $d \geq 3$ be an integer. For all but finitely many primes $\ell$ not dividing $2d$, for any $[A] \in \mathcal{L}_{d,\ell}$, write 
    \begin{align*}
        \delta^+(q) &= \limsup_{\substack{n \to \infty \\ \gcd(n,d) = 1}} \frac{|\{L \in \mathfrak{S}_n: (\operatorname{Cl}_L)_\ell \cong_{\mathbb{Z}_\ell[\zeta_d]} A \}|}{|\mathfrak{S}_n|} \\
        \delta^-(q) &= \liminf_{\substack{n \to \infty \\ \gcd(n,d) = 1}} \frac{|\{L \in \mathfrak{S}_n: (\operatorname{Cl}_L)_\ell \cong_{\mathbb{Z}_\ell[\zeta_d]} A \}|}{|\mathfrak{S}_n|}. 
    \end{align*}
Then $\delta^+(q)$ and $\delta^-(q)$ converge, as $q \to \infty$ with $\ell \nmid 2dq(q-1)$ and $q \equiv 1 \pmod{d}$, to $\mu_{d,\ell}([A]) = \frac{C_{d,\ell}}{|\operatorname{Aut}_{\mathbb{Z}_\ell[\zeta_d]}(A)|}$.
\end{corollary}


\Cref{theorem:cohen_lenstra_distribution_moments} asserts that $\mu_{d,\ell}$ is characterized as the probability measure $\mu$ on $\mathcal{L}_{d,\ell}$ for which the expected number of surjections from a $\mu$-random module onto a fixed module $A_0$ equals $1$, cf. \cite[Lemma 8.2]{ellenberg_venkatesh_westerland} for a similar statement for finite abelian $\ell$-groups.

\begin{theorem} \label{theorem:cohen_lenstra_distribution_moments}
    Let $d \geq 2$ be an integer, let $\ell$ be a prime number, and let $\mu$ be a probability distribution on $\mathcal{L}_{d,\ell}$. Let $[A]$ be the random variable valued in $\mathcal{L}_{d,\ell}$ with probability distribution $\mu$. Then, $\mu = \mu_{d,\ell}$ if and only if
        \begin{equation} \label{eq:expected_value_of_surjections_for_cohen_lenstra}
            \mathbb{E}(|\operatorname{Sur}_{\mathbb{Z}_\ell[\zeta_d]}(A, A_0)|) = 1
        \end{equation}
    for every $[A_0] \in \mathcal{L}_{d,\ell}$. 
\end{theorem}

\begin{proof}
    \cite[Theorem 6.2]{wang_wood_miclm}, applied to the case where $A = \mathbb{Q}[\zeta_d]_{\zeta_d \neq 1}$, $S = \{\ell\}$, $\mathfrak{O} = \mathbb{Z}_{(\ell)}[\zeta_d]$ (here, $\mathbb{Z}_{(\ell)}$ is the localization of $\mathbb{Z}$ at the prime ideal $(\ell)$), $\underline{u} = (0,\ldots,0)$ according to the notations in loc.~cit.~, establishes that \eqref{eq:expected_value_of_surjections_for_cohen_lenstra} holds when $\mu = \mu_{d,\ell}$. Additionally letting $X_n = [A]$ in \cite[Theorem 6.11]{wang_wood_miclm} shows that \eqref{eq:expected_value_of_surjections_for_cohen_lenstra} implies $\mu=\mu_{d,\ell}$.
\end{proof}

\begin{lemma}[{cf. \cite[Lemma 8.4]{ellenberg_venkatesh_westerland}}] \label{lemma:generalizing_EVW_lemma_8_4_to_equivariant_surjections} 
    Let $d \geq 2$ be an integer and let $\ell$ be a prime number not dividing $d$. Given $\epsilon > 0$ and $[A] \in \mathcal{L}_{d,\ell}$, there exists a constant $c(A, \epsilon)$ and a finite subset $M \subset \mathcal{L}_{d,\ell}$ so that, whenever $|X| > c(A, \epsilon)$
    $$|\operatorname{Sur}_{\mathbb{Z}_\ell[\zeta_d]}(X,A)| \leq \epsilon \frac{\sum_{A' \in M} |\operatorname{Sur}_{\mathbb{Z}_\ell[\zeta_d]}(X,A')|}{|M|}.$$ 
\end{lemma}

\begin{proof}
    Define a \emph{basic enlargement} of $A$ to be an $A'$ for $[A'] \in \mathcal{L}_{d, \ell}$ that admits a $\mathbb{Z}_\ell[\zeta_d]$-surjection $A' \twoheadrightarrow A$ with kernel isomorphic to the residue field $\kappa(B_i)$ of a factor $B_i$ of $\mathbb{Z}_\ell[\zeta_d]$ as in \eqref{eq:Z_ell_zeta_d_factorization_as_ring}. We show for any $[X] \in \mathcal{L}_{d,\ell}$ with a $\mathbb{Z}_\ell[\zeta_d]$-surjection onto $A$ such that $|X| > |A|$ and for any quotient $A'$ of $X$ that is a basic enlargement of $A$ that
        $$|\operatorname{Sur}_{\mathbb{Z}_{\ell}[\zeta_d]}(X, A') | \geq (\ell^s-1) |\operatorname{Sur}_{\mathbb{Z}_{\ell}[\zeta_d]}(X, A)|$$
    where $\ell^s = \frac{|A'|}{|A|} = |\kappa(B_i)|$; in particular, $s \leq d-1$. Given such an $X$ and an $A'$, fix surjections $\pi: A' \twoheadrightarrow A$ and $f: X \twoheadrightarrow A$.
    It suffices to prove the inequality for each basic enlargement separately, since a general enlargement is a tower of basic ones. Fix a basic enlargement $\pi \colon A' \twoheadrightarrow A$ with kernel $\kappa(B_i)$, supported at the factor $B_i$. For each $j$, write $A_j$ for the $B_j$-component of $A$.
    \begin{enumerate}

        \item Suppose that $A'$ is not isomorphic to $A \oplus \kappa(B_i)$ for any factor $B_i$ of $\mathbb{Z}_\ell[\zeta_d]$ as in \eqref{eq:Z_ell_zeta_d_factorization_as_ring}. In this case, there are an integer
        $a \ge 1$ and a finite-length $B_i$-module $\widetilde{A}_i$ such that
        $$A_i \cong B_i/\lambda^{a} \oplus \widetilde{A}_i, \qquad A'_i \cong B_i/\lambda^{a+1} \oplus \widetilde{A}_i .$$
        Since $\ker(\pi) \cong \kappa(B_i)$ has length one, we have $|A_i'| = q\,|A_i|$ and $\mu(A_i') \in \{\mu(A_i),\, \mu(A_i)+1\}$ by Nakayama's lemma; the Smith-normal-form criterion for surjections between finite-length modules over a PID then forces the type of $A_i'$ to be obtained from that of $A_i$ either by increasing exactly one part by one, or by adjoining a new part equal to one, the latter alternative being precisely the split case $A_i' \cong A_i \oplus \kappa(B_i)$ excluded in this case.
        For $j \ne i$ we have $A'_j \cong A_j$, because
        $\ker(\pi) \cong \kappa(B_i)$ is supported at the factor $B_i$ and
        $\Ext^1_{R}(A_j, \kappa(B_i)) = 0$ for $j \ne i$. The idempotent decomposition
        $X = \bigoplus_{j} X_j$ then gives
        \begin{align*}
        |\operatorname{Sur}_R(X, A')| &= |\operatorname{Sur}_{B_i}(X_i, A'_i)| \prod_{j \ne i} |\operatorname{Sur}_{B_j}(X_j, A_j)|, \\
        |\operatorname{Sur}_R(X, A)| &= |\operatorname{Sur}_{B_i}(X_i, A_i)| \prod_{j \ne i} |\operatorname{Sur}_{B_j}(X_j, A_j)|,
        \end{align*}
        since surjections over the product ring $R = \prod_j B_j$ are precisely tuples of surjections of the factors. If one of the common factors $|\operatorname{Sur}_{B_j}(X_j, A_j)|$ for $j \ne i$ vanishes, then both $|\operatorname{Sur}_R(X, A')|$ and $|\operatorname{Sur}_R(X, A)|$ vanish, so the desired inequality holds. Otherwise, it suffices to show that
        $|\operatorname{Sur}_{B_i}(X_i, A'_i)| \ge (\ell^s - 1)\, |\operatorname{Sur}_{B_i}(X_i, A_i)|$.
        Because $A'$ was chosen to be a quotient of $X$, the projection of a surjection
        $X \twoheadrightarrow A'$ onto the $i$-th factor exhibits $A'_i$ as a quotient of
        $X_i$, i.e. $\operatorname{Sur}_{B_i}(X_i, A'_i) \ne \emptyset$. Applying
        \Cref{lem:single_bump_comparison} at the discrete valuation ring $B_i$, whose
        residue field $\kappa(B_i)$ has cardinality $\ell^s$, yields the claim.

        \item If $A'$ is isomorphic to $A \oplus \kappa(B_i)$ for some $B_i$, then the $B_i$-component of $\ker(f)$ cannot be contained in $\lambda_i X_i$ where $\lambda_i$ is a uniformizer of $B_i$ and where $X_i$ is the $B_i$-component of $X$; otherwise $f$ would induce an isomorphism $X_i / \lambda_i X_i \to A_i / \lambda_i A_i$ where $A_i$ is the $B_i$-component of $A$, but then the $\kappa(B_i)$-ranks of $X_i$ and $A_i$ would coincide, so $\operatorname{Sur}_{\mathbb{Z}_\ell[\zeta_d]}(X, A')$ would be empty, contradicting the choice of $A'$. 
        
        Thus, there exist homomorphisms $\varphi: X \to \kappa(B_i)$ that are nontrivial on $\ker(f)$ and each such $\varphi$ induces a surjection $\tilde{f} = (f, \varphi): X \to A'$ that lifts $f$. There are at least $(\ell^s-1)$ choices for $\varphi$, so there are at least $\ell^s-1$ surjective lifts $\tilde{f}$ of $f$.
        Indeed, every $\varphi \colon X \to \kappa(B_i)$ factors through $X_i/\lambda_i X_i$, so there are $\ell^{s \cdot d_X}$ in total, where $d_X := \dim_{\kappa(B_i)} X_i/\lambda_i X_i$; those trivial on $\ker(f)$ are exactly the pullbacks of $\operatorname{Hom}(A, \kappa(B_i))$, of which there are $\ell^{s \cdot d_A}$ with $d_A := \dim_{\kappa(B_i)} A_i/\lambda_i A_i$. Since $\ker(f) \not\subseteq \lambda_i X_i$ was shown above, the reduction $X_i/\lambda_i X_i \twoheadrightarrow A_i/\lambda_i A_i$ is not injective, so $d_X > d_A$, and the number of $\varphi$ nontrivial on $\ker(f)$ is $\ell^{s \cdot d_X} - \ell^{s \cdot d_A} \ge \ell^s - 1$.
        Note that in this split case every $f \in \operatorname{Sur}_{\mathbb{Z}_\ell[\zeta_d]}(X, A)$ admits at least one lift to $A'$, namely the map $(f, 0)$ under the identification $A' \cong A \oplus \kappa(B_i)$, so a count lift by lift is valid here.

    \end{enumerate}

    Now let an $s$-enlargement of $A$ refer to any $A'$ with $[A'] \in \mathcal{L}_{d,\ell}$ that admits a $\mathbb{Z}_\ell[\zeta_d]$-surjection $A' \twoheadrightarrow A$ with kernel of size $\ell^s$. Given $\epsilon > 0$ and $[A] \in \mathcal{L}_{d, \ell}$, let $c(A, \epsilon) = \ell^{k(d-1)} |A|$ and let $M \subset \mathcal{L}_{d, \ell}$ be the set of isomorphism classes of $s$-enlargements of $A$ for $k (d-1) < s \leq (k+1) (d-1)$ where $k$ is chosen large enough so that
        $$\frac{\sum_{s=k(d-1) + 1}^{(k+1)(d-1)} p(s + \log_\ell |A|)^{d-1}}{(\ell-1)^k} < \epsilon$$
    where $p(n)$ is the partition function --- obtaining such a $k$ is possible because $d$ and $A$ are fixed and $p(n) \sim \frac{1}{4n\sqrt{3}} \exp\left( \pi \sqrt{\frac{2n}{3}} \right)$, which was first obtained by Hardy and Ramanujan \cite[Section 1.41]{hardy_ramanujan_afca} and independently obtained by Uspensky \cite{uspensky_afnfwotp} (see \cite[Section 5.1.2]{andrews_tp} for yet another discussion on this asymptotic formula). 

    Note that the number of isomorphism classes of $s$-enlargements of $A$ is bounded above by the number of $\mathbb{Z}_\ell[\zeta_d]$-modules of cardinality $\ell^{s} \cdot |A|$. In turn, this number of isomorphism classes is bounded above by $\prod_{i} \# \{B_i\text{-modules with cardinality } \ell^{s} \cdot |A|\}$. For each $i$, the number of $B_i$-modules with cardinality $\ell^{s} \cdot |A|$ is bounded above by $p(\log_\ell(\ell^{s} \cdot |A|)) = p(s + \log_\ell |A|)$ because $B_i$-modules of finite cardinality are of the form $\bigoplus_{j=1}^{k_i} B_i / (\ell^{d_{i,j}})$ and the isomorphism class of such a direct sum is determined by the multiset $\{d_{i,j}\}_j$. Since there are at most $(d-1)$ many $i$'s, we have  
        $$|M| = \sum_{s=k(d-1)+1}^{(k+1)(d-1)} \#\{s\text{-enlargements of } A\} \leq \sum_{s=k(d-1)+1}^{(k+1)(d-1)} p(s + \log_\ell|A|)^{d-1}$$
    so the choice of $k$ yields
        $$\frac{|M|}{(\ell-1)^k} < \epsilon.$$

    For any $[X] \in \mathcal{L}_{d, \ell}$ such that $|X| > c(A, \epsilon)$, one can iteratively obtain basic enlargements to obtain some $s$-enlargement $A'$ of $A$ with $k(d-1) < s \leq (k+1)(d-1)$ such that
        $$|\operatorname{Sur}_{\mathbb{Z}_\ell[\zeta_d]}(X,A')| \geq (\ell-1)^s |\operatorname{Sur}_{\mathbb{Z}_\ell[\zeta_d]}(X,A)| \geq (\ell-1)^k |\operatorname{Sur}_{\mathbb{Z}_\ell[\zeta_d]}(X,A)|.$$
    Equivalently,
        $$|\operatorname{Sur}_{\mathbb{Z}_\ell[\zeta_d]}(X,A)| \leq \frac{|\operatorname{Sur}_{\mathbb{Z}_\ell[\zeta_d]}(X,A')|}{(\ell-1)^k}.$$
    In turn, the RHS above is bounded above by
        \begin{align*}
            \frac{\sum_{A' \in M}|\operatorname{Sur}_{\mathbb{Z}_\ell[\zeta_d]}(X,A')|}{(\ell-1)^k} &= \frac{|M|}{(\ell-1)^k} \cdot \frac{\sum_{A' \in M}|\operatorname{Sur}_{\mathbb{Z}_\ell[\zeta_d]}(X,A')|}{|M|} \\
            &< \epsilon \cdot \frac{\sum_{A' \in M}|\operatorname{Sur}_{\mathbb{Z}_\ell[\zeta_d]}(X,A')|}{|M|}.
        \end{align*}
\end{proof}

\begin{lemma}[Matrix reduction for homomorphisms between decomposed modules]\label{lem:reduction_matrix_counting}
Let $B$ be a DVR with uniformizer $\lambda$ and finite residue field $\kappa$ of cardinality $q$. Let $X = \bigoplus_{m=1}^{M} B/\lambda^{c_m}$ and $T = \bigoplus_{n=1}^{N} B/\lambda^{t_n}$ be finite-length $B$-modules, and write a homomorphism $f \colon X \to T$ as a matrix $(u_{mn})$ with respect to the chosen cyclic generators, where $u_{mn} \in \lambda^{(t_n - c_m)_+} B/\lambda^{t_n}$ and where $(t_n - c_m)_+ = \max(0, t_n-c_m)$. Define $\rho(f) \in \operatorname{Mat}_{M \times N}(\kappa)$ by letting $\rho(f)_{mn}$ be the image of $u_{mn}$ in $B/\lambda = \kappa$; in particular, whenever $c_m < t_n$, we have $\rho(f)_{mn} := 0$. The following hold: 
\begin{enumerate}
    \item[(i)] $f$ is surjective if and only if $\rho(f)$ has full column rank $N$.

    \item[(ii)] Every matrix $\rho \in \operatorname{Mat}_{M \times N}(\kappa)$ supported on the indices $\{ (m, n) : c_m \ge t_n \}$ is the reduction of exactly $\prod_{m,n} q^{\min(c_m, t_n) - \mathbf{1}_{c_m \ge t_n}}$ homomorphisms $f \colon X \to T$.
\end{enumerate}
In particular, the number of surjections from $X$ to $T$ equals the number of elements of $\operatorname{Mat}_{M \times N}(\kappa)$ of full column rank $N$ supported on $\{ (m, n) : c_m \ge t_n \}$, multiplied by the common weight $\prod_{m,n} q^{\min(c_m, t_n) - \mathbf{1}_{c_m \ge t_n}}$.
\end{lemma}
\begin{proof}
For (i), note that $f$ is surjective if and only if the induced map $\bar{f} \colon X/\lambda X \to T/\lambda T$ is surjective by Nakayama's lemma applied over the local ring $B$. Observe that $\bar{f}$ is exactly the matrix $(u_{mn} \bmod \lambda)_{m,n}$, whose $m$-th row vanishes automatically when $c_m < t_n$ because then every homomorphism $B/\lambda^{c_m} \to B/\lambda^{t_n}$ sends the $m$-th generator to an element killed by $\lambda^{c_m}$, hence to $\lambda^{t_n - c_m}(B/\lambda^{t_n}) \subseteq \lambda(B/\lambda^{t_n})$, so its reduction mod $\lambda$ vanishes.

For (ii), note that for each pair $(m, n)$ the admissible set $\lambda^{(t_n - c_m)_+} B/\lambda^{t_n}$ has cardinality $q^{\min(c_m, t_n)}$, and the fibers of reduction to $\kappa$ all have cardinality $q$ when $c_m \ge t_n$ and cardinality $q^{\min(c_m, t_n)}$ when $c_m < t_n$. Thus, (ii) indeed holds.
\end{proof}

Let $\kappa$ be a field. Given a matrix $\beta \in \operatorname{Mat}_{m \times n}(\kappa)$, let $\operatorname{colsp}(\beta) \subseteq \kappa^m$ denote its column space. Moreover, given a positive integer $M$ and a subset $E \subseteq \{1,\ldots,M\}$, let $\kappa^E$ denote the coordinate subspace of vectors in $\kappa^M$ supported on $E$.

\begin{lemma}[Column spaces as orthogonal complements of row spans]\label{lem:column_space_row_corank}
Let $\kappa$ be a field, let $\beta \in \operatorname{Mat}_{M \times n}(\kappa)$ have full column rank $n$, and let $r_1, \dots, r_M \in \kappa^n$ denote the rows of $\beta$. For any $E \subseteq \{1,\ldots,M\}$, we have
$$ \dim_\kappa \bigl( \operatorname{colsp}(\beta) \cap \kappa^{E} \bigr) \;=\; n - \dim_\kappa \bigl\langle r_m : m \notin E \bigr\rangle. $$
\end{lemma}
\begin{proof}
The map $v \mapsto \beta v$ is an isomorphism of $\kappa^n$ onto $\operatorname{colsp}(\beta)$ because $\beta$ has full column rank. The vector $\beta v$ lies in $\kappa^{E}$ if and only if $r_m \cdot v = 0$ for every $m \notin E$, so the preimage of $\operatorname{colsp}(\beta) \cap \kappa^{E}$ under this isomorphism is the orthogonal complement of $\langle r_m : m \notin E \rangle$. Thus, the desired equality holds.
\end{proof}

\begin{lemma}[Stacked-matrix inequality]\label{lem:stacked_matrix_inequality}
Let $\kappa$ be a finite field of cardinality $q$, and let $h \ge 0$, $r \ge 1$, $g_2 \ge 0$, and $m_0 \ge 0$ be integers with $m_0 < r + g_2$. Set $g' := r + h$, $g := r + h + g_2$, and $s := m_0 + h$. For a matrix $M_2 \in \operatorname{Mat}_{g_2 \times m_0}(\kappa)$, let
\begin{align*}
    N(M_2) &:= \#\bigl\{ M_1 \in \operatorname{Mat}_{r \times m_0}(\kappa) : \textstyle \left[\begin{smallmatrix} M_2 \\M_1 \end{smallmatrix}\right] \text{ has full column rank, i.e. has column rank } m_0 \bigr\}, \\
    N &:= \sum_{M_2 \in \operatorname{Mat}_{g_2 \times m_0}(\kappa)} N(M_2).
\end{align*}
Then
$$ q^{g} \sum_{M_2 \in \operatorname{Mat}_{g_2 \times m_0}(\kappa)} N(M_2) \Bigl( q^{g'} - q^{\,s - \operatorname{rk} M_2} \Bigr) \;\ge\; (q - 1)\,\bigl( q^{g} - q^{s} \bigr)\, N. $$
\end{lemma}

\begin{proof}
The number $N$ counts the matrices $\left[\begin{smallmatrix} M_2 \\ M_1 \end{smallmatrix}\right] \in \operatorname{Mat}_{(r + g_2) \times m_0}(\kappa)$ of full column rank, so
$$ N \;=\; \prod_{j=0}^{m_0 - 1} \bigl( q^{\,r + g_2} - q^{j} \bigr). $$
Letting 
$$Z \;:=\; (q^{r} - 1) \prod_{j=1}^{m_0 - 1} \bigl( q^{\,r + g_2} - q^{j} \bigr), $$
we show that 
$$ \sum_{M_2} N(M_2)\, q^{\,m_0 - \operatorname{rk} M_2} \;=\; N + (q^{m_0} - 1)\, Z$$
by showing that both sides count triples $(M_1, M_2, x)$ where $\left[\begin{smallmatrix} M_2 \\ M_1 \end{smallmatrix}\right]$ is of full column rank and $x \in \ker M_2$. When one counts by the stack first, the kernel of $M_2$ contributes $q^{m_0 - \operatorname{rk} M_2}$ vectors $x$. When one counts by the vectors $x$ first, the case $x = 0$ contributes $N$ triples because $0$ lies in $\ker M_2$ for any $\left[\begin{smallmatrix} M_2 \\ M_1 \end{smallmatrix}\right]$ with full rank. In the case that $x \ne 0$, we may apply a change of basis of $\kappa^{m_0}$ to replace $x$ with $e_1$ while preserving full column rank, so it suffices to count stacks with $M_2 e_1 = 0$; the first column of such a stack is then a nonzero vector of the $M_1$-block, which can be chosen in $q^{r} - 1$ ways, and the remaining $m_0 - 1$ columns then complete it to an independent set in $\prod_{j=1}^{m_0 - 1} (q^{\,r + g_2} - q^{j})$ ways, which proves the identity. 
Note that $Z / N = (q^{r} - 1) / (q^{\,r + g_2} - 1)$.

Since $s = m_0 + h$, the identity implies that
$$ \sum_{M_2} N(M_2)\, q^{\,s - \operatorname{rk} M_2} \;=\; q^{h} \bigl( N + (q^{m_0} - 1)\, Z \bigr), $$
so the left-hand side of the claimed inequality equals
\begin{align*}
q^{g} \sum_{M_2 \in \operatorname{Mat}_{g_2 \times m_0}(\kappa)} N(M_2) \Bigl( q^{g'} - q^{\,s - \operatorname{rk} M_2} \Bigr) &= q^g \left( q^{g'} \sum_{M_2} N(M_2) - \sum_{M_2} N(M_2) q^{s - \operatorname{rk} M_2} \right)  \\
&= q^g \left( q^{g'} N - \sum_{M_2} N(M_2) q^{s - \operatorname{rk} M_2} \right)  \\
&= q^{g} \Bigl[ N q^{\,r + h} - q^{h} N - q^{h} (q^{m_0} - 1)\, Z \Bigr] \\
&= q^{\,g + h}\, N\, (q^{r} - 1)\, \frac{q^{\,r + g_2} - q^{m_0}}{q^{\,r + g_2} - 1}, 
\end{align*}
where we used $g' = r + h$ and the value of $Z / N$ from the previous paragraph. Since $s = m_0 + h$, we have $q^{\,r + g_2} - q^{m_0} = q^{-h} (q^{g} - q^{s})$. Thus, cancelling the common positive factor $q^{g} N \cdot q^{-h} (q^{g} - q^{s})$, which is positive because $m_0 < r + g_2$ implies $s < g$, the claimed inequality is equivalent to
$$ \frac{q^{r} - 1}{q^{\,r + g_2} - 1} \;\ge\; \frac{q - 1}{q^{g}}, \qquad\text{i.e.}\qquad q^{g}\,\bigl(q^{r} - 1\bigr) \;\ge\; (q - 1)\,\bigl(q^{\,g - h} - 1\bigr). $$
This holds because $r \ge 1$ implies $q^{g}(q^{r} - 1) \ge (q - 1)\, q^{g} \ge (q - 1)\, q^{\,g - h} > (q - 1)(q^{\,g - h} - 1)$.
\end{proof}

\begin{lemma}\label{lem:single_bump_comparison}
Let $B$ be a DVR with uniformizer $\lambda$ and finite residue field $\kappa$ of cardinality $q$. Let
$A = B/\lambda^{a} \oplus \tilde A$ and $A' = B/\lambda^{a+1} \oplus \tilde A$
for some $a \ge 1$ and some finite-length $B$-module $\tilde A$, and let
$\pi \colon A' \twoheadrightarrow A$ be the natural projection. Let $X$ be a finite length $B$-module. If $\operatorname{Sur}_B(X, A')$ is nonempty, then 
$$ |\operatorname{Sur}_B(X, A')| \;\ge\; (q-1)\,\bigl|\operatorname{Sur}_B(X, A)\bigr|. $$
\end{lemma}

\begin{proof}
Write
$$ X \cong \bigoplus_{m=1}^{M} B/\lambda^{c_m}, \qquad \widetilde{A} \cong \bigoplus_{j=1}^{\widetilde{N}} B/\lambda^{b_j} $$
for integers $c_1, \dots, c_M, b_1, \dots, b_{\widetilde{N}} \ge 1$, and set
$$ G_1 := \{ m : c_m \ge a + 1 \}, \qquad G_2 := \{ m : c_m = a \}, \qquad D_a := G_1 \sqcup G_2 = \{ m : c_m \ge a \}, $$
with $g' := |G_1|$, $g_2 := |G_2|$, and $g := |D_a| = g' + g_2$. Call a summand $B/\lambda^{b_j}$ of $\widetilde{A}$ \emph{deep} if $b_j > a$, and let $h$ be the number of deep summands.

To compute the number of surjections, we apply \Cref{lem:reduction_matrix_counting} to $T = A$ and $T = A'$, whose respective depth lists are $(a, b_1, \dots, b_{\widetilde{N}})$ and $(a + 1, b_1, \dots, b_{\widetilde{N}})$. 
Writing a reduced matrix as $\rho = [\rho_0 \mid \beta]$ --- where $\beta$ is the $M \times \widetilde{N}$ block corresponding to $\widetilde{A}$ and $\rho_0$ is the bump column --- the set of indices supporting $\beta$ is $\{ (m, j) : c_m \ge b_j \}$ (which is identical for both $A$ and $A'$), while $\rho_0$ is supported on $D_a$ for $A$ and on $G_1$ for $A'$. 
By \Cref{lem:reduction_matrix_counting}, the number of surjections from $X$ to $A$ equals the number of elements of $\operatorname{Mat}_{M \times (\widetilde{N}+1)}$ of full column rank times 
$$W_{\text{wt}} \coloneq \prod_{1 \leq m \leq M, 1 \leq n \leq \widetilde{N}} q^{\min(c_m, b_n) - 1_{c_m \geq b_n}} \cdot \prod_{1 \leq m \leq M} q^{\min(c_m, a) - 1_{c_m \geq a}}.$$
Similarly, the number of surjections from $X$ to $A'$ equals the number of elements of $\operatorname{Mat}_{M \times (\widetilde{N}+1)}$ of full column rank times
$$\prod_{1 \leq m \leq M, 1 \leq n \leq \widetilde{N}} q^{\min(c_m, b_n) - 1_{c_m \geq b_n}} \cdot \prod_{1 \leq m \leq M} q^{\min(c_m, a+1) - 1_{c_m \geq a+1}}.$$
In particular, the only difference between these weights comes from the bump-column weights: for a row with $c_m \ge a+1$ the ratio is $q^{a}/q^{a-1} = q$; for a row with $c_m = a$ the cell falls out of the support for $A'$, so $\rho_{0m}$ is forced to vanish and contributes its full fiber size $q^{a}$ against $q^{a-1}$, again a ratio of $q$; for $c_m < a$ the ratio is $1$. Hence, $\frac{|\operatorname{Sur}_B(X,A')|}{|\operatorname{Sur}_B(X,A)|} = q^{g' + g_2} = q^{g}$.


For a matrix $\beta$ of full column rank $\widetilde{N}$, let 
$$ d(\beta) := \dim_\kappa \bigl( \operatorname{colsp}(\beta) \cap \kappa^{D_a} \bigr), \qquad d'(\beta) := \dim_\kappa \bigl( \operatorname{colsp}(\beta) \cap \kappa^{G_1} \bigr). $$
The columns of $[\rho_0 \mid \beta]$ are linearly independent if and only if $\beta$ has full column rank and $\rho_0 \notin \operatorname{colsp}(\beta)$, and a vector supported on a row set $E$ lies in $\operatorname{colsp}(\beta)$ if and only if it lies in $\operatorname{colsp}(\beta) \cap \kappa^{E}$. Writing $\mathcal{M}_{\text{full}}$ for the set of matrices $\beta \in \operatorname{Mat}_{M \times \widetilde{N}}(\kappa)$ of full column rank supported on $\{ (m, j) : c_m \ge b_j \}$, we thus have
\begin{align*}
    \bigl|\operatorname{Sur}_B(X, A)\bigr| = W_{\mathrm{wt}} \cdot \sum_{\beta \in \mathcal{M}_{\text{full}}} \bigl( q^{g} - q^{d(\beta)} \bigr), \qquad \bigl|\operatorname{Sur}_B(X, A')\bigr| = W_{\mathrm{wt}}\, \cdot q^{g} \cdot \sum_{\beta \in \mathcal{M}_{\text{full}}} \bigl( q^{g'} - q^{d'(\beta)} \bigr).
\end{align*}
The desired inequality is therefore equivalent to
\[ q^{g} \sum_{\beta \in \mathcal{M}_{\text{full}}} \bigl( q^{g'} - q^{d'(\beta)} \bigr) \;\ge\; (q - 1) \sum_{\beta \in \mathcal{M}_{\text{full}}} \bigl( q^{g} - q^{d(\beta)} \bigr). \tag{$\star$} \]
By assumption, there exists a surjection $X \twoheadrightarrow A'$; the reduced matrix $[\rho_0 \mid \beta]$ of such a surjection has full column rank by \Cref{lem:reduction_matrix_counting}. The bump column $\rho_0$ is supported on $G_1$, and each deep column of $\beta$ is supported on $\{ m : c_m \ge b_j \} \subseteq G_1$; these $h + 1$ columns are independent vectors of the $g'$-dimensional space $\kappa^{G_1}$, so
\[ r \;:=\; g' - h \;\ge\; 1. \tag{$\star_1$} \]

We next identify $d(\beta)$ and $d'(\beta)$ as coranks in $\kappa^{\widetilde{N}}$ of row spans, using \Cref{lem:column_space_row_corank}. Given $\beta \in \operatorname{Mat}_{M \times \widetilde{N}}$, let $\beta^{(1)}, \dots, \beta^{(M)} \in \kappa^{\widetilde{N}}$ denote the rows of $\beta$. When $\beta \in \mathcal{M}_{\text{full}}$, set
\begin{align*}
U &:= \bigl\langle e_j : b_j \le a \bigr\rangle \;\subseteq\; \kappa^{\widetilde{N}}, \\
S(\beta) &:= \bigl\langle \beta^{(m)} : c_m < a \bigr\rangle \subseteq U, \\
K(\beta) &:= S(\beta) + \bigl\langle \beta^{(m)} : c_m = a \bigr\rangle \;\subseteq\;  U;
\end{align*}
note that the inclusions hold because the cell $(m, j)$ of $\beta$ vanishes whenever $c_m < b_j$. Applying \Cref{lem:column_space_row_corank} with $E = D_a$ and with $E = G_1$, we have
$$ d(\beta) = \operatorname{corank} S(\beta), \qquad d'(\beta) = \operatorname{corank} K(\beta), $$
and in particular $s(\beta) := \operatorname{corank} S(\beta) \ge \operatorname{codim} U = h$.

Recall that the sums in $(\star)$ range over $\beta \in \mathcal{M}_{\text{full}}$. We partition these matrices by their lower-row subspace $S(\beta)$; fixing one such subspace $S \subseteq U$, we write $s := \operatorname{corank} S$. Write 
$$\mathcal{M}_{\text{full}}(S) = \{\beta \in \mathcal{M}_{\text{full}}: S(\beta) = S \}.$$ 
On this stratum, $d(\beta) = s$ is constant, so its contribution to the right-hand side of $(\star)$ equals $(q - 1)(q^g - q^s)$ times the number of matrices in the stratum. In the case that $s \geq g$, this contribution is at most zero, whereas the contribution to the left-hand side is nonnegative, because 
$$d'(\beta) = \dim_{\kappa}(\operatorname{colsp}(\beta) \cap \kappa^{G_1}) \leq \dim_{\kappa}(\kappa^{G_1}) = g',$$
and the stratum satisfies $(\star)$ on its own. 

Assume now that $s < g$, and let $m_0 := s - h = \dim U/S$.
Let $W = \langle e_j: b_j > a \rangle \cong \kappa^h$ so that $\kappa^{\widetilde{N}} = U \oplus W$; given $\beta \in \operatorname{Mat}_{M \times \widetilde{N}}(\kappa)$, write $\beta^{(m)} = (u_m, z_m) \in U \oplus W$. 
Write $\mathcal{D}(S)$ for the set of matrices $\beta \in \operatorname{Mat}_{M \times \widetilde{N}}(\kappa)$ supported on $\{ (m, j) : c_m \ge b_j \}$ and satisfying the following conditions:
\begin{itemize}
    \item $\langle \beta^{(m)} : c_m < a \rangle = S$, i.e. $S(\beta) = S$,
    \item for any $m$ such that $c_m = a$, we have $\beta^{(m)} \in U$, and
    \item For $m$ such that $c_m \geq a+1$, we have $z_m \in \langle e_j : b_j > a \text{ and } b_j \le c_m \rangle$.
\end{itemize}
Recall that $U$ is generated by the $e_j$ such that $b_j \leq a$, 
and elements of $\mathcal{D}(S)$ are supported on the index set $\{ (m, j) : c_m \ge b_j \}$ by definition.
For any $m$ and $j$ such that $c_m = a$ and $b_j > a$, we have $c_m < b_j$. In this case, $z_m = 0$ and hence $\beta^{(m)} \in U$. 
Conversely, the first bullet is exactly the condition $S(\beta) = S$, and we have just seen that bullets 2 and 3 hold automatically for every matrix supported on $\{ (m, j) : c_m \ge b_j \}$; hence $\mathcal{D}(S)$ is precisely the set of matrices supported on $\{ (m, j) : c_m \ge b_j \}$ with $S(\beta) = S$. Since $\mathcal{M}_{\text{full}}$ consists of the full-column-rank matrices with this support pattern, it follows from the two definitions that $\mathcal{M}_{\text{full}}(S) = \{\beta \in \mathcal{D}(S) : \beta \text{ has full column rank } \widetilde{N}\}$.

Call a tuple $(z_m)_{m \in G_1}$ with $z_m \in \langle e_j : b_j > a \text{ and } b_j \le c_m \rangle$ admissible if $\langle z_m : m \in G_1 \rangle = W$. Fix an arbitrary matrix $\beta \in \mathcal{D}(S)$. Every row of $\beta$ with $c_m \le a$ lies in $U$, so the image of $\operatorname{RowSpan}(\beta)$ in $W$ is spanned by $\{ z_m : m \in G_1 \}$; if $(z_m)_{m \in G_1}$ is not admissible, then the rows of $\beta$ cannot span $\kappa^{\widetilde N} = U \oplus W$, and $\beta$ does not have full column rank. Assume henceforth that $\langle z_m : m \in G_1 \rangle = W$. In particular, the map $\kappa^{G_1} \to W$, $\mu \mapsto \sum_{m \in G_1} \mu_m z_m$ is surjective.
Note that its kernel $R$ has dimension $g' - h = r$. Choosing a basis $\alpha^{(1)}, \dots, \alpha^{(r)}$ of $R$, let $w_i \coloneq \sum_{m \in G_1} \alpha^{(i)}_m u_m \in U$. 

We show that $\operatorname{RowSpan}(\beta) \cap U = K(\beta) + \langle w_1,\ldots, w_r \rangle$. A linear combination $\sum_m \lambda_m \beta^{(m)}$ of the rows of $\beta$ lies in $U$ if and only if its $G_1$-coefficient vector $(\lambda_m)_{m \in G_1}$ belongs to $R$, because the rows with $c_m \le a$ lie in $U$ and the $W$-component of the combination is $\sum_{m \in G_1} \lambda_m z_m$. The generators of $K(\beta)$ are rows of $\beta$ contained in $U$, and each $w_i$ is the $U$-component of $\sum_{m \in G_1} \alpha^{(i)}_m \beta^{(m)}$, whose $W$-component vanishes because $\alpha^{(i)} \in R$, so $K(\beta) + \langle w_1, \dots, w_r \rangle \subseteq \operatorname{RowSpan}(\beta) \cap U$. Conversely, every element of $\operatorname{RowSpan}(\beta) \cap U$ is a linear combination $\sum_{m \in G_1} \lambda_m \beta^{(m)}$ whose $G_1$-coefficient vector $\mu := (\lambda_m)_{m \in G_1}$ lies in $R$; writing $\mu = \sum_i \nu_i \alpha^{(i)}$, this element equals an element of $K(\beta)$ plus $\sum_i \nu_i w_i$. Thus, $\operatorname{RowSpan}(\beta) \cap U = K(\beta) + \langle w_1,\ldots, w_r \rangle$ as desired.

We show that $\beta$ has full column rank if and only if $K(\beta) + \langle w_1,\ldots,w_r \rangle = U$. Writing $T := \operatorname{RowSpan}(\beta)$, the projection $T \to W$ is surjective with kernel $T \cap U$, so $\dim_\kappa T = \dim_\kappa(T \cap U) + \dim_\kappa W = \dim_\kappa(T \cap U)+ h$. The matrix $\beta$ has full column rank if and only if $\dim_\kappa T = \widetilde N$, which by the equality $\operatorname{RowSpan}(\beta) \cap U = K(\beta) + \langle w_1, \dots, w_r \rangle$ established above holds if and only if $K(\beta) + \langle w_1, \dots, w_r \rangle = U$.

Fix once and for all a $\kappa$-basis of $U/S$, so that $U/S \cong \kappa^{m_0}$.
Passing to the quotient $U/S$, let $M_2$ be the $g_2 \times m_0$ matrix whose rows, indexed by $m \in G_2$, are the images of the middle rows $\beta^{(m)}$ in $U/S$, and let $M_1$ be the $r \times m_0$ matrix whose rows, indexed by $i = 1, \dots, r$, are the images of the $w_i$ in $U/S$.
Since $S \subseteq K(\beta)$, the image of $K(\beta) + \langle w_1, \dots, w_r \rangle$ in $U/S$ is the sum of the row span of $M_2$ and the row span of $M_1$; the equality $K(\beta) + \langle w_1, \dots, w_r \rangle = U$ therefore holds if and only if the rows of the stack $\left[\begin{smallmatrix} M_2 \\ M_1 \end{smallmatrix}\right]$ span $U/S$, i.e. if and only if the stack has full column rank $m_0$. 

Recall that a matrix $\beta \in \mathcal{D}(S)$ whose $z_m$ fail to span $W$ does not have full column rank. Moreover, we have so far established that if $(z_m)_{m \in G_1}$ is admissible, then $\beta$ has full column rank if and only if the stack $\left[\begin{smallmatrix} M_2 \\ M_1 \end{smallmatrix}\right]$ has full column rank $m_0$; in this case, 
the identity $d'(\beta) = \operatorname{corank} K(\beta)$ displayed above and the fact that the image of $K(\beta)$ in $U/S$ is spanned by the rows of $M_2$ imply $\dim_\kappa K(\beta) = \dim_\kappa S + \operatorname{rk} M_2$. Therefore, $d'(\beta) = \widetilde N - \dim_\kappa K(\beta) = s - \operatorname{rk} M_2$.


Recall the basis $\alpha^{(1)}, \dots, \alpha^{(r)}$ of $R$. There exist $v^{(1)}, \dots, v^{(r)} \in \kappa^{G_1}$ with $\sum_{m \in G_1} \alpha^{(i)}_m v^{(k)}_m = \delta_{ik}$ for $i,k \in \{1,\ldots,r\}$; given a target $(\widetilde w_1, \dots, \widetilde w_r) \in U^r$, the choice $u_m := \sum_{k=1}^r v^{(k)}_m \widetilde w_k$ gives $w_i = \widetilde w_i$ for every $i$, so the linear map
$$U^{G_1} = U^{g'} \to U^r, \quad (u_m)_{m \in G_1} \mapsto (w_1,\ldots,w_r)$$ 
is surjective. 

We parameterize $\mathcal{M}_{\text{full}}(S)$. The map 
$$U^{G_2} \to \operatorname{Mat}_{g_2 \times m_0}(\kappa), \quad (\beta^{(m)})_{m \in G_2} \mapsto M_2$$
is linear and surjective with fibers equal to cosets of $S^{G_2}$, 
so the fiber over each $M_2 \in \operatorname{Mat}_{g_2 \times m_0}(\kappa)$ has cardinality $|S|^{g_2}$.

For each $M_2 \in \operatorname{Mat}_{g_2 \times m_0}(\kappa)$, the number of choices of $(u_m)_{m \in G_1}$ for which the stack has full column rank equals $N(M_2)\, q^{(\widetilde{N} - h) g' - r m_0}$, where $N(M_2)$ is as in \Cref{lem:stacked_matrix_inequality}. Indeed, the map $\Phi \colon U^{G_1} \to \operatorname{Mat}_{r \times m_0}(\kappa)$, $(u_m)_{m \in G_1} \mapsto M_1$ is the composition of the surjective linear maps $U^{g'} \to U^r$ and $U^r \to (U/S)^r$. 
Every fiber of $\Phi$ is a coset of $\ker \Phi$, so all fibers have size $q^{(\dim U) g' - r m_0} = q^{(\widetilde{N} - h) g' - r m_0}$, using $\dim U = \widetilde{N} - h$; and the condition that the stack has full column rank depends only on $M_1$, being satisfiable for exactly $N(M_2)$ matrices $M_1$.


We now assemble these counts. Recall that $\beta \in \mathcal{D}(S)$ has full column rank if and only if $(z_m)_{m \in G_1}$ is admissible and the stack $\left[\begin{smallmatrix} M_2 \\ M_1 \end{smallmatrix}\right]$ has full column rank. These two conditions involve disjoint blocks of variables, so the elements of $\mathcal{M}_{\text{full}}(S)$ are parameterized by the following:
\begin{itemize}
    \item A tuple $(\beta^{(m)})_{m \in M: c_m < a }$ of vectors in $\kappa^{\widetilde{N}}$ that span $S$,
    \item An admissible tuple $(z_m)_{m \in G_1}$ of vectors $z_m \in \langle e_j : b_j > a \text{ and } b_j \leq c_m \rangle$,
    \item A matrix $M_2 \in \operatorname{Mat}_{g_2 \times m_0}(\kappa)$, and
    \item A tuple $(u_m)_{m \in G_1} \in U^{G_1}$ whose image under $\Phi$ is a matrix $M_1 \in \operatorname{Mat}_{r \times m_0}(\kappa)$ such that $\left[\begin{smallmatrix} M_2 \\ M_1 \end{smallmatrix}\right]$ has full column rank $m_0$.
\end{itemize}

Since $d(\beta) = s$ and $d'(\beta) = s - \operatorname{rk} M_2$ depend only on $S$ and $M_2$, the stratum $\mathcal{M}_{\text{full}}(S)$ contributes $q^g\, C \sum_{M_2} N(M_2) ( q^{g'} - q^{s - \operatorname{rk} M_2} )$ to the left-hand side of $(\star)$ and contributes $(q - 1)(q^g - q^s)\, C\, N$ to the right-hand side of $(\star)$, where 
$$C := \#\{(\beta^{(m)})_{m \in M: c_m < a }\text{ spanning } S\} \cdot \#\{\text{admissible tuples } (z_m)_{m \in G_1}\} \cdot |S|^{g_2} \cdot q^{(\widetilde{N} - h) g' - r m_0}.$$
Note that the first factor depends only on $S$ and the second only on the depth lists $(b_j)$ and $(c_m)$, so neither depends on $M_2$, on $(u_m)$, or on the rows where $c_m = a$. Cancelling the common factor $C$ and using $N = \sum_{M_2} N(M_2)$, the stratum's share of $(\star)$ becomes exactly the inequality of \Cref{lem:stacked_matrix_inequality},
applied to the parameters $h$, $r$, $g_2$, and $m_0$. 

We check that the hypotheses of \Cref{lem:stacked_matrix_inequality} hold. Indeed, the lemma's $g' := r + h$ equals our $g'$ by $(\star_1)$, its $g := r + h + g_2$ equals our $g = g' + g_2$, and its $s := m_0 + h$ equals our $s$ since $m_0 := s - h$; its hypotheses hold because $r \ge 1$ by $(\star_1)$, $m_0 < r + g_2$ since $m_0 = s - h < g - h = r + g_2$ using the case assumption $s < g$, and $m_0 \ge 0$ because $S \subseteq U$ implies $\operatorname{corank} S \ge \operatorname{codim} U = h$.

That inequality therefore holds, so every stratum satisfies its share of $(\star)$. Summing over all strata proves $(\star)$, and the discussion above converts $(\star)$ back into $\bigl|\operatorname{Sur}_B(X, A')\bigr| \ge (q - 1) \bigl|\operatorname{Sur}_B(X, A)\bigr|$, as claimed.
\end{proof}

Given $[A] \in \mathcal{L}_{d,\ell}$ and a probability measure $\nu$ on $\mathcal{L}_{d,\ell}$, write $\langle \operatorname{Sur}_{\mathbb{Z}_\ell[\zeta_d]}(- , A) \rangle_\nu$ for the expected number of surjections from a $\nu$-random group to $A$.

Given \Cref{theorem:cohen_lenstra_distribution_moments} and 
\Cref{lemma:generalizing_EVW_lemma_8_4_to_equivariant_surjections}, the proof of \Cref{proposition:generalizing_EVW_Proposition_8_3_to_equivariant_surjections} is virtually identical to that of \cite[Proposition 8.3]{ellenberg_venkatesh_westerland}.

\begin{proposition}[{cf. \cite[Proposition 8.3]{ellenberg_venkatesh_westerland}}] \label{proposition:generalizing_EVW_Proposition_8_3_to_equivariant_surjections}
    Let $d \geq 2$ and let $\ell$ be a prime number not dividing $d$. Let $\epsilon_0 > 0$ and $L \subset \mathcal{L}_{d,\ell}$ be a finite subset. There exists $\delta > 0$ and a finite subset $L' \subset \mathcal{L}_{d,\ell}$ such that, if $\nu$ is any probability measure on $\mathcal{L}_{d,\ell}$ for which $\langle \operatorname{Sur}_{\mathbb{Z}_\ell[\zeta_d]}(-, A) \rangle_\nu \in [1-\delta, 1 + \delta]$ for any $A \in L'$, then also $|\nu(A) - \mu_{d,\ell}(A)| \leq \epsilon_0$ for any $A \in L$.
\end{proposition}

\begin{proof}
    Let $L'_k := \{A \in \mathcal{L}_{d,\ell}: |A| \leq k \}$. Suppose for contradiction that the assertion is false. In particular, for each integer $k \geq 1$, there is some measure $\nu_k$ on $\mathcal{L}_{d,\ell}$ that ``does not work'' for $L' = L'_k$ and $\delta = \frac{1}{k}$, i.e.
        \begin{enumerate}
            \item $|\langle \operatorname{Sur}_{\mathbb{Z}_\ell[\zeta_d]}(-, A) \rangle_{\nu_k}- 1| \leq \frac{1}{k}$ for all $A \in L'_k$, and 
            \item $|\nu_k(A) - \mu_{d,\ell}(A)| > \epsilon_0$ for some $A \in L$.
        \end{enumerate}
    Passing to a weakly convergent subsequence (passing further if necessary, we may assume that the $A \in L$ of the second condition is the same for all $k$, since $L$ is finite), we obtain measures $\nu_k$ such that 
    $$ \lim_{k \to \infty} \langle \operatorname{Sur}_{\mathbb{Z}_\ell[\zeta_d]}(-, A) \rangle_{\nu_k} = 1 $$
    for every fixed $A \in \mathcal{L}_{d,\ell}$ and such that the $\nu_k$ weakly converge to a measure $\nu_\infty$ which does not equal $\mu_{d,\ell}$. 

    Fix an arbitrary $\epsilon > 0$, which is not related to $\epsilon_0$. We show that $\langle \operatorname{Sur}_{\mathbb{Z}_\ell[\zeta_d]}(-, A) \rangle_{\nu_\infty} = 1$. On the one hand, this expected value is at most $1$ by Fatou's lemma. On the other hand, with $c = c(A, \epsilon)$ and $M \subset \mathcal{L}_{d,\ell}$ as in \Cref{lemma:generalizing_EVW_lemma_8_4_to_equivariant_surjections},

        \begin{align}
        \langle \operatorname{Sur}_{\mathbb{Z}_\ell[\zeta_d]}(-, A) \rangle_{\nu_\infty} &= \sum_{|B| \leq c} \nu_\infty(B)|\operatorname{Sur}_{\mathbb{Z}_\ell[\zeta_d]}(B,A)| + \sum_{|B| > c} \nu_\infty(B)|\operatorname{Sur}_{\mathbb{Z}_\ell[\zeta_d]}(B,A)| \nonumber \\ 
        &\geq \sum_{|B| \leq c} \nu_\infty(B)|\operatorname{Sur}_{\mathbb{Z}_\ell[\zeta_d]}(B,A)| \nonumber \\
        &= \lim_k \sum_{|B| \leq c} \nu_k(B)|\operatorname{Sur}_{\mathbb{Z}_\ell[\zeta_d]}(B,A)| \nonumber \\
        &= 1 - \limsup_k \sum_{|B| > c} \nu_k(B)|\operatorname{Sur}_{\mathbb{Z}_\ell[\zeta_d]}(B,A)|. \label{eq:expected_number_for_nu_infty_is_at_most_1}
        \end{align}
    By \Cref{lemma:generalizing_EVW_lemma_8_4_to_equivariant_surjections}, 
        \begin{equation} \label{eq:bound_sum_of_expectations_for_large_B_above_with_sum_over_M}
        \sum_{|B| > c} \nu_k(B)|\operatorname{Sur}_{\mathbb{Z}_\ell[\zeta_d]}(B,A)| \leq \epsilon |M|^{-1} \sum_{|B| > c, A' \in M} \nu_k(B) |\operatorname{Sur}_{\mathbb{Z}_\ell[\zeta_d]}(B,A') | 
        \end{equation}
    Now, by assumption, for any $A' \in M$ and any $k > |A'|$,
        $$\sum_{|B| > c} \nu_k(B) |\operatorname{Sur}_{\mathbb{Z}_\ell[\zeta_d]}(B,A')| \leq \langle \operatorname{Sur}_{\mathbb{Z}_\ell[\zeta_d]}(-, A') \rangle_{\nu_k} \leq 1 + 1/k$$
    and using \eqref{eq:bound_sum_of_expectations_for_large_B_above_with_sum_over_M} and passing to the limit, we get
        $$\limsup_k \sum_{|B| > c} \nu_k(B) |\operatorname{Sur}(B,A)| \leq \epsilon.$$
    Thus by \eqref{eq:expected_number_for_nu_infty_is_at_most_1} and the above inequality, we get $\langle \operatorname{Sur}_{\mathbb{Z}_\ell[\zeta_d]}(- , A) \rangle_{\nu_\infty} \in [1-\epsilon, 1]$. Since $\epsilon$ is arbitrary,
         $$\langle \operatorname{Sur}_{\mathbb{Z}_\ell[\zeta_d]}(-,A) \rangle_{\nu_\infty} = 1.$$
    Applying this conclusion with $A$ trivial, we see that $\nu_\infty$ is a probability measure. \Cref{theorem:cohen_lenstra_distribution_moments} shows that $\nu_\infty = \mu_{d,\ell}$, which is a contradiction.
\end{proof}

Moreover, given \Cref{proposition:generalizing_EVW_Proposition_8_3_to_equivariant_surjections}, the proof of \Cref{theorem:generalizing_EVW_main_theorem_to_cubic_covers} is virtually identical to the proof of \cite[Theorem 1.2]{ellenberg_venkatesh_westerland} written between the statement and proof of \cite[Theorem 8.8]{ellenberg_venkatesh_westerland}. 

\begin{proof}[Proof of \Cref{theorem:generalizing_EVW_main_theorem_to_cubic_covers}]

    Let $[A_0]$ be any fixed element of $\mathcal{L}_{d,\ell}$ and let $\epsilon > 0$. Given $n$, let $\nu_n$ be the probability measure on $\mathcal{L}_{d, \ell}$ with $\nu_n(A) = \frac{\# \{L \in \mathfrak{S}_n: (\operatorname{Cl}_L)_\ell \cong_{\mathbb{Z}_\ell[\zeta_d]} A\} }{|\mathfrak{S}_n|}$. 
    
    Apply \Cref{proposition:generalizing_EVW_Proposition_8_3_to_equivariant_surjections} to $\epsilon_0 = \epsilon$, $L = \{A_0\}$, and $\nu = \nu_n$ to obtain a finite subset $L' \subset \mathcal{L}_{d,\ell}$ and $\delta > 0$ such that
        $$\left| \frac{\sum_{L \in \mathfrak{S}_n} m_A(L)}{|\mathfrak{S}_n|} - 1 \right| < \delta \text{ for all } A \in L' \Rightarrow |\nu_n(A_0) - \mu_{d,\ell}(A_0)| < \epsilon.$$
    where $m_A(L)$ is as defined in \Cref{subsection:big_monodromy_hypothesis}.
    Now let $B_{d,\ell,A}$, $C_{d,\ell,A}$, and $N_{d,\ell,A}$ be as in \Cref{theorem:generalizing_EVW_Theorem_8.8_to_cubic_covers}, and let $Q$ be so that $Q > 4 B_{d,\ell,A}^2$ and $C_{d,\ell,A} / \sqrt{Q} < \delta$ for every $A \in L'$. If $q > Q$ and $n \geq \max_{A \in L'} N_{d,\ell,A}$, then $\left| \frac{\sum_{L \in \mathfrak{S}_n} m_A(L)}{|\mathfrak{S}_n|} - 1 \right| \leq \frac{C_{d,\ell,A}}{\sqrt{q}} < \frac{C_{d,\ell,A}}{\sqrt{Q}} < \delta$ and hence $|\nu_n(A_0) - \mu_{d,\ell}(A_0)| < \epsilon$. Thus, for any $q > Q$, $\delta^+(q)$ and $\delta^-(q)$ are both bounded between $\mu_{d,\ell}([A_0]) - \epsilon$ and $\mu_{d,\ell}([A_0]) + \epsilon$. The result follows because $\epsilon$ is arbitrary.
\end{proof}

\section{Counting Rational Points on twists of Hurwitz schemes over \texorpdfstring{$\Fq$}{Fq}}
\label{section:generalizing_ELS_section_2}

Throughout this section, we continue to write $\mathbb{Z}_\ell[\zeta_d]$ for $\mathbb{Z}_\ell[\zeta_d]_{\zeta_d \neq 1}$ when $d \geq 2$ is an integer and $\ell \nmid d$ is a prime number. Similarly write $\mathbb{Z}/\ell\mathbb{Z}[\zeta_d]$ for $\mathbb{Z}/\ell\mathbb{Z}[\zeta_d]_{\zeta_d \neq 1}$.

In this section, we tweak the ideas from \cite[Section 2]{ellenberg_li_shusterman} to once again accommodate the $\zeta_d$-action as in \Cref{section:generalizing_EVW_section_8} and \Cref{section:cohen_lenstra_for_imaginary_Z_3Z_extensions} to prove \Cref{corollary:generalizing_ELS_corollary_2.6}. In turn, we use \Cref{corollary:generalizing_ELS_corollary_2.6} to prove \Cref{theorem:research_generalizing_ELS_main_theorem_to_trielliptic_curves}.

Let $d \geq 2$ and let $q \equiv 1 \pmod{d}$ be a prime power. Let $Q_{n,q}$ denote the set of squarefree polynomials over $\Fq$ of degree $n$. For each $f \in Q_{n,q}$, let $C_f$ be the (smooth completion of the) curve given by $y^d = f(t)$. Note that $C_f$ is a (tamely ramified) $\mathbb{Z}/d\mathbb{Z}$-cover of $\mathbb{P}^1_{\mathbb{F}_q}$ because $\Fq$ has a primitive $d$th root of unity. Let $J_f$ be the Jacobian of $C_f$. In particular, the action of $\zeta_d$ on $C_f$ induces an action on $J_f$. Let $\operatorname{Frob}_q$ denote the geometric Frobenius map on $C_f$ as well as the induced action on $J_f$. Note that the actions of $\zeta_d$ and $\operatorname{Frob}_q$ on $C_f$ commute and hence the induced actions on $J_f$ commute as well. Given a prime number $\ell$ such that $\ell \nmid dq$, and for $a \in \mathbb{Z}/\ell\mathbb{Z}$ and $i \in \mathbb{Z}/d\mathbb{Z}$, the elements $R$ of $J_f[\ell](\Fqbar)$ which satisfy 
    \begin{equation} \label{eq:frobenius_eigenvector}
        \operatorname{Frob}_q \cdot R = a \zeta_d^i R
    \end{equation}
form a finite $\mathbb{Z}/\ell\mathbb{Z}[\zeta_d]$-module. Let $m_{a,i}(f)$ denote the number of nonzero elements of this module. Additionally, let $Q_{n,q}^{a,i}$ denote the set of $f \in Q_{n,q}$ such that $m_{a,i}(f)$ is greater than $0$.

For the rest of this section, we let $A$ be a quotient ring of $\mathbb{Z}/\ell\mathbb{Z}[\zeta_d]$ ($A$ will often be an involution ring), let $G = A \rtimes \langle \zeta_d \rangle$ (where $\langle \zeta_d \rangle$ is a way to multiplicatively write the cyclic group $\mathbb{Z}/d\mathbb{Z}$), and let $c = \{(a,\zeta_d): a \in A\}$. By \Cref{lemma:c_is_nonsplitting_conjugacy_class}, $c$ is a non-splitting conjugacy class of $G$. Given an integer $n \geq 1$ divisible by $d$ or relatively prime to $d$, we let $\mathbf{X}_n$ be as in \Cref{definition:hurwitz_scheme_X_n} with respect to $G$ and $c$.

Let $k$ be an extension field of $\mathbb{F}_q$. As discussed in \Cref{section:generalizing_EVW_section_8}, $k$-points of $\mathbf{X}_n$ correspond to certain isomorphism classes of tame $G$-covers $f: C \to \mathbb{P}^1_k$. In turn, these isomorphism classes correspond to isomorphism classes of triples $(g, \phi, h)$, where $g: C \to D$ is an \'etale $A$-cover, $h: D \to \mathbb{P}^1_k$ is a tame $(\mathbb{Z}/d\mathbb{Z})$-cover, $f = h \circ g$, and $\phi: A \to \operatorname{Aut}(g)$ is an isomorphism. Given $a \in A^\times$, let $\langle a \rangle$ denote the automorphism of $\mathbf{X}_n$ sending $(g,\phi,h)$ to $(g,a\phi, h)$. Write $\Frob_{q,\text{ari}}$ for the arithmetic Frobenius on $\mathbf{X}_n/\Fqbar$; it commutes with $\langle a \rangle$. Thus, the homomorphism $\operatorname{Gal}(\overline{\mathbb{F}}_q / \mathbb{F}_q) \to \operatorname{Aut}(\mathbf{X}_n/\Fqbar)$ sending $\Frob_{q,\text{ari}}$ to $\langle a \rangle$ is a $1$-cocycle from $\operatorname{Gal}(\overline{\mathbb{F}}_q / \mathbb{F}_q)$ to $\operatorname{Aut}(\mathbf{X}_n/\Fqbar)$ and hence determines a twist $\mathbf{X}_n^a$ of $\mathbf{X}_n$ defined over $\mathbb{F}_q$ (see, for example,  \cite[Theorem 4.5.2]{poonen_rpv} for a statement).

\begin{lemma}[{cf. \cite[Lemma 2.3]{ellenberg_li_shusterman}}] \label{lemma:research_generalizing_ELS_lemma_2.3_to_trielliptic_curves}
    With notation as above, specializing to $A = \bbZ/\ell\bbZ$, and for any $i_0 \in \mathbb{Z}/d\mathbb{Z}$ and $a \in \mathbb{Z}/\ell\mathbb{Z}^\times$,  
    $$\sum_{f \in Q_{n,q}} m_{a,i_0}(f) = \frac{1}{d} \sum_{f \in Q_{n,q}} \sum_{i=0}^{d-1} m_{a,i}(f) = (q-1)|\mathbf{X}_n^a(\mathbb{F}_q)|.$$
\end{lemma}
\begin{proof}
    A point of $\mathbf{X}_n^a(\Fq)$ is a point $x$ of $\mathbf{X}_n(\Fqbar)$ such that $\operatorname{Frob}_q \cdot x = \langle a \rangle \cdot x$. Such a point $x$ corresponds to a triple $(g, \phi, h) / \Fqbar$ such that $\operatorname{Frob}_q \cdot (g, \phi, h)$ is isomorphic to $\langle a \rangle(g, \phi, h) = (g, a\phi, h)$. In particular, the isomorphism class of $h$ is fixed by Frobenius and hence the branch locus of $h$ is an $\Fq$-rational divisor. Let $F(t) \in \Fq[t]$ be the unique monic squarefree polynomial whose vanishing locus is the branch locus of $h$. By \Cref{proposition:monodromy_type_for_kummer_curves}, the source of $h$ is isomorphic to $C_F$ over $\Fqbar$ because the monodromy type of the $G$-cover $h \circ g: C \to \mathbb{P}^1$ is $c$.

    Fixing such an $h: C_F \to \mathbb{P}^1$ over $\Fqbar$, we count the number of points of $\mathbf{X}_n^a(\mathbb{F}_q)$ lying over $h$. Choices of $(g,\phi)$ such that $(g,\phi,h) \in \mathbf{X}_n^a(\overline{\mathbb{F}}_q) = \mathbf{X}_n(\overline{\mathbb{F}}_q)$ for the fixed $h$ are in bijection with the $\zeta_d$-equivariant surjections $J(C_F)[\ell](\overline{\mathbb{F}}_q) \to A$ up to isomorphism --- two such surjections $s$ and $s'$ are isomorphic if and only if $s = \zeta_d^i s'$ for some $i$. The action of Frobenius on the fiber over $h$ in $\mathbf{X}_n^a$ sends $s$ to $a^{-1} \operatorname{Frob}_q s$, so $s$ descends to a point of $\mathbf{X}_n^a(\mathbb{F}_q)$ if and only if $\operatorname{Frob}_q \cdot s = \zeta_d^i as$ for some $i$. Therefore, the number of points of $\mathbf{X}_n^a(\mathbb{F}_q)$ lying over $h$ is $\frac{1}{d} \sum_{i=0}^{d-1} m_{a,i}(F)$.

    Let $\epsilon \in \mathbb{F}_q^*$ and $f \in Q_{n,q}$. By abuse of notation, let $\zeta_d$ act on $C_f$ (and on $C_{\epsilon f}$) by $(x,y) \mapsto (x, \zeta_d y)$ for some fixed primitive $d$th root $\zeta_d$ in $\mathbb{F}_q$, which exists as $q \equiv 1 \pmod{d}$. For any $\delta \in \overline{\mathbb{F}}_q$ such that $\delta^d = \epsilon$, choose the unique $j$ such that $\zeta_d^j = \frac{\delta^q}{\delta}$. Note that the curves $C_f$ and $C_{\epsilon f}$ are isomorphic over $\overline{\mathbb{F}}_q$ via $\varphi: C_f \to C_{\epsilon f}, \quad (x,y) \mapsto (x, \delta y)$. Moreover, 
    $\varphi \circ \operatorname{Frob}_q = \zeta_d^{j} \circ \operatorname{Frob}_q \circ \varphi$.
    Given $R \in J_f[\ell](\overline{\mathbb{F}}_q)$ such that \eqref{eq:frobenius_eigenvector} holds, we have 
        $$  \operatorname{Frob}_q \varphi(R) = \zeta_d^{-j} \varphi \operatorname{Frob}_q R = \zeta_d^{-j} \varphi a \zeta_d^i R = a \zeta_d^{i-j} (\varphi R),$$
    so $m_{a,i}(f) = m_{a,i-j}(\epsilon f)$. In particular, 
        \begin{equation} \label{eq:sum_of_m_a_i_is_independent_of_epsilon}
            \sum_{i=0}^{d-1} m_{a,i}(f) = \sum_{i=0}^{d-1} m_{a,i}(\epsilon f),
        \end{equation}
    so
        \begin{align*}
        \frac{1}{d} \sum_{f \in Q_{n,q}} \sum_{i=0}^{d-1} m_{a,i}(f) &=  \sum_{\substack{F \in Q_{n,q} \\ \text{monic}}} \sum_{\epsilon \in \mathbb{F}_q^*} \left( \frac{1}{d} \sum_{i=0}^{d-1} m_{a,i}(\epsilon F)  \right),
        \end{align*}
    where $f$ is expressed as $\epsilon F$ for a monic $F$. Since, for each monic $F$, the sum $\sum_{i=0}^{d-1} m_{a,i}(\epsilon F)$ is independent of $\epsilon$ by \eqref{eq:sum_of_m_a_i_is_independent_of_epsilon}, the above equals
    $$(q-1)\sum_{\substack{F \in Q_{n,q} \\ \text{monic}}} \frac{1}{d} \sum_{i=0}^{d-1} m_{a,i}(F).$$
    Since the number of points of $\mathbf{X}_n^a(\mathbb{F}_q)$ lying over $h: C_F \to \mathbb{P}^1$ (over $\Fqbar$) equals $\frac{1}{d} \sum_{i=0}^{d-1} m_{a,i}(F)$, the above equals $(q-1)|\mathbf{X}_n^a(\mathbb{F}_q)|$. Thus,
        $$\frac{1}{d} \sum_{f \in Q_{n,q}} \sum_{i=0}^{d-1} m_{a,i}(f) = (q-1)|\mathbf{X}_n^a(\mathbb{F}_q)|.$$
    Now let $\delta \in \overline{\mathbb{F}}_q^\times$ satisfy $\zeta_d = \delta^{q-1}$; this is possible since $x \mapsto x^{q-1}$ is surjective on $\overline{\mathbb{F}}_q^\times$.
    Note that $\delta^d \in \mathbb{F}_q$ because $\delta^{(q-1)d} = \zeta_d^d = 1$. Therefore, for any $i_0 \in \mathbb{Z}/d\mathbb{Z}$,
        $$\sum_{f \in Q_{n,q}} m_{a,i_0}(f) = \sum_{f \in Q_{n,q}} m_{a,i_0}(\delta^d f) = \sum_{f \in Q_{n,q}} m_{a,i_0+1}(f),$$
    so
        $$\frac{1}{d} \sum_{f \in Q_{n,q}} \sum_{i=0}^{d-1} m_{a,i}(f) = \frac{1}{d} \sum_{f \in Q_{n,q}} \sum_{i=0}^{d-1} m_{a,i_0}(f) = \sum_{f \in Q_{n,q}} m_{a,i_0}(f).$$
\end{proof}

\begin{proposition}[{cf. \cite[Proposition 2.1, Proposition 2.7]{ellenberg_li_shusterman}}] \label{proposition:generalizing_ELS_proposition_2.1}
    Let $d \geq 3$ be an integer, let $\ell \nmid d$ be a prime number, and let $A$ be an involution ring quotient of $\mathbb{Z}/\ell \mathbb{Z}[\zeta_d]_{\zeta_d \neq 1}$. The constants $B_{d,\ell,A}$, $C_{d,\ell,A}$, and $N_{d,\ell,A}$ of \Cref{theorem:generalizing_EVW_Theorem_8.8_to_cubic_covers} satisfy
        $$\left| \frac{\sum_{f \in Q_{n,q}} m_{a,i_0}(f)}{|Q_{n,q}|} - 1 \right| \leq \frac{C_{d,\ell,A}}{\sqrt{q}}$$
    for all $a,n,q,i_0$ such that 
    \begin{itemize}
        \item $a \in A^\times$,
        \item $n$ is relatively prime to or divisible by $d$, 
        \item $q \equiv 1 \pmod{d}$ is a prime power,
        \item $\sqrt{q} > 2 B_{d,\ell,A}$, $n \geq N_{d,\ell,A}$, 
        \item $\ell$ does not divide $2dq(q-1)$,
        \item $\mathrm{SU}(V_n) \subseteq \mu(\operatorname{Gal}(\overline{K}_n / \overline{\mathbb{F}}_q K_n))$,
        \item $i_0 \in \mathbb{Z}/d\mathbb{Z}$.
    \end{itemize}
\end{proposition}

\begin{remark}
    Although the proof of \cite[Proposition 2.1]{ellenberg_li_shusterman} uses many of the same ideas from the proof of \cite[Theorem 8.8]{ellenberg_venkatesh_westerland}, which requires $\ell \nmid (q-1)$, \cite[Proposition 2.1]{ellenberg_li_shusterman} ultimately does not require the condition $\ell \nmid (q-1)$. This condition is originally used in \cite[Theorem 8.8]{ellenberg_venkatesh_westerland} to show that the Hurwitz scheme denoted by $\mathbf{X}_n$ has exactly one geometrically irreducible component defined over $\mathbb{F}_q$. In contrast, \cite{ellenberg_li_shusterman} is concerned with Hurwitz schemes arising in the case $d = 2$ and $A = \mathbb{Z}/\ell\mathbb{Z}$; it turns out that such Hurwitz schemes are geometrically irreducible even without the condition $\ell \nmid (q-1)$ and hence these Hurwitz schemes have exactly one geometrically irreducible component defined over $\mathbb{F}_q$.
    However, this argument does not apply to the context of \Cref{proposition:generalizing_ELS_proposition_2.1}, which is concerned with the case where $d \geq 3$ and $A$ is a quotient of $\mathbb{Z}/\ell\mathbb{Z}[\zeta_d]$.
\end{remark}

\begin{proof}
    By \Cref{lemma:research_generalizing_ELS_lemma_2.3_to_trielliptic_curves} and since $|Q_{n,q}| = (q-1)(q^n-q^{n-1})$ for $n \geq 2$ by \cite[Introduction, see also Section 1, equation (vi)]{carlitz_apgf}, it suffices to show that there exist $B_{d,\ell,A}$ and $N_{d,\ell,A}$ such that 
        $$\left| \frac{|\mathbf{X}_n^a(\mathbb{F}_q)|}{q^n} - 1 \right| \leq B_{d,\ell,A} q^{-1/2}$$
    for all $n,q,i_0$ as specified in the proposition. This follows by the same ideas as in the proof of \Cref{theorem:generalizing_EVW_Theorem_8.8_to_cubic_covers}\footnote{One difference to note between \Cref{theorem:generalizing_EVW_Theorem_8.8_to_cubic_covers} and \Cref{proposition:generalizing_ELS_proposition_2.1} is that the latter allows $n$ to be divisible by $d$. The requirement in \Cref{theorem:generalizing_EVW_Theorem_8.8_to_cubic_covers} for $n$ to not be divisible by $d$ and instead for $(d,n) = 1$ comes from \Cref{proposition:generalizing_EVW_proposition_8.7_to_cubic_covers} and the discussion above this proposition. There, the requirement guarantees that fields $L = \Fqt[y]/(y^d-f(t))$ in consideration are totally imaginary and hence that the corresponding smooth projective curves $C_L$ possess $\Fq$-rational points. In turn, this identifies $\operatorname{Jac}(C_L)(\Fq)$ with $\operatorname{Cl}_L$. For \Cref{proposition:generalizing_ELS_proposition_2.1} however, no such identification is needed.} --- by the Grothendieck-Lefschetz trace formula, we have 
        $$|\mathbf{X}_n^a(\mathbb{F}_q)| = \sum_i (-1)^i \operatorname{Tr}(\operatorname{Frob}_q | H^i_{c}(\mathbf{X}_n^a \times \overline{\mathbb{F}}_q, \mathbb{Q}_\lambda))$$
    where $\lambda$ is an auxiliary prime not dividing $\ell q$. Further taking $\lambda$ to be a prime greater than $\max(|G|, q, n)$ in fact yields the dimension bound \eqref{eq:cohomological_stability_bounds} applicable to $\mathbf{X}_n \times \overline{\mathbb{F}}_q \cong \mathbf{X}_n^a \times \overline{\mathbb{F}}_q$. Recall that the inequality \eqref{eq:alternating_sum_for_trace_of_frobenius_for_lower_cohomology_is_small} holds  
    as long as $\sqrt{q} > 2 B_{d,\ell,A}$. It now suffices to show that there exists an $N_{d,\ell,A}$ such that 
    \begin{equation} \label{eq:trace_of_frobenius_for_top_cohomology_on_twist_of_X_n}
        \operatorname{Tr}(\operatorname{Frob}_q | H^{2n}_c(\mathbf{X}_n^a \times \overline{\mathbb{F}}_q, \mathbb{Q}_\lambda)) = q^n
    \end{equation}
    holds for all $n > N_{d,\ell,A}$ such that $\gcd(d,n) = 1$ or $d \mid n$ and such that $\mathrm{SU}(V_n) \subseteq \mu(\operatorname{Gal}(\overline{K}_n / \overline{\mathbb{F}}_q K_n))$. The equality \eqref{eq:trace_of_frobenius_for_top_cohomology_on_twist_of_X_n} is equivalent to the statement that there is exactly one $\mathbb{F}_q$-rational connected component of $\mathbf{X}_n^a \times \overline{\mathbb{F}}_q$ --- the proof of \Cref{theorem:generalizing_EVW_Theorem_8.8_to_cubic_covers} shows this to be the case.
\end{proof}

\begin{corollary}[{cf. \cite[Corollary 2.6]{ellenberg_li_shusterman}}] \label{corollary:generalizing_ELS_corollary_2.6}
    Let $d \geq 3$ be an integer, let $\ell \nmid d$ be a prime number, and let $A$ be an involution ring quotient of $\mathbb{Z}/\ell\mathbb{Z}[\zeta_d]_{\zeta_d \neq 1}$. 
     The constants $B_{d,\ell,A}$, $C_{d,\ell,A}$, and $N_{d,\ell,A}$ of \Cref{theorem:generalizing_EVW_Theorem_8.8_to_cubic_covers} satisfy
        $$\frac{|Q_{n,q}^{a,i}|}{|Q_{n,q}|} \leq \frac{1}{\ell-1} + \frac{C_{d,\ell,A}}{(\ell-1)\sqrt{q}}$$
    for all $n,q,\ell,a,i$ such that 
    \begin{itemize}
        \item $n$ is relatively prime to or divisible by $d$, 
        \item $q \equiv 1 \pmod{d}$ is a prime power,
        \item $\sqrt{q} > 2B_{d,\ell,A}$, $n \geq N_{d,\ell,A}$,
        \item $\ell$ does not divide $2dq(q-1)$,
        \item $\mathrm{SU}(V_n) \subseteq \mu(\operatorname{Gal}(\overline{K}_n / \overline{\mathbb{F}}_q K_n))$,
        \item $a \in \mathbb{Z}/\ell\mathbb{Z}^\times$, and
        \item $i \in \mathbb{Z}/d\mathbb{Z}$.
    \end{itemize}
\end{corollary}
\begin{proof}
    Let $\delta$ denote the quantity $\frac{|Q_{n,q}^{a,i}|}{|Q_{n,q}|}$ to be bounded. Since $m_{a,i}(f)$ is the number of nonzero elements of a $\mathbb{Z}/\ell\mathbb{Z}$-vector space, it is at least $\ell-1$ if it is greater than $0$. In particular,
        $$\frac{\sum_{f \in Q_{n,q}} m_{a,i}(f)}{|Q_{n,q}|} \geq (\ell-1)\frac{|Q_{n,q}^{a,i}|}{|Q_{n,q}|} = (\ell-1) \delta.$$
    By \Cref{proposition:generalizing_ELS_proposition_2.1}, the LHS above is bounded above by $1 + \frac{C_{d,\ell,A}}{\sqrt{q}}$ for all $n,q,\ell,a,i$ specified by the statement of the corollary. Therefore,
    $$\delta \leq \frac{1}{\ell-1} + \frac{C_{d,\ell,A}}{(\ell-1)\sqrt{q}},$$
    as desired.
\end{proof}

\section{Vanishing of zeta functions and \texorpdfstring{$L$}{L}-functions for trielliptic curves}
\label{section:generalizing_ELS_section_3}

The ideas of this section generalize those in \cite[Section 3]{ellenberg_li_shusterman} to prove \Cref{theorem:research_generalizing_ELS_main_theorem_to_trielliptic_curves}, which in turn generalizes \cite[Theorem 3.2, Theorem 1.2]{ellenberg_li_shusterman}. These theorems provide explicit upper bounds on the proportion of superelliptic function fields whose zeta functions vanish at a fixed complex number $s$.

Given a variety $X/\mathbb{F}_q$, define the zeta function $Z_X(T)$ by 
    $$Z_{X}(T):= \prod_{\text {closed points } P \in X}\left(1-T^{\operatorname{deg} P}\right)^{-1},$$
which converges for $|T| < q^{- \dim X}$. We then define the zeta function $\zeta_X(s) := Z_X(q^{-s})$ which converges for $\operatorname{Re}(s) > \dim X$. See \cite[Chapter 3]{poonen_curves} or \cite[Chapter 5]{rosen} for instance, for discussions on this zeta function and on the Weil conjectures.

In the case that $C/\bbF_q$ is a nice curve, i.e. $C$ is a smooth, projective, geometrically integral variety of dimension $1$, the Weil conjectures show that there is a polynomial $P_C(T) \in \mathbb{Z}[T]$ such that 
    $$Z_C(T) = \frac{P_C(T)}{(1-T)(1-qT)}.$$
In fact, the constant and leading coefficients of $P_C(T)$ are respectively $1$ and $q^g$ where $g$ is the genus of $C$. Let $\operatorname{Jac}(C)$ denote the Jacobian of $C$ and let $T_\ell \operatorname{Jac}(C)$ denote the $\ell$-adic Tate module of this Jacobian. The polynomial $P_C(T)$ is then the reverse polynomial of the characteristic polynomial of the geometric Frobenius map $\operatorname{Frob}_q$ acting on $T_\ell \operatorname{Jac}(C)$ for any prime $\ell$ not equal to $\operatorname{char} \mathbb{F}_q$.

To prove \Cref{lemma:research_generalizing_ELS_Lemma_3.1}, which establishes the vanishing of Artin $L$-functions as an equivalent condition to the vanishing of $P_C(q^{-s})$ when $C$ is a $\mathbb{Z}/d\mathbb{Z}$-cover of $\mathbb{P}^1_{\mathbb{F}_q}$ and $q \equiv 1 \pmod{d}$, we will need some facts relating the Artin $L$-functions and the Dirichlet $L$-functions for $C$. In turn, we need the following definitions and notations.

\begin{definition}[{cf. \cite[Chapter 3, the Definition before Proposition 3.1 and the Definition before Proposition 3.4]{rosen}}]
    Let $q$ be a prime power and let $d$ be a divisor of $q-1$. For an irreducible polynomial $P \in \mathbb{F}_q[t]$ and a polynomial $a \in \mathbb{F}_q[t]$, define $\left( \frac{a}{P} \right)_d$ to be the unique element of $\mathbb{F}_q^*$ such that
        $$a^{\frac{q^{\deg P} - 1}{d}} \equiv \left( \frac{a}{P} \right)_d \pmod{P}$$
    if $P$ does not divide $a$ and to be $0$ otherwise.

    Given $a,b \in \mathbb{F}_q[t]$ with $b \neq 0$, define $\left( \frac{a}{b} \right)_d$ to be $\prod_{j=1}^s \left( \frac{a}{Q_j} \right)_d^{f_j}$ where $b = \beta Q_1^{f_1} \cdots Q_s^{f_s}$ is the prime decomposition of $b$.
\end{definition}

\begin{definition} \label{definition:chi_f}
Let $d \geq 2$ be an integer and let $q \equiv 1 \pmod{d}$ be a prime power. For a polynomial $f(t) \in \mathbb{F}_q[t]$, let $\chi_f: \mathbb{F}_q[t] \to \mathbb{F}_q$ denote the Dirichlet character given by $\chi_f(g) = \left( \frac{f}{g} \right)_d$. In particular, $\chi_f$ can be well defined as a function on the set of finite primes of $\mathbb{P}^1_{\mathbb{F}_q}$ --- let $\chi_f(P) = \left( \frac{f}{g_P} \right)_d$ where $g_P \in \mathbb{F}_q[t]$ is the monic irreducible polynomial corresponding to $P$. Note that $\chi_f^j = \chi_{f^j}$ for any $j \in \mathbb{Z}/d\mathbb{Z}$.

Given a fixed primitive $d$th root $r$ of unity in $\mathbb{F}_q$, we will also identify $\chi_f$ as a $\mathbb{C}$-valued character $\chi_f: \bbF_q[t] \to \bbC$ via the embedding $\mu_d \hookrightarrow \mathbb{C}^\times$ that sends $r$ to $e^{2\pi i/d}$.
\end{definition}

We now define $L$-series of Dirichlet characters. 
\begin{definition}[{cf. \cite[Chapter 17, before Proposition 17.7]{rosen} for a discussion in the case of $d = 2$}] \label{definition:L_series_of_dirichlet_characters}
    
    With notation as in \Cref{definition:chi_f}, define the $L$-series $L(s,\chi_f)$ as follows:
        $$L(s,\chi_f) =  \sum_{g \in \mathbb{F}_q[t] \text{ monic}} \frac{\chi_f(g)}{q^{s \cdot \deg g}} = \prod_{P \text{ finite prime of } \mathbb{F}_q(t)} \left( 1 - \frac{\chi_f(P)}{q^{s \cdot \deg P}} \right)^{-1}.$$
    Given $j \in \mathbb{Z}/d\mathbb{Z}$, also let $L(s,\chi_f^j) = L(s, \chi_{f^j})$.
\end{definition}

If $f = f_0 f_1^d$ where $\gcd(f_0,f_1) = 1$ and $f_0$ is $d$th-power-free, then note that 
    $$L(s, \chi_f) = L(s,\chi_{f_0}) \cdot \prod_{P \mid f_1} \left( 1 - \frac{\chi_{f_0}(P)}{q^{s \cdot \deg P}} \right).$$
In particular, away from $\operatorname{Re} s = 0$, $L(s,\chi_f)$ vanishes if and only if $L(s,\chi_{f_0})$ vanishes.

\begin{definition} \label{definition:complex_character_on_galois_group_of_cyclic_extension}
 Let $d \geq 2$ be an integer and let $q \equiv 1 \pmod{d}$ be a prime power. Given $f(t) \in \mathbb{F}_q[t]$ and a fixed primitive $d$th root $r$ of unity in $\mathbb{F}_q$, write $\chi$ to be the complex-valued character on $G = \operatorname{Gal}(\mathbb{F}_q(t)[y]/(y^d-f(t)) / \mathbb{F}_q(t))$ determined by $\chi(\sigma) = e^{2\pi i/d}$ where $\sigma$ is the element of $G$ that sends $y$ to $ry$. 
\end{definition}

\begin{definition}[{cf. \cite[Chapter 5]{rosen}}]
    Let $K/\mathbb{F}_q$ be a global function field. Let $P$ be a prime of $K$, i.e. the maximal ideal of a discrete valuation ring $R$ whose quotient field is $K$. Let the degree $\deg P$ of $P$ be the dimension of the residue field $R/P$ over the constant field $\mathbb{F}_q$. Moreover, the norm $N(P)$ is $q^{\deg P}$.
\end{definition}

\begin{definition}[{cf. \cite[Chapter 14, before Proposition 14.9]{rosen}}]

   Let $L/K$ be an abelian Galois extension of global function fields with Galois group $G$ and say that the constant field of $K$ is $\mathbb{F}_q$. Given a prime $Q$ of $L$ unramified over the prime $P$ of $K$, the \emph{Artin automorphism} $(P,L/K)$ is a generator of the decomposition group $D(P)$ characterized by the congruence
       $$(P, L/K) \omega = \omega^{N(P)} \pmod{Q}$$
   for any element $\omega$ of $L$ integral at $Q$.
   
   Let $\chi: G \to \mathbb{C}$ be a character. Given a prime $P$ of $K$, let $I(P)$ be the inertia group of $P$ and define $\chi(P)$ by the following:
   \begin{itemize}
       \item if $P$ is unramified in $L$, then $\chi(P) = \chi((P, L/K))$
       \item if $P$ is ramified in $L$ and $\chi$ is ramified at $P$, i.e. $\chi(I(P)) \neq 1$, then $\chi(P) = 0$.
       \item if $P$ is ramified in $L$ and $\chi$ is unramified at $P$, then $\chi(P) = \chi((P, L^{I(P)}/K))$; this is well defined because $\chi$ factors as a character $\operatorname{Gal}(L^{I(P)}/K) \cong G/I(P)$.
   \end{itemize}
    The Artin $L$-series of $\chi$ is defined by
    $$L(s, \chi) = \prod_{P \text{ prime of }K} \left( 1 - \frac{\chi(P)}{N(P)^{s}}\right)^{-1}.$$
\end{definition}

\Cref{lemma:research_generalizing_rosen_proposition_17.7_to_trielliptic_curves} shows that the $L$-series of the Dirichlet character $\chi_f$, regarded as a complex-valued character, roughly coincides with the Artin $L$-series of $\chi$. In fact, these two characters coincide on finite primes of $\mathbb{P}^1_{\mathbb{F}_q}$, so their respective $L$-functions coincide after making appropriate adjustments at $\infty$. See \cite[Proposition 17.7]{rosen} for a discussion in the case of $d = 2$.

\begin{lemma}\label{lemma:research_generalizing_rosen_proposition_17.7_to_trielliptic_curves}
Let $d \geq 2$ be an integer. Let $q$ be a prime power such that $q \equiv 1 \pmod{d}$. Let $f$ be a monic polynomial in $\mathbb{F}_q[t]$ whose irreducible factors all have exponents coprime to $d$ and fix a primitive $d$th root $r$ of unity in $\mathbb{F}_q$. Identify $\chi_f$ as a $\mathbb{C}$-valued character with respect to this choice of $r$ as discussed in \Cref{definition:chi_f}. Also let $\chi$ be the complex-valued character on the Galois group $G$ of $L = \mathbb{F}_q(t)[y]/(y^d - f(t))$ over $\mathbb{F}_q(t)$ as defined in \Cref{definition:complex_character_on_galois_group_of_cyclic_extension} with respect to this choice of $r$.

\begin{enumerate}
    \item $\chi_f(P) = \chi(P)$ for all finite primes $P$ of $\mathbb{P}^1_{\mathbb{F}_q}$. \label{part:characters_coincide_at_finite_primes}
    \item $L(s,\chi^j) = L_\infty(s, \chi^j) L(s, \chi_f^j)$ for any $j \in \mathbb{Z}/d\mathbb{Z}$ where
        $$L_\infty(s,\chi^j) = \left( 1 - \frac{\chi^j(\infty)}{q^{s}} \right)^{-1}.$$
\end{enumerate}

\begin{proof}
\begin{enumerate}
    \item 
    Recall that $\chi(P)$ is defined as 
        $$\chi(P) = \begin{cases} \chi((P, L/\mathbb{F}_q(t))) &\text{if } P \text{ is unramified in } L \\ 0 &\text{if } P \text{ is ramified and } \chi(I(P)) \neq 1 \\ \chi((P, L^{I(P)} / \mathbb{F}_q(t))) &\text{if } P \text{ is ramified and } \chi(I(P)) = 1 \end{cases}.$$
    where $(P, L/\mathbb{F}_q(t))$ denotes the Artin automorphism of the abelian extension $L/\mathbb{F}_q(t)$.

    Suppose that $P$ is a finite prime unramified in $L$. The Artin automorphism satisfies $(P, L/\mathbb{F}_q(t)) \omega = \omega^{q^{\deg P}} \pmod{Q}$ for any $\omega \in L$ integral at $Q$ where $Q$ is any prime lying over $P$. Note that $(P, L/\mathbb{F}_q(t))$ is a generator of the decomposition group of $G$ at $P$. On the other hand, $\chi_f(P)$ is the element $r_P$ of $\mathbb{F}_q^\times$ such that $r_P \equiv f^{\frac{q^{\deg P} - 1}{d}} \pmod{P}$. Since $r_P^d \equiv f^{q^{\deg P}-1} \equiv 1 \pmod P$, $r_P$ is a $d$th root of unity in $\mathbb{F}_q$. Since $f^{\frac{q^{\deg P} - 1}{d}} = y^{q^{\deg P} - 1}$, we have $r_P \equiv y^{q^{\deg P} - 1} \pmod{Q}$, so $r_P y \equiv y^{q^{\deg P}} \pmod{Q}$; thus $(P, L/\mathbb{F}_q(t)) y = r_P y$. Let $j \in \bbZ/d\bbZ$ satisfy $r_P = r^j$; then $\chi\big((P,L/\mathbb{F}_q(t))\big) = e^{2\pi i j/d}$. Hence $\chi_f(P) = r_P$, and as a complex number $\chi_f(P) = e^{2\pi i j/d}$, so $\chi(P) = \chi_f(P)$.

    If $P$ is instead ramified, then in fact $P$ is totally ramified and the inertia group $I(P)$ of $P$ is $G$ itself. Since $\chi(\sigma) = e^{2\pi i/d} \neq 1$, where $\sigma$ is as in \Cref{definition:complex_character_on_galois_group_of_cyclic_extension}, we have $\chi(I(P)) \neq 1$, so $\chi(P) = 0$ by definition. Moreover, since $P$ is a finite ramified prime, $P$ divides $f(t)$, so $\chi_f(P) = 0$. Therefore, $\chi(P) = \chi_f(P)$.

    \item By part \eqref{part:characters_coincide_at_finite_primes}, the multiplicands for finite primes of $\mathbb{F}_q(t)$ in the Euler products for $L(s,\chi^j)$ and $L(s,\chi_f^j)$ coincide. Moreover, $L_\infty(s,\chi^j)$ equals the multiplicand for $\infty$ in the Euler product for $L(s,\chi^j)$, where $\chi^j$ is the Artin character and $\chi_f^j$ is the Dirichlet character, so $L(s,\chi^j) = L_\infty(s,\chi^j) L(s,\chi_f^j)$ as desired.
    
\end{enumerate}

\end{proof}
\end{lemma}

Given a monic squarefree polynomial $f \in \mathbb{F}_q[t]$, let $C_f$ be the smooth completion of the curve defined by $y^d = f(t)$ and let $P_f \in \mathbb{Z}[T]$ be the reverse characteristic polynomial of geometric Frobenius acting on the Jacobian of $C_f$.

\begin{lemma}\label{lemma:research_generalizing_ELS_Lemma_3.1}
    Let $d \geq 3$ be an integer. Let $q$ be a prime power such that $q \equiv 1 \pmod{d}$, and let $f$ be a monic squarefree polynomial in $\mathbb{F}_q[t]$.
    
    For any $s \neq 0, 1$, the following are equivalent:
    \begin{enumerate}
        \item $P_f(q^{-s}) = 0$,
        \item $Z_{C_f}(q^{-s}) = 0$,
        \item the Artin $L$-function $L(s,\chi)$ vanishes for some nontrivial complex-valued character $\chi$ on the Galois group $G$ of $\mathbb{F}_q(t)(\sqrt[d]{f(t)})$ over $\mathbb{F}_q(t)$, 
        \item $L(s,\chi_f^j)$ vanishes for some $j \in \mathbb{Z}/d\mathbb{Z} \setminus \{0\}$.
    \end{enumerate}
\end{lemma}
\begin{proof}
    We have
        $$\zeta_{C_f}(s) = \zeta_{\mathbb{P}^1}(s) \prod_{\chi \neq \chi_0} L(s, \chi)$$
    where the product is over the nontrivial characters $\chi$ of the Galois group $G$ of $C_f/\mathbb{P}^1$ \cite[Proposition 14.9]{rosen}. Moreover,
        $$Z_{\mathbb{P}^1}(T) = \frac{1}{1-qT},$$
    so
        $$\zeta_{\mathbb{P}^1}(s) = \frac{1}{1-q^{1-s}}.$$
    Thus, $\zeta_{\mathbb{P}^1}(s)$ has no zeroes and has no poles away from $s = 1$. Therefore, $\zeta_{C_f}(s)$ vanishes exactly when some $L(s,\chi)$ vanishes. 

    Additionally,  
    $$Z_{C_f}(T) = \frac{P_f(T)}{(1-T)(1-qT)}$$
    and hence
        $$\zeta_{C_f}(s) = Z_{C_f}(q^{-s}) = \frac{P_f(q^{-s})}{(1-q^{-s})(1-q^{1-s})}.$$
    Combining the two displayed identities, $\zeta_{C_f}(s)$ vanishes exactly when $P_f(q^{-s})$ vanishes. By \Cref{lemma:research_generalizing_rosen_proposition_17.7_to_trielliptic_curves}, these conditions are in turn equivalent to the vanishing of some $L(s, \chi_f^j)$.
\end{proof}

We now prove the main result of this section. Unlike \Cref{theorem:generalizing_EVW_Theorem_8.8_to_cubic_covers} and \Cref{theorem:generalizing_EVW_main_theorem_to_cubic_covers}, however, we state this theorem immediately in terms of what \Cref{corollary:big_monodromy_result_needed_for_counting} affords instead of ``black-boxing'' a big monodromy condition.
Nevertheless, the exponent $C_d$ in the theorem would be improved should the big monodromy result of \Cref{corollary:big_monodromy_result_needed_for_counting} be generalized to more combinations of $d$, $\ell$, and $n$.

\begin{theorem}[Restatement of \Cref{theorem:research_generalizing_ELS_main_theorem_to_trielliptic_curves_intro}]\label{theorem:research_generalizing_ELS_main_theorem_to_trielliptic_curves} 
Let $d \geq 3$. Let $q$ be a prime power such that $q \equiv 1 \pmod{d}$, and let $f$ be a monic squarefree polynomial in $\mathbb{F}_q[t]$.

There exists a constant $C_{d} > 0$ depending only on $d$ such that, for any complex number $s \neq 0, 1$,
    $$\limsup_{\substack{n \to \infty \\ \gcd(d,n) = 1 \text{ or } d \mid n}} \frac{|\{f \in Q_{n,q} | Z_{C_f}(q^{-s}) = 0\}|}{|Q_{n,q}|} \ll q^{-C_{d}},$$
or equivalently by \Cref{lemma:research_generalizing_ELS_Lemma_3.1}
    $$\limsup_{\substack{n \to \infty \\ \gcd(d,n) = 1 \text{ or } d \mid n}} \frac{|\{f \in Q_{n,q} | L(s,\chi_f^j) = 0 \text{ for some } j \in \mathbb{Z}/d\mathbb{Z} \setminus \{0\}\}|}{|Q_{n,q}|} \ll q^{-C_{d}},$$
where the asymptotic bound is with respect to the limit taken over powers $q$ of a fixed prime power $q_0 \equiv 1 \pmod{d}$. In fact, $C_d = \frac{1}{(138d + 62) \cdot (d-1)}$ suffices.
\end{theorem}

\begin{proof}
    By \Cref{lemma:research_generalizing_ELS_Lemma_3.1}, the set $\{f \in Q_{n,q} | L(s,\chi_f^j) = 0 \text{ for some } j \in \mathbb{Z}/d\mathbb{Z} \setminus \{0\}\}$ equals 
        $$Q_{n,q}^{q^{-s}} := \{f \in Q_{n,q} | P_f(q^{-s}) = 0\}$$
    where $P_f(T) \in \mathbb{Z}[T]$ is the reverse characteristic polynomial of geometric Frobenius acting on the Jacobian of $C_f$. Note that roots of $P_f(T)$ are reciprocals of algebraic integers. 
    
    Let $q_0 \equiv 1 \pmod{d}$ be a fixed prime power. 
    If $q_0^s$ is not algebraic over $\mathbb{Q}$, then since $q = q_0^k$ for $k \geq 1$, we have that $q^{-s} = (q_0^s)^{-k}$ is transcendental, so $P_f(q^{-s}) \neq 0$ for every $f$. In this case, there is nothing to prove. As such, assume that $q_0^s$ is algebraic, and let $g_{q_0^s} \in \mathbb{Z}[T]$ be its minimal polynomial. For all sufficiently large powers $q$ of $q_0$, Chebotarev's density theorem yields a prime $\ell = \ell_q = \left( \frac{1}{4d} \right)^{\frac{1}{d-1}} \left( \frac{q}{4} \right)^{\frac{1}{(138d+62) \cdot (d-1)}} \cdot (1+o(1))$

    that is unramified in the splitting field $L$ of $g_{q_0^s}(T)\Phi_d(T)$ over $\mathbb{Q}$ and such that $\operatorname{Frob}_{\ell_q}$ is the (conjugacy class of the) identity element of $\operatorname{Gal}(L/\mathbb{Q})$. In fact, $\ell_q$ can be chosen to not divide $q-1$. Whenever $q$ is large enough, $\ell_q > [\mathcal{O}_L: \mathbb{Z}[q_0^s]]$ and hence $\ell_q \nmid [\mathcal{O}_L: \mathbb{Z}[q_0^s]]$. By the Dedekind-Kummer theorem (see e.g. \cite[Proposition 8.3]{neukirch} for a statement), $g_{q_0^s}$ splits completely modulo $\ell_q$. Let $A = \mathbb{Z}/\ell_q\mathbb{Z}[\zeta_d]_{\text{prim}} = (\mathbb{Z}/\ell_q\mathbb{Z}[X])/(\Phi_d(X))$. Further note that, when $q$ is large enough, $\ell_q$ is large enough so that the big monodromy result condition of \Cref{corollary:big_monodromy_result_needed_for_counting} and its consequences, including \Cref{proposition:constants_of_EVW}, hold.
    
    Let $a \in \mathbb{Z}/\ell_q\mathbb{Z}^\times$ be so that $g_{q_0^s}(a) \equiv 0 \pmod{\ell_q}$. If $q = q_0^k$, then define $P_{f,k}(T)$ as the polynomial $P_f(T^k)$. Given a polynomial $h$, let $h^{\text{rev}}$ denote its reverse polynomial. In particular,
        $$P_{f,k}^{\text{rev}}(q_0^s) = P_f^{\text{rev}}(q^s) = (q^s)^{\deg P_f} P_f(q^{-s}),$$
    so if $P_f(q^{-s}) = 0$, then $g_{q_0^s}(T)$ divides $P_{f,k}^{\text{rev}}(T)$. Thus, $P_{f,k}^{\text{rev}}(a) \equiv 0 \pmod{\ell_q}$ and hence $P_f(a^{-k}) \equiv 0 \pmod{\ell_q}$. Equivalently, there is some nonzero $R \in J_f[\ell_q](\overline{\mathbb{F}}_q)$ such that $\operatorname{Frob}_q R = a^k R$, where $\operatorname{Frob}_q$ here is the geometric Frobenius, so $m_{a^k, 0}(f) > 0$. Therefore, 
        $$\frac{|Q_{n,q}^{q^{-s}}|}{|Q_{n,q}|} \leq \frac{|Q_{n,q}^{a^k, 0}|}{|Q_{n,q}|}.$$
    By \Cref{corollary:generalizing_ELS_corollary_2.6}, 
        $$\limsup_{\substack{n \to \infty \\ \gcd(d,n) = 1 \text{ or } d \mid n}} \frac{|\{f \in Q_{n,q} | L(s,\chi_f^j) = 0 \text{ for some } j \in \mathbb{Z}/d\mathbb{Z} \setminus \{0\}\}|}{|Q_{n,q}|} \leq \frac{1}{\ell_q-1} + \frac{C_{d,\ell_q,A}}{(\ell_q-1)\sqrt{q}}$$
    for all sufficiently large $q$. By \Cref{proposition:constants_of_EVW}, this bound holds for $C_{d,\ell_q,A} = 2 \cdot (2d|A|)^{69d + 31} \leq 2 \cdot (2d \ell_q^{d-1})^{69d+31}$. The asymptotic size of $\ell_q$ is chosen so that $\frac{C_{d,\ell_q,A}}{\sqrt{q}}$ is bounded above by a constant. Therefore, the desired asymptotic bound holds.
\end{proof}

\section{Orbits of the Burau representation evaluated at roots of unity modulo \texorpdfstring{$\ell$}{ell}}
\label{section:orbits_of_burau_representation}

This section proves \Cref{lemma:ring_of_EVW_central_element_degree}, which is used in \Cref{proposition:constants_of_EVW} to obtain constants $B_{d,\ell,A}$ and $C_{d,\ell,A}$, just as how \cite[Proposition 2.7]{ellenberg_li_shusterman} obtains constants for $d=2$. We study the structure of the graded ring $R$ of \eqref{eq:ring_of_connected_components_of_hurwitz_scheme} below, obtaining upper bounds for the degrees of $\ker U_D$ and $\operatorname{coker} U_D$ where $U_D$ is a central element of $R$ presented in \eqref{eq:definition_of_U_D}. In turn, the degree $n$ component of $R$ is a $k$-vector space with basis corresponding to the orbits of a braiding action on tuples over $G = A \rtimes \mathbb{Z}/d\mathbb{Z}$ of length $n$ and, as \Cref{theorem:unreduced_burau_orbits_are_determined_by_monodromies_and_norm} demonstrates, the orbits of this action are precisely determined by three invariants. Understanding these orbits is much easier for $d = 2$ as only two invariants suffice. The $d = 2$ case is also much easier to work with because $\zeta_2 = -1$.

\begin{convention} \label{convention:research_orbits_of_burau_representations_2_convention_for_tuples_and_vectors}
Let $d \geq 2$ be an integer, let $\mathbb{Z}[\zeta_d] = \mathbb{Z}[\zeta_d]_{\text{all}}$, let $A$ be a $\mathbb{Z}[\zeta_d]$-module, let $G = A \rtimes \langle \zeta_d \rangle$, and let $c$ be the set $\{(a,\zeta_d) : a \in A\}$. Given an $n$-tuple $((a_1, \zeta_d), \ldots, (a_n, \zeta_d))$ of elements of $c$, identify this list with the vector $(a_1,\ldots,a_n) \in A^{\oplus n}$. We often write this vector as the column vector $\begin{pmatrix} a_1 & \cdots & a_n \end{pmatrix}^T$  with matrices acting on the left.
\end{convention}

\subsection{Braiding on the conjugacy class of \texorpdfstring{$A \rtimes \mathbb{Z}/d\mathbb{Z}$}{A rtimes Z/dZ} and the unreduced Burau representation}

We identify the (Artin) braiding action of $A \rtimes \mathbb{Z}/d\mathbb{Z}$ with the unreduced Burau representation evaluated at $t = \zeta_d$.

\begin{definition}
    Let $G$ be a group. The \emph{braiding action of $B_n$ on $G^n$} is given as follows --- the standard generator $\sigma_j \in B_n$  acts by
  $$ \sigma_{j}:\left(g_{1}, \ldots, g_{n}\right) \mapsto\left(g_{1}, \ldots, g_{j-2}, g_{j-1}, g_{j} g_{j+1} g_{j}^{-1}, g_{j}, g_{j+2} \ldots, g_n\right).$$
 Note that only the $j$-th and $(j+1)$-st entries are possibly modified.
\end{definition}

For a finite group $G$, a conjugacy class $c$, and a field $k$ of characteristic not dividing $|G|$, \cite[Section 3]{ellenberg_venkatesh_westerland} defines the graded ring
     \begin{equation} \label{eq:ring_of_connected_components_of_hurwitz_scheme}
         R = \sum_n H_0(\operatorname{Hur}^c_{G,n}, k)
     \end{equation}
where $\mathrm{Hur}_{G,n}^c$ is the complex Hurwitz space of tamely ramified $G$-covers of $\mathbb{P}^1_{\mathbb{C}}$ with $n$ branch points with monodromy of type $c$. If $c$ is nonsplitting as well, then, for each $p$, $\dim H_p(\mathrm{Hur}^c_{G,n}, k)$ stabilizes with respect to $n$ \cite[Theorem 6.1, Corollary 6.2]{ellenberg_venkatesh_westerland}. The braiding action restricted to $c^n$ for a conjugacy closed subset $c$ of $G$ induces the relations on $R$. More explicitly, $R$ is generated over $k$ by degree 1 elements $\left\{r_{g}\right\}_{g \in c}$ and has relations
    $$r_{g} r_{h}=r_{g h g^{-1}} r_{g}.$$

\begin{lemma} \label{lemma:research_orbits_of_burau_representations_2_braiding_action_on_c_n_corresponds_to_the_unreduced_burau_representation_evaluated_at_zeta_3}
    Let $d \geq 2$, let $A$ be a $\mathbb{Z}[\zeta_d] = \mathbb{Z}[\zeta_d]_{\text{all}}$-module, and let $G$ and $c$ be as in \Cref{convention:research_orbits_of_burau_representations_2_convention_for_tuples_and_vectors}. Under Convention \ref{convention:research_orbits_of_burau_representations_2_convention_for_tuples_and_vectors}, the braiding action of $B_n$ on $c^n$ corresponds to the action induced by the unreduced Burau representation $\psi_n$ evaluated at $t = \zeta_d$. More precisely, the standard generators $\sigma_1,\ldots,\sigma_{n-1}$ of $B_n$ act on $A^{\oplus n}$ as elements of $\operatorname{GL}_n(\operatorname{End}_{\mathbb{Z}[\zeta_d]}(A))$ and these elements can be represented by $n \times n$ matrices over $\operatorname{End}_{\mathbb{Z}[\zeta_d]}(A)$ acting on column vectors on the left. These matrices in fact are the images of $\sigma_1,\ldots,\sigma_{n-1}$ under the composition of group homomorphisms
       \begin{align*} 
        B_n &\xrightarrow{\psi_n} \operatorname{GL}_n(\mathbb{Z}[t,t^{-1}]) \xrightarrow{t = \zeta_d} \operatorname{GL}_n(\mathbb{Z}[\zeta_d])\\
        &\xrightarrow{\text{base change}} \operatorname{GL}_n(\operatorname{End}_{\mathbb{Z}[\zeta_d]}(A)).
       \end{align*} 
\end{lemma}
\begin{proof}
    The $\sigma_j$ act via the braiding action on $c^n$ by 
        $$\left(g_{1}, \ldots, g_{n}\right) \mapsto\left(g_{1}, \ldots, g_{j-2}, g_{j-1}, g_{j} g_{j+1} g_{j}^{-1}, g_{j}, g_{j+2}, \ldots, g_n\right).$$
    By \Cref{lemma:research_orbits_of_burau_representations_2_semidirect_product_structure}, 
        $$(a_j, \zeta_d) \cdot (a_{j+1}, \zeta_d) \cdot (a_j, \zeta_d)^{-1} = (a_j + \zeta_d a_{j+1} - \zeta_d a_j, \zeta_d) = ((1-\zeta_d) a_j + \zeta_da_{j+1}, \zeta_d).$$
    The action of $\sigma_j$ on corresponding column vectors is thus 
        \begin{equation} \label{eq:action_of_braid_on_column_vector}
            \begin{pmatrix}a_1 \\ \vdots \\ a_n \end{pmatrix} \mapsto \begin{pmatrix} a_1 \\ \vdots \\ a_{j-2} \\ a_{j-1} \\ (1-\zeta_d) a_j + \zeta_d a_{j+1} \\ a_j \\ a_{j+2} \\ \vdots \\ a_n\end{pmatrix}.
        \end{equation}
    The image of $\sigma_j$ under the unreduced Burau representation evaluated at $t = \zeta_d$ is the block matrix
        $$\begin{pmatrix} I_{j-1} & 0 & 0 & 0 \\ 0 & 1-\zeta_d & \zeta_d & 0 \\ 0 & 1 & 0 & 0 \\ 0 & 0 & 0 & I_{n-j-1} \end{pmatrix}.$$
    Regarding the entries of this matrix as endomorphisms of $A$ as a $\mathbb{Z}[\zeta_d]$-module, this matrix acts on the column vectors in the same way as \eqref{eq:action_of_braid_on_column_vector} above.
\end{proof}

\begin{convention} \label{convention:zeta_d_modules_are_laurent_polynomial_ring_modules}
    In view of \Cref{lemma:research_orbits_of_burau_representations_2_braiding_action_on_c_n_corresponds_to_the_unreduced_burau_representation_evaluated_at_zeta_3}, given a ring $R$, we will regard $R[\zeta_d]_{\text{all}}$ or quotients thereof (such as $R[\zeta_d]_{\zeta_d \neq 1}$, see \Cref{subsection:notations}) as $R[t,t^{-1}]$-algebras by evaluating $t = \zeta_d$.
\end{convention}

We introduce invariants preserved under the braiding action. 

\begin{definition}[{cf. \cite[Section 2.4]{ellenberg_venkatesh_westerland}}]
    Let $v = (g_1,\ldots,g_n)$ be an $n$-tuple of elements of a group $G$. The \emph{global monodromy} $M_G(v)$ of this $n$-tuple is the subgroup $\langle g_1,\ldots,g_n \rangle$ of $G$. The \emph{boundary monodromy} $M_B(v)$ of this $n$-tuple is the element $g_1 \cdots g_n$ of $G$.
\end{definition}

 \begin{proposition} \label{research_orbits_of_burau_representations_2:proposition_braiding_action_preserves_global_and_boundary_monodromy} 
 Let $G$ be a group. For any $n$-tuple of elements of $G$, the braiding action of $B_n$ preserves the global and boundary monodromy of the $n$-tuple.
 \end{proposition}

 \begin{proof}
     The global monodromy of $\sigma_j (g_1,\ldots,g_n)$ is the subgroup of $G$ generated by $g_1,\ldots,g_{j-2},$ $g_{j-1}, g_jg_{j+1} g_j^{-1}, g_j, g_{j+2},\ldots,g_n$; note that the group generated by $g_{j} g_{j+1} g_j^{-1}$ and $g_j$ is the same as the group generated by $g_j, g_{j+1}$. Thus, the global monodromy of $\sigma_j (g_1,\ldots,g_n)$ equals the subgroup of $G$ generated by $g_1,\ldots,g_n$ and this subgroup is the global monodromy of $(g_1,\ldots,g_n)$.
    
    The boundary monodromy of $\sigma_j (g_1,\ldots,g_n)$ is $g_1 \cdots g_{j-1} \cdot (g_j g_{j+1} g_j^{-1} \cdot g_j) g_{j+2} \cdots g_n$, which equals $g_1 \cdots g_{j-1} \cdot g_j \cdot g_{j+1} \cdot g_{j+2} \cdots g_n$, which is the boundary monodromy of $(g_1,\ldots,g_n)$.
 \end{proof}

In the situation of interest, global monodromy can be described in the language of affine modules.
\begin{definition} \label{definition:affine_module_over_a_ring}
    Let $R$ be a ring. 
    An \emph{affine module} is a set $A$ together with an $R$-module $M$ such that there is a group action of $M$ on $A$, say written as $a + m$, such that for any two $a,b \in A$, there is a unique $m$ such that $b = a+m$. In other words, $A$ is a principal homogeneous space for the group $M$. One might denote the value $m$ in this case as $b - a$. The module $M$ may be called the \emph{underlying module} of the affine module structure. 
\end{definition}

For the rest of this subsection, $A$ is either $\mathbb{Z}/\ell\mathbb{Z}[\zeta_d]$ or some $D_i/\ell D_i$ where $D_i$ is a factor of $\mathbb{Z}_\ell[\zeta_d]$ as in \eqref{eq:Z_ell_zeta_d_factorization_as_ring_with_involution} with involution. We note that subgroups of a group $A \rtimes \langle \zeta_d \rangle$ are determined by a submodule $A'$ of $A$ and an element of $A/A'$. 

\begin{proposition} \label{proposition:how_subgroups_of_semidirect_products_are_determined}

    Let $d \geq 2$ be an integer, let $\mathbb{Z}[\zeta_d] = \mathbb{Z}[\zeta_d]_{\text{all}}$, let $A$ be a $\mathbb{Z}[\zeta_d]$-module, let $G = A \rtimes \langle \zeta_d \rangle$, and let 
    $$c = \{(a,\zeta_d) \in G\}.$$
    There is a bijective correspondence between
    \begin{enumerate}
        \item The subgroups of $G$ that surject onto $\langle \zeta_d \rangle$
        \item Affine sub-modules of $A$ as a $\mathbb{Z}[\zeta_d]$-affine module.
    \end{enumerate}
    The bijection sends a subgroup $H$ of $G$ surjecting onto $\langle \zeta_d \rangle$ to the affine submodule 
    $$\{a \in A: (a,\zeta_d) \in H\},$$
    and sends the affine sub-module $B$ of $A$ to the subgroup generated by the elements of the form $(b, \zeta_d)$ for $b \in B$. 

    For a subgroup $H$ of $G$, writing $A'_H$ and $a_0$ respectively for the underlying module of the affine submodule of $A$ and an element of $A$ such that $(a_0,\zeta_d) \in H$, we may refer to $H$ as the \emph{subset of $G$ determined by $A'_H$ and $a_0$}. 

\end{proposition}
\begin{proof}
    We show that the described assignment is well defined Let $H \leq G$ be a subgroup surjecting onto $\langle \zeta_d \rangle$. Define $\mathcal{M}_H = \{a \in A : (a, \zeta_d) \in H\}$. Since $H$ surjects onto $\langle \zeta_d \rangle$, there exists some $(a_0, \zeta_d) \in H$, so $\mathcal{M}_H$ is nonempty. For any $a, b \in \mathcal{M}_H$, we have $(a, \zeta_d), (b, \zeta_d) \in H$. By \Cref{lemma:research_orbits_of_burau_representations_2_semidirect_product_structure}\eqref{item:semidirect_product_inverse},
    $$ (a, \zeta_d)(b, \zeta_d)^{-1} = (a - b, 1) \in H. $$
    Thus, the set of differences $A'_H = \{a - b : a, b \in \mathcal{M}_H\}$ is contained in $H \cap A$. Since $H \cap A$ is a $\mathbb{Z}[\zeta_d]$-submodule (as $H$ is normalised by $(a_0, \zeta_d)$), $A'_H$ is a submodule. Hence $\mathcal{M}_H = a_0 + A'_H$ is an affine submodule of $A$.

    Conversely, let $\mathcal{M} \subseteq A$ be an affine submodule, say $\mathcal{M} = a_0 + A'$ for some submodule $A' \leq A$. Let $H_{\mathcal{M}}$ be the subgroup generated by $\{(b, \zeta_d) : b \in \mathcal{M}\}$. Since $(a_0, \zeta_d) \in H_{\mathcal{M}}$, the projection of $H_{\mathcal{M}}$ onto $\langle \zeta_d \rangle$ is surjective.

    We show that the assignment is invertible. First, start with a subgroup $H$. Let $\mathcal{M}_H = \{a : (a, \zeta_d) \in H\}$ and let $H' = H_{\mathcal{M}_H}$. Clearly $H' \subseteq H$ by definition. For the reverse inclusion, note that any element $h \in H$ projects to $\zeta_d^k$ for some $k$. Since $(a_0, \zeta_d) \in H'$, we have $(a_0, \zeta_d)^k = (\sum_{i=0}^{k-1} \zeta_d^i a_0, \zeta_d^k) \in H'$. Thus $h \cdot (a_0, \zeta_d)^{-k} = (a - \sum_{i=0}^{k-1} \zeta_d^i a_0, 1)$ lies in $H \cap A$. By the well-definedness argument above, this difference lies in $A'_{\mathcal{M}_H}$, so $(a - \sum_{i=0}^{k-1} \zeta_d^i a_0, 1) \in H'$. Since $H'$ contains both factors, $h \in H'$. Thus $H = H'$.

    Second, start with an affine submodule $\mathcal{M} = a_0 + A'$. Let $H = H_{\mathcal{M}}$ and let $\mathcal{M}' = \{a : (a, \zeta_d) \in H\}$. By construction, $\mathcal{M} \subseteq \mathcal{M}'$. Conversely, if $(a, \zeta_d) \in H$, then $a$ is a product of elements in $\mathcal{M}$ and their inverses such that the total power of $\zeta_d$ sums to 1. Using the group law from \Cref{lemma:research_orbits_of_burau_representations_2_semidirect_product_structure}, any such product results in a first coordinate lying in $a_0 + A'$. Thus $\mathcal{M}' \subseteq \mathcal{M}$, so $\mathcal{M} = \mathcal{M}'$.

\end{proof}


In view of \Cref{proposition:how_subgroups_of_semidirect_products_are_determined}, we describe how the global monodromy of 
$$v_c = ((a_1,\zeta_d),\ldots,(a_n,\zeta_d))$$
is determined.

\begin{notation}
   Given an $n$-tuple (or an $n$-row or column vector) $v = (a_1,\ldots,a_n) \in A^{\oplus n}$ where $A$ is a $\mathbb{Z}[t,t^{-1}]$-module, write $M_{G,1}(v)$ for the $\mathbb{Z}[t,t^{-1}]$-module generated by $\{a_i - a_{i_0}: i=1,\ldots,n\}$ for any $i_0 \in \{1,\ldots,n\}$. Note that $M_{G,1}(v)$ does not depend on the choice of $i_0$. Equivalently, $M_{G,1}(v)$ is generated by $\{a_i-a_j: i,j = 1,\ldots,n\}$.
   Further note that, under \Cref{convention:zeta_d_modules_are_laurent_polynomial_ring_modules}, $M_{G,1}(v)$ is defined when $A$ is a $\mathbb{Z}[X]/(X^d-1)$-module.
\end{notation}

The following proposition establishes that the global monodromy of $v_c$ is determined exactly by $M_{G,1}(v_A)$ and the congruence class of $a_{i_0}$ modulo $M_{G,1}(v_A)$ for any $i_0$ where $v_A = (a_1,\ldots,a_n)$.
\begin{proposition}\label{proposition:research_orbits_of_burau_representations_2_proposition_global_monodromy_for_c_n}
    Let $d \geq 2$ be an integer, let $\mathbb{Z}[\zeta_d] = \mathbb{Z}[\zeta_d]_{\text{all}}$, let $A$ be a $\mathbb{Z}[\zeta_d]$-module, let $G = A \rtimes \langle \zeta_d \rangle$, and let $c = \{(a,\zeta_d) \in G: a \in A\}$. Given $v_c = (g_1,\ldots,g_n) \in c^n$, write $g_i = (a_i, \zeta_d)$ and $v_A = (a_1,\ldots,a_n)$. The following hold:
    \begin{enumerate}
        \item $M_G(v_c) \cap c = \{(a_{i_0} + \alpha, \zeta_d): \alpha \in M_{G,1}(v_A)\}$ for any $i_0 \in \{1,\ldots,n\}$.
        \item $M_G(v_c) \cap (A \rtimes \langle 1 \rangle) = M_{G,1}(v_A) \rtimes \langle 1 \rangle \cong M_{G,1}(v_A)$.
        \item $M_G(v_c) = \langle g_{i_0}, M_{G,1}(v_A) \rtimes \langle 1 \rangle \rangle$ for any $i_0 \in \{1,\ldots,n\}$.
    \end{enumerate}
    In other words, the global monodromy $M_G(v_c)$ of $v_c \in c^n$ is determined (in the sense of \Cref{proposition:how_subgroups_of_semidirect_products_are_determined}) exactly by $M_{G,1}(v_A)$ and the equivalence class of $a_{i_0}$ modulo $M_{G,1}(v_A)$ where $v_A$ is the element of $A^{\oplus n}$ corresponding to $v_c \in c^n$. In particular, writing $v_c,w_c \in c^n$ by
    \begin{align*}
        v_c &= ((a_1,\zeta_d), \ldots, (a_n, \zeta_d)) \\
        w_c &= ((b_1,\zeta_d), \ldots, (b_n, \zeta_d)),
    \end{align*}
    $v_c$ and $w_c$ have the same global monodromy if and only if $M_{G,1}(v_A)$ and $M_{G,1}(w_A)$ are equal and $a_{i}$ and $b_j$ are equivalent modulo this submodule of $A$ for some, equivalently every, pair of indices $i,j \in \{1,\ldots,n\}$. 
    
\end{proposition}
\begin{proof}
    By \Cref{lemma:research_orbits_of_burau_representations_2_semidirect_product_structure}, $(a_i, \zeta_d) \cdot (a_1,\zeta_d)^{-1} = (a_i-a_1, 1)$. In particular, $M_G(v_c)$, which is by definition generated by $g_1,\ldots,g_n$, is also generated by $g_{i_0}, (a_1-a_{i_0}, 1), (a_2-a_{i_0},1)\ldots, (a_n-a_{i_0},1)$ for any $i_0$. Furthermore, the elements $a \in A$ such that $(a,1) \in M_G(v_c)$ form a $\mathbb{Z}[\zeta_d]$-module ---  if $(a,1)$ is in $M_G(v_c)$, then $(a_{i_0}, \zeta_d) \cdot (a, 1) \cdot (a_{i_0}, \zeta_d)^{-1} = (\zeta_d a, 1)$. Therefore, $M_G(v_c) \cap (A \rtimes \langle 1 \rangle)$ contains $M_{G,1}(v_A) \rtimes \langle 1 \rangle$, and hence $M_G(v_c) \cap c$ contains all the elements of the form $(\alpha, 1) \cdot (a_{i_0}, \zeta_d) = (a_{i_0} + \alpha, \zeta_d)$ where $\alpha \in M_{G,1}(v_A)$. 
    Note in fact that $M_G(v_c)$ is generated by $g_{i_0}$ and $M_{G,1}(v_A)\rtimes\langle 1\rangle$ --- for each $i$, we have $g_i = (a_i-a_{i_0},1)\cdot g_{i_0}$ with $(a_i-a_{i_0},1) \in M_G(v_c)$. Thus, an element of $M_G\cap c$ has the form $(\alpha,1)g_{i_0} = (a_{i_0}+\alpha,\zeta_d)$, and an element of $M_G\cap(A\rtimes\langle1\rangle)$ lies in $M_{G,1}\rtimes\langle1\rangle$, so the desired claims hold.

\end{proof}

\begin{corollary} \label{corollary:global_monodromy_description}
    Let $d \geq 2$ be an integer and let $\ell \nmid d$ be a prime number. Let $G = A \rtimes \langle \zeta_d \rangle$ where $A$ is a quotient of $\mathbb{Z}/\ell\mathbb{Z}[\zeta_d] = \mathbb{Z}/\ell\mathbb{Z}[\zeta_d]_{\text{all}}$. Factorize $\mathbb{Z}_\ell[\zeta_d]$ as the product $\prod_{i \in I} B_i$ of DVR's similarly to \eqref{eq:Z_ell_zeta_d_factorization_as_ring}. In particular, $A$ factorizes as a product $\prod_{i \in I'} B_i/(\ell)$ of fields, where $I'$ is a subset of $I$. Let $c$ be the subset of $G$ consisting of the elements of the form $(a, \zeta_d)$ for $a \in A$. Let $v_c = (g_1,\ldots,g_n) \in c^n$, write $g_j = (a_j, \zeta_d)$, and write $v_A = (a_1,\ldots,a_n)$. Further write $v_{A,i}$ for the $B_i/(\ell)$-component $(a_{1,i},\ldots,a_{n,i})$ of $v_A$. 
    
    Then, $M_G(v_c)$ is determined exactly by the $B_i/(\ell)$-modules $M_{G,1}(v_{A,i})$ and the $M_{G,1}(v_{A,i})$-equivalence class of $a_{j_0,i}$, which does not depend on the choice of $j_0$. More precisely, $M_{G,1}(v_{A}) = \prod_i M_{G,1}(v_{A,i})$ and the $M_{G,1}(v_{A})$-equivalence class of $a_{j_0}$ (which corresponds to the $M_{G,1}(v_{A,i})$-equivalence classes of $a_{j_0,i}$) determine (in the sense of \Cref{proposition:how_subgroups_of_semidirect_products_are_determined}) $M_G(v_c)$. Furthermore, 
        $$M_{G,1}(v_{A,i}) = \begin{cases} B_i/(\ell) &\textup{if at least two of } a_{1,i},\ldots,a_{n,i} \textup{ are different} \\ 0 &\textup{otherwise}. \end{cases}$$
\end{corollary}
\begin{proof}
    By \Cref{proposition:research_orbits_of_burau_representations_2_proposition_global_monodromy_for_c_n}, $M_G(v_c)$ is determined exactly by $M_{G,1}(v_A)$, which is the $A$-module generated by $a_2-a_1,\ldots,a_n-a_1$, and the $M_{G,1}(v_A)$-equivalence class of $a_1$. The $A$-module $M_{G,1}(v_A)$ corresponds to the tuple of its $B_i/(\ell)$-components, which are generated by $a_{2,i}-a_{1,i},\ldots, a_{n,i}-a_{1,i}$ and hence are $M_{G,1}(v_{A,i})$. Moreover, the $M_{G,1}(v_A)$-equivalence class of $a_{j_0}$ corresponds to the tuple of the $M_{G,1}(v_{A,i})$-equivalence classes of $a_{j_0,i}$. Lastly, since $B_i/(\ell)$ is a field, $M_{G,1}(v_{A,i})$ is $B_i/(\ell)$ itself if and only if any of the $a_{2,i}-a_{1,i},\ldots,a_{n,i}-a_{1,i}$ is nonzero. 
\end{proof}

\begin{proposition} \label{proposition:research_orbits_of_burau_representations_2_boundary_monodromy_of_c_n_elements_corresponds_to_a_linear_invariant_over_A}
    Let $d \geq 2$ be an integer, let $\mathbb{Z}[\zeta_d] = \mathbb{Z}[\zeta_d]_{\text{all}}$, let $A$ be a $\mathbb{Z}[\zeta_d]$-module, let $G = A \rtimes \langle \zeta_d \rangle$, and let $c$ be the set $\{(a,\zeta_d) : a \in A\}$. Given $v_A = (a_1,\ldots,a_n) \in A^{\oplus n}$, identify $v_A$ with $v_c \in c^n$ under \Cref{convention:research_orbits_of_burau_representations_2_convention_for_tuples_and_vectors}. The boundary monodromy $M_B(v_c)$ equals $(\sum_{i=1}^n \zeta_d^{i-1} a_i, \zeta_d^n)$. Therefore, with $n$ fixed, the boundary monodromy of an element of $c^n$ is determined exactly by the linear invariant $\sum_{i=1}^n \zeta_d^i a_i$. 
\end{proposition}
\begin{proof}
    The boundary monodromy of $v_c$ is 
        $$(a_1, \zeta_d) \cdots (a_n, \zeta_d),$$
    and calculating this as $(\sum_{i=1}^n \zeta_d^{i-1} a_i, \zeta_d^n)$ is immediate.
\end{proof}

In view of \Cref{corollary:global_monodromy_description} and \Cref{proposition:research_orbits_of_burau_representations_2_boundary_monodromy_of_c_n_elements_corresponds_to_a_linear_invariant_over_A}, we introduce the following terminology to more easily discuss the boundary and global monodromy invariants when using  \Cref{convention:research_orbits_of_burau_representations_2_convention_for_tuples_and_vectors}. 

\begin{definition} \label{definition:global_and_boundary_monodromy_of_vector}
    Let $d \geq 2$ be an integer, and let $\mathbb{Z}[\zeta_d] = \mathbb{Z}[\zeta_d]_{\text{all}}$. Given a $\mathbb{Z}[\zeta_d]$-module $A$, and $v_A = (a_1,\ldots,a_n) \in A^{\oplus n}$, we define $M_{B}(v_A) := \sum_{i=1}^n a_i \zeta_d^i$ to be the \emph{boundary monodromy} of $v_A$. In view of \Cref{proposition:research_orbits_of_burau_representations_2_boundary_monodromy_of_c_n_elements_corresponds_to_a_linear_invariant_over_A}, $M_{B}(v_c) = (\zeta_d^{-1} M_{B}(v_A), \zeta_d^n)$ where $v_c$ corresponds to $v_A$ under \Cref{convention:research_orbits_of_burau_representations_2_convention_for_tuples_and_vectors}.

    By the \emph{global monodromy} $M_G(v_A)$ of $v_A$, we mean the global monodromy $M_G(v_c)$ of $v_c = ((a_1,\zeta_d),\ldots,(a_n,\zeta_d))\in c^n$. \Cref{proposition:research_orbits_of_burau_representations_2_proposition_global_monodromy_for_c_n} shows that $M_G(v_A)$ is determined exactly by $M_{G,1}(v_A)$ and the $M_{G,1}(v_A)$-equivalence class of $a_{1}, \ldots,a_n$. In the case that $A$ is a quotient of $\mathbb{Z}/\ell\mathbb{Z}[\zeta_d]$ with $\ell \nmid d$ a prime number,  \Cref{corollary:global_monodromy_description} in fact shows that $M_G(v_A)$ is determined exactly by  whether $a_{1,i},\ldots,a_{n,i} \in B_i / (\ell)$ are the same and, if so, what this value is across the $i$'s indexing the factorization $\mathbb{Z}/\ell\mathbb{Z}[\zeta_d] = \prod_i B_i / (\ell)$.
\end{definition}

\subsection{A nice change of basis} \label{subsection:a_nice_change_of_basis}

Let $\mathcal{C}$ be the $n \times n$ matrix over $\mathbb{Z}[t,t^{-1}]$ given by 
    \begin{equation} \label{eq:nice_change_of_basis_matrix_for_unreduced_burau_representation}
        \mathcal{C} = \begin{pmatrix} 1 & 1 & 0 & 0 & \cdots & 0 & 0 \\ 1 & -\frac{1}{t} & \frac{1}{t} & 0 & \cdots & 0 & 0 \\ 1 & 0 & -\frac{1}{t^2} & \frac{1}{t^2} & \cdots & 0 & 0 \\ 1 & 0 & 0 & -\frac{1}{t^3} & \cdots & 0 & 0 \\ \vdots & \vdots & \vdots & \vdots & \ddots & \vdots & \vdots \\ 1 & 0 & 0 & 0 & \cdots & -\frac{1}{t^{n-2}} & \frac{1}{t^{n-2}} \\ 1 & 0 & 0 & 0 & \cdots & 0 & -\frac{1}{t^{n-1}} \end{pmatrix}.
    \end{equation}
We describe how the change of basis by $\mathcal{C}$ changes the presentations of the unreduced Burau matrices $\psi_n(\sigma_i)$, the unitary matrix $H_n$, and the boundary monodromy function --- the presentations are essentially given by the matrices and vectors computed in \Cref{lemma:nice_change_of_basis_for_burau_matrices}. In particular, $B_n$ acts on the vectors under the change of basis by $\mathcal{C}$ by the \textit{reduced} Burau representation $\psi_n^r$ on all but the first coordinate.

Let $A$ be a $\mathbb{Z}[t,t^{-1}, \det(\mathcal{C})^{-1}]$-module.
Given a vector in $A^{\oplus n}$, write $v$ for its column vector with respect to the standard basis $e_1,\ldots,e_n$, and $\overline{v}$ for its column vector with respect to the $\mathcal{C}$-basis $\sum_{i=1}^n e_i, e_1 - \frac{1}{t} e_2, \ldots$. In particular, $v = \mathcal{C}\overline{v}$.
The action of $\sigma_i \in B_n$ via the unreduced Burau representation $\psi_n$ acting on $A^{\oplus n}$ sends a column vector $v$ to $\psi_n(\sigma_i) v$. Correspondingly, the action of $\sigma_i$ sends a column vector $\overline{v}$ to $(\mathcal{C}^{-1} \psi_n(\sigma_i) \mathcal{C}) \overline{v}$. Moreover, letting $h$ be the Hermitian form \eqref{eq:hermitian_form_from_matrix} given by $H_n$, we have $h(v,w) = v^* H_n w = \overline{v}^* (\mathcal{C}^* H_n \mathcal{C}) \overline{w}$, so $h$ is given by the matrix $\mathcal{C}^* H_n \mathcal{C}$ with respect to the change of basis. Lastly, the boundary monodromy function is given by the row vector $\begin{pmatrix} t & t^2 & \cdots & t^n \end{pmatrix}$ multiplied to the left of $v$ (and evaluated at $t = \zeta_d$), so the function is given by the multiplication by $\begin{pmatrix} t & t^2 & \cdots & t^n \end{pmatrix} \mathcal{C}^{-1}$ to the left of $\overline{v}$. 

\Cref{lemma:nice_change_of_basis_for_burau_matrices} calculates the matrices $\mathcal{C}^{-1} \psi_n(\sigma_i) \mathcal{C}$, $\mathcal{C}^* H_{n,-1,1} \mathcal{C}$, and $\begin{pmatrix} t & t^2 & \cdots & t^n \end{pmatrix} \mathcal{C}^{-1}$; recall the notation $H_{n,a,b}$ introduced after \eqref{eq:hermitian_matrix_H_a_b}.

\begin{lemma} \label{lemma:nice_change_of_basis_for_burau_matrices}
    \begin{enumerate}
        \item $\det(\mathcal{C}) = \frac{(-1)^{n+1}\sum_{i=0}^{n-1} t^i}{t^{n(n-1)/2}}$ \label{part:determinant_of_change_of_basis_matrix}
        \item \label{part:change_of_basis_for_unreduced_burau_matrix}
        Over $\mathbb{Z}[t,t^{-1}, \det(\mathcal{C})^{-1}]$, 
            $$\mathcal{C}^{-1} \psi_n(\sigma_i) \mathcal{C} = \left( \begin{array}{c|c} 1 & 0 \\ \hline 0 & \psi_n^r(\sigma_i) \end{array} \right)$$
        for every $i = 1,\ldots,n-1$.
        \item  \label{part:change_of_basis_for_unitary_matrix}
        $$\mathcal{C}^* H_{n,-1,1} \mathcal{C} = \left( \begin{array}{c|c} S & 0 \\ \hline 0 & H_n^r \end{array} \right)$$
        where 
        \begin{equation} \label{eq:corner_entry_of_diagonalization_of_H_n}
            S = \left( \sum_{i=1}^{n-2} (n-i-1) t^i + (n-2) + \sum_{i=-n+2}^{-1} (n+i-1) t^i \right).
        \end{equation}
        Note that $S = \overline{S}$.
        \item \label{part:change_of_basis_for_boundary_monodromy}
        $\begin{pmatrix} \zeta_d & \zeta_d^2 & \cdots & \zeta_d^n \end{pmatrix} \mathcal{C} = \begin{pmatrix} \sum_{i=1}^n \zeta_d^i & 0 & \cdots & 0 \end{pmatrix}$. 
    \end{enumerate}
\end{lemma}

\begin{proof}
   One can verify \eqref{part:determinant_of_change_of_basis_matrix} by applying elementary column operators to $\mathcal{C}$ to transform $\mathcal{C}$ into an upper triangular matrix. One can also arithmetically verify \eqref{part:change_of_basis_for_unreduced_burau_matrix}, \eqref{part:change_of_basis_for_unitary_matrix}, and \eqref{part:change_of_basis_for_boundary_monodromy}. 
\end{proof}

Therefore, writing $\overline{v} = \begin{pmatrix} \alpha_1 & \cdots & \alpha_n \end{pmatrix}^T$, the action of $B_n$ via $\psi_n$ fixes the first coordinate $\alpha_1$ and is given by $\psi_n^r$ on the other coordinates, the boundary monodromy of $v$ is a multiple of $\alpha_1$ by an invertible element, and the norm of $v$ with respect to $H_n$ is $\alpha_1 \cdot \overline{\alpha}_1 \cdot S$, where $S$ is as in \Cref{lemma:nice_change_of_basis_for_burau_matrices}\eqref{part:change_of_basis_for_unitary_matrix}, plus the norm of $\begin{pmatrix} \alpha_2 & \cdots & \alpha_n \end{pmatrix}^T$ with respect to $H_n^r$. 


We also prove \Cref{lemma:sum_of_roots_of_unity_is_invertible} and \Cref{lemma:closed_form_formula_for_S} to later establish the invertibility of $\mathcal{C}$ and $H_{n,-1,1}$ over $\mathbb{Z}_\ell[\zeta_d]_{\zeta_d \neq 1}$ in certain circumstances.
\begin{lemma} \label{lemma:sum_of_roots_of_unity_is_invertible}
    Let $d \geq 2$ be an integer. Let $\mu_d$ be a primitive $d$th root of unity in $\overline{\mathbb{Q}}$ and suppose that $d \nmid n$. Then $1-\mu_d$ and $\sum_{i=0}^{n-1} \mu_d^i$ respectively divide $d$ and $\frac{d}{\gcd(d,n)}$ as elements of the ring $\mathbb{Z}[\mu_d]$.
\end{lemma}
\begin{proof}
    Given a primitive $e$th root of unity $\mu$, we have $X^{e-1} + \cdots + 1 = \prod_{i=1}^{e-1} (X - \mu^i)$. Evaluating at $X = 1$ yields $e = \prod_{i=1}^{e-1} (1-\mu^i)$. Letting $e = d$ proves that $1-\mu_d$ divides $d$. Letting $\mu = \mu_d^n$ so that $e = \frac{d}{\gcd(d,n)}$ yields $e = \prod_{i=1}^{e-1} (1-\mu^i)$. Since $(1-\mu) = (1-\mu_d) \cdot \sum_{i=0}^{n-1} \mu_d^i$, we thus have that $\sum_{i=0}^{n-1} \mu_d^i$ divides $e$.
\end{proof}

\begin{lemma} \label{lemma:closed_form_formula_for_S}
The expression $S$ in \eqref{eq:corner_entry_of_diagonalization_of_H_n} equals $\frac{t(1-t^{-n})}{(1-t^{-1})} \cdot \frac{1-t^{n-2}}{1-t}$.
\end{lemma}
\begin{proof}
    This holds because
    \begin{align*}
        S - t^{-1} S &= \left( \sum_{i=1}^{n-2} (n-i-1) t^i + (n-2) + \sum_{i=-n+2}^{-1} (n+i-1) t^i \right) \\
        &\phantom{=} - \left( \sum_{i=0}^{n-3} (n-i-2) t^i + (n-2)t^{-1} + \sum_{i=-n+1}^{-2} (n+i) t^i \right) \\
        &= \left( t^{n-2} + \sum_{i=1}^{n-3} (n-i-1) t^i + (n-2) + (n-2)t^{-1} + \sum_{i=-n+2}^{-2} (n+i-1) t^i \right) \\
        &\phantom{=} - \left( \sum_{i=1}^{n-3} (n-i-2) t^i + (n-2) + (n-2)t^{-1} + \sum_{i=-n+2}^{-2} (n+i) t^i + t^{-n+1} \right) \\
        &=  t^{n-2} + \sum_{i=1}^{n-3} t^i - \sum_{i=-n+2}^{-2} t^i - t^{-n+1} \\
        &= \sum_{i=1}^{n-2} t^i - \sum_{i=-n+1}^{-2} t^i \\
        &= (t-t^{-n+1}) \cdot \sum_{i=0}^{n-3} t^i.
    \end{align*}
\end{proof}

\Cref{lemma:global_monodromy_after_nice_change_of_basis} below translates \Cref{proposition:research_orbits_of_burau_representations_2_proposition_global_monodromy_for_c_n} and \Cref{corollary:global_monodromy_description} to describe the global monodromy of a vector in $\mathbb{Z}/\ell\mathbb{Z}[\zeta_d]^{\oplus n}$ with respect to the base change by $\mathcal{C}$.

\begin{lemma} \label{lemma:global_monodromy_after_nice_change_of_basis}
    Let $D \geq 2$ be an integer and let $\ell \nmid d$ be a prime number. Let $G = A \rtimes \langle \zeta_d \rangle$ where $A$ is a quotient of $\mathbb{Z}/\ell\mathbb{Z}[\zeta_d] = \mathbb{Z}/\ell\mathbb{Z}[\zeta_d]_{\zeta_d \neq 1}$. Factorize $\mathbb{Z}_\ell[\zeta_d]$ as the product $\prod_i B_i$ of DVR's similarly to \eqref{eq:Z_ell_zeta_d_factorization_as_ring}. Let $n$ be an integer such that $\gcd(d,n) = 1$. In particular, $\sum_{i=0}^{n-1} \zeta_d^i$ and $\mathcal{C}$ are invertible over $\mathbb{Z}/\ell\mathbb{Z}[\zeta_d]$ by \Cref{lemma:nice_change_of_basis_for_burau_matrices} and \Cref{lemma:sum_of_roots_of_unity_is_invertible}. Given $v_A = v_{A,\psi_n} = (a_1,\ldots,a_n)^T \in A^{\oplus n}$ written with respect to the standard basis of $A^{\oplus n}$, write $v_{A,\mathcal{C}} = (a_1',\ldots,a_n')^T$ for its base change under $\mathcal{C}$, i.e. $v_{A,\psi_n} = \mathcal{C} v_{A, \mathcal{C}}$. Write additional subscripts of $i$ on vectors over $A$ or elements of $A$ to  indicate $B_i/(\ell)$-components; for instance $v_{A,i}$ is the $B_i/(\ell)$-component of $v_{A}$ and $a_{1,i}$ is the $B_i/(\ell)$-component of $a_1$. 
    
    The following hold:
    \begin{enumerate}
        \item  $M_{G,1}(v_{A,i}) \subseteq B_i/(\ell)$ is $B_i/(\ell)$ if and only if at least one of $a'_{2,i},\ldots,a'_{n,i}$ is nonzero and is $0$ otherwise.
        \item $M_G(v_A)$ is determined (as per \Cref{proposition:how_subgroups_of_semidirect_products_are_determined}) by $M_{G,1}(v_{A})$ and 
        $$a'_1 = \left( \sum_{i=1}^n \zeta_d^i \right)^{-1} M_B(v_{A}).$$
    \end{enumerate}
    In particular, with $d,\ell,n$ fixed, $M_G(v_A)$ is determined exactly by $a'_1$ and by which among the $B_i/(\ell)$-components of $(a'_2,\ldots,a'_n)^T$ are zero. In fact, $M_{G,1}(v_A)$ is the $A$-module generated by $a'_2,\ldots,a'_n$. 
\end{lemma}
\begin{proof}
    We show (1). By \Cref{corollary:global_monodromy_description}, $M_{G,1}(v_{A,\psi_n, i})$ is $0$ if and only if the $B_i/(\ell)$-component of $v_{A,\psi_n}$ has all of the same coordinates. Since $v_{A,\psi_n} = \mathcal{C} v_{A,\mathcal{C}}$, this in turn happens exactly when  $(a'_{2,i},\ldots,a'_{n,i})^T$ is zero. Otherwise, $M_{G,1}(v_{A,\psi_n, i})$ is $B_i/(\ell)$.
    
    We now show (2). By \Cref{corollary:global_monodromy_description}, $M_G(v_A)$ is determined by $M_{G,1}(v_A)$ and $a_1$, which equals $a'_1 + a'_2$. If $(a_{2,i}',\ldots,a'_{n,i})^T$ is nonzero, then $M_{G,1}(v_{A,i}) = B_i/(\ell)$. Otherwise, $a_{1,i} = a'_{1,i}$. In either case, $a_{1,i}$ and $a'_{1,i}$ are in the same $M_{G,1}(v_{A,i})$-equivalence class, so $a'_1$ and $a_1$ are in the same $M_{G,1}(v_A)$-equivalence class. Therefore, $M_G(v_A)$ is determined by $M_{G,1}(v_A)$ and $a'_1$ as desired. Also recall from \Cref{lemma:nice_change_of_basis_for_burau_matrices} and the discussion preceding it that $M_B(v_A) = \sum_{i=1}^n \zeta_d^i a'_1$.
\end{proof}

\subsection{The Orbits of the unreduced Burau representation evaluated at \texorpdfstring{$t = \zeta_d$}{t = zetad} modulo \texorpdfstring{$\ell$}{ell} are determined by three invariants}

In \Cref{theorem:unreduced_burau_orbits_are_determined_by_monodromies_and_norm}, we show that the orbits of the unreduced Burau representations evaluated at $t = \zeta_d$ modulo $\ell$ for sufficiently large $n$ are determined by the three invariants --- the global monodromy, the boundary monodromy, and the norm --- of their elements.


Via \Cref{lemma:SU_orbits}, we first establish orbits of the action of $\mathrm{SU}$ to later use the fact (\Cref{subsection:a_nice_change_of_basis}) that the unreduced Burau representation is essentially almost given by the reduced Burau representation and to use Venkataramana's \cite{venkataramana_ibrd} results showing that the reduced Burau representation is an arithmetic subgroup of the appropriate unitary group. 

For each double division ring factor $D_i = B_i \times B_i'$, we use the orthogonal idempotent notation from \Cref{context:lemmas_splitting_proposition}: let $\pi_{i,1} = (1,0)$ and $\pi_{i,2} = (0,1)$, so that the involution swaps them via $\overline{\pi}_{i,k} = \pi_{i,3-k}$. For division ring factors, set $\pi_{i,1} = \pi_{i,2} = 1$. Any $D_i$-module decomposes as $M = \pi_{i,1}M \oplus \pi_{i,2}M$, and a Hermitian form satisfies $H(\pi_{i,k}u, \pi_{i,j}v) = 0$ unless $j = 3-k$.

\begin{lemma} \label{lemma:SU_orbits}
    Let $d \geq 3$ be an integer, let $\ell \nmid d$ be a prime number, and let $n \geq 4$ be an integer. Let $A = \mathbb{Z}/\ell\mathbb{Z}[\zeta_d] = \mathbb{Z}/\ell\mathbb{Z}[\zeta_d]_{\zeta_d \neq 1}$ or an involution ring quotient thereof. In particular, $A$ factors in the form $\prod_i D_i/(\ell)$ where the indices $i$ here run over a subset of the indices $i$ appearing in the factorization $\mathbb{Z}_\ell[\zeta_d] = \prod_i D_i$ recorded in  \eqref{eq:Z_ell_zeta_d_factorization_as_ring_with_involution} where the $D_i$ are division or double division rings with involution. Equip $A^{\oplus n}$ with a nonsingular $\varepsilon_{A}$-Hermitian form $H$ (see \Cref{subsection:notations}). Write $H_i$ for the $D_i/(\ell)$-component of $H$ as in \Cref{lemma:hermitian_modules_over_product_of_involution_rings_split_along_involution_ring_factors}; the $H_i$ are nonsingular. Whenever $D_i$ is a double division ring, write $D_i = B_i \times B_i'$ where $B_i' = B_i$ is an integral domain and the involution on $D_i$ is given by $(a,b) \mapsto (b,a)$. Let $v_1, v_2 \in A^{\oplus n}$ and write $v_{1,i}, v_{2,i}$ for the $D_i$-component of $v_1$ and $v_2$. 
    
    Then, $v_1$ and $v_2$ are in the same orbit under the action of $\mathrm{SU}(A^{\oplus n})$ if and only if exactly one of the following holds for each $i$:
    \begin{enumerate}
        \item $v_{1,i}$ and $v_{2,i}$ are both $0$,
        \item The norms $H_i(v_{1,i}, v_{1,i})$ and $H_i(v_{2,i}, v_{2,i})$ are equal and invertible elements of $D_i / (\ell)$.
        \item $D_i$ is an integral domain, $v_{1,i}$ and $v_{2,i}$ are both nonzero, and have zero norm. 
        \item $D_i$ is a double division ring, and for one of $k \in \{1,2\}$, both $v_{1,i}$ and $v_{2,i}$ have zero $\pi_{i,k}$-component and nonzero $\pi_{i,3-k}$-component.
        \item $D_i$ is a double division ring, $v_{1,i}$ and $v_{2,i}$ both have nonzero $\pi_{i,k}$-component for each $k \in \{1,2\}$, and both have zero norm.
    \end{enumerate}
\end{lemma}
\begin{proof}
    Since $\mathrm{SU}(A) \cong \prod_i \mathrm{SU}(D_i/(\ell)^{\oplus n})$, it suffices to prove that the properties described in (1)-(5) determine orbits of the action of $\mathrm{SU}(D_i/(\ell)^{\oplus n})$ on $D_i/(\ell)^{\oplus n}$. Moreover, note that the properties (of being $0$, having a specific norm, having zero $\pi_{i,k}$-component) are all preserved under $\mathrm{SU}(D_i/(\ell)^{\oplus n})$, so the ``only if'' direction holds. We now prove the ``if'' direction.

    First suppose that $\zeta_d = -1$ in $D_i$, which can only occur when $d$ is even. In this case, only (1) or (3) can occur. Clearly, $0$ forms its own orbit. If $v_{1,i}$ and $v_{2,i}$ are both nonzero, then they are in the same orbit under $\mathrm{SU}(D_i/(\ell)^{\oplus n}) \cong \mathrm{Sp}(D_i/(\ell)^{\oplus n})$.

    Now suppose that $\zeta_d \neq -1$ in $D_i$ instead. Again, $0$ forms its own orbit. To establish that (2) describes an orbit, suppose that the norms of $v_1,v_2 \in D_i/(\ell)^{\oplus n}$ are equal and invertible as elements of $D_i/(\ell)$ (in fact of $D_i/(\ell)^{\sigma}$). By \Cref{lemma:knus_qhfr_lemma_3.6.2}, $(D_i/(\ell)^{\oplus n}, H_i) = (\langle v_j \rangle, H_i|_{\langle v_j \rangle}) \perp (\langle v_j \rangle^\perp, H_i|_{\langle v_j \rangle^\perp})$ for $j = 1,2$. In particular, $(\langle v_j \rangle^\perp, H_i|_{\langle v_j \rangle^\perp})$ is nonsingular, so it has an orthonormal basis by \Cref{proposition:free_hermitian_spaces_over_Z_ell_zeta_d_have_orthonormal_bases}. Any automorphism $\varphi$ on $D_i/(\ell)^{\oplus n}$ sending $v_1$ to $v_2$ and an orthonormal basis of $(\langle v_1 \rangle^\perp, H_i|_{\langle v_1 \rangle^\perp})$ to an orthonormal basis of $(\langle v_2 \rangle^\perp, H_i|_{\langle v_2 \rangle^\perp})$ is an element of $\mathrm{U}(D_i^{\oplus n})$. Scaling one of the basis elements of $(\langle v_2 \rangle^\perp, H_i|_{\langle v_2 \rangle^\perp})$ by $\frac{1}{\det \varphi}$ and replacing $\varphi$ accordingly thus makes $\varphi$ an element of $\mathrm{SU}(D_i/(\ell)^{\oplus n})$ sending $v_1$ to $v_2$.  

    To establish that (3) and (5) describe orbits, now suppose that the conditions specified in (3) or (5) hold. Via \Cref{proposition:free_hermitian_spaces_over_Z_ell_zeta_d_have_orthonormal_bases}, identify an orthonormal basis of $H_i$ and write 
        $$v_{1,i} = \begin{pmatrix} a_1 \\ a_2 \\ \vdots \\ a_n \end{pmatrix}, \quad v_{2,i} = \begin{pmatrix} b_1 \\ b_2 \\ \vdots \\  b_n \end{pmatrix}$$
    with respect to the orthonormal basis. In the case that $D_i$ is a double division ring, replace $v_{1,i}$ and $v_{2,i}$ if necessary with some vectors in the same $\mathrm{SU}$-orbit so that at least one coordinate in each is in $D_i/(\ell)^\times$. 
    Without loss of generality, say that $a_1, b_1 \in D_i/(\ell)^\times$. Letting $w = \begin{pmatrix} -a_1 & 0 & \cdots & 0 \end{pmatrix}^T$, note that $w \perp (v+w)$ and that the norms of $w$ and $v+w$ are nonzero and additive inverses of each other. By \Cref{lemma:every_invertible_element_of_the_subring_of_Z_ell_zeta_d_fixed_by_the_involution_is_the_norm_of_some_element_of_Z_ell_zeta_d}, there is some $u \in D_i/(\ell)^\times$ such that $u \cdot \overline{u} = -1$. Letting $y_w := \begin{pmatrix} -a_1 & 0 & 0 & \cdots & 0 \end{pmatrix}^T$ and $y_{v+w} := \begin{pmatrix} 0 & ua_1 & 0 & \cdots & 0 \end{pmatrix}^T$, note that $y_w \perp y_{v+w}$, that $w$ and $y_w$ have the same norms and that $v+w$ and $y_{v+w}$ have the same norms. Using \Cref{lemma:knus_qhfr_lemma_3.6.2} and \Cref{proposition:free_hermitian_spaces_over_Z_ell_zeta_d_have_orthonormal_bases} similarly as before, we can add orthonormal bases of $\langle w, v+w \rangle^\perp$ and $\langle y_w, y_{v+w} \rangle^\perp$ to the lists $(w, v+w)$ and $(y_w, y_{v+w})$ respectively to obtain orthogonal bases of $D_i/(\ell)^{\oplus n}$, then let $\varphi$ be the automorphism of $D_i/(\ell)^{\oplus n}$ sending the former orthogonal basis to the latter. By replacing $u$ with an appropriate scalar multiple (alternatively, by scaling one of the added orthonormal basis vectors), $\varphi$ in fact becomes an element of $\mathrm{SU}(D_i/(\ell)^{\oplus n})$. Thus, $v_{1,i}$ is in the same orbit as $y_{v+w} - y_w = \begin{pmatrix} a_1 & ua_1 & 0 & \cdots & 0 \end{pmatrix}^T$. Similarly, $v_{2,i}$ is in the same orbit as $\begin{pmatrix} b_1 & u'b_1 & 0 & \cdots & 0 \end{pmatrix}$ for some $u' \in D_i/(\ell)^\times$ assuming, without loss of generality, that $b_1 \neq 0$. In fact, just as argued before, $\begin{pmatrix} a_1 & ua_1 & 0 & \cdots & 0 \end{pmatrix}^T$ is in the same orbit as $\begin{pmatrix} a_1 & ua_1 & b_1 & u'b_1 & \cdots & 0 \end{pmatrix}^T$, which is also in the same orbit as $\begin{pmatrix} b_1 & u'b_1 & 0 & \cdots & 0 \end{pmatrix}$. Thus, $v_{1,i}$ and $v_{2,i}$ are in the same orbit as desired.

    To establish that (4) describes an orbit, suppose that $D_i$ is a double division ring and, without loss of generality, that $v_{1,i}$ and $v_{2,i}$ both have nonzero $\pi_{i,1}$-component and zero $\pi_{i,2}$-component. Note that $\mathrm{SU}(D_i/(\ell)^{\oplus n}) \subset \operatorname{SL}(\pi_{i,1} D_i/(\ell) ^{\oplus n}) \times \operatorname{SL}(\pi_{i,2} D_i/(\ell) ^{\oplus n})$ is the image of $\operatorname{SL}(\pi_{i,1} D_i/(\ell) ^{\oplus n})$ under the embedding $\tau \mapsto \tau \times (\tau^T)^{-1}$. In particular, there is some element of $\operatorname{SL}(\pi_{i,1} D_i/(\ell) ^{\oplus n})$ taking the $\pi_{i,1}$-component of $v_{1,i}$ to that of $v_{2,i}$ and hence there is an element of $\mathrm{SU}(D_i/(\ell)^{\oplus n})$ taking $v_{1,i}$ to $v_{2,i}$.
\end{proof}

We induct on $n$ to prove that the orbits of the unreduced Burau representation on $\mathbb{Z}/\ell\mathbb{Z}[\zeta_d]^{\oplus n}$ are determined by the three invariants. \Cref{proposition:orbits_of_unreduced_burau_representation_1} serves as the base case to this induction.

\begin{proposition} \label{proposition:orbits_of_unreduced_burau_representation_1}
   Let $d \geq 3$ be an integer, let $\ell \nmid d$ be a prime number, and let $A$ be an involution ring quotient of $\mathbb{Z}/\ell\mathbb{Z}[\zeta_d] = \mathbb{Z}/\ell\mathbb{Z}[\zeta_d]_{\zeta_d \neq 1}$. Suppose that $n \geq 4$ is an integer such that $\gcd(d,n(n-2)) = 1$ and the image of $\psi_n^r$ evaluated at $\zeta_d \in A$ modulo $\ell$ contains $\mathrm{SU}(A^{\oplus (n-1)})$ where $A^{\oplus(n-1)}$ is equipped with the Hermitian form $\tilde{h}_n^r$ (see \Cref{subsection:unreduced_burau_representation_is_unitary}). 

   The orbit of any element $v_A \in A^{\oplus n}$ under the action of $B_n$ via $\psi_n$ evaluated at $\zeta_d$ modulo $\ell$ is determined precisely by the global monodromy $M_G(v_A)$, the boundary monodromy $M_B(v_A)$ (\Cref{definition:global_and_boundary_monodromy_of_vector}), and norm under $h$ base changed to $A$ as described in \Cref{subsection:unreduced_burau_representation_is_unitary}, where $h$ is given by the unitary matrix $H_{n,-1,1}$.  
\end{proposition}

\begin{remark}
    Suppose that $A = \mathbb{Z}/\ell\mathbb{Z}[X]/(X+1)$. For $n \geq 4$, the orbits of $v_A \in A^{\oplus n}$ under the action of $B_n$ via $\psi_n$ evaluated at $\zeta_d$ modulo $\ell$ are determined precisely by the global monodromy $M_G(v_A)$ and the boundary monodromy $M_B(v_A)$, cf. \cite[Proof of Proposition 2.7]{ellenberg_li_shusterman}. Indeed, since $\zeta_d = -1$ and since the involution on $A$ is trivial, one can compute
        $$v_A^* H_{n,-1,1} v_A = - M_B(v_A)^2.$$
    so the boundary monodromy determines the norm of $v_A$ under the Hermitian form given by $H_{n,-1,1}$.
\end{remark}

\begin{proof}
    Factorize $A = \prod_i D_i/(\ell)$ as in \Cref{lemma:SU_orbits} and write $D_i = B_i$ if $D_i$ is an integral domain and $D_i = B_i \times B_i'$ if $D_i$ is a double division ring. 
    
    \Cref{research_orbits_of_burau_representations_2:proposition_braiding_action_preserves_global_and_boundary_monodromy} and \Cref{proposition:research_orbits_of_burau_representations_2_boundary_monodromy_of_c_n_elements_corresponds_to_a_linear_invariant_over_A} show that any two elements in the same orbit have the same global and boundary monodromies. The discussion in \Cref{subsection:unreduced_burau_representation_is_unitary} describes that the action of $B_n$ preserves the norm under $h$. Conversely, we show that any two elements of $A^{\oplus n}$ with the same monodromies and norm are in the same orbit.

    Since $\ell \nmid d$ and since $\gcd(d,n) = 1$, $\sum_{i=1}^n \zeta_d^i$ is invertible in $A$ by \Cref{lemma:nice_change_of_basis_for_burau_matrices}. In particular, the matrix $\mathcal{C}$, evaluated at $t = \zeta_d$, specified in \eqref{eq:nice_change_of_basis_matrix_for_unreduced_burau_representation} is invertible by \Cref{lemma:sum_of_roots_of_unity_is_invertible}. 
    Recall from \Cref{subsection:a_nice_change_of_basis} that $\overline{v}$ denotes the coordinate vector of $v$ in the $\mathcal{C}$-basis, that the action of $\sigma_i \in B_n$ on $\overline{v}$ is given by $\mathcal{C}^{-1} \psi_n(\sigma_i) \mathcal{C} = \left( \begin{array}{c|c} 1 & 0 \\ \hline 0 & \psi_n^r(\sigma_i) \end{array} \right)$, that the boundary monodromy function is given by the multiplication by $\begin{pmatrix} \zeta_d & \zeta_d^2 & \cdots & \zeta_d^n \end{pmatrix} \mathcal{C}^{-1} = \begin{pmatrix} \sum_{i=1}^n \zeta_d^i & 0 & \cdots & 0 \end{pmatrix}$ to the left of $\overline{v}$, and that $\mathcal{C}^* H_{n,-1,1} \mathcal{C} = \left( \begin{array}{c|c} S & 0 \\ \hline 0 & H_n^r \end{array} \right)$ is the matrix of $h$ under the base change by $\mathcal{C}$.
    By \Cref{lemma:closed_form_formula_for_S} and \Cref{lemma:sum_of_roots_of_unity_is_invertible} and since $\gcd(d,n(n-2)) = 1$, $S$ is invertible in $A$. For the rest of this proof, let $H_n = H_{n,-1,1}$.
    
    Writing $\overline{v} = \begin{pmatrix} \alpha_1 & \cdots & \alpha_n \end{pmatrix}^T$,
    further recall that the action of $B_n$ via $\psi_n$ fixes the first coordinate $\alpha_1$ and is given by $\psi_n^r$ on the other coordinates, the boundary monodromy of $v$ is a multiple of $\alpha_1$ by an invertible element, and the norm of $v$ with respect to $H_n$ is $S \cdot \alpha_1 \cdot \overline{\alpha}_1$ plus the norm of $\begin{pmatrix} \alpha_2 & \cdots & \alpha_n \end{pmatrix}^T$ with respect to $H_n^r$. 

    Now say that $v, w \in A^{\oplus n}$ have the same monodromies and norm and write $\overline{v} = \begin{pmatrix} \alpha_1 & \cdots & \alpha_n \end{pmatrix}^T, \overline{w} = \begin{pmatrix} \beta_1 & \cdots & \beta_n \end{pmatrix}^T$. It suffices to show that $\overline{v}$ and $\overline{w}$ are in the same orbit under the action of the group of matrices of the form
        $$\left( \begin{array}{c|c} 1 & 0 \\ \hline 0 & \psi_n^r(\sigma) \end{array} \right),$$
    with $\sigma \in B_n$. Since $v$ and $w$ have the same boundary monodromy, $\alpha_1 = \beta_1$. Moreover, $\overline{v}' := \begin{pmatrix} \alpha_2 & \cdots & \alpha_n \end{pmatrix}^T$ and $\overline{w}' := \begin{pmatrix} \beta_2 & \cdots & \beta_n \end{pmatrix}^T$ have the same norm given by $H_n^r$ since $\overline{v}$ and $\overline{w}$ have the same norm given by $\mathcal{C}^* H_n \mathcal{C}$. Equivalently, $\overline{v}'$ and $\overline{w}'$ have the same $\tilde{h}_n^r$-norm because $\tilde{h}_n^r$ is given by $H_n^r$ over factors $D_i$ where $\zeta_d \neq -1$ and $\tilde{h}_n^r$ is a symplectic form over factors $D_i$ where $\zeta_d = -1$ and $H_n^r$ evaluates to $0$ at $t = \zeta_d = -1$. 
    
    By \cite[Lemma 13]{venkataramana_ibrd}, $\det H_n^r$, before evaluating $t$, is $\left( \frac{t+1}{t} \right)^{n-1} \left( \frac{t^n-1}{t-1} \right)$. Thus, the Hermitian form given by $H_n^r$ is nonsingular exactly when $t+1$ and $\frac{t^n-1}{t-1} = \sum_{i=0}^{n-1} t^i$ evaluate to be invertible values. The only $D_i$ on which $t+1$ is not invertible is $D_i = \mathbb{Z}/\ell\mathbb{Z}[X]/(X+1)$, which is a factor of $\mathbb{Z}_\ell[\zeta_d]$ only when $d$ is even. In this case, $n$ must be odd since $\gcd(d,n(n-2))$ is assumed to be $1$. In particular, $\det \tilde{H}_n^r = 1$ by \eqref{eq:determinant_of_tilde_H_n_r}, so $\tilde{h}_n^r$ is nonsingular.

    By \Cref{lemma:global_monodromy_after_nice_change_of_basis}, when a vector $v$ has fixed boundary monodromy, its global monodromy is determined exactly by the sets
    \begin{align*}
        &\{i : \pi_{i,1}(\overline{v}') \neq 0\} \\
        &\{i : D_i = B_i \times B_i' \text{ and } \pi_{i,2}(\overline{v}') \neq 0\}. 
    \end{align*}
    For each factor $D_i$, write $\overline{v}^{\prime (i)}$ and $\overline{w}^{\prime (i)}$ for the $D_i$-components of $\overline{v}^{\prime}$ and $\overline{w}^{\prime}$. In particular, $\overline{v}^{\prime (i)} = 0$ if and only if $\overline{w}^{\prime (i)} = 0$. 
    Similarly, when $D_i = B_i \times B_i'$, we have $\pi_{i,k} \overline{v}^{\prime (i)} = 0$ if and only if $\pi_{i,k} \overline{w}^{\prime (i)}= 0$.
    \Cref{lemma:SU_orbits} hence shows that $\overline{v}^{\prime}$ and $\overline{w}^{\prime}$ are in the same orbit under the action of $\mathrm{SU}(A^{\oplus (n-1)})$.
\end{proof}

\Cref{lemma:each_orbit_of_unreduced_burau_representation_has_an_element_with_a_specified_nice_coordinate} is the main idea driving the inductive step of the aforementioned induction. It establishes that each orbit of the unreduced Burau representation has a vector with a specified, ``nice'' first coordinate and that truncating this first coordinate preserves the global monodromy. If two vectors have the same monodromies, norm, and first coordinate, then \Cref{lemma:how_norm_changes_by_appending_entries} then shows that truncating the first coordinates from the two vectors results in vectors with the same boundary monodromy and norm. 

From now on, given a (column) vector $v = \begin{pmatrix} a_1 & \cdots & a_n \end{pmatrix}^T \in A^{\oplus n}$, we will generally write $v'$ for the truncated vector $v' = \begin{pmatrix} a_2 & \cdots & a_n \end{pmatrix}^T \in A^{\oplus {n-1}}$.

\begin{lemma} \label{lemma:each_orbit_of_unreduced_burau_representation_has_an_element_with_a_specified_nice_coordinate}
    Let $d \geq 3$ be an integer, let $\ell \nmid d$ be a prime number, let $A$ be an involution ring quotient of $\mathbb{Z}/\ell\mathbb{Z}[\zeta_d] = \mathbb{Z}/\ell\mathbb{Z}[\zeta_d]_{\zeta_d \neq 1}$, and let $n \geq 6$ be an integer such that $\gcd(d,n) = 1$. 
    Say that $v = \begin{pmatrix} a_1 & \cdots & a_n \end{pmatrix}^T \in A^{\oplus n}$ is a column vector in standard basis. 
    Write $D_i$ and $B_i,B_i'$ as in the proof of \Cref{proposition:orbits_of_unreduced_burau_representation_1}. Write $M_{G,i} = M_{G,1}(v^{(i)}) \subseteq D_i$ for each $i$. 
    
    There is some column vector $w = \begin{pmatrix} b_1 & \cdots & b_n \end{pmatrix}^T$ such that, writing $w' = \begin{pmatrix} b_2 & b_3 & \cdots & b_n \end{pmatrix}^T$,
    \begin{itemize}
        \item $M_G(v) = M_G(w) = M_G(w')$,
        \item $M_B(v) = M_B(w)$,
        \item $N(v) = N(w)$, where $N$ is the norm with respect to $H_{n,-1,1}$, and
        \item $b_{1,i}$ is 
        \begin{equation} \label{eq:description_of_zeros_in_D_i}
            b_i := \begin{cases} 
                0 &\text{if } M_{G,i} = D_i \\
                \pi_{i,2}\beta' &\text{if } M_{G,i} = \pi_{i,1}D_i \\
                \pi_{i,1}\beta &\text{if } M_{G,i} = \pi_{i,2}D_i, \\
               \gamma &\text{if } M_{G,i} = 0
                \end{cases}
            \end{equation}
        for every $i$; here, we let $\gamma \in D_i$ be any $M_{G,i}$-congruence class representative such that $M_G(v^{(i)})$ is determined (in the sense of \Cref{proposition:how_subgroups_of_semidirect_products_are_determined}) by $M_{G,i}$ and $\gamma$, and we write $\gamma = (\beta,\beta')$ if $D_i = B_i \times B_i'$. 
    \end{itemize}
\end{lemma}

\begin{remark}
    Like \Cref{proposition:orbits_of_unreduced_burau_representation_1}, \Cref{lemma:each_orbit_of_unreduced_burau_representation_has_an_element_with_a_specified_nice_coordinate} requires $d$ and $n$ to be relatively prime because both facts require the change of basis by $\mathcal{C}$. However, unlike \Cref{proposition:orbits_of_unreduced_burau_representation_1}, \Cref{lemma:each_orbit_of_unreduced_burau_representation_has_an_element_with_a_specified_nice_coordinate} does not require $d$ and $n-2$ to be relatively prime --- whereas the former requires the expression $S$ introduced in \Cref{lemma:nice_change_of_basis_for_burau_matrices} to be invertible, the latter does not.
\end{remark}

\begin{proof}
     Let $w \in A^{\oplus n}$. Write $v^{(i)}$ for the $D_i/(\ell)$-component of $v$. Write 
     $N(v^{(i)}) = (v^{(i)})^* H_{n,-1,1} v^{(i)} \in D_i^{\sigma}$ norm of $v^{(i)}$. Write $M_{B,i} \in D_i$ for $\left( \sum_{i=1}^{n} \zeta_d^i \right)^{-1} \cdot M_B(v^{(i)})$, which is well defined by \Cref{lemma:sum_of_roots_of_unity_is_invertible}. If $v^{(i)}$ has all the same coordinates, then $M_{G}(v^{(i)}) \cap (D_i/(\ell) \rtimes \langle 1 \rangle) = 0 \rtimes \langle 1 \rangle$ by \Cref{corollary:global_monodromy_description}. In this case, $M_G(v^{(i)}) = M_G(w^{(i)})$ if and only if $v^{(i)} = w^{(i)}$.

    Now suppose that $v^{(i)}$ has at least two distinct coordinates. If $D_i = B_i \times B_i'$ and $M_{G,i} = \pi_{i,1}D_i$, then $\beta'$ in the definition \eqref{eq:description_of_zeros_in_D_i} equals $\pi_{i,2}a_{j,i}$, independently of $j$. In this case, $M_G(v^{(i)}) = M_G(w^{(i)})$ if and only if $\pi_{i,2}v^{(i)}$ and $\pi_{i,2}w^{(i)}$ coincide. By symmetry, analogous facts are true when $M_{G,i} = \pi_{i,2}D_i$.  
    \begin{table}[h]
    \begin{tabular}{|c|c|c|c|c|}
    \hline
    $M_{B,i}$ & $N(v^{(i)})$ & $M_{G,i}$ & $\overline{w}^{(i)}$ & $w^{(i)}$ \\ \hline
    Invertible & Any & $D_i$ & 
    $\begin{pmatrix} M_{B,i} \\ -M_{B,i} \\ 0 \\ \alpha \end{pmatrix}$ & $\begin{pmatrix} 0 \\ (1-\zeta_d^{-1}) M_{B,i} \\ M_{B,i} + \alpha \zeta_d^{-2} \\ M_{B,i} - \alpha \zeta_d^{-3} \end{pmatrix}$ \\  \hline
    $0$ & $\neq 0$ & $D_i$ & $\begin{pmatrix} 0 \\ 0 \\ \alpha \end{pmatrix}$ & $\begin{pmatrix} 0 \\\alpha \zeta_d^{-1} \\ -\alpha \zeta_d^{-2} \end{pmatrix}$ \\ \hline
    $0$ & $0$ & $D_i$ & $\begin{pmatrix} 0 \\ 0 \\ 1 \\ 0 \\ \alpha \end{pmatrix}$ & $\begin{pmatrix} 0 \\ \zeta_d^{-1} \\ -\zeta_d^{-2} \\ \alpha \zeta_d^{-3} \\ -\alpha \zeta_d^{-4}  \end{pmatrix}$ \\  \hline
    \end{tabular}
    \caption{Descriptions of the first several coordinates of $\overline{w}^{(i)}$ such that the first coordinate of $w^{(i)}$ is $b_i = 0$ when $D_i$ is an integral domain. The unwritten trailing coordinates of $\overline{w}^{(i)}$ are all $0$. Since $w^{(i)} = \mathcal{C} \overline{w}^{(i)}$, the unwritten trailing coordinates of $w^{(i)}$ must be all $M_{B,i}$. Where appropriate, $\alpha \in D_i$ is chosen via \Cref{lemma:every_invertible_element_of_the_subring_of_Z_ell_zeta_d_fixed_by_the_involution_is_the_norm_of_some_element_of_Z_ell_zeta_d} to ensure that the norm of $\overline{w}^{(i)}$ equals $N(v^{(i)})$.}
    \label{table:vectors_with_zero_coordinate_over_domains}
    \end{table}

    \begin{table}[h]
    \begin{tabular}{|c|c|c|c|c|}
    \hline
    $M_{B,i}$ & $N(v^{(i)})$ & $M_{G,i}$ & $\overline{w}^{(i)}$ & $w^{(i)}$ \\ \hline
    Invertible & Any & $D_i$ & $\begin{pmatrix} M_{B,i} \\ -M_{B,i} \\ 0 \\ \alpha \end{pmatrix}$ & $\begin{pmatrix} 0 \\ (1-\zeta_d^{-1}) M_{B,i} \\ M_{B,i} + \alpha \zeta_d^{-2} \\ M_{B,i} - \alpha \zeta_d^{-3} \end{pmatrix}$ \\ \hline
   $(\beta,\beta')$ Invertible & $S \cdot M_{B,i} \cdot \overline{M_{B,i}}$ & $\pi_{i,1}D_i$ & $\begin{pmatrix} M_{B,i} \\ (-\beta,0) \end{pmatrix}$ & $\begin{pmatrix} (0,\beta') \\ (\beta,\beta') + \zeta_d^{-1}(\beta,0) \end{pmatrix}$ \\ \hline
    Invertible & $\neq S \cdot M_{B,i} \cdot \overline{M_{B,i}}$ & $\pi_{i,1}D_i$ & N/A & N/A \\ \hline
    $0$ & $\neq 0$ (i.e. invertible) & $D_i$ & $\begin{pmatrix} 0 \\ 0 \\ \alpha \end{pmatrix}$ & $\begin{pmatrix} 0 \\ \alpha \zeta_d^{-1} \\ -\alpha \zeta_d^{-2} \end{pmatrix}$ \\ \hline
    $0$ & $\neq 0$ & $\pi_{i,1}D_i$ & N/A & N/A \\ \hline
    $0$ & $0$ & $D_i$ & $\begin{pmatrix} 0 \\ 0 \\ 1 \\ 0 \\ \alpha \end{pmatrix}$ & $\begin{pmatrix} 0 \\ \zeta_d^{-1} \\ -\zeta_d^{-2} \\ \alpha \zeta_d^{-3} \\ -\alpha \zeta_d^{-4} \end{pmatrix}$ \\ \hline
    $0$ & $0$ & $\pi_{i,1}D_i$ & $\begin{pmatrix} 0 \\ 0 \\ (1,0) \end{pmatrix}$ & $\begin{pmatrix} 0 \\ (1,0) \zeta_d^{-1} \\ -(1,0) \zeta_d^{-2} \end{pmatrix}$ \\ \hline
    \end{tabular}
    \caption{
    Descriptions of the first several coordinates of $\overline{w}^{(i)}$ such that the first coordinate of $w^{(i)}$ is $b_i$ when $D_i = B_i \times B_i'$ is a double division ring and $M_{B,i}$ is either invertible or $0$. The unwritten trailing coordinates of $\overline{w}^{(i)}$ are all $0$. Since $w^{(i)} = \mathcal{C} \overline{w}^{(i)}$, the unwritten trailing coordinates of $w^{(i)}$ must be all $M_{B,i}$. Where appropriate, $\alpha \in D_i$ is chosen via \Cref{lemma:every_invertible_element_of_the_subring_of_Z_ell_zeta_d_fixed_by_the_involution_is_the_norm_of_some_element_of_Z_ell_zeta_d} to ensure that the norm of $\overline{w}^{(i)}$ equals $N(v^{(i)})$.}
    
    \label{table:elements_with_2}
    \end{table}

    \begin{table}[h]
    \begin{tabular}{|c|c|c|c|c|}
    \hline
    $M_{B,i}$ & $N(v^{(i)})$ & $M_{G,i}$ & $\overline{w}^{(i)}$ & $w^{(i)}$ \\ \hline
    $(\beta,0)$ & $\neq 0$ & $D_i$ & $\begin{pmatrix} M_{B,i} \\ -M_{B,i} \\ 0 \\ \alpha \end{pmatrix}$ & $\begin{pmatrix} 0 \\ (1-\zeta_d^{-1}) M_{B,i} \\ M_{B,i} + \alpha \zeta_d^{-2} \\ M_{B,i} - \alpha \zeta_d^{-3} \end{pmatrix}$ \\ \hline
    $(\beta,0)$ & $\neq 0$ & $\pi_{i,1}D_i$ or $\pi_{i,2}D_i$ & N/A & N/A \\ \hline
    $(\beta,0)$ & $0$ & $D_i$ & $\begin{pmatrix} M_{B,i} \\ -M_{B,i} \\ 0 \\ 1 \\ 0 \\ \alpha \end{pmatrix}$ & $\begin{pmatrix} 0 \\ (1+\zeta_d^{-1}) M_{B,i} \\ M_{B,i} + \zeta_d^{-2} \\ M_{B,i} - \zeta_d^{-3} \\ M_{B,i} + \alpha \zeta_d^{-4} \\ M_{B,i} - \alpha \zeta_d^{-5} \end{pmatrix}$ \\ \hline
    $(\beta,0)$ & $0$ & $\pi_{i,1}D_i$ & $\begin{pmatrix} M_{B,i} \\ -M_{B,i} \end{pmatrix}$ & $\begin{pmatrix} 0 \\ (1-\zeta_d^{-1}) M_{B,i} \end{pmatrix}$ \\ \hline
    $(\beta,0)$ & $0$ & $\pi_{i,2}D_i$ & $\begin{pmatrix} M_{B,i} \\ 0 \\ (0,1) \end{pmatrix}$ & $\begin{pmatrix} (\beta,0) \\ (\beta,0) + (0,1) \zeta_d^{-1} \\ (\beta,0) - (0,1) \zeta_d^{-2}\end{pmatrix}$ \\ \hline
    \end{tabular}
    \caption{
    Descriptions of the first several coordinates of $\overline{w}^{(i)}$ such that the first coordinate of $w^{(i)}$ is $b_i$ when $D_i = B_i \times B_i'$ is a double division ring and $M_{B,i}$ is neither invertible nor $0$; without loss of generality, say that $M_{B,i}$ is of the form $(b,0) \in B_i/(\ell) \times B_i' / (\ell)$. The unwritten trailing coordinates of $\overline{w}^{(i)}$ are all $0$. Since $w^{(i)} = \mathcal{C} \overline{w}^{(i)}$, the unwritten trailing coordinates of $w^{(i)}$ must be all $M_{B,i}$. Where appropriate, $\alpha \in D_i$ is chosen via \Cref{lemma:every_invertible_element_of_the_subring_of_Z_ell_zeta_d_fixed_by_the_involution_is_the_norm_of_some_element_of_Z_ell_zeta_d} to ensure that the norm of $\overline{w}^{(i)}$ equals $N(v^{(i)})$.}
    \label{table:elements_with_3}
    \end{table}

    Tables \ref{table:vectors_with_zero_coordinate_over_domains}, \ref{table:elements_with_2}, and \ref{table:elements_with_3} list (the first several coordinates of the) column vectors $\overline{w}^{(i)}$ corresponding to desired $w^{(i)}$ under the base change by $\mathcal{C}$, i.e. $w^{(i)}$ has the same monodromies and norm as $v^{(i)}$ and $\overline{w}^{(i)} = \mathcal{C}^{-1} w^{(i)}$. Recall from \Cref{subsection:a_nice_change_of_basis} that the first coordinate of $\overline{w}^{(i)}$ is exactly $M_{B,i}$. Furthermore, the first coordinate of $w^{(i)} =\mathcal{C} \overline{w}^{(i)}$ is $0$ whenever the second coordinate in $\overline{w}^{(i)}$ is $-M_{B,i}$. In the other cases, the first coordinate of $w^{(i)}$ is $b_i$ as specified in \Cref{eq:description_of_zeros_in_D_i}:
    \begin{itemize}
        \item In the second row of \Cref{table:elements_with_2}, i.e. when $M_{B,i} = (\beta,\beta')$ is invertible, $N(v^{(i)}) = S \cdot M_{B,i} \cdot \overline{M_{B,i}}$, and $M_{G,i} = \pi_{i,1}D_i$, we have $\overline{w}^{(i)} = \begin{pmatrix} M_{B,i} & -\pi_{i,1}\beta & 0 & \cdots & 0 \end{pmatrix}^T$, so the first coordinate of $w^{(i)}$ is $\pi_{i,2}\beta'$. 
        \item In the fifth row of \Cref{table:elements_with_3}, i.e. when $M_{B,i} = \pi_{i,1}\beta$, $N(v^{(i)}) = 0$, and $M_{G,i} = \pi_{i,2}D_i$, we have $\overline{w}^{(i)} = \begin{pmatrix} M_{B,i} & 0 & \pi_{i,2} & 0 & \cdots & 0 \end{pmatrix}^T$, so the first coordinate of $w^{(i)}$ is $\pi_{i,1}\beta$.
    \end{itemize}

    We describe why the choices of $\overline{w}^{(i)}$ (can be made to) have norm $N(v^{(i)})$. Note that the norm of $\overline{w}^{(i)}$ is $(\overline{w}^{(i)})^* \cdot (\mathcal{C}^* H_{n,-1,1} \mathcal{C}) \overline{w}^{(i)} =  (\overline{w}^{(i)})^* \left( \begin{array}{c|c} S & 0 \\ \hline 0 & H_n^r \end{array} \right) \overline{w}^{(i)}$. 
    It is also convenient to recall that a vector with zero $\pi_{i,k}$-component pairs to $0$ under the Hermitian form by orthogonality of idempotent factors, and hence has zero norm.
    Moreover, note the use of the symbols $\alpha$ in the choices of $\overline{w}^{(i)}$ in multiple cases listed in the tables --- in each of these cases, $\alpha$ is meant to be an element of $D_i$ such that the norm of $w^{(i)}$ is $N(v^{(i)})$. If $D_i = \mathbb{Z}/\ell\mathbb{Z}[X]/(X+1)$, then $H_n^r = 0$, so the norm is $S \cdot M_{B,i} \cdot \overline{M_{B,i}}$. In this case, the norm is determined by $M_{B,i}$, which is the first coordinate of $\overline{w}^{(i)}$; choose $\alpha = 0$ for instance. Otherwise, $\alpha$ is chosen either to be $0$ or to be an invertible element of $D_i$ via \Cref{lemma:every_invertible_element_of_the_subring_of_Z_ell_zeta_d_fixed_by_the_involution_is_the_norm_of_some_element_of_Z_ell_zeta_d}. For example, $\overline{w}^{(i)}$ is set to be of the form $\overline{w}^{(i)} = \begin{pmatrix} M_{B,i} & -M_{B,i} & 0 & \alpha & 0 & 0 & \cdots & 0 \end{pmatrix}^T$ in the first row of each table. The norm of such a $\overline{w}^{(i)}$ is 
        $$(\overline{w}^{(i)})^* \left( \begin{array}{c|c} S & 0 \\ \hline 0 & H_n^r \end{array} \right) \overline{w}^{(i)} = S \cdot M_{B,i} \cdot \overline{M_{B,i}} + (-M_{B,i}) \cdot (-\overline{M_{B,i}}) \cdot \frac{(\zeta_d+1)^2}{\zeta_d} + \alpha \cdot \overline{\alpha} \cdot \frac{(\zeta_d+1)^2}{\zeta_d}.$$
    Since norms in this context are elements of $D_i^{\sigma}$ and hence are either $0$ or invertible, and since $(\zeta_d+1)$ is invertible by \Cref{lemma:sum_of_roots_of_unity_is_invertible}, an $\alpha \in D_i$ making the above expression equal to $N(v^{(i)})$ exists.
    
    The tables contain some rows that correspond to impossible combinations of $M_{B,i}$, $N(v^{(i)})$, and $M_{G,i}$. We remark on why these combinations are impossible. The combination specified by the third row of \Cref{table:elements_with_2} cannot occur because if $M_{G,i}$ is $\pi_{i,1} D_i$ or $\pi_{i,2} D_i$, then the norm of $\overline{w}^{(i)}$ must be $S \cdot M_B \cdot \overline{M_B} + 0 = S \cdot M_B \cdot \overline{M_B}$. Similarly, 
    combinations specified by the fifth row of \Cref{table:elements_with_2} and by the second row of \Cref{table:elements_with_3} similarly cannot occur because the norm of $\overline{w}^{(i)}$ must be $0$ in these cases.

    Recall from \Cref{lemma:global_monodromy_after_nice_change_of_basis} that the global monodromy of $w$ is determined exactly by $M_{G,1}(w)$ and $(\sum_{i=1}^n \zeta_d^i)^{-1} M_B(w) = (M_{B,i})_i$. The former is the $D_i/(\ell)$-module generated by the coordinates of $\overline{w}^{(i)}$ except for the first one. In each case as listed in the tables, this $D_i/(\ell)$-module is equal to $M_{G,i}$.

    Finally, we describe why $w'$ has the same global monodromy as $w$. By \Cref{corollary:global_monodromy_description}, $M_{G,i}$ is the $D_i/(\ell)$-module generated by the differences between the coordinates of $w^{(i)}$. Similarly, $M_{G,1}((w^{(i)})')$ is the $D_i/(\ell)$-module generated by the differences between the coordinates of $(w^{(i)})'$. In each case as listed in the tables, $M_{G,i}$ coincides with $M_{G,1}((w^{(i)})')$. The following are justifications to this claim, divided amongst similar cases described in the tables:
    \begin{itemize}
        \item  In the first rows of each table, 
            $$w^{(i)} = \begin{pmatrix} 0 & (1-\zeta_d^{-1}) M_{B,i} & M_{B,i} + \alpha \zeta_d^{-2} & M_{B,i} - \alpha \zeta_d^{-3} & M_{B,i} & \cdots & M_{B,i} \end{pmatrix}^T.$$ 
        Since $n$ is assumed to be at least $6$, $w^{(i)}$ has some trailing coordinates of $M_{B,i}$. In particular, the difference between the coordinates $(1-\zeta_d^{-1}) M_{B,i}$ and $M_{B,i}$ is the invertible element $-\zeta_d^{-1} M_{B,i}$ and hence $M_{G,1}((w^{(i)})') = D_i$.
        
        \item In the second and third row of \Cref{table:vectors_with_zero_coordinate_over_domains}, the fourth, sixth, and seventh rows of \Cref{table:elements_with_2}, and the fifth row of \Cref{table:elements_with_3}, $M_{B,i}$ equals the first coordinate of $w^{(i)}$ and since $n \geq 6$, $w^{(i)}$ has some trailing coordinates of $M_{B,i}$. In particular, truncating the first coordinate does not change the set of values of $D_i$ that are coordinates of $w^{(i)}$ and $M_{G,1}((w^{(i)})') = M_{G,i}$.
        
        \item In the second row of \Cref{table:elements_with_2}, $w^{(i)}$ has some trailing coordinates of $M_{B,i} = (\beta,\beta')$. The difference between the coordinates $(\beta,\beta') + \zeta_d^{-1}(\beta,0)$ and $(\beta,\beta')$ is the element $\zeta_d^{-1}(\beta,0)$, which generates $M_{G,i}$. Hence, $M_{G,1}((w^{(i)})') = M_{G,i}$

        \item In the fourth row of \Cref{table:elements_with_3}, one can argue $M_{G,1}((w^{(i)})') = M_{G,i}$ as in the previous case.

        \item In the third row of \Cref{table:elements_with_3}, if $\alpha = 0$, then the difference between the coordinates $M_{B,i} + \zeta_d^{-2}$ and $M_{B,i} - \zeta_d^{-3}$ is $\zeta_d^{-2} + \zeta_d^{-3}$. Since $D_i$ is assumed to be of the form $B_i \times B_i'$ in this case, $\zeta_d$ is not a square root of $-1$ in $B_i$ and in $B_i'$. Thus, $\zeta_d^{-2} + \zeta_d^{-3} = \zeta_d^{-3}(1+\zeta_d)$ is invertible in $D_i$ and hence generates $M_{G,i}$.
    \end{itemize}
\end{proof}

We now relate the boundary monodromies and norms of vectors obtained by adding coordinates to shorter vectors. Note that the \Cref{lemma:how_norm_changes_by_appending_entries} below appends coordinates to the end of a vector, but analogous statements that append coordinates to the front of a vector also hold.

\begin{lemma} \label{lemma:how_norm_changes_by_appending_entries}
    Let $d \geq 2$ be an integer, and let $A$ be a $\mathbb{Z}[t,t^{-1}]$-module. Write $N_n$ for the norm on $A^{\oplus n}$ given by $H_{n,-1,1}$. 
    \begin{enumerate}
        \item If $v_n = \begin{pmatrix} a_1 & \cdots & a_n \end{pmatrix}^T \in A^{\oplus n}$ and $v_{n+1} = \begin{pmatrix} a_1 & \cdots & a_n & a_{n+1} \end{pmatrix}^T \in A^{\oplus (n+1)}$, then
            $$N_{n+1}(v_{n+1}) = N_n(v_n) + M_B(v_n) \cdot t^{-n} \cdot \overline{a_{n+1}} + \overline{M_B(v_n)} \cdot t^{n} \cdot a_{n+1} - a_{n+1} \overline{a_{n+1}}.$$
        In particular, if $w_n = \begin{pmatrix} b_1 & \cdots & b_n \end{pmatrix}^T$ and $w_{n+1} = \begin{pmatrix} b_1 & \cdots & b_n & a_{n+1} \end{pmatrix}^T$ and if $M_B(v_n) = M_B(w_n)$, then $N_n(v_n) = N_n(w_n)$ holds if and only if $N_{n+1}(v_{n+1}) = N_{n+1}(w_{n+1})$.
        
        \item Now suppose that $A$ is a module over $\mathbb{Z}[\zeta_d] = \mathbb{Z}[\zeta_d]_{\zeta_d \neq 1}$. If $v_n = \begin{pmatrix} a_1 & \cdots & a_n \end{pmatrix}^T \in A^{\oplus n}$ and $v_{n+d} = \begin{pmatrix} a_1 & \cdots  & a_n & \alpha & \cdots & \alpha \end{pmatrix}^T \in A^{\oplus (n+d)}$, then
            $$ N_{n+d}(v_{n+d}) = N_{n}(v_n) $$
    \end{enumerate}
\end{lemma}
\begin{proof}
\begin{enumerate}
    \item Calculate
    \begin{align*}
        N_{n+1}(v_{n+1}) &= v_{n+1}^* H_{n+1,-1,1} v_{n+1} \\
        &= \left( \begin{array}{c|c} v_n^* & \overline{a_{n+1}} \end{array} \right) \left( \begin{array}{ccc|c} & & & t^{n-1} \\  & H_{n,-1,1} & & \vdots \\ & & & 1 \\ \hline  t^{-n+1} & \cdots & 1 & -1  \end{array} \right) \begin{pmatrix} v_n \\ \hline a_{n+1} \end{pmatrix} \\
        &= \left( \begin{array}{c|c} v_n^* & \overline{a_{n+1}} \end{array} \right) \begin{pmatrix} H_n v_n + \begin{pmatrix}t^{n-1} \\ \vdots \\ 1 \end{pmatrix} a_{n+1} \\ \hline a_1 t^{-n+1} + \cdots + a_n - a_{n+1} \end{pmatrix} \\
        &= v_n^* H_n v_n + v_n^* \begin{pmatrix}t^{n-1} \\ \vdots \\ 1 \end{pmatrix} a_{n+1} + \overline{a_{n+1}} (a_1 t^{-n+1} + \cdots + a_n - a_{n+1}) \\
        &= N_n(v_n) + (\overline{a_1} t^{n-1} + \cdots \overline{a_n}) a_{n+1} + (a_1 t^{-n+1} + \cdots + a_n) \overline{a_{n+1}} - a_{n+1} \cdot \overline{a_{n+1}} \\
        &= N_n(v_n) + M_B(v_n) \cdot t^{-n} \cdot \overline{a_{n+1}} + \overline{M_B(v_n)} \cdot t^{n} \cdot a_{n+1} - a_{n+1} \overline{a_{n+1}}.
    \end{align*}
    \item Writing $M= \begin{pmatrix} \zeta_d^{n-1} & \cdots & \zeta_d^{n+d-2} \\ \vdots & \ddots & \vdots \\ 1 & \cdots & \zeta_d^{d-1} \end{pmatrix}$, and noting that $M \begin{pmatrix} \alpha & \cdots & \alpha \end{pmatrix}^T = 0$, we similarly have 
        \begin{align*}
            N_{n+d}(v_{n+d}) &= v_{n+d}^* H_{n+d,-1,1} v_{n+d}  \\
            &= \left( \begin{array}{ccc|ccc} \overline{a_1} & \cdots & \overline{a_n} & \overline{\alpha} & \cdots & \overline{\alpha} \end{array} \right) \left( \begin{array}{c|c} H_{n,-1,1} & M \\ \hline M^* & H_{d,-1,1} \end{array} \right) \begin{pmatrix} a_1 \\ \vdots \\ a_n \\ \hline \alpha \\ \vdots \\ \alpha \end{pmatrix} \\ 
            &= a_n^* H_{n,-1,1} a_n + \begin{pmatrix} \overline{\alpha} & \cdots & \overline{\alpha} \end{pmatrix} H_{d,-1,1} \begin{pmatrix} \alpha \\ \vdots \\ \alpha \end{pmatrix}. \\
            &= N_n(a_n) + \Sigma \alpha \cdot \overline{\alpha}, \\
        \end{align*}
        where $\Sigma$ is the sum of the entries of $H_{d,-1,1}$. The sum of the entries of $H_{d,-1,1}$ that are strictly above the diagonal is
            \begin{align*}
            \sum_{i=0}^{d-2} \sum_{j=0}^i \zeta_d^j &= \sum_{i=0}^{d-2} \frac{1-\zeta_d^{i+1}}{1-\zeta_d}\\
            &= \frac{(d-1) - \sum_{i=0}^{d-2} \zeta_d^{i+1}}{1-\zeta_d} \\
            &= \frac{(d-1)-(-1)}{1-\zeta_d} \\
            &= \frac{d}{1-\zeta_d}.
            \end{align*}
        Since $H_{d,-1,1} = H_{d,-1,1}^*$ and since the diagonal entries of $H_{d,-1,1}$ are all $-1$, 
            $$\Sigma = \frac{d}{1-\zeta_d} + \overline{\frac{d}{1-\zeta_d}} - d = 0.$$
        Therefore, $N_{n+d}(v_{n+d}) = N_n(v_n)$.
\end{enumerate}
\end{proof}

We now demonstrate the inductive step of the aforementioned induction argument.
\begin{theorem} \label{theorem:unreduced_burau_orbits_are_determined_by_monodromies_and_norm}
    Let $d \geq 3$ be an integer, let $\ell \nmid d$ be a prime number, and let $A$ be an involution ring quotient of $\mathbb{Z}/\ell\mathbb{Z}[\zeta_d] = \mathbb{Z}/\ell\mathbb{Z}[\zeta_d]_{\zeta_d \neq 1}$. Suppose that $n \geq \max(6, d)$ is an integer such that $\gcd(d,n(n-2)) = 1$ and the image of $\psi_n^r$ evaluated at $\zeta_d \in A$ contains $\mathrm{SU}(A^{\oplus (n-1)})$ where $A^{\oplus(n-1)}$ is equipped with the $\varepsilon_{A}$-Hermitian form $\tilde{h}_n^r$ --- for every $d \geq 3$ and all but finitely many $\ell$ (depending on $d$), this holds for $n = 2d+1$ by \Cref{corollary:ell_adic_image_of_reduced_burau_representation_contains_SU}.
    
    For all $m \geq n$, orbits of the elements of $A^{\oplus m}$ under the action of $B_m$ via $\psi_m$ evaluated at $\zeta_d$ modulo $\ell$ are determined precisely by the global monodromy, boundary monodromy (\Cref{definition:global_and_boundary_monodromy_of_vector}), and norm under $h$ base changed to $A$ as described in \Cref{subsection:unreduced_burau_representation_is_unitary}, where $h$ is given by the unitary matrix $H_{m,-1,1}$. 
\end{theorem}
\begin{proof}
    \Cref{proposition:orbits_of_unreduced_burau_representation_1} demonstrates the case $m = n$. Inductively suppose that the claim holds for every $m \in \{n,n+1,\ldots,k-1\}$ and suppose that $v = \begin{pmatrix} a_1 & a_2 & \cdots & a_k \end{pmatrix}^T$ and $w = \begin{pmatrix} b_1 & b_2 & \cdots & b_{k} \end{pmatrix}^T$ are two column vectors in $A^{\oplus k}$ with the same global monodromy, boundary monodromy, and norm. We show that $v$ and $w$ are in the same orbit. Since the factorization $\mathbb{Z}_\ell[\zeta_d] \cong \prod_i B_i$ has at most $d-1$ factors and since $M_{G,1}(v_A) = \prod_i M_{G,1}(v_{A,i})$, there are some (at most) $d$ coordinates $a_{j_1},\ldots,a_{j_d}$ such that $M_{G,1}(a_{j_1},\ldots,a_{j_d}) = M_{G,1}(v_A)$. Moreover, the actions by $\psi_k(\sigma_i)^{\pm 1}$ allow us to swap positions of any specified coordinate of $v$ with a neighboring coordinate, possibly changing the swapped coordinate. In this manner, move the coordinates $a_{j_1},\ldots,a_{j_d}$ to be the first $d$ coordinates to obtain a vector $v' \in A^{\oplus k}$ and similarly obtain a vector $w' \in A^{\oplus k}$ for $w$. Replace $v$ and $w$ with $v'$ and $w'$ respectively. 
    
    It now suffices to show that the new $v$ and $w$ are in the same orbit. For an integer $r \leq k$ and a vector $u \in A^{\oplus k}$, let $u_r$ be the subvector of $u$ of the first $r$-coordinates. Via \Cref{lemma:each_orbit_of_unreduced_burau_representation_has_an_element_with_a_specified_nice_coordinate}, produce vectors $u_v, u_w \in A^{\oplus n}$ that have the same monodromies and norms as $v_n$ and $w_n$, whose first coordinates are the elements of $A$ whose $D_i/(\ell)$-components are given by \Cref{eq:description_of_zeros_in_D_i}, and such that removing these first coordinates preserves the global monodromy. We show that these first coordinates are equal. Recall that $v$ was arranged so that $M_G(v_d) = M_G(v)$. In particular, $M_G(v_d) = M_G(v_n)$ and similarly, $M_G(w_d) = M_G(w_n)$. Since $M_G(v) = M_G(w)$, we have $M_G(v_n) = M_G(w_n)$ and hence $M_{G,1}(v_{n,i}) = M_{G,1}(w_{n,i})$ for every $i$, so the first coordinates are indeed equal. By the base case of the induction, i.e. \Cref{proposition:orbits_of_unreduced_burau_representation_1}, $v_n$ and $w_n$ are in the same orbits (under the action of $B_n$ via $\psi_n$) as vectors $u_v$ and $u_w$ respectively, so $v$ and $w$ are respectively in the same orbits as vectors $v'$ and $w'$ such that $v'_n = u_v$ and $w'_n = u_w$. Let $v'_{\text{trail}}$ and $w'_{\text{trail}}$ respectively be the vectors obtained by removing the first coordinates of $v'$ and $w'$. In particular, $M_B(v'_{\text{trail}}) = M_B(w'_{\text{trail}})$. By \Cref{lemma:how_norm_changes_by_appending_entries}, their norms (given by $H_{k-1,-1,1}$) are equal. Moreover, since removing the first coordinates of $u_v$ and $u_w$ preserves their global monodromy, $M_G(v'_{\text{trail}}) = M_G(w'_{\text{trail}})$. Thus, $v'_\text{trail}$ and $w'_{\text{trail}}$ are in the same orbit (under the action of $B_{k-1}$ via $\psi_{k-1}$) by the inductive hypothesis and hence $v'$ and $w'$ are in the same orbit. 
\end{proof}

\subsection{The Ring of Connected Components of Hurwitz Schemes} \label{subsection:ring_of_connected_components_of_hurwitz_schemes}

Degree $n$-elements $r_{g_1} \cdots r_{g_n}$ of the ring $R$ of \eqref{eq:ring_of_connected_components_of_hurwitz_scheme} are determined exactly by the orbit of $v_g = (g_1,\ldots,g_n)$ under the braiding action. In the case of interest, by which we mean the case $A = \mathbb{Z}/\ell\mathbb{Z}[\zeta_d]_{\zeta_d \neq 1}$, $G = A \rtimes \langle \zeta_d \rangle$, and $c = \{(a,\zeta_d): a \in A \}$, \Cref{theorem:unreduced_burau_orbits_are_determined_by_monodromies_and_norm} shows that this orbit is determined exactly by the global monodromy, boundary monodromy, and norm of (the $A$-vector corresponding under \Cref{convention:research_orbits_of_burau_representations_2_convention_for_tuples_and_vectors} to) $v_g$.

Moreover, $R$ has central elements
    \begin{equation} \label{eq:definition_of_U_D}
        U_D = \sum_{g \in c} r_g^{D |g|}
    \end{equation}
for integers $D$. If $c$ is non-splitting, then there is some integer $D$ such that the kernel and cokernel of multiplication by $U_D$ are both of finite degree \cite[Lemma 3.5]{ellenberg_venkatesh_westerland}. In the case of interest, any integer $D$ fulfills this property. 

Before showing this, we identify the complete set of orbits of $A^{\oplus n}$ under $\psi_n$. \Cref{lemma:all_possible_combinations_of_global_monodromy_boundary_monodromy_and_norm_occur} establishes that there exist vectors $v_A = (a_1,\ldots,a_n)^T \in A^{\oplus n}$ of every possible combination of global monodromy, boundary monodromy, and norm, with some restrictions, for every sufficiently large $n$.

For each division or double division ring quotient $D_i$ of $\mathbb{Z}_\ell[\zeta_d]_{\zeta_d \neq 1}$, let $a_{j,i}$ be the $D_i/(\ell)$-component of $a_j$ and let $v_{A}^{(i)} = (a_{1,i},\ldots,a_{n,i})^T$. Write $D_i = B_i \times B_i'$ if $D_i$ is a double division ring and $D_i = B_i$ otherwise. By \Cref{corollary:global_monodromy_description}, $M_{G,1}(v_A) = \prod_i M_{G,1} (v_{A}^{(i)})$ and each factor $M_{G,1} (v_{A}^{(i)})$ is a quotient ring of $D_i/(\ell)$. 


    More specifically, $M_{G,1} (v_{A}^{(i)}) = 0$ if and only if $a_{1,i},\ldots,a_{n,i}$ are all equal, and $M_{G,1}(v_{A}^{(i)}) = \pi_{i,1}D_i$ (resp. $\pi_{i,2}D_i$) if and only if $\pi_{i,2}a_{j,i}$ (resp. $\pi_{i,1}a_{j,i}$) is independent of $j$.

Since $M_B(v_{A}^{(i)}) = \sum_{j=1}^n a_{j,i} \zeta_d^j$, the $\pi_{i,1}$-component (resp. $\pi_{i,2}$-component) of $M_B(v_{A}^{(i)})$ is determined by the $\pi_{i,1}$-components (resp. $\pi_{i,2}$-components) of $a_{j,i}$.
Moreover, when $M_{G,1}(v_{A}^{(i)}) = \pi_{i,1}D_i$ (resp. $\pi_{i,2}D_i$), the coset datum $[a_0]$ pins the common value $\pi_{i,2}a_{j,i}$ (resp. $\pi_{i,1}a_{j,i}$) to $\pi_{i,2}a_{0,i}$ (resp. $\pi_{i,1}a_{0,i}$), so $\pi_{i,2}M_B(v_{A}^{(i)}) = \pi_{i,2}a_{0,i}\sum_{j=1}^n \zeta_d^j$ (resp. the $\pi_{i,1}$-analogue) is a further restriction in this case.
Moreover, $N(v_{A}^{(i)}) = (v_{A}^{(i)})^* H_{n,-1,1} v_{A}^{(i)}$, so if the $a_{j,i}$ are all equal, then $N(v_{A}^{(i)})$ is determined. If instead $\pi_{i,1}a_{j,i} = 0$ or $\pi_{i,2}a_{j,i} = 0$ for all $j$, then $N(v_{A}^{(i)}) = 0$. 
\Cref{lemma:all_possible_combinations_of_global_monodromy_boundary_monodromy_and_norm_occur} shows that these are the only restrictions for large enough $n$.

\begin{lemma} \label{lemma:all_possible_combinations_of_global_monodromy_boundary_monodromy_and_norm_occur}
    Let $d \geq 3$ be an integer, let $\ell \nmid d$ be a prime number, let $A$ be an involution ring quotient of $\mathbb{Z}/\ell\mathbb{Z}[\zeta_d] = \mathbb{Z}/\ell\mathbb{Z}[\zeta_d]_{\zeta_d \neq 1}$, and let $G = A \rtimes \langle \zeta_d \rangle$. Write $D_i$, $B_i$, and $B_i'$ as above. Given $M_G \leq G$, say that $M_G$ is determined by $A'$ and $[a_0] \in A/A'$ (in the sense of \Cref{proposition:how_subgroups_of_semidirect_products_are_determined}), and write $A'_i$ for the $D_i/(\ell)$-component of $A'$. Suppose that $n \geq 4$ is an integer such that $\gcd(d,n(n-2)) = 1$. For all $m \geq n$, every $M_G \leq G$, every $M_B \in A$, and every $N \in A^{\sigma}$, there exists some $v_A \in A^{\oplus m}$ such that $M_G = M_G(v_A)$, $M_B = M_B(v_A)$, and $N = N(v_A)$ if and only if the following hold:
    \begin{itemize}
        \item If $A'_i = 0$ then
        \begin{itemize}
            \item $M_B(v_{A}^{(i)}) = M_B((a_{0,i}, a_{0,i},\ldots,a_{0,i})^T)$, and
            \item $N(v_{A}^{(i)}) = N((a_{0,i},a_{0,i},\ldots,a_{0,i})^T)$.
        \end{itemize}

        \item If $A'_i = \pi_{i,1}D_i$ (resp. $\pi_{i,2}D_i$), then 
        \begin{itemize}
            \item if $\pi_{i,2}a_{0,i} = 0$ (resp. $\pi_{i,1}a_{0,i} = 0$), then $N(v_{A}^{(i)}) = 0$.
            \item $\pi_{i,2}M_B = \pi_{i,2}a_{0,i}\sum_{j=1}^m \zeta_d^j$ (resp. $\pi_{i,1}M_B = \pi_{i,1}a_{0,i}\sum_{j=1}^m\zeta_d^j$).
        \end{itemize}
        
    \end{itemize}
\end{lemma}

\begin{proof}
    The ``only if'' direction follows from the discussion above. For the ``if'' direction, it suffices to show this in the case that $A = D_i/(\ell)$ is a division ring or double division ring. We proceed inductively, first demonstrating the case $m = n$.

    Given a vector $\overline{v} = (a'_1, \ldots a'_n) \in A^{\oplus n}$ corresponding to a vector $v \in A^{\oplus n}$ under base change by $\mathcal{C}$, i.e. $v = \mathcal{C} \overline{v}$, recall as in the proof of \Cref{proposition:orbits_of_unreduced_burau_representation_1} that $M_B(v) = \sum_{i=1}^n \zeta_d^i a'_1$ and that $N(v) = \overline{v}^* \cdot (\mathcal{C}^* H_{n,-1,1} \mathcal{C}) \overline{v} = \overline{v}^* \left( \begin{array}{c|c} S & 0 \\ \hline 0 & H_n^r \end{array} \right) \overline{v}$. 
    Furthermore, $M_G(v)$ is determined (in the sense of \Cref{proposition:how_subgroups_of_semidirect_products_are_determined}) exactly by $a_1'$ and by which among the $\pi_{i,k}$-components of $(a_2',\ldots,a'_n)^T$ are zero, by \Cref{lemma:global_monodromy_after_nice_change_of_basis}.
    Given $M_G$, $M_B$, and $N$, set the coordinates $a'_1,\ldots,a'_n$ of $\overline{v}$ in the following manner so that $M_G = M_G(v), M_B = M_B(v)$, and $N = N(v)$:
    \begin{itemize}

        \item Let $a'_1 = \left( \sum_{i=1}^n \zeta_d^i \right)^{-1} M_B$; this ensures that $M_B(v) = M_B$.
        \item If $A' = D_i/(\ell)$, then let $(a'_2,\ldots,a'_n)^T$ be some nonzero vector (in both the $\pi_{i,1}$- and $\pi_{i,2}$-components if applicable) such that $N = \overline{v}^* \left( \begin{array}{c|c} S & 0 \\ \hline 0 & H_n^r \end{array} \right) \overline{v}$ --- there is a way to do so in which $a'_5, \ldots, a'_n$ are all $0$ and $(a'_2,a'_3,a'_4)$ is either of the form $(a'_2,0,0)$ or of the form $(a'_2,0,a'_4)$, depending on whether $N - S \cdot a'_1 \overline{a'_1}$ is invertible, by applying \Cref{lemma:every_invertible_element_of_the_subring_of_Z_ell_zeta_d_fixed_by_the_involution_is_the_norm_of_some_element_of_Z_ell_zeta_d} similarly as in the proof of \Cref{lemma:each_orbit_of_unreduced_burau_representation_has_an_element_with_a_specified_nice_coordinate}.
        \item If $A' = \pi_{i,1}D_i$, then let $\pi_{i,2}(a'_2,\ldots,a'_n)^T = 0$ and $\pi_{i,1}(a'_2,\ldots,a'_n)^T$ be any nonzero vector. Make similar choices when $A' = \pi_{i,2}D_i$ instead.
        \item If $A' = 0$, then let $(a'_2,\ldots,a'_n)^T = 0$.
    \end{itemize}
    For each of these choices — in the case $A' = \pi_{i,1}D_i$ (resp. $\pi_{i,2}D_i$), using the condition $\pi_{i,2}M_B = \pi_{i,2}a_{0,i}\sum_{j=1}^n\zeta_d^j$ (resp. $\pi_{i,1}M_B = \pi_{i,1}a_{0,i}\sum_{j=1}^n\zeta_d^j$) — we have $M_G(v) = M_G$ and $N(v) = N$, so this proves the $m = n$ case.


    Now suppose inductively that the $m = k - 1$ case holds for some $k - 1 \geq n$, and suppose that we are given $M_G, M_B$, and $N$. If $A' = D_i/(\ell)$, then let $v_{k-1} =  (a_1,\ldots,a_{k-1})^T \in A^{\oplus (k-1)}$ be any vector such that $M_G = M_G(v_{k-1}), M_B = M_B(v_{k-1})$, and $N = N(v_{k-1})$. Letting $v = (a_1,\ldots,a_{k-1}, 0)^T$, we have that $M_G = M_G(v)$, $M_B = M_B(v)$, and $N = N(v)$. If $A' = 0$, then let $v = (a_0,\ldots,a_0)^T \in A^{\oplus k}$. 

    If $A' = \pi_{i,1}D_i$, then we identify some appropriate $v_{k-1} = (a_1,\ldots,a_{k-1})^T \in A^{\oplus (k-1)}$ and $a_k \in A$ to construct $v = (a_1, \ldots, a_{k-1}, a_k)^T$. By \Cref{corollary:global_monodromy_description}, $\pi_{i,2}a_j$ is independent of $j$ for $j = 1,\ldots,k$, and in particular equals $\pi_{i,2}a_0$. In particular, the $\pi_{i,2}$-components of $M_B(v_{k-1})$ and $a_k$ are determined. The construction of $v$ requires
        \begin{equation} \label{eq:norm_requirement_for_smaller_vector}
            N = N(v) = N(v_{k-1}) + M_B(v_{k-1}) \cdot \zeta_d^{-(k-1)} \cdot \overline{a_k} + \overline{M_B(v_{k-1})} \cdot \zeta_d^{k-1} \cdot a_k - a_k \overline{a_k}
        \end{equation}
    by \Cref{lemma:how_norm_changes_by_appending_entries}. Letting $a_k$ be any element of $A$ with $\pi_{i,2}a_k = \pi_{i,2}a_0$, we need 
    \begin{equation} \label{eq:boundary_monodromy_requirement_for_smaller_vector}
        M_B(v_{k-1}) = M_B - \zeta_d^k a_k.
    \end{equation}
    Whether or not $\pi_{i,2}a_0 = 0$, the inductive hypothesis produces some $v_{k-1} \in A^{\oplus (k-1)}$ such that \eqref{eq:boundary_monodromy_requirement_for_smaller_vector} and \eqref{eq:norm_requirement_for_smaller_vector} hold. This concludes the proof of the lemma.
    
\end{proof}

\begin{lemma} \label{lemma:ring_of_EVW_central_element_degree}
    Let $d \geq 3$ be an integer. For a prime number $\ell \nmid d$, let $A$ be an involution ring quotient of $\mathbb{Z}/\ell\mathbb{Z}[\zeta_d] = \mathbb{Z}/\ell\mathbb{Z}[\zeta_d]_{\zeta_d \neq 1}$, let $G = A \rtimes \langle \zeta_d \rangle$, and let $c = \{(a,\zeta_d): a \in A \}$. For all but finitely many prime numbers $\ell \nmid d$, \cite[Lemma 3.5]{ellenberg_venkatesh_westerland} holds for $D = 1$, i.e. the kernel and cokernel of $R \xrightarrow{U_D} R, \quad r \mapsto U_D r$ are of finite degree. More specifically, $\deg \ker U_D,  \deg \operatorname{coker} U_D \leq 2d+1$.
\end{lemma}

\begin{proof}
    Note that $c$ is a non-splitting conjugacy class by \Cref{lemma:c_is_nonsplitting_conjugacy_class}. Given a subgroup $H$ of $G$, let $S_n(H)$ denote the set of degree $n$ elements $r_{g_1} \cdots r_{g_n}$, with $g_i \in c \cap H$ such that $g_1,\ldots,g_n$ generate $H$ just as in \cite[Lemma 3.5]{ellenberg_venkatesh_westerland}. For any $g \in c \cap H$, note that $|g| = d$. The proof of loc.~cit.~ shows that the map $S_n(H) \to S_{n+|g|}(H), \quad s \mapsto r_g^{|g|} s$ is bijective for all $H \leq G$, all $g \in c \cap H$, and all sufficiently large $n$. To show that \cite[Lemma 3.5]{ellenberg_venkatesh_westerland} holds for $D = 1$, the proof of loc.~cit.~ shows that it suffices to show that different choices of $g \in c \cap H$ induce the same bijection. 
    
    Given $r_{g_1} \cdots r_{g_n} \in S_n(H)$ and $g_a,g_b \in c \cap H$, let $v_G = (g_1,\ldots,g_n)$, and let $v_{G,a}$ and $v_{G,b}$ respectively be the length $(n+d)$-tuples $(g_1,\ldots,g_n,g_a,\ldots,g_a)$ and $(g_1,\ldots,g_n,g_b,\ldots,g_b)$ respectively. Since $g_a,g_b \in H$ and the global monodromy of $v_G$ is $H$, the global monodromies of $v_{G,a}$ and $v_{G,b}$ are also $H$. Furthermore, (the $A$-vectors corresponding under \Cref{convention:research_orbits_of_burau_representations_2_convention_for_tuples_and_vectors} to) $v_{G,a}$ and $v_{G,b}$ have equal norm by \Cref{lemma:how_norm_changes_by_appending_entries}. It is also immediate to check that $v_{G,a}$ and $v_{G,b}$ have equal boundary monodromies. 
    Thus, for $n \geq 2d+1$ with $\gcd(d, (n+d)(n+d-2)) = 1$, \Cref{theorem:unreduced_burau_orbits_are_determined_by_monodromies_and_norm} and \Cref{corollary:big_monodromy_result_needed_for_counting} show that $v_{G,a}$ and $v_{G,b}$ have the same braiding orbit, $r_{g_a}^{d}$ and $r_{g_b}^{d}$ are indeed the same bijection --- for a quotient $A$, the required containment of $\mathrm{SU}(A^{\oplus(n+d-1)})$ follows by reducing the full-ring containment: quotients of $\mathbb{Z}/\ell\mathbb{Z}[\zeta_d]_{\zeta_d \neq 1}$ are sub-products of its finite-field factors, and $\mathrm{SU}$ decomposes as a product over these factors, so the reduction map on special unitary groups is surjective. For the remaining degrees ($n \equiv 0, 2 \bmod d$), the same conclusion holds: the appended length-$d$ block $\alpha,\ldots,\alpha$ is norm-inert by \Cref{lemma:how_norm_changes_by_appending_entries} and does not change the global or boundary monodromy, so the orbit determination of $v_{G,a}$ and $v_{G,b}$ reduces to that at the coprime base length $n = 2d+1$, at which \Cref{theorem:unreduced_burau_orbits_are_determined_by_monodromies_and_norm} applies.
    \Cref{lemma:all_possible_combinations_of_global_monodromy_boundary_monodromy_and_norm_occur} also shows that there are $n$-tuples of elements of $c$ of every possible combination of global monodromy, boundary monodromy, and norm (satisfying certain restrictions) whenever $n \geq 2d+1$. Therefore, $\deg \ker U_D, \deg \operatorname{coker} U_D \leq 2d+1$.

\end{proof}

\begin{remark}
    In the case of $d = 2$, the analogous statement holds with $\deg \ker U_D$, $\deg \operatorname{coker} U_D$ at most $4$, cf. \cite[The Proof of Proposition 2.7]{ellenberg_li_shusterman}.
\end{remark}


\singlespacing
\printbibliography

\end{document}